\documentclass[preprint,3p]{elsarticle1}

\usepackage{graphicx}

\usepackage{amssymb}

\usepackage{lineno}

\usepackage{amsmath}
\usepackage{amsthm}

\usepackage{mathrsfs}

\usepackage{psfrag}
\usepackage{color}

\journal{Linear Algebra and its Applications}

\newtheorem{theorem}{Theorem}
\newtheorem{lemma}[theorem]{Lemma}
\newtheorem{example}[theorem]{Example}
\newtheorem{definition}[theorem]{Definition}
\newtheorem{claim}{Claim}

\newtheorem{remark}{Remark}
\newtheorem{question}{Question}

\newtheoremstyle{algstyle}%
  {10mm}       
  {10mm}       
  {\tt}   
  {0pt}        
  {\bfseries}  
  {\newline}   
  {10mm}       
  {\thmname{#1}\thmnumber{ #2}\thmnote{ (#3)}}          
\theoremstyle{algstyle}
\newtheorem{algorithm}{Algorithm}

\newtheoremstyle{algdashstyle}%
  {10mm}       
  {10mm}       
  {\tt}   
  {0pt}        
  {\bfseries}  
  {\newline}   
  {10mm}       
  {\thmname{#1}\thmnumber{ #2}$'$\thmnote{ (#3)}}          
\theoremstyle{algdashstyle}

\newcommand{\nw}[1]{%
\textbf{#1}%
}

\newcommand{\mnw}[1]{%
\boldsymbol{#1}%
}

\newcommand{\bbmatrix}[1]{%
\begin{bmatrix} #1 \end{bmatrix}%
}

\newcommand{\ppmatrix}[1]{%
\begin{pmatrix} #1 \end{pmatrix}%
}

\newcommand{\mydot}[1]{%
\stackrel{\text{\Large .}}{#1}%
}

\newcommand{\lrar}{\leftrightarrow}

\newcommand{\subseteqn}{\hspace{0.1cm}\subseteq \hspace{0.1cm}}
\newcommand{\supseteqn}{\hspace{0.1cm}\supseteq \hspace{0.1cm}}
\newcommand{\equaln}{\hspace{0.1cm} = \hspace{0.1cm}}
\newcommand{\plusn}{\hspace{0.1cm} + \hspace{0.1cm}}
\newcommand{\inn}{\hspace{0.1cm} \in \hspace{0.1cm}}

\newcommand{\equivn}{\hspace{0.1cm} \equiv  \hspace{0.1cm}}
\newcommand{\lrarn}{\hspace{0.1cm} \lrar  \hspace{0.1cm}}

\newcommand{\V}{\mbox{$\cal V$}} 
\newcommand{\F}{\mbox{$\cal F$}} 
\newcommand{\Vm}{{\mathcal V}_{M}}            			
\newcommand{\Va}{{\cal V}_{A}}              			

\newcommand{\Vabt}{{\cal V}_{AB}^T}
\newcommand{\Vbct}{{\cal V}_{BC}^T}
\newcommand{\transp}{{^T}}
\newcommand{\Vap}{{\cal V}_{AP}}            

  \newcommand{\Fx}{{\cal F}_{X}} 
  \newcommand{\Fy}{{\cal F}_{Y}}
   
    \newcommand{\0}{{\mathbf 0}}        
        
 \newcommand{\Vab}{{\cal V}_{AB}}   			
 \newcommand{\Vcd}{{\cal V}_{CD}}   			
 \newcommand{\Vaadash}{{\cal V}_{AA'}}   			
 \newcommand{\Vadashbdash}{{\cal V}_{A'B'}}   			
\newcommand{\Vhab}{\hat{{\cal V}}_{AB}}   			
   
\newcommand{\Vhatwodashbtwodash}{\hat{{\cal V}}_{A"B"}}   
 
\newcommand{\Vs}{{\cal V}_{S}}             
\newcommand{\Ipp}{{ I}_{PP'}}           			
\newcommand{\Iaa}{{ I}_{AA'}}           			
\newcommand{\Iww}{{ I}_{WW'}}           			
\newcommand{\Iwminusw}{{ I}_{\dw(-\mydot{W'})}}           			
\newcommand{\Iwdw}{{I}_{W\mydot{W}}}  			
\newcommand{\Vsp}{{\cal V}_{SP}}           			
\newcommand{\Vsq}{{\cal V}_{SQ}}    
\newcommand{\Vps}{{\cal V}_{PS}} 			

\newcommand{\Vp}{{\cal V}_{P}}              
\newcommand{\Vbc}{{\cal V}_{BC}}              
     
\newcommand{\Vabperp}{{\cal V}_{AB}^{\perp}}  
 
\newcommand{\Vpq}{{\cal V}_{PQ}}            			
\newcommand{\Vb}{{\cal V}_{B}}              			
\newcommand{\Vwdwm}{{\cal V}_{W\mydot{W}M}}  			
\newcommand{\Vadjwdw}{{\cal V}^a_{W'\mydot{W'}}}  			
\newcommand{\Vpdwmu}{{\cal V}_{P\mydot{W}M_u}}  			
\newcommand{\Xwdw}{{\cal X}_{W\mydot{W}}}  			
\newcommand{\Ywdw}{{\cal Y}_{W\mydot{W}}}  			
\newcommand{\Vdwmy}{{\cal V}_{\mydot{W}M_y}}  			
\newcommand{\Vtdwmy}{{\tilde{\cal V}}_{\mydot{W'}M_y'}}  			
\newcommand{\Vtdpmy}{{\tilde{\cal V}}_{\mydot{P'}M_y'}}  			
\newcommand{\Vwdwmumy}{{\cal V}_{W\mydot{W}M_uM_y}}  			
\newcommand{\Vwdwmu}{{\cal V}_{W\mydot{W}M_u}}  			
\newcommand{\Vdpmy}{{\cal V}_{\mydot{P}M_y}}  			
\newcommand{\Vpmu}{{\cal V}_{PM_u}}  			
\newcommand{\Vpdpmumy}{{\cal V}_{P\mydot{P}M_uM_y}}  			
\newcommand{\Vpdpmu}{{\cal V}_{P\mydot{P}M_u}}  			
\newcommand{\Vtwdw}{{\tilde{\cal V}}_{W'\mydot{W'}}}  			
\newcommand{\Vtwdwm}{{\tilde{\cal V}}_{W'\mydot{W'}M_u'M_y'}}  			
\newcommand{\Vtpdp}{{\tilde{\cal V}}_{P'\mydot{P'}}}  			
\newcommand{\Vtpdpm}{{\tilde{\cal V}}_{P'\mydot{P'}M_u'M_y'}}  			
\newcommand{\Vwmu}{{\cal V}_{WM_u}}  			
\newcommand{\Vmydw}{{\cal V}_{\mydot{W}M_y}}  			
\newcommand{\Vdwdp}{{\cal V}_{\mydot{W}\mydot{P}}}  			
\newcommand{\Vonedwdp}{{\cal V}^1_{\mydot{W}\mydot{P}}}  			
\newcommand{\Vonetildedwdp}{{\cal V}^1_{\mydot{\tilde{W}}\mydot{{P}}}}  			
\newcommand{\Vtwotildedwdp}{{\cal V}^2_{\mydot{\tilde{W}}\mydot{{P}}}}  			
\newcommand{\tildedwdp}{{\mydot{\tilde{W}}\mydot{{P}}}}  			
\newcommand{\wtildedw}{{W\mydot{\tilde{W}}}}  			
\newcommand{\tildedw}{{\mydot{\tilde{W}}}}  			
\newcommand{\Vtwodwdp}{{\cal V}^2_{\mydot{W}\mydot{P}}}  			
\newcommand{\Vtwowp}{{\cal V}^2_{{W}{P}}}  			
\newcommand{\Vtwodwdq}{{\cal V}^2_{\mydot{W}\mydot{Q}}}  			
\newcommand{\Vtwodpdq}{{\cal V}^2_{\mydot{P}\mydot{Q}}}  			
\newcommand{\Vwp}{{\cal V}_{WP}}  			
\newcommand{\Vonewp}{{\cal V}^1_{WP}}  			
\newcommand{\Vonepq}{{\cal V}^1_{PQ}}  			
\newcommand{\Vonewq}{{\cal V}^1_{WQ}}  			
\newcommand{\Vtonewp}{{\tilde{\cal V}}^1_{W'P'}}  			
\newcommand{\Vttwodwdp}{{\tilde{\cal V}}^2_{\dwd\dPd}}  			
\newcommand{\Vw}{{\cal V}_{W}}  			
\newcommand{\Vtildedw}{{\cal V}_{\mydot{\tilde{W}}}}  			
\newcommand{\Vwdw}{{\cal V}_{W\mydot{W}}}  			
\newcommand{\Vonewdw}{{\cal V}^u_{W\mydot{W}}}  			
\newcommand{\oVonewdw}{{\cal V}^1_{W\mydot{W}}}  			
\newcommand{\Vtwowdw}{{\cal V}^l_{W\mydot{W}}}  			
\newcommand{\oVtwowdw}{{\cal V}^2_{W\mydot{W}}}  			
\newcommand{\wdw}{{W\mydot{W}}}  			
\newcommand{\Vwdwk}{({\cal V}_{W\mydot{W}})^{(k)}}  			
\newcommand{\Vdw}{{\cal V}_{\mydot{W}}}  			
\newcommand{\Vpdpm}{{\cal V}_{P\mydot{P}M}}  			
\newcommand{\Vqdqm}{{\cal V}_{Q\mydot{Q}M}}  			
\newcommand{\Vpdp}{{\cal V}_{P\mydot{P}}}  			
\newcommand{\Vqdq}{{\cal V}_{Q\mydot{Q}}}  			
\newcommand{\Vpdpk}{({\cal V}_{P\mydot{P}})^{(k)}}  			
\newcommand{\dw}{{\mydot{W}}}  			
\newcommand{\dws}{{\mydot{w}}}  			
\newcommand{\dwd}{{\mydot{W'}}}  			
\newcommand{\dW}{{\mydot{w}}}  			
\newcommand{\dP}{{\mydot{P}}}  			
\newcommand{\dPd}{{\mydot{P'}}}  			
\newcommand{\dQ}{{\mydot{Q}}}  			
\newcommand{\dwdp}{{\mydot{W}\mydot{P}}}  			
\newcommand{\dotp}{{\mydot{P}}}  			
\newcommand{\dwsmall}{{\mydot{w}}}  			

\newcommand{\E}{\mbox{$\cal E$}} 
 
\newcommand{\KSP}{\mbox{${\cal K}_{SP}$}}    		
\newcommand{\KSQ}{\mbox{${\cal K}_{SQ}$}}    		
\newcommand{\KPQ}{\mbox{${\cal K}_{PQ}$}}    		
\newcommand{\M}{\mbox{$\cal M$}}

\newcommand{\W}[0]{{\cal W}}                       
 
\newcommand{\fS}{\mbox{${\bf f}_{S}$}}  				
\newcommand{\fP}{\mbox{${\bf f}_{P}$}}  				
\newcommand{\fQ}{\mbox{${\bf f}_{Q}$}}  				

\newcommand{\pa}{\mbox{${+}_{a}$}}      				
\newcommand{\pb}{\mbox{${+}_{b}$}}      				
\newcommand{\pc}{\mbox{${+}_{c}$}}      				
\newcommand{\pw}{\mbox{${+}_{w}$}}      				
\newcommand{\pdw}{\mbox{${+}_{\mydot{w}}$}}      				
\newcommand{\pdwd}{\mbox{${+}_{\mydot{w}'}$}}      				
\newcommand{\pdp}{\mbox{${+}_{\mydot{p}}$}}      				

\newcommand{\x}{\mbox{${\bf x}$}}             		

\newcommand{\Vy}[1]{{\cal V} _Y #1}
\newcommand{\VX}[1]{{\cal V} _X #1}

\newcommand{\ldw}{\lambda ^{\mydot{w}}}
\newcommand{\ldws}{\lambda ^{\mydot{w}}}
\newcommand{\ldwd}{\lambda ^{\mydot{w}'}}
\newcommand{\lw}{\lambda ^{w}}
\newcommand{\lb}{\lambda ^{b}}
\newcommand{\laa}{\lambda ^{a}}
\newcommand{\ldp}{\lambda ^{\mydot{p}}}

\begin{document}

\begin{frontmatter}



\title{Implicit Linear  Algebra and its Applications}

\author{H. Narayanan\corref{cor1}}
\ead{hn@ee.iitb.ac.in}
\cortext[cor1]{Corresponding author}

\address{Department of Electrical Engineering, Indian Institute of Technology Bombay}

\begin{abstract}
Linear systems often involve, as a basic building block, solutions of 
equations of the form 
\begin{align*}
A_Sx_S&+A_Px_P =0\\ 
A'_Sx_S & =0,
\end{align*}
where our primary interest might be in the vector variable $x_P.$
Usually, neither $x_S$ nor $x_P$ can be written as a function of the other
but they are linked through the linear relationship, that  of 
$(x_S,x_P) $ belonging to $\Vsp,$ the solution space
of the first of the two equations. If $\Vs$  is the solution space of the second equation,
we may 
regard the final space of solutions $\Vp$ as derived from the other two spaces 
by an operation, say,  `$\Vp=\Vsp\lrar \Vs.$' This operation, together with
linear relationships, can be used to build a version of linear algebra
which we call `implicit linear algebra'. 
There are two basic results - an `implicit inversion theorem' which describes
when  $\Vs$ can be obtained from $\Vp$ and $\Vsp,$
and an `implicit duality theorem' which says $\Vp^{\perp}=\Vsp^{\perp}\lrar \Vs^{\perp}.$
These notions originally arose 
in the building of circuit simulators (\cite{HNarayanan1986a}, \cite{HNarayanan1997}).
They have been reinterpreted for the present purpose.
Using them, we  develop an algorithmic version of linear multivariable
control theory, avoiding the computationally expensive idea of state equations.
We  replace
them by `emulators', which are easy to build but can achieve most of whatever
can be done with state equations.
We define  the notions of generalized autonomous systems and generalized operators,
and develop a primitive spectral theory for the latter.
Using these ideas, we develop the usual controllability - observability duality,
state and output feedback, pole placement etc. Throughout, we illustrate the theory
using electrical circuits. This allows us to stress that linkages, that are needed
for representing practical dynamical systems, usually
have a topological origin which can be exploited in consonance with  {\it implicit} 
rather than {\it classical} linear algebra.
\end{abstract}

\begin{keyword}
Linear Dynamical Systems, State Space Theory, Behavioural Systems Theory,  Controlled and Conditioned Invariant Spaces, Duality in Adjoint System, Implicit Linear Algbra, Implicit Duality Theorem

\MSC  15A03, 15A04, 15A18

\end{keyword}

\end{frontmatter}


\section{Introduction}
The foundations of the work described in this paper originally arose in the study of electrical
network topology with the aim of building efficient circuit simulators
\cite{HNarayanan}, \cite{HNarayanan1986a}, \cite{narayanan1987topological},\cite{HNarayanan1997} (open 2nd edition available at \cite{HNarayanan2009}). 
An electrical network is defined by 
{\it topological constraints}, namely Kirchhoff's current and voltage laws
and {\it device characteristic constraints}. Our approach, towards building circuit simulators, was to get maximum 
advantage out of manipulating the topological - rather than the device characteristic, constraints.  Some fundamental questions here are:
how to connect subnetworks, how to decompose networks into multiports,
what is the minimum number of variables through which blocks of a partition of 
edges of the network interact, how to make minimal changes  in the topology 
of the network for faster simulation etc. 

From the point of view of application of linear algebra to electrical networks, it is more useful to deal with subspaces than with transformations.  The variables of interest  are currents and voltages associated with edges
of the graph of the network and the injected currents and potentials at nodes.
The spaces, that arise, can be regarded as subspaces of a `global' space with 
a specified basis - in a vector $(x_1,\cdots, x_n), x_i$ refers to the variable
associated with say the $i^{th}$ edge or node.
A characteristic feature of the approach is, as far as possible, to avoid 
speaking of networks in terms of matrices but rather in terms of other derived networks. The variables  of interest would often be buried in a much larger set
of network variables and we have to deal with the former {\it implicitly}
without trying to {\it eliminate} the  latter. 

When one connects subnetworks, taking the device independent point of view,
there is no {\it direction} to the flow of `information', merely {\it interaction,} which could be regarded as working both ways. 
To capture this situation, linear transformations are not suitable, 
rather one needs linear relations. One is naturally led to the question whether
parts of linear algebra can be developed based on linear relations and, more importantly,
whether there are direct applications. As mentioned before, connections of networks
provide such instances. However, to carry the similarity to standard linear algebra
to its limit, 
one needs to introduce direction to the flow of information. A beginning was made in \cite{HNPS2013}, where
it was shown that linear time invariant systems can be regarded as indexed vector spaces, and, when this is done, Kalman dual  etc. have simple descriptions. Some basic results
and algorithms of multivariable  control come through neatly and in a computationally
efficient manner. The present paper carries this work farther. A primitive `spectral' theory
is developed for `generalized operators'. A generalization  of state variables which we call
an `emulator' is used to describe relationships between essentially identical dynamical systems
in an implicit way. The processes involved are more refined than similarity transformations,
capable of linking systems with different numbers of variables
and yielding more information about the interrelationship between the systems.

We now give a brief account of the basic ideas, which fortunately are just a few.\\
A {\it {vector}} ${f}$ on $X$ over $\mathbb{F}$ is a function $f:X\rightarrow \mathbb{F}$ where $\mathbb{F}$ is a field. It would usually be represented as $f_X.$ A collection of such vectors closed under addition and scalar multiplication is a {\it vector space} denoted by $\V_X.$ 

A {\it linkage} $\V_{S_1\cdots S_k}$ is a vector space on $S_1\cup\cdots \cup S_k$ with a specified partition $\{S_1,\cdots , S_k\}$ of $S_1\cup\cdots \cup S_k.$ 

The usual {\it sum} and {\it intersection} of vector spaces are given extended meanings as follows.
$$\V_X+\V_Y\equiv (\VX\oplus \ \0_{Y\setminus X})+ (\Vy\oplus \ \0_{X\setminus Y}),$$
$$\V_X\cap \V_Y\equiv (\VX\oplus \F_{Y\setminus X})\bigcap (\Vy\oplus \F_{X\setminus Y}),$$
where $\0_Z$ represents the space containing only the $0$ vector on $Z$ and $ \F_Z$ represents the collection of all vectors on $Z.$

Given $\V_S,T\subseteq S,$ the {\it restriction} of $\V_S$ to $T,$ denoted by 
$\V_S\circ T,$ is the 
collection of all $f_T,$ where $(f_T,f'_{S\setminus T}),$ for some $f'_{S\setminus T},$
belongs to $\V_S.$ 

The {\it contraction} of $\V_S$ to $T$  denoted by 
$\V_S\times T,$  is the
collection of all $f_T,$ where $(f_T,0_{S\setminus T}),$ 
belongs to $\V_S.$ 

The {\it matched composition} of $\Vsp,\Vpq,$ with  $S,P,Q$ pairwise disjoint, is denoted by $\Vsp\lrar \Vpq$ and is defined to be the collection of all
$(f_S,h_Q)$ such that there exists some $g_P$ with \\
$(f_S,g_P)\in \Vsp,
(g_P,h_Q)\in \Vpq .
$ 

The  
 {\it skewed composition} of $\Vsp,\Vpq,$ with $ S,P,Q$ pairwise disjoint,
 is denoted by $\Vsp\rightleftharpoons \Vpq$ and is defined to be the collection of all
$(f_S,h_Q)$ such that there exists some $g_P$ with \\
$(f_S,-g_P)\in \Vsp,
(g_P,h_Q)\in \Vpq.$ 

The {\it dot product} of $f_S, g_S $ is denoted by $ \langle f_S, g_S \rangle $ and is defined to be $\sum_{e\in S}f(e)g(e).$ 

The {\it complementary orthogonal space} to $\V_S$ 
is denoted by 
${\V_S}^\perp $ and is defined to be\\
 $\{ g_S: \langle f_S, g_S \rangle = 0,\ f_S\in \V_S \}.$

The two basic results (\cite{HNarayanan1986a}) are as follows.

(IIT) The {\it implicit inversion theorem} states that the linkage equation
$$ \Vsp\lrar \Vpq=\Vsq,$$ with  specified $ \Vsp, \Vsq,$ but $\Vpq$ as unknown,  has a solution iff $\Vsp\circ S\supseteq \Vsq \circ S,
\Vsp\times S\subseteq \Vsq \times S$\\
 and, further, under the additional conditions
$\Vsp\circ P\supseteq \Vpq \circ P,
\Vsp\times P\subseteq \Vpq \times P,$ it has a unique solution.

IIT could be regarded as a generalization of the usual existence-uniqueness result
for the solution of the equation $ Ax=b.$ In the context of this paper, it is very much more useful. 

(IDT) The {\it implicit duality theorem} states that 
$$(\Vsp\lrar \Vpq)^{\perp}=(\Vsp^{\perp}\rightleftharpoons \Vpq^{\perp}).$$
IDT was a folklore result in electrical network theory
probably known informally to G.Kron (\cite{kron39},\cite{kron63}). An equivalent result  is stated
without proof in \cite{belevitch68}.
It can be regarded as a generalization 
of the result $(AB)^T=B^TA^T,$

A central theme of topological network theory  is that of 
dealing efficiently with vector spaces of the kind $\Vsp \cap \Vs$ 
where $\Vsp$ is in some way topological, for  instance it could be 
the direct sum $\V_{S'P'}\oplus \V_{S"P"}$ \\ ($S=S'\cup S",  \ P=P'\cup P",
\ \ S', S", P', P", \ \ \textup{being\  pairwise\  disjoint}$),
where the two spaces that make up the direct sum are the current and voltage
spaces of a graph. The space $\Vs$ would usually have the form $\V_{S_1}\oplus \cdots \oplus\V_{S_k}$ where the $S_i$ are very small (usually of size not more than
$10$). Topological spaces have the characteristic feature that their contractions
and restrictions remain topological. There are powerful ways of dealing with them
which use only topology - for instance `multiport decomposition' (see Section \ref{subsec:MultiportDecomposition}) and `topological transformations' (\cite{narayanan1987topological}). The first  half of the book
\cite{HNarayanan1997} deals with these ideas and can be regarded
as suggesting the need for {\it implicit linear algebra }
by bringing out two facts: 
\begin{itemize}
\item
that many practical
dynamical systems occur naturally as a composition of linkages,
with the more complicated linkages having a topological origin,
\item that the topological linkages can be simplified efficiently
through topological methods within the ambit of implicit linear algebra.
\end{itemize}

The present paper also exploits topological ideas but is closer to standard linear algebra in the sense that a primitive implicit spectral theory of generalized operators (defined below) is developed.
For a description of these ideas we need to define the notion of 
\nw{copies} of sets and \nw{copies} of collections of vectors 
on copies of sets.  We say sets $X$, $X'$ are \nw{copies of each other} iff they are disjoint and there is a bijection, usually clear from the context, mapping  $e\in X$ to $e'\in X'$. 
We need this operation in order to talk, for instance, of the vector $(v,i)$ where $v$ is a voltage vector on the edge set of a graph while $i$ is  a current vector on the same set. We handle this by building a copy $E'$ of the edge set $E$ and say $v$ is on $E$ and $i$ is on $E'$.
Let $A,A'$ be two sets which are copies of each other and let $e'\in A'$ denote the copy of $e\in A.$ We say  vectors $f_A,f_{A'}$ are copies of each other iff
we have $f_{A}(e)=f_{A'}(e'), \ e\in A.$
By $(\V_A)_{A'}$ we mean the  copy of $\V_A$ obtained by collecting  copies $f_{A'}$ of $f_{A}\in \V_A .$

 While there is no direction to interaction in a linkage, to deal with
operators implicitly we have to bring in such direction. A {\it generalized dynamical system} $\Vwdwm$ is on the set $W\uplus \dw\uplus M,$
where $\uplus$ denotes disjoint union, $W,\dw$ are disjoint copies of each other and $M $ is a set disjoint from $W,\dw,$ intended to take care of manifest (visible) variables.
{\it A generalized autonomous system}
({\bf genaut}) $\Vwdw$ is on the set $W\uplus \dw.$ 
It is called a {\it generalized operator} ({\bf genop}) if, additionally, 
it satisfies the conditions $$(\Vwdw  \circ W)_{\dw} \supseteq \Vwdw\circ \dw,\ \ \ 
(\Vwdw  \times W )_{\dw}\supseteq \Vwdw\times \dw$$



An `ordinary' operator $A$ which takes $w^T$ to $\mydot{w}^T$ through 
$\mydot{w}^T=w^TA$ can be treated as the collection of all $(w^T,\mydot{w}^T)$
which satisfy the equation and can be seen to be  a genop.

Genauts and genops differ from ordinary operators  in the following respects.
\begin{itemize}
\item
The domain for an operator can be regarded as $\F_W,$ the collection of all vectors on $W.$
For the more general case of a genaut, the role of the domain is played by $\Vwdw\circ W.$ 
In the case of  an operator, $0$ maps to $0.$
In the more general case of a  genaut, $0$ is linked to $ \Vwdw\times \dw.$

\item
In the case of  an operator,
the range is a subspace of the domain. The analogous idea in the case of the genop
is the condition $(\Vwdw  \circ W)_{\dw} \supseteqn \Vwdw\circ \dw.$
\item
In the case of  an operator, $0$ maps to $0.$
The analogous idea in the case of the genop is\\
$(\Vwdw  \times W)_{\dw} \supseteqn \Vwdw\times \dw.$
This has essentially the same power as  $0$ mapping to $0.$
\end{itemize}
Most controllability-observability results can be stated in terms  of genauts
which satisfy one of the two conditions that make up a genop.

An {\it upper semi genop} ({\bf USG}) $\Vwdw$ 
satisfies $(\Vwdw  \circ W)_{\dw} \supseteq \Vwdw\circ \dw,
$
while a {\it lower semi genop} ({\bf LSG}) $\Vwdw$
 satisfies $(\Vwdw  \times W)_{\dw} \supseteq \Vwdw\times \dw.
 $

For instance, if we start with a dynamical system defined by state equations
and keep input and output free, then we would get a USG, while if we set both
input and output to zero, then we would get an LSG. 

Practically everything that needs to be done with operators computationally
can be done  with genops, but with the added advantage that things can
be done implicitly, exploiting, for instance, topology in the case of an electrical
network. 
The notion corresponding to sum of operators is the `intersection-sum $\pdw$' defined  by
$$\oVonewdw\  \pdw \ \oVtwowdw \equivn  \{(f_W,g_{\dw}):(f_W,g^1_{\dw})\in \oVonewdw ,\ \  (f_W,g^2_{\dw}) \in \oVtwowdw,\ \  g_{\dw}= g^1_{\dw}+g^2_{\dw} \}.$$
Scalar multiplication is defined by
$$\ldw\Vwdw \equiv  \{(f_W,\lambda g_{\dw}): (f_W, g_{\dw})\in \Vwdw \}+ \Vwdw\times {\dw}.$$
Operator multiplication is defined by
$$\Vwdw *\Vwdw \equiv (\Vwdw)_{WW_1}\lrar (\Vwdw)_{W_1\dw}.$$
Here $(\Vwdw)_{WW_1}$ means that we build a copy of $W\uplus \dw$ and call it $W\uplus W_1,$
$\dw$ being copied onto $W_1$ and $(\Vwdw)_{W_1\dw}$ means $W\uplus \dw$ is copied onto
$W_1\uplus \dw,$
$W_1$ being a copy of $W .$ The three subsets  $W,\dw,W_1$ are pairwise disjoint.

Polynomials of genops can then be defined in the obvious manner and these 
behave similarly to polynomials of ordinary operators - for instance, there
is a unique minimal annihilating polynomial. Here we say $p(s)$ annihilates
$\Vwdw$ iff $p(\Vwdw)$ is decoupled, i.e., is of the form $\Vw^1 \oplus \V_{\dw}^2.$
As is required, when ordinary operators are treated as genops, their
annihilation agrees with the generalized definition.

In control theory, controllability, observability, pole placement
through state feedback, output injection play fundamental roles. 
We show how to deal with these ideas implicitly using genauts and genops.

As an illustration of the power of implicit linear algebra we define
emulators for dynamical systems and genauts.
We say $\Vwdwm,\Vpdpm$ are emulators of each other iff there are  linkages
$\Vonewp,\Vtwodwdp$ such that $$\Vpdpm\equaln\Vonewp\lrar \Vwdwm\lrar\Vtwodwdp,\ \ \ \ 
\Vwdwm\equaln\Vonewp\lrar \Vpdpm\lrar\Vtwodwdp.$$
We  note that each of  $\Vwdwm,\Vpdpm, \Vonewp,\Vtwodwdp$ would usually be
written as an expression involving other more basic linkages associated for instance
with elementary dynamical systems and graphs linking them.
The pair of linkages $\Vonewp,\Vtwodwdp$ behave like a similarity transformation but 
have greater flexibility in dealing implicitly with systems of differing numbers of
dynamical variables. They can also be built implicitly in linear time using purely
topological methods in the case of electrical networks.
One can do all computations associated with control  including feedback, pole placement, observer 
building etc. using emulators.
This approach can be regarded as being consistent with the point of view
of behaviourists (\cite{willems1991paradigms},\cite{Willems1997}) while retaining the computational power of the state variable approach. The theory developed yields a unified description of both standard linear dynamical systems and linear singular systems (or linear descriptor systems) (see for e.g.,  \cite{gantmacher1959theoryVol2}, \cite{lewis1986survey}).
There is an additional simplification over standard linear algebra-
we deal only with subspaces which arise through operations of the kind 
$\Vsp\lrar\Vp,$ sums and intersections  and not with
quotient spaces.

Throughout, the theory  that we build respects duality perfectly.
At a primitive level, one may take the dual of a linkage $\Vab$
to be $\Vab^{\perp}.$ If one wants the analogy with standard linear algebra
stronger we could define the transpose $\Vab^T$ to be $(\Vab^{\perp})_{(-A)B},$
where this latter space is obtained by changing every $(f_A,f'_B)$
in $\Vab^{\perp}$ to  $(-f_A,f'_B).$  It will then turn out by IDT, that
$$(\Vab\lrar \Vbc)^T
 \equaln  \Vab ^T\lrar \Vbc ^T,$$
$$(\oVonewdw \pdw \oVtwowdw)^T\equaln (\oVonewdw )^T\pw (\oVtwowdw) ^T;\ \ \ \  (\ldw\Vwdw)^T\equaln\lw\Vwdw^T.$$
The adjoints of dynamical systems are defined through similar ideas
and we get the usual controllability - observability duality.

We now present an outline of this paper.

Section \ref{sec:Preliminaries} is on preliminary definitions and results.
Among others, matched and skewed compositions  are defined.
An important role is played by the operations of building copies, of sets
and of vector spaces on copies of index sets, in increasing the effectiveness of these operations.
A very general version of the implicit inversion theorem (IIT) is described with its proof. The implicit duality theorem (IDT) is stated and proved.
A minimal set of definitions from graph theory required for this paper  is presented. Multiport decomposition is described along with an algorithm 
for minimal decomposition.

Section \ref{sec:link} presents definitions on linkages. A simple example that
relates IIT to standard existence - uniqueness of solution of linear equations
is given. It is also shown how to make the `$\lrar$' operation associative
by the judicious use of copies of index sets. A natural variation (pseudoidentity) 
of the notion of identity is presented. The special case of two block
linkages is described and the notion of transpose derived for it.
The IDT is rewritten in terms of the notion of transpose.
The intersection-sum notion ($\Vab \pb \Vhab$) for two block linkages is defined
and, using IDT, it is shown that $$(\Vab \pb \Vhab )^{\transp}=\Vab^{\transp} \pa \Vhab^{\transp}.$$ 
The notion of scalar multiplication $\lb\Vab$ is defined.
It is shown, again using IDT, that $$(\lb\Vab)^T= \laa \Vab^T.$$
A useful result on the distributivity of the intersection-sum operation
over matched composition and scalar multiplication is presented.
This result is true under certain conditions which are stated.
Fortunately, these are sufficient for  defining factorisation of polynomials of genops later in Section \ref{sec:genoppoly}.

Section \ref{subsec:DynSys} presents  basic definitions about
dynamical systems such as regular dynamical systems, generalized autonomous
systems (genauts), generalized operators (genops). The ideas are illustrated using
an electrical circuit with resistors, capacitors, inductors and sources.

Section \ref{sec:Duality1} describes the construction of dual statements
to a given statement of the form
\begin{align*}
 \epsilon(\V_1,\ldots,\V_k,+,\cap,\oplus ,\circ , \times ,\leftrightarrow,\rightleftharpoons,\supseteq,\subseteq,=).
\end{align*}
The adjoint of a dynamical system is defined and a basic result on adjoints
is proved.

Section \ref{sec:Invsub} describes conditioned and controlled invariant subspaces of a genaut. The duality between the two subspaces is brought out. It is shown that if we start 
with a USG $\Vwdw^u$ and a conditioned invariant subspace $\Vw$ in it, then 
$\Vwdw^u +\Vw$ is a genop. Dually if we start
with a LSG $\Vwdw^l$ and a controlled invariant subspace $\Vw$ in it, then
$\Vwdw^l \cap \Vw$ is a genop.  
These genops are analogous to uncontrollable and unobservable spaces respectively.

Section \ref{sec:feedback_injection} is on $wm_u-$feedback and $m_y\mydot{w}-$injection. These are  analogues of state feedback and output injection
in our framework.
We prove that the two notions are dual to each other and give necessary and sufficient
conditions for a genop $\Vwdw$ to be obtainable from a genaut $\Vwdw '$ by 
 $wm_u-$ feedback or $m_y\mydot{w}-$ injection.

Section \ref{sec:genoppoly}
is on polynomials of genops. These are defined in general for genauts but for genops we show 
that the minimal annihilating polynomial is unique. Here annihilation is interpreted as resulting
in a space of the kind $\Vw^1\oplus \Vdw^2.$
Proceeding analogously to multivariable control theory, given a USG $\Vwdw',$
we determine what polynomials 
can be annihilating polynomials of a genop $\Vwdw\subseteq \Vwdw'$ that satisfies  $\Vwdw'\circ W= \Vwdw\circ W.$
We also state the dual and indicate its proof.

Section \ref{sec:eigen_ann} is a description of the usual pole placement theory and algorithm
in our framework. When the desired annihilating polynomial $p(s)$ satisfies conditions derived in Section \ref{sec:genoppoly}, we solve the problem of 
constructing state feedback and output injection, which, when imposed on the original dynamical 
system, result in a system with $p(s)$ as its annihilating polynomial.
In \ref{sec:ppa_dualization},
the algorithms of this section have been recast in a `line by line dualizable form'.
The rule is simply to work only with vector spaces (not with vectors), avoid the `belongs to' relation
and use `contained in' instead.
In this case the vectors are converted to one dimensional  subspaces and their duals to 
$|W|-1$ dimensional subspaces.

Section \ref{sec:emu}
is on emulators which have been defined earlier in this section. The idea is motivated through an example on linear 
dynamical circuits. The linear time topological procedure of multiport decomposition when applied to a capacitor $+$ inductor $+$
static circuit ($RLC$ circuit) is shown to yield an emulator.  
 Theorem \ref{thm:emulatorpoly}
shows that if $\Vwdw,\Vpdp$ are emulators of each other, so are $p(\Vwdw),p(\Vpdp),$
provided the polynomial $p(s)$ has no constant term. Thus if $p(s)$ annihilates one of them,
it will also annihilate the other. The `no constant term' condition is shown to be easy to get around, since
it means merely that eigenvectors corresponding to zero eigenvalues have to be handled separately.
Theorem \ref{thm:emuadjoint} shows that adjoints of emulators can be obtained by constructing 
emulators for the original dynamical system.
Theorem \ref{thm:invlinking} shows that invariant space computations can be performed on an
emulator and the results transferred to the original space.
Theorem \ref{thm:controlobservelink}  repeats this in the case of controllability-observability and
Theorems \ref{thm:wmfeedback},
\ref{thm:dwminjection}, respectively, for state feedback and output injection.

\ref{app:ComputingBasicOperations} discusses the computation of basic operations explicitly and implicitly.
The operations considered include vector space sum, vector space intersection and the $\lrar$ operation. 

\ref{sec:ppa_dualization}
has been described a couple of paragraphs earlier.

\ref{sec:implicitemulatormultiport} discusses, through an example, how to build an emulator implicitly for an  $RLC$ circuit 
using multiport decomposition.

\ref{sec:exadjoint}
discusses,  through an example, the construction of the adjoint  of 
an $RLC$ circuit and the adjoint of its
multiport 
decomposition based emulator.

In addition, the appendices also contain some routine proofs of theorems stated in the main body of the paper.

\section{Preliminaries}
\label{sec:Preliminaries}
The preliminary results and the notation used are from \cite{HNarayanan1997} (open 2nd edition available at \cite{HNarayanan2009}) and more specifically from
\cite{HNPS2013}. 
A \nw{vector} $\mnw{f}$ on $X$ over $\mathbb{F}$ is a function $f:X\rightarrow \mathbb{F}$ where $\mathbb{F}$ is a field. In this paper we work primarily with the real field.  When $X$, $Y$ are disjoint, $X\uplus Y$ denotes the disjoint 
union of $X$ and $Y.$ A vector $f_{X\uplus  Y}$ on $X\uplus Y$ would be written as $f_{XY}$ and would often be written as $(f_X,f_Y)$ during operations dealing with such vectors. The {\bf sets} on which vectors are defined will  always be {\bf finite}. When a vector $x$ figures in an equation, we will use the 
convention that $x$ denotes a column vector and $x^T$ denotes a row vector such as
in $Ax=b,x^TA=b^T.$ Let $f_Y$ be a vector on $Y$ and let $X \subseteq Y$ . Then the \textbf{restriction} of $f_Y$ to $X$ is defined as follows:
\begin{align*}
f_Y/X \equiv g_X, \textrm{ where } g_X(e) = f_Y(e), e\in X.
\end{align*}
When $f$ is on $X$ over $\mathbb{F}$, $\lambda \in \mathbb{F}$ then  $\mnw{\lambda f}$ is on $X$ and is defined by $(\lambda f)(e) \equiv \lambda [f(e)]$, $e\in X$. When $f$ is on $X$ and $g$ on $Y$ and both are over $\mathbb{F}$, we define $\mnw{f+g}$ on $X\cup Y$ by 
\begin{align*}
 (f+g)(e) &\equiv \left\{ \begin{matrix}
                          f(e) + g(e),& e\in X \cap Y\\
			    f(e),& e\in X \setminus Y\\
			    g(e),& e\in Y \setminus X.
                         \end{matrix}
\right.
\end{align*}
When $X\cap Y = \emptyset$  then $f_X+g_Y$ is written as $\mnw{f_X\oplus g_Y}$ or as $\mnw{(f_X, g_Y)}.$ When $f,g$ are on $X$ over $\mathbb{F}$ the \textbf{dot product} $\langle f, g \rangle$ of $f$ and $g$ is defined by 
\begin{align*}
 \langle f,g \rangle \equiv \sum_{e\in X} f(e)g(e).
\end{align*}
We say $f$, $g$ are \textbf{orthogonal} iff $\langle f,g \rangle = 0$.

An \nw{arbitrary  collection} of vectors on $X$ with $0_X$ as a member would be denoted by $\mnw{\mathcal{K}_X}$. $\mathcal{K}$ on $X$ is a \nw{vector space} iff $f,g\in \mathcal{K}$ implies $\lambda f + \sigma g \in \mathcal{K}$ for $\lambda, \sigma \in \mathbb{F}$. We will use the symbol $\mnw{\V_X}$ for vector space on $X$ as opposed to $\mathcal{K}_X$ for arbitrary collections of vectors on $X$ with a zero vector as a member. The symbol $\mnw{\mathcal{F}_X}$ refers to the \nw{collection of all vectors} on $X$ and $\mnw{0_X}$ to the \nw{zero vector space} on $X$ as well as the \nw{zero vector} on $X$. 

When the index set is clear from the context, to improve readability, we sometimes
represent a vector space as  $\{[K]\}$ or as $\{Kx=0\}.$ In the former case the vector
space referred to is the row space of $K$ and in the latter  case, it is the solution space of $Kx=0.$ 

When $X$, $Y$ are disjoint we usually write $\mathcal{K}_{XY}$ in place of $\mathcal{K}_{X\uplus Y}$.\\
The collection
$\{ (f_{X},\lambda f_Y) : (f_{X},f_Y)\in \mathcal{K}_{XY} \}$
is denoted by
$ \mathcal{K}_{X(\lambda Y)}.
$ 
When $\lambda = -1$ we would write more simply $\mathcal{K}_{X(-Y)}.$
Observe that $(\mathcal{K}_{X(-Y)})_{X(-Y)}=\mathcal{K}_{XY}.$\\ 

\subsection{Basic operations}
The basic operations we use in this paper are as follows:
\subsubsection{Building Copies}
An operation that occurs often during the course of this paper is that of building \nw{copies} of sets and \nw{copies} of collections of vectors on copies of sets. We say set $X$, $X'$ are \nw{copies of each other} iff they are disjoint and there is a bijection, usually clear from the context, mapping  $e\in X$ to $e'\in X'$. 

When $X,X'$ are copies of each other, the vectors $f_X$ and $f_{X'}$ are said to be copies of each other, i.e., $f_{X'}(e') = f_X(e).$ If the vectors on $X$ and $X'$ are not copies of each other, they  would be distinguished using   notation, such as,  $f_X$, $\widehat{f}_{X'}$ etc. Vectors without subscripts like,  $x$, $\mydot{x}$ and $w$, $\mydot{w}$   are not necessarily copies of each other.

When $X$ and $X'$ are copies of each other, the collections of vectors $\mathcal{K}_X$, $\mathcal{K}_{X'}$  always represent copies of each other, i.e., 
\begin{align*}
 \mathcal{K}_{X'} \equiv \{ f_{X'} : f_{X'}(e') = f_X(e), f_X\in \mathcal{K}_X \}.
\end{align*}
This holds also for vector spaces
$\V_X$, $\V_{X'}.$
  If they are not copies they would be clearly distinguished from each other by the notation, for instance, $\mathcal{K}_X$, $\widehat{\mathcal{K}}_{X'}$, $\V_{X}$, $\widehat{\V}_{X'}$, etc. 

When $X$ and $X'$ are copies of each other, the notation for interchanging the positions of variables $X$ and $X'$ in a collection $\mathcal{K}_{XX'Y}$ is given by $\mnw{(\mathcal{K}_{XX'Y})_{X'XY}}$, that is 
$$(\mathcal{K}_{XX'Y})_{X'XY}$$
$$ \equiv \{(g_X,f_{X'},h_Y)\ |\ (f_X,g_{X'},h_Y) \in \mathcal{K}_{XX'Y},\ g_X\textrm{ being copy of }g_{X'},\ f_{X'}\textrm{ being copy of }f_X  \}.$$
When $P, P'$ are copies of each other, 
$\mnw{\Ipp} $ is defined to be  the vector space $ \{ (f_P,f_P'):f_P\in \mathcal{F}_P\}.$ 
\subsubsection{Sum and Intersection}
Let $\mathcal{K}_X$, $\mathcal{K}_Y$ be collections of vectors on sets $X$, $Y$ respectively. The \nw{sum} $\mnw{\mathcal{K}_X+\mathcal{K}_Y}$ of $\mathcal{K}_X$, $\mathcal{K}_Y$ is defined over $X\cup Y$ as follows:
\begin{align*}
 \mathcal{K}_X + \mathcal{K}_Y &\equiv  \{  (f_X,0_{Y\setminus X}) + (0_{X\setminus Y},g_Y), \textrm{ where } f_X\in \mathcal{K}_X, f_Y\in \mathcal{K}_Y \},
 \end{align*}
When $X$, $Y$ are disjoint, $\mathcal{K}_X + \mathcal{K}_Y$ is usually written as $\mnw{\mathcal{K}_X \oplus \mathcal{K}_Y}$ and is called the \nw{direct sum}.
Thus,
\begin{align*}
\mathcal{K}_X + \mathcal{K}_Y &\equiv (\mathcal{K}_X \oplus \0_{Y\setminus X}) + (\0_{X\setminus Y} \oplus \mathcal{K}_Y).
\end{align*}

The \nw{intersection} $\mnw{\mathcal{K}_X \cap \mathcal{K}_Y}$ of $\mathcal{K}_X$, $\mathcal{K}_Y$ is defined over $X\cup Y$ as follows:
\begin{align*}
\begin{split}
\mathcal{K}_X \cap \mathcal{K}_Y \equiv \{ f_{X\cup Y} &: f_{X\cup Y} = (f_X,x_{Y\setminus X}),\qquad f_{X\cup Y} = (y_{X\setminus Y},f_Y), \\& \textrm{ where } f_X\in\mathcal{K}_X,\qquad f_Y\in\mathcal{K}_Y,   \qquad x_{Y\setminus X},\ y_{X\setminus Y} \\&\textrm{ are arbitrary vectors on }  Y\setminus X,\ X\setminus Y \textrm{ respectively} \}.
\end{split}
\end{align*}
Thus,
\begin{align*}
\mathcal{K}_X \cap \mathcal{K}_Y\equiv (\mathcal{K}_X \oplus \mathcal{F}_{Y\setminus X}) \cap (\mathcal{F}_{X\setminus Y} \oplus \mathcal{K}_Y).
\end{align*}

\subsubsection{Matched and Skewed Composition}
The \nw{matched composition} $\mnw{\mathcal{K}_X \leftrightarrow \mathcal{K}_Y}$ is on $(X\setminus Y)\uplus (Y \setminus X)$ and is defined as follows:
\begin{align*}
 \mathcal{K}_X \leftrightarrow \mathcal{K}_Y &\equiv \{
                f : f=(g_{X\setminus Y} , h_{Y \setminus X}), \textrm{ where } g\in \mathcal{K}_X, h\in \mathcal{K}_Y \textrm{ \& } g_{X\cap Y} = h_{Y \cap X}
\}.
\end{align*}
Matched composition is referred to as matched sum in \cite{HNarayanan1997}.
In the special case where $Y\subseteq X$, matched composition is called 
\nw{generalized minor} operation (generalized minor of $\mathcal{K}_X $
with respect to $\mathcal{K}_Y$). 

The \nw{skewed composition} $\mnw{\mathcal{K}_X \rightleftharpoons \mathcal{K}_Y}$ is on $(X\setminus Y)\uplus (Y \setminus X)$ and is defined as follows:
$$ \mathcal{K}_X \rightleftharpoons \mathcal{K}_Y \equiv  \mathcal{K}_{(X\setminus Y)(-X\cap Y)}\lrar \mathcal{K}_Y$$
$$=\{
                f :f= (g_{X\setminus Y},  h_{Y \setminus X}), \textrm{ where } g\in \mathcal{K}_X, h\in \mathcal{K}_Y \textrm{ \& } g_{X\cap Y} = -h_{Y \cap X}
\}.$$
When $X$, $Y$ are disjoint, the matched and skewed composition both correspond to direct sum.
The \nw{generalized minor} of $\mathcal{K}_{XY}$ relative to $\mathcal{K}_Y$ is $\mnw{\mathcal{K}_{XY}\leftrightarrow \mathcal{K}_Y}$. When $\mathcal{K}_{XY}$, $\mathcal{K}_Y$ are vector spaces,   observe that $\mathcal{K}_{XY}\leftrightarrow \mathcal{K}_Y = \mathcal{K}_{XY}\rightleftharpoons \mathcal{K}_Y.$ The operations $\mnw{\mathcal{K}_{XY} \leftrightarrow \mathcal{F}_Y}$, $\mnw{\mathcal{K}_{XY}\leftrightarrow \0_Y}$ are called the \nw{restriction and contraction} of $\mathcal{K}_{XY}$ and are also denoted by $\mnw{\mathcal{K}_{XY}\circ X}$, $\mnw{\mathcal{K}_{XY} \times X}$, respectively.
Here again $\mnw{\mathcal{K}_{XYZ}\circ XY}$, $\mnw{\mathcal{K}_{XYZ} \times XY}$, respectively
when $X,Y,Z$ are pairwise disjoint, would denote  $\mnw{\mathcal{K}_{XYZ}\circ X\uplus Y}$, $\mnw{\mathcal{K}_{XYZ} \times X \uplus Y}.$
\subsubsection{Vector Space Results}
\label{ssec:vspaceresults}
If the collections of vectors, which go to make up the expression, are all closed under addition, it is clear that $+,\cap, \leftrightarrow, \rightleftharpoons ,$ yield collections
which are closed under addition.
 It is also clear that $+,\cap, \leftrightarrow, \rightleftharpoons ,$ all yield vector spaces when they operate on vector spaces. 

If $\mathcal{K}$ is on $X$, then $\mnw{\mathcal{K}^\perp}$ is defined by
\begin{align*}
 \mathcal{K}^\perp = \{ g: \langle f, g \rangle = 0,f\in \mathcal{K} \}.
\end{align*}
  Clearly $\mathcal{K}^\perp$ is a vector space whether or not  $\mathcal{K}$ is. When $\V$ is a vector space on a finite set $S$, it can be shown that $(\V^\perp)^\perp = \V$. This is fundamental for the results of the present work. The following results are easy to see
\begin{align*}
 (\V_X + \widehat{\V}_{X})^\perp & = \V_X^\perp \cap \widehat{\V}_{X}^\perp, \qquad X \textrm{ finite } \\
(\V_X \cap \widehat{\V}_X)^\perp & = \V_X^\perp + \widehat{\V}_X^\perp,\qquad X \textrm{ finite }.
\end{align*}
When $X$, $Y$ are disjoint, it is easily verified that $(\V_X \oplus \V_Y)^\perp =\V_X^\perp \oplus \V_Y^\perp$. When $X$, $Y$ are not disjoint, $\V_X + \V_Y \equiv (\V_X \oplus \0_{Y\setminus X}) + (\V_Y \oplus \0_{X\setminus Y}) $. So 
\begin{align*}
 (\V_X + \V_Y)^\perp &= (\V_X^\perp \oplus \mathcal{F}_{Y\setminus X}) \cap (\V_Y^\perp \oplus \mathcal{F}_{X\setminus Y}) \\
&= \V_X^\perp \cap \V_Y^\perp
\end{align*}
by the definition of intersection of vector spaces on two distinct sets. 

Using $(\V^\perp)^\perp = \V$, we have $(\V_X \cap \V_Y)^\perp  = \V_X^\perp + \V_Y^\perp$. 

The following results can also be easily verified:
\begin{align*}
 (\V_{XY}\circ X )^\perp &= \V_{XY}^\perp \times X \\
 (\V_{XY}\times X )^\perp & = \V_{XY}^\perp \circ X \textrm{ using } (\V^\perp)^\perp = \V. 
\end{align*}
The above pair of results will be referred to as 
the \textbf{dot-cross duality}.  

\subsubsection{Results on Generalized Minor}
The results in this subsection were originally derived in \cite{HNarayanan,HNarayanan1986a,narayanan1987topological}. A comprehensive account is in 
\cite{HNarayanan1997}.

Let $X$, $Y$ and $Z$ be disjoint sets, then
\begin{align*}
(\mathcal{V}_{XYZ} \leftrightarrow  \mathcal{V}_{X}) \leftrightarrow  \mathcal{V}_{Y} = (\mathcal{V}_{XYZ} \leftrightarrow  \mathcal{V}_{Y}) \leftrightarrow  \mathcal{V}_{X} = \mathcal{V}_{XYZ} \leftrightarrow  (\mathcal{V}_{X} \oplus  \mathcal{V}_{Y}).
\end{align*}

When $ \mathcal{V}_{X}\equiv \0_X, \mathcal{V}_{Y}\equiv \F_Y,$ the above reduces to 
$$ \mathcal{V}_{XYZ}\times YZ\circ Z\equaln \mathcal{V}_{XYZ}\circ XZ\times Z.$$
The next lemma is a slightly modified version of a result (Problem 7.5) in \cite{HNarayanan1997}. 
\begin{lemma}
\label{lem:Kgenminor}
\begin{enumerate}
\item Let $\KSP,\KSQ $ be collections of vectors on $S \uplus P,S\uplus Q $  respectively.

Then there exists a collection of vectors
$\KPQ$ on $P\uplus Q$ s.t. ${\bf 0}_{PQ} \in \KPQ $ and $\KSP \lrar \KPQ = \KSQ,$ 

only if
$\KSP\circ S \supseteq \KSQ \circ S$
and
$ \KSP \times S \subseteq \KSQ \times S.$
\vspace{0.25cm}
\item Let $\KSP$ be a collection of vectors closed under subtraction  on $S \uplus P$ 

and let  $\KSQ$ be a collection of vectors on $S\uplus Q,$
closed under addition. 

Further let $ \KSP \times S \subseteq \KSQ \times S$ and $\KSP\circ S \supseteq \KSQ \circ S.$

Then the collection of vectors
$\KSP \lrar \KSQ,$ is closed under addition, with ${\bf 0}_{PQ}$ as a member 

and further we have that $\KSP \lrar (\KSP\lrar \KSQ)= \KSQ.$
\vspace{0.25cm}
\item Let $\KSP$ be a collection of vectors closed under subtraction, 
and let $\KSQ$ satisfy
the conditions, 

closure under addition, 
${\bf 0}_{SQ}\in \KSQ, \KSP \times S \subseteq \KSQ \times S$ and $\KSP\circ S \supseteq \KSQ \circ S.$

Then the equation
$$\KSP\lrar \KPQ = \KSQ,$$
where $\KPQ$ has to satisfy
closure under addition, 
has a unique solution under the condition
$$\KSP \times P \subseteq \KPQ \times P \ and\  \KSP\circ P \supseteq \KPQ \circ P.$$

If the solution $\KPQ$ does not satisfy these conditions, then there exists another solution,
i.e., $$\KSP\lrar \KSQ,$$ that satisfies these conditions.
\end{enumerate}
\end{lemma}

\begin{proof}
\begin{enumerate}
\item
Suppose $\KSP \lrar \KPQ = \KSQ$ and 
${\bf 0}_{PQ} \in \KPQ .$

It is clear
from the definition of the matched composition operation that
$\KSP \circ S \supseteq \KSQ\circ S.$  

Since 
${\bf 0}_{PQ} \in \KPQ,  $  if 
$\fS \oplus {\bf 0}_{P} \in \KSP,$ we must have that $\fS \oplus {\bf 0}_{Q} \in \KSQ.$

Thus, $\KSP \times S \subseteq \KSQ\times S.$ 
\vspace{0.5cm}
\item
On the other hand suppose
$ \KSP \times S \subseteq \KSQ \times S$ and $\KSP\circ S \supseteq \KSQ \circ S.$
\\
Let $\KPQ \equiv \KSP \lrar \KSQ,$ i.e., $\KPQ$ is the collection of all vectors $\fP\oplus \fQ $ s.t. for some
vector $\fS,$ 
$\fS\oplus \fQ \in \KSQ$, $\fS \oplus \fP \in \KSP.$  

Since $\KSP$ is closed under subtraction,
it  contains the zero vector, the negative of every vector in it and is closed
under addition. 

Since  ${\bf 0}_S \oplus{\bf 0}_P\in \KSP,$ we must have that   ${\bf 0}_S \in \KSP \times S$ and therefore   ${\bf 0}_S \in \KSQ\times S.$
It follows that 
${\bf 0}_S \oplus {\bf 0}_{Q} \in \KSQ. $ 

Hence, by definition of $\KPQ,
{\bf 0}_{PQ} \in \KPQ.$ Further, since both $\KSP, \KSQ,$ are closed under addition, so is $\KPQ$ since $\KPQ= \KSP \lrar \KSQ.$\\
Let $\fS\oplus \fQ \in \KSQ.$

Since $\KSP\circ S \supseteq \KSQ \circ S,$ for some $\fP,$ we must have that $\fS\oplus \fP \in \KSP.$
By the definition of $\KPQ,$ we have that $\fP\oplus \fQ \in \KPQ.$ 

Hence,
$\fS\oplus \fQ \in \KSP \lrar \KPQ.$ 

Thus,
$\KSP \lrar \KPQ \supseteq \KSQ.$

\vspace{0.5cm}
Next, let $\fS\oplus \fQ\in \KSP \lrar \KPQ,$ i.e., for some $\fP, \fS \oplus \fP \in \KSP$ and $\fP\oplus \fQ \in \KPQ.$

We know, by the definition of $\KPQ,$ that there exists $\fS'\oplus \fQ \in \KSQ$ s.t. $\fS' \oplus \fP \in \KSP.$

Since $\KSP$ is closed under subtraction, we must have, $(\fS - \fS') \oplus
{\bf 0}_{P} \in \KSP.$  

Hence, $\fS - \fS' \in \KSP \times S
\subseteq \KSQ\times S.$  

Hence $(\fS - \fS') \oplus
{\bf 0}_{Q} \in \KSQ.$ 

Since $\KSQ$ is closed under addition and
$\fS'\oplus \fQ \in \KSQ$,

it follows that $(\fS - \fS')\oplus {\bf 0}_{Q} + \fS'\oplus \fQ = \fS \oplus \fQ$ also
belongs to $\KSQ$.  

Thus, $\KSP \lrar \KPQ \subseteq \KSQ.$

\vspace{0.5cm}
\item From parts (1) and (2) above, the equation can be satisfied by some $\KPQ$ if and only if 
$\KSP \circ S \supseteq \KSQ\circ S$ and  $\KSP \times S \subseteq \KSQ\times S.$ 

Next, let $\KPQ$ satisfy the equation  $\KSP\lrar \KPQ =\KSQ $ and be closed under addition. 

From part (2), we know that if $\KPQ$ satisfies $\KSP \circ P \supseteq \KPQ\circ P$ and  $\KSP \times P \subseteq \KPQ\times P,$

then $\KSP\lrar (\KSP\lrar \KPQ) =\KPQ.$ But $\KPQ$ satisfies $\KSP\lrar \KPQ =\KSQ.$

It follows that for any such $\KPQ,$ we have $\KSP\lrar \KSQ=\KPQ.$ 

This proves that $\KSP\lrar \KSQ$
is the only solution to the equation $\KSP\lrar \KPQ =\KSQ, $ under the condition $\KSP \circ P \supseteq \KPQ\circ P$ and  $\KSP \times P \subseteq \KPQ\times P.$ 

\end{enumerate}
\end{proof}
\begin{remark}
We note that the  collections of vectors in Lemma \ref{lem:Kgenminor} can be over rings rather than over fields- in particular over the ring of integers. Also only $\KSP$ has to be closed over subtraction.
The other two  collections $\KPQ,\KSQ$ are only closed over addition with a zero vector as a member. In particular $\KSP$ could be a vector space over rationals while
 $\KPQ,\KSQ$  could be cones over rationals. In all these cases the proof would go through 
 without change.
\end{remark}
The following theorem is a simple consequence.
We will call this the {\bf implicit 
inversion theorem}(IIT).
\begin{theorem}
\label{thm:inverse}
Consider the equation $\Vsp \lrar \Vpq =\Vsq $. Then 
\begin{enumerate}
\item given $\Vsp, \Vsq,  \exists \Vpq, $ satisfying the equation iff $\Vsp\circ S\supseteq \Vsq \circ S$ and $\Vsp\times S\subseteq \Vsq \times S;$
\item  given $\Vsp, \Vpq,\Vsq$ satisfying the equation, we have $\Vsp \lrar \Vsq =\Vpq $ iff
$\Vsp\circ P\supseteq \Vpq \circ P$ and $\Vsp\times P\subseteq \Vpq \times P.$
\item given $\Vsp, \Vsq,$ assuming that the equation  $\Vsp \lrar \Vpq =\Vsq $
is satisfied by some $\Vpq $ it is unique under the condition $\Vsp\circ P\supseteq \Vpq \circ P$ and $\Vsp\times P\subseteq \Vpq \times P.$
\end{enumerate}

\end{theorem}
It can be shown that the special case where $Q=\emptyset$ is equivalent to the above result. For convenience we state it below.
\begin{theorem}
\label{thm:Vgenminor}
Let $\Vs = \Vsp \lrar \Vp.$  Then, $\Vp = \Vsp \lrar \Vs$ iff
$\Vsp \circ P \supseteq \Vp$ and $\Vsp \times P \subseteq \Vp.$
\end{theorem}

\subsubsection{Implicit Duality Theorem}
\label{ssec:idt}
From the definition of matched  composition, when $(X,Y,Z)$ are disjoint, 
\begin{align*}
\mathcal{V}_{XY}\leftrightarrow \mathcal{V}_{YZ} = (\mathcal{V}_{XY}+ \mathcal{V}_{(-Y)Z}) \times (X\cup Z)
\end{align*}
and also equal to 
\begin{align*}
 (\mathcal{V}_{XY}\cap \mathcal{V}_{YZ})\circ (X\cup Z).
\end{align*}
Similarly from the definition of skewed composition,
\begin{align*}
 \mathcal{V}_{XY} \rightleftharpoons \mathcal{V}_{YZ} = (\mathcal{V}_{XY}+ \mathcal{V}_{YZ}) \times (X\cup Z)
\end{align*}
and also equal to 
\begin{align*}
 (\mathcal{V}_{XY}\cap \mathcal{V}_{(-Y)Z})\circ (X\cup Z).
\end{align*}
Hence we have 
\begin{align*}
 (\mathcal{V}_{XY}\leftrightarrow \mathcal{V}_{YZ})^\perp 
& = [(\mathcal{V}_{XY} + \mathcal{V}_{(-Y)Z}) \times (X\cup Z)]^\perp \\
& = (\mathcal{V}_{XY} + \mathcal{V}_{(-Y)Z})^\perp \circ (X\cup Z) \\
& =(\mathcal{V}_{XY}^\perp \cap \mathcal{V}_{(-Y)Z}^\perp) \circ(X\cup Z)\\
& = \ \ \  \mathcal{V}_{XY}^\perp \rightleftharpoons \mathcal{V}_{YZ}^\perp 
\end{align*}
In particular,
\begin{align*}
 (\mathcal{V}_{XY}\leftrightarrow \mathcal{V}_{Y})^\perp &\equaln   
\mathcal{V}_{XY}^\perp \leftrightarrow \mathcal{V}_{Y}^\perp
\end{align*}
since $\mathcal{V}_Y = \mathcal{V}_{(-Y)}$.
To summarize
\begin{theorem}
\label{thm:idt}
$(\mathcal{V}_{XY}\leftrightarrow \mathcal{V}_{YZ})^\perp 
\ \equaln\ \mathcal{V}_{XY}^\perp \rightleftharpoons \mathcal{V}_{YZ}^\perp 
.$ In particular,
$(\mathcal{V}_{XY}\leftrightarrow \mathcal{V}_{Y})^\perp \ \equaln\ \mathcal{V}_{XY}^\perp \leftrightarrow \mathcal{V}_{Y}^\perp
.$
\end{theorem}
The above result will be referred to as the \nw{implicit duality theorem} (IDT).

Observe that the dot-cross duality is also a consequence of the implicit duality theorem since
\begin{align*}
(\mathcal{V}_{XY}\times X )^\perp = (\mathcal{V}_{XY}\leftrightarrow 0_X )^\perp = \mathcal{V}_{XY}^\perp \leftrightarrow \mathcal{F}_X = \mathcal{V}_{XY}^\perp\circ X.
\end{align*}

Implicit duality theorem and its applications are dealt with in detail in \cite{HNarayanan1986a},  \cite{narayanan1987topological}, \cite{HNarayanan}, \cite{HNarayanan1997}. 
There has been recent interest in the applications of this result in \cite{Forney2004},
in the context of `Pontryagin duality'.
The above proof is based on the one in  \cite{HNarayanan}. Versions for other contexts (such as orthogonality being replaced by polarity, or by the constraint of
dot product being an integer) are available in 
\cite{HNarayanan1997}, for operators in \cite{HN2000}, \cite{HN2000a}, for matroids in \cite{STHN2014}. 

Computational techniques for the basic operations are discussed in \ref{app:ComputingBasicOperations}.

\subsection{Graphs}

The motivation for the ideas in this  paper arose from the attempt to understand the topological processes of electrical circuit theory
and exploit them efficiently. So we require a few definitions from graph theory to state results from electrical networks. We follow \cite{HNarayanan1997}.

An \nw{undirected graph} is a triple $(V,E,f^u)$, where $V$ is the set of \nw{vertices}, $E$ is the set of \nw{edges}, and $f^u$ is the \nw{incidence function} which associates  a pair of vertices with each edge. 
The \nw{incidence function} defines the end vertices of the edges. 

Similarly, a \nw{directed graph} is also a triplet $(V,E,f^d)$, where  $f^d$ is the \nw{incidence function} which associates  with each edge an ``ordered pair'' of vertices and the sets $V$, $E$ define the \nw{vertices},   \nw{edges}, respectively of the graph. 
%
%
%
%
The \nw{incidence function} defines the end vertices and also the direction of arrow for the edges. 




An \nw{undirected path} (of a graph) from vertex $v_{1}$ to vertex $v_{k}$ is a disjoint alternating vertex-edge sequence $v_{1},e_{l1},v_{2},e_{l2},\ldots,  e_{l(k-1)},v_{k} $, such that every edge $e_{lr}$ is incident on vertices $v_{r}$ and $v_{r+1}$. 
When it is clear from the context, a path is simply denoted by its edge sequence. 
A graph is said to be \nw{connected}, if there exists an undirected path between every pair of nodes. Otherwise it is said to be \nw{disconnected}.  
A disconnected graph has  \nw{connected components} which are individually connected  with no edges between the components.


A  \nw{loop} (or a \nw{circuit}) (of any graph) is an undirected  path in which the starting and ending vertices are  the same and no other vertices or edges are repeated. 


A \nw{tree} of a graph is a sub-graph of the original graph with no loops. The edges of a tree are called \nw{branches}. A \nw{spanning tree} is a maximal tree with respect to the edges of a graph. Hence, the set obtained by adding  any other edge of the graph to a spanning tree contains a (single) loop. For this reason, a spanning tree is also called \nw{maximal circuit free set}.  The loops obtained in this manner are called \nw{fundamental circuits} associated with the tree.   




It can be verified that every spanning tree of a connected graph with $n$ nodes consists of $n-1$ edges. Since addition of each edge to a spanning tree creates only one loop, the number of fundamental circuits (associated with a spanning tree) is $e-n+1$, where $e$ is the number of edges and $n$ the number of vertices of the graph. 

A \nw{cotree} of a graph is defined with respect to a spanning tree. It consists of all the edges of the graph which are not in the spanning tree. The edges of a cotree are called \nw{links}. 




A \nw{forest} of a disconnected graph is a disjoint union of the spanning trees of its connected components. The complement of a forest is called \nw{coforest}. In this paper, the terms tree and cotree will be used to mean forest and coforest when it is clear from the context.

The number of edges in a forest of a graph is called the \nw{rank} of the graph. For a graph $\mathcal{G}$ it is denoted by $r(\mathcal{G})$.

A \nw{crossing edge} set of a graph is the set of all edges which lie between two complementary subsets of the vertex set of the graph.



   A \nw{cutset} is a minimal crossing edge set, i.e., a minimal subset which when deleted from the graph increases the count of connected components by one.



A spanning tree of the graph can be used to systematically generate cutsets. Since in a spanning tree of a connected graph there exists only one path between any two vertices, deletion of a branch from the tree disconnects the tree into two components. The set of all edges in the graph, between these two components of the tree, forms a cutset of the graph.  Since a spanning tree of a connected graph with $n$ vertices contains $n-1$ edges, one could associate $n-1$ cutsets to it. 
%
%
%
The cutsets obtained in this manner are called \nw{fundamental cutsets} of the graph with respect to the tree.

In a graph, an edge set $E_1$ is said to \nw{span} another edge set $E_2$ if for each $e_2 \in E_2,$ there exists a   circuit $L_1 \subseteq e_2 \cup E_1$ with $e_2 \in L_1$. For example, a spanning tree of any graph spans the cotree of the same graph. Dually, an edge set $E_1$ is said to \nw{cospan} another edge set $E_2$ if for each $e_2 \in E_2,$ there exists a   cutset  $C_1 \subseteq e_2 \cup E_1$ with $e_2 \in C_1$. For instance, the cotree with respect to a spanning tree of any graph cospans the same spanning tree of the graph.





Let $\mathcal{G}$ be a graph and $E\equiv E(\mathcal{G})$ be the edge set of $\mathcal{G}$ and $T\subseteq E$. Then $\mnw{\mathcal{G} \times (E-T)}$ denotes the graph obtained by removing the edges $T$ from $\mathcal{G}$ and fusing the end vertices of the removed edges. $\mnw{\mathcal{G} \circ (E-T)}$ denotes the graph obtained by removing the edges $T$ from $\mathcal{G}$ and removing the isolated vertices.
(Note that $\times, \circ $ are also used to denote vector space operations. However, the context would make clear
whether the objects involved are graphs or vector spaces.)

We refer to the space of vectors $v_E,$ which satisfy Kirchhoff's Voltage Law (KVL) of the graph $\mathcal{G},$
by $\mnw{\V^v(\mathcal{G})}$ and to the space of vectors $i_E,$ which satisfy Kirchhoff's Current Law (KCL) of the graph $\mathcal{G},$
by $\mnw{\V^i(\mathcal{G})}.$

{\bf Tellegen's Theorem} (\cite{HNarayanan1997}) states that $\V^v(\mathcal{G})= \V^i(\mathcal{G})^{\perp}.$ The following result on vector spaces associated with graphs is useful.
$$ \V^v(\mathcal{G}\circ T)= \V^v(\mathcal{G})\circ T, \ \ \  \V^v(\mathcal{G}\times T)= \V^v(\mathcal{G})\times T,\ \ T\subseteq E(\mathcal{G});$$
$$ \V^i(\mathcal{G}\circ T)= \V^i(\mathcal{G})\times T, \ \ \  \V^i(\mathcal{G}\times T)= \V^i(\mathcal{G})\circ T,\ \ T\subseteq E(\mathcal{G}).$$

An {\bf electrical network $\mathcal{N}$} is a pair $(\mathcal{G},\mathcal{D}).$ Here $\mathcal{G}$ is a directed graph
and $\mathcal{D}$ is a collection of pairs of vector functions of time $(v(\cdot),i(\cdot)),$ where  $v(t),i(t)$ are vectors on the edge set of the graph. $\mathcal{D}$ is called the `device characteristic' of the network. 

A {\bf solution} of $\mathcal{N}$ is a pair 
 $(v(\cdot),i(\cdot)),$ satisfying 
$$v(t)\in \V^v(\mathcal{G}),\ \ i(t) \in \V^i(\mathcal{G})\ \  (KVL,KCL)\ \  and \ \ (v(\cdot),i(\cdot))\in \mathcal{D}.$$ 
\subsection{Topological transformation and multi-port decomposition}
\label{subsec:MultiportDecomposition}
There are two important ideas of network analysis, which are essentially notions of implicit linear algebra, but 
which appear to have no counterparts in classical linear algebra.
These are the methods called topological transformation and multiport decomposition (see \cite{HNarayanan1986a}, \cite{narayanan1987topological}, \cite{HNarayanan1997}). 

In the method of topological transformation,
an electrical circuit is treated as though it has a more convenient topology
but with some additional constraints. The idea is to exploit the new topology
while meeting the additional constraints. In the language of this paper,
we have two vector spaces $\Vs^1, \Vs^2$ and we need to find a linkage
$\Vsp$ and vector spaces $\Vp^1,\Vp^2,$ such that
$$\Vsp\lrar \Vp^1 = \Vs^1, \ \ \Vsp\lrar \Vp^2 = \Vs^2,\ \ 
\textup{and\   $|P|$\   is\  the\  minimum\  possible.}$$ One can show that this minimum 
size is $d(\Vs^1,\Vs^2)\equiv r(\Vs^1+\Vs^2)-r(\Vs^1\cap \Vs^2).$ A much harder problem 
is where $\Vs^1$ is given and $\Vs^2$ belongs to a class of vector spaces
on $S,$ closed under sum and intersection, and we have to pick 
$\Vs^2$ to minimize $d(\Vs^1,\Vs^2).$

In this paper we use only multiport decomposition.
A description follows.

Let $\mathcal{G}$ be a directed graph and let $E_1$, $E_2$ be a partition of edges of the graph. Then it can be shown that we can build three graphs $\mathcal{G}_{E_1P_1}$, $\mathcal{G}_{E_2P_2}$, $\mathcal{G}_{P_1P_2}$ such that 
\begin{align*}
 \mathcal{V}^i(\mathcal{G}) = \Big(\mathcal{V}^i(\mathcal{G}_{E_1P_1}) \oplus \mathcal{V}^i(\mathcal{G}_{E_2P_2})\Big) \leftrightarrow \mathcal{V}^i(\mathcal{G}_{P_1P_2})
\end{align*}
and therefore through Tellegen's theorem  and the implicit duality theorem 
\begin{align*}
 \mathcal{V}^v(\mathcal{G}) = \Big(\mathcal{V}^v(\mathcal{G}_{E_1P_1}) \oplus \mathcal{V}^v(\mathcal{G}_{E_2P_2})\Big) \leftrightarrow \mathcal{V}^v(\mathcal{G}_{P_1P_2}),
\end{align*}
where $P_1$, $P_2$ are additional sets of branches. This process is called multi-port decomposition. The two graphs $\mathcal{G}_{E_1P_1}$, $\mathcal{G}_{E_2P_2}$
are called multiports and the graph $\mathcal{G}_{P_1P_2}$ is called 
the port connection diagram for the multiport decomposition.

Let network $\mathcal{N}$ be on directed graph  $\mathcal{G},$  with a partition $E_1$, $E_2$  of edges of the graph specified such that the device characteristics
of $E_1$, $E_2$ are independent in $\mathcal{N}.$  Let $\mathcal{G}_{E_1P_1}$,$\mathcal{G}_{E_2P_2}$,
 $\mathcal{G}_{P_1P_2}$
be a multiport decomposition of $\mathcal{G}.$ Define the networks $\mathcal{N}_{E_1P_1}$, $\mathcal{N}_{E_2P_2}$  on graphs $\mathcal{G}_{E_1P_1}$, $\mathcal{G}_{E_2P_2}$,
by retaining the device characteristics of $E_1,E_2$ as in $\mathcal{N}$
and having no device characteristic constraints on $P_1,P_2.$ Then we say that
 $\mathcal{N}_{E_1P_1}$, $\mathcal{N}_{E_2P_2}$
is a multiport decomposition of $\mathcal{N}$
with port connection diagram $\mathcal{G}_{P_1P_2}.$

We present below a linear time algorithm (\cite{HNarayanan1986a}) for decomposing a graph minimally, that is, in such a way that $|P_1|$, $|P_2|$ have minimum size, when the partition $E_1,E_2$ is  specified. We can show that this size is 
\begin{align*}
 |P_1| = |P_2|& = r(\mathcal{G}\circ E_1) - r(\mathcal{G}\times E_1) \\
 & = r(\mathcal{G}\circ E_2) - r(\mathcal{G}\times E_2).
\end{align*}

{
{\bf Algorithm minimal port decomposition}

Let $\mathcal{G}$ be a graph on edge set $E$ and let a partition $E_1,E_2$ be given.
Let $t_1, t_2$ be forests of $\mathcal{G}\circ E_1, \mathcal{G}\circ E_2,$ respectively.

Start with  $t_1$ and grow a forest $t_1\uplus t_{12}, t_{12}\subseteq t_2,$ of $\mathcal{G}.$

Start with  $t_2$ and grow another forest $t_2\uplus t_{21}, t_{21}\subseteq t_1,$ of $\mathcal{G}.$

Build the graph $\mathcal{G}\circ (E_1\uplus t_2)\times (E_1\uplus (t_2-t_{12})).$
Rename the edges $t_2-t_{12}$ as $P_1,$ and the graph as $\mathcal{G}_{E_1P_1}.$ 

Build the graph $\mathcal{G}\circ (E_2\uplus t_1)\times (E_2\uplus (t_1-t_{21})).$
Rename the edges $t_1-t_{21}$ as $P_2,$ and the graph as $\mathcal{G}_{E_2P_2}.$ 

Build the graph $\mathcal{G}_{E_1P_1}\circ (t_1\uplus P_1) \times (P_1\uplus (t_1-t_{21})).$
Rename $t_1-t_{21}$ as $P_2$ as before, and the graph as $\mathcal{G}_{P_1P_2}.$

Output $\mathcal{G}_{E_1P_1}, \mathcal{G}_{E_2P_2}$ as multiports and $\mathcal{G}_{P_1P_2}$
as the port connection diagram.

Algorithm ends.
}

When the decomposition is minimal, the sets $P_i$ will not contain circuits or cutsets in $\mathcal{G}_{E_1P_1}$, $\mathcal{G}_{P_1P_2}$, $\mathcal{G}_{P_2E_2}$.
The following result is an immediate consequence.
\begin{theorem}
\label{thm:minimalmultiport}
Let the decomposition of the directed graph $\mathcal{G}$ into $\mathcal{G}_{E_1P_1}$, $\mathcal{G}_{E_2P_2}$, $\mathcal{G}_{P_1P_2}$ be minimal. We then have the following.
\begin{enumerate}
\item The sets $P_i, \ i=1,2,$ do not contain circuits or cutsets in $\mathcal{G}_{E_1P_1}$, $\mathcal{G}_{P_1P_2}$, $\mathcal{G}_{E_2P_2}$.
\item
Let $i_E$, $v_E$ be current and voltage vectors of $\mathcal{G}$. Let $(i_{E_1}, i_{P_1})$, $(i_{E_2}, i_{P_2})$, $(i_{P_1}, i_{P_2})$ be current vectors and let $(v_{E_1}, v_{P_1})$, $(v_{E_2}, v_{P_2})$, $(v_{P_1}, v_{P_2})$ be voltage vectors of $\mathcal{G}_{E_1P_1}$, $\mathcal{G}_{E_2P_2}$, $\mathcal{G}_{P_1P_2}$, respectively. Then $i_{P_1}$, $i_{P_2}$ are uniquely determined, given $i_{E_1}$ or $i_{E_2},$ and  $v_{P_1}$, $v_{P_2}$  are uniquely determined, given $v_{E_1}$ or $v_{E_2}$.
\end{enumerate}

\end{theorem}

In Section \ref{sec:emu}, we will show how to build `emulators' for electrical
networks. These are based on multiport decomposition, can be built in linear time and can serve as a substitute for state equations.

\section{Linkages and maps}
\label{sec:link}
A {\bf linkage} is a vector space $\mathcal{V}_{S}$ with a partition $\{S_1\cdots S_k\}$ of $S,$  specified.

We will denote it by  $\mathcal{V}_{S_1\cdots S_k}.$ 

Observe that a map $x^TK=y^T$ can be regarded as
a linkage $\mathcal{V}_{XY},$ with a typical vector $(x^T,y^T)$,  whose basis is the set of rows of $\bbmatrix{
 I & K},$ the first set of columns of the matrix being indexed by $X$ and the second set by $Y.$

A linkage $\mathcal{V}_{S_1\cdots S_k}$ is said to be {\bf decoupled} 
iff $\mathcal{V}_{S_1\cdots S_k}\circ S_i =\mathcal{V}_{S_1\cdots S_k}\times S_i, i=1,2,\cdots ,k.$
This is equivalent to saying that it has the form $\bigoplus _i\V_{S_i}, i=1,2,\cdots ,k .$

The matched composition operation $"\leftrightarrow "$ can be regarded as a generalization of the composition operation for ordinary maps. 

Thus if $x^TK=y^T,y^TP=z^T,$ it can be seen that $\mathcal{V}_{XY}\leftrightarrow\mathcal{V}_{YZ}  $ 
would have as the basis, the rows of  $\bbmatrix{
I & KP },$ the first set of columns of the matrix being indexed by $X$ and the second set by $Z.$

\begin{remark}
\begin{enumerate}
\item When we compose matrices, the order is important. In the case of linkages, since vectors
are indexed by subsets, order does not matter. Thus $\Vsp=\Vps$ and $\Vsp\lrar \Vpq\equaln \Vpq\lrar \Vps.$

\item
$\mathcal{V}_{XY} $ mimics the map $K$ by linking 
 row vectors  $x^T,y^T$ which correspond to each other when $x^T$ is operated upon by $K,$ into 
 the long vector $(x^T,y^T).$
The collection of all such vectors forms a vector space which contains all the information that 
the map $K$ and $K^{-1}$ (if it exists)  contain and represents both of them implicitly.

But there are two essential differences: a map takes vectors in  
$\Fx$ to vectors in $\Fy,$ but takes the vector $0 _X$ only to the vector $0 _Y.$
In the case of linkages only a subspace of $\Fx$  may be involved as 
$\mathcal{V}_{XY}\circ X$ and the vector 
$0 _X$ may be linked to a nontrivial subspace of $\Fy.$
\end{enumerate}
\end{remark}
\begin{example}
\textup{Consider the equation  $Ax=b.$ In the present  frame work, we will derive the usual existence and 
uniqueness result 
i.e.,}\\ 

a solution exists iff $b$ lies in the span of columns of $A,$
equivalently $b \in (col(A))^{\perp\perp},$

and if the solution exists, it  is unique iff  columns of $A$ are independent.\\

\noindent \textup{We will associate with $A,x,b$ the vector spaces $\Vsp,\Vpq,\Vsq$ respectively,
where\\ $\Vsp$ is a space of  vectors $(f_S,g_P)$ with $f_S^TA=g_P^T,$ \\
$\Vpq$ is a space of  vectors $(g_P,h_Q)$ with $g_P^Tx=h_Q^T,$ \\
$\Vsq$ is a space of  vectors $(f_S,h_Q)$ with $f_S^Tb=h_Q^T.$\\
Observe that from  $\Vsp,\Vpq,\Vsq$ we can uniquely infer $A,x,b.$}

\textup{The equation  $Ax=b$ translates to $\Vsp \lrar \Vpq = \Vsq.$}

\textup{By Theorem \ref{thm:inverse}, we have that} 
\textup{$\Vpq$ exists such that 
$\Vsp \lrar \Vpq = \Vsq$ iff 
$$\Vsp\circ S\supseteq \Vsq\circ S\ \ \   {\textup{and}}\ \ \  \Vsp\times S\subseteq \Vsq \times S.$$} 
\textup{The first condition is automatic for maps since the domain is $\F_S$
which translates to $\Vsp\circ S=\F_S.$ The second condition is equivalent to
stating that `whenever $f_S^T A=0,$ we must also have   $f_S^T b=0,$'
which is the same as saying $b \in (col(A))^{\perp\perp}.$}


\textup{We need next to show that existence of a solution $\Vpq $ to the equation also implies that there is another solution
of the form $\{[I \ \ \hat{x}]\}$ .}

\textup{We do this as follows.
If $$\Vsp\equiv \{[I\ \  A]\}, \Vsq\equiv \{[I\ \  b]\}$$ then  $$\Vpq\equiv \Vsp \lrar \Vsq = \{[A\ \  b]\}
\  \ \textup{and}\ \   \Vsp.P=\Vpq.P.$$ 
We may assume wlog rows of $ \{[A\ \  b]\}$
 to be independent.
This implies that the rows of $A$ are linearly independent (using 
`whenever $f_S^T A=0,$ we must also have   $f_S^T b=0$').
Now extend $\Vpq $ to a vector space $\Vpq'$ such that $\Vpq'.P=\F_P.$
It can be seen that $\Vsp \lrar \Vpq'$ is equal to  $\Vsp \lrar \Vpq$
and further has the form $\{[I\ \ \hat{x}]\}.$}

%

\textup{Next let $\Vsp,\Vsq$ be such that there exists some $\Vpq$ satisfying $\Vsp \lrar \Vpq = \Vsq.$\\
We know by Theorem \ref{thm:inverse}, it is unique iff 
  $$\Vsp\circ P\supseteq \Vpq\circ P \ \  \textup{and}\ \  \Vsp\times P\subseteq \Vpq \times P.$$
  The second condition is trivial since $0_S^T A=0_P,$ so that $\Vsp\times P=\0_P.$\\
  Next note that since $x$  is the map $g_P^Tx=h_Q^T,$ we must have $
\Vpq\circ P=\F_P.$\\
 So the first condition implies that the range of the map
$f_S^TA=g_P^T,$  must be $\F_P.$\\
 This means that columns of $A$ are independent.}
\\Example ends
\end{example}
The `$\lrar $' operation is not inherently associative.
For instance, consider the expression 
$$(\V_{AB}\lrar \V_{BC})\lrar(\V_{CA}\lrar \V_{PQ}).$$
 Here, if we remove the brackets,
we get a subexpression $\V_{AB}\lrar \V_{BC}\lrar\V_{CA},$ where the index set becomes null
and the `$\lrar$' operation is not defined for such a situation.
 We say such an expression contains a `null' subexpression.
Let us only look at expressions which do not contain null subexpressions.
For such expressions,
there is a special case
where the brackets can be got rid of without  ambiguity.
This is when no index set occurs more than twice in the expression.
Consider for instance the expression 
$$\V_{A_1B_1}\lrar \V_{A_2B_2}\lrar 
\V_{A_1B_2},$$ 
where all the $A_i,B_j$ are mutually disjoint. It is clear that this expression
has a unique meaning, namely, the space of all $(f_{A_2},g_{B_1})$ such that
there exist $ h_{A_1},k_{B_2}$ with $$(h_{A_1},g_{B_1})\in \V_{A_1B_1},\ \ (f_{A_2},k_{B_2})\in \V_{A_2B_2},\ \ (h_{A_1},k_{B_2})\in \V_{A_1B_2}.$$ Essentially the 
terms of the kind $p_{D_i}$ survive if they occur only once and they get coupled through the terms which occur twice. These latter dont survive in the final expression. Such an interpretation is valid even if we are dealing with 
signed index sets of the kind ${- A_i}.$ (We remind the reader that
$\V_{(-A)B}$ is the space of all vectors $(-f_A,g_B)$ where
$(f_A,g_B)$ belongs to $\V_{AB}.$ ) If either ${- A_i},$
or ${ A_i}$ occurs, it counts as one occurrence of the index set $  { A_i}.$
We state this result as a theorem but omit the routine proof.
\begin{theorem}
\label{thm:notmorethantwice}
An expression of the form $\lrar _{i,j,k,\cdots}(\V_{(\pm A_i)(\pm B_j)(\pm C_k)\cdots}),$ where the index sets $\pm A_i, \pm  B_j,\pm C_k, \cdots$ are all mutually disjoint
has a unique meaning provided it contains no null subexpresssion and no index set occurs more than twice in it.
\end{theorem}
\subsection{Pseudoidentity}
\begin{definition}
\label{def:pseudoidentity}
Let $A'$ be a copy of $A.$
We say the linkage $\Vaadash$ is a pseudoidentity for a linkage 
$\Vab$ iff $$\Vaadash \lrar \Vab =(\Vab)_{A'B}.$$ 
\end{definition}
When $A'$ is a copy of $A,$
we remind the reader that
$\Iaa $ is   the vector space $ \{ (f_A,f_A'):f_A\in \mathcal{F}_A\}.$ 
\begin{lemma}
Let  $\Vaadash \supseteq \Iaa\bigcap (\Vaadash \circ A)$ 
and let $\Vaadash=(\Vaadash)_{A'A}.$ 
Then $\Vaadash$ satisfies the following conditions.
\begin{enumerate}
\item $(\Vaadash \circ A)_{A'}= \Vaadash \circ {A'}.$ If $f_A \in \Vaadash \circ A$ then $(f_A,f_{A'}) \in \Vaadash .$
Further,  if $(f_A,g_{A'}) \in \Vaadash ,$ then $(g_A,f_{A'}) \in \Vaadash .$ 
\item If $(f_A,g_{A'}), (g_A,h_{A'})\in \Vaadash,$ then 
$(f_A,h_{A'}) \in \Vaadash .$
\end{enumerate}
\end{lemma}
\begin{proof}
\begin{enumerate}
\item These are immediate from the hypothesis.
\item  If $(f_A,g_{A'}) \in \Vaadash ,$ then since $(f_A,f_{A'}) \in \Vaadash $
and $\Vaadash $
is a vector space, we must have \\
$(0_A,g_{A'}-f_{A'})\in \Vaadash. $
Since $(g_A,h_{A'})\in \Vaadash,$ we must have 
$(h_A,g_{A'})\in \Vaadash .$\\
So $(h_A,g_{A'})- (0_A,g_{A'}-f_{A'})= (h_A,f_{A'}) \in \Vaadash .$
So $(f_A,h_{A'}) \in \Vaadash .$
\end{enumerate}
\end{proof}
\begin{definition}
We will call a linkage $\Vaadash,$ where $A'$ is a copy of $A,$ symmetric,\\
if $\Vaadash \supseteq \Iaa\bigcap (\Vaadash \circ A)$
and  $\Vaadash=(\Vaadash)_{A'A}.$
\end{definition}
We now have a simple procedure for building a symmetric pseudoidentity for 
a linkage $\Vab.$
\begin{theorem}
\label{thm:pseudoidentity}
Let $\Vaadash $ be a symmetric linkage. We then have the following.
\begin{enumerate}
\item $\Vaadash $ is a pseudoidentity for a linkage $\Vab$ iff
$\Vaadash\circ A\supseteq \Vab \circ A$ and $\Vaadash\times A\subseteq \Vab \times A.$ 
\item
$\Vap\lrar(\Vap)_{A'P},$ where $A'$ is a copy of $A$ is a symmetric 
pseudoidentity for $\Vab,$ if 
$$\Vap\circ A\supseteq \Vab \circ A, \ \ \Vap\times A\subseteq \Vab \times A.$$
\item
$\Vab\lrar(\Vab)_{A'B},$ where $A'$ is a copy of $A$ is a symmetric 
pseudoidentity for $\Vab.$
\end{enumerate}
\end{theorem}
\begin{proof}
\begin{enumerate}
\item ({\bf only if}) Suppose $\Vaadash\circ A\not \supseteq \Vab \circ A.$

Then there exists $(f_A,g_B)\in \Vab$ such that $(f_A,f_{A'})\not\in \Vaadash.$

But then $(f_{A'},g_B) \not\in (\Vaadash \lrar \Vab),$ which contradicts the fact that $\Vaadash \lrar \Vab=(\Vab)_{A'B}.$\\
Next suppose $\Vaadash\times A\not \subseteq \Vab \times A.$

Then there exists $(f_A,0_{A'})\in \Vaadash,$ such that $(f_A,0_{B})\not \in \Vab.$ 

Now $(0_A,f_{A'})\in \Vaadash$ and $(0_A,0_{B})\in \Vab.$ 
So $(f_{A'},0_{B})\in \Vaadash \lrar \Vab=(\Vab)_{A'B},$
a contradiction.

\vspace{0.2cm}

({\bf if}) Let $(f_{A'},g_B) \in (\Vaadash \lrar \Vab).$ 

Then there exist 
$(h_A,f_{A'})\in \Vaadash,\ \  (h_A,g_B)\in \Vab.$ \\
Now $(f_A,f_{A'})\in \Vaadash$ so that  $(h_A-f_A,0_{A'})\in\Vaadash.$\\
Thus $(h_A-f_A)\in \Vaadash\times A\subseteq \Vab\times A.$\\
Since $(h_A,g_B), (h_A-f_A,0_{B})\in \Vab$ and $\Vab $ is a vector space, it follows that
$(f_A,g_B)\in \Vab.$\\
Next let $(f_A,g_B)\in \Vab.$ 

Since $f_A \in \Vab\circ A\subseteq \Vaadash\circ A,$ it follows from the symmetry of $\Vaadash$ that $(f_A,f_{A'})\in \Vaadash.$\\
But then $(f_{A'},g_B) \in (\Vaadash \lrar \Vab).$

It thus follows that $(\Vaadash \lrar \Vab)= (\Vab)_{A'B},$
i.e., $\Vaadash $ is a pseudoidentity for $\Vab.$\\

\item It is easily verified that $\Vap\lrar(\Vap)_{A'P},$ is symmetric.\\
Also $(\Vap\lrar(\Vap)_{A'P})\circ A = \Vap\circ A$
and $(\Vap\lrar(\Vap)_{A'P})\times A=\Vap\times A.$

The result now follows from the preceding section of the present theorem.\\

\item This is immediate from the preceding section of the present theorem.
\end{enumerate}
\end{proof}
\subsection{Special operations for  linkages with two block partitions}
For linkages whose partition has exactly two blocks, there are simple generalizations
of operations on maps which are available. The first is the { transpose}.
\begin{definition}
Let $\Vab$ be a linkage.  The {\bf transpose} of $\Vab, $ $ \Vabt\equiv (\Vabperp)_{A(-B)}.$
\end{definition}
For convenience we define the transpose of $\Va,$ $ (\Va)^T\equiv \Va^{\perp}.$
The following lemma is routine.
\begin{lemma}
\label{lem:transposesign}
\begin{enumerate}
\item  $(\Va)_{-A}= \Va $,
\item $(\Vab )_{A(-B)}= (\Vab )_{(-A)B)}$,
\item $((\Vab )_{A(-B)})^{\perp}= (\Vab ^{\perp})_{A(-B)},$
\item  $(\Vab)_{A(-B)}\circ A= \Vab\circ A$ and $(\Vab)_{A(-B)}\times A= \Vab\times A.$ 
\end{enumerate}
\end{lemma}
Using the results from Subsection \ref{ssec:vspaceresults}  and Lemma \ref{lem:transposesign} we have from the definition of the transpose,
\begin{lemma}
\label{lem:transposetransposeminor}
Let $\Va, \Vab $ be  linkages. Then
\begin{enumerate}
\item $\Va^{TT}=\Va , $  $\Vab^{\transp\transp}= \Vab .$
\item  $(\Vab\circ B)^{T}= \Vab ^{\transp}\times B .$
\end{enumerate}
\end{lemma}
The IDT (Theorem \ref{thm:idt}) could be restated as 
{\theorem{
$$\Vabt\lrar \Vbct \equaln (\Vab\lrar \Vbc )\transp.$$
\label{thm:transpose}}}
\begin{proof}
The left side of the expression in the statement of the theorem is 
$$(\Vab^{\perp})_{A(-B)} \lrar (\Vbc^{\perp})_{(-B)C}\equaln (\Vab^{\perp}) \lrar (\Vbc^{\perp}).$$
The right side of the expression is 
$$(\Vab\lrar \Vbc )^{\perp}_{A(-C)}.$$ Changing the sign of the index set $C$ on both sides,
we need to verify whether 
$$(\Vab^{\perp} \lrar \Vbc^{\perp})_{A(-C)}  \ \textup{equals}\   (\Vab\lrar \Vbc )^{\perp},$$\ i.e.,  whether
$$ \Vab^{\perp} \lrar (\Vbc^{\perp})_{B(-C)}\equaln \Vab^{\perp} \lrar (\Vbc^{\perp})_{(-B)C}\equaln (\Vab^{\perp} \rightleftharpoons \Vbc^{\perp})
\ \textup{equals}\  (\Vab\lrar \Vbc )^{\perp} $$
(we note that  $(\Vab' \lrar (\Vbc')_{(-B)C})\equaln (\Vab' \rightleftharpoons \Vbc')$).

\vspace{0.1cm}

Since this equality holds by  IDT (Theorem \ref{thm:idt}), this completes the proof.
\end{proof}
The above result generalizes the usual transpose rule for linear maps: $(FG)\transp = G\transp F\transp.$ For the `$\lrar $' operation, order does not matter
since the spaces are named with index sets.

For linkages whose partition has exactly two blocks, a useful operation which generalizes the usual sum of maps is the { intersection-sum} defined as follows.
\begin{definition}
Let $\Vab ,\Vhab$ be linkages. The {\bf intersection-sum}
$$\Vab \pb \Vhab \equiv \{(f_A,f_B):(f_A,f^1_B)\in \Vab , (f_A,f^2_B) \in \Vhab, f_B= f^1_B+f^2_B \}.$$
\end{definition}
Such a definition, which speaks both of `domain' and `range' instead of only `range' in the case of
ordinary maps, is natural because for ordinary maps the domain has the form $\F_A$ whereas
for linkages, $\Vab\circ A$ and $\Vhab \circ A $ need not be the same.

An immediate consequence of the definition is the following
\begin{lemma}
\label{lem:intsumasscomm}
The `$\pb$' operation is associative and commutative.

\end{lemma}

We give below a useful way of visualizing the intersection-sum operation.
We remind the reader that   $\{[K]\}$ denotes the space spanned by the rows of the matrix $K.$
Thus $$\{[I_A \ I_{A'}\ I_{A"}]\}$$
denotes the collection of all vectors $$f_{AA'A"},\ \textup{where} \ f_{AA'A"}/A=  (f_{AA'A"}/A')_{A}= (f_{AA'A"}/A")_{A}, \ \ (A,A',A"\ \textup{being\  copies\  of\  each\  other}).$$
\begin{align}
\left\{
\left[
\begin{matrix}
I_A & I_{A'}& 0\\
I_A & 0       & I_{A'}
\end{matrix}
\right]
\right\}
\end{align}
denotes the collection of all vectors  
 $f_{AA'A"},$ where  $f_{AA'A"}/A=  (f_{AA'A"}/A')_{A}
+ (f_{AA'A"}/A")_{A}.$

It is clear that $\Vab\lrar \{[I_A \ I_{A'}]\}= (\Vab)_{A'B}.$ 

Next consider $\Vadashbdash\bigoplus\Vhatwodashbtwodash.$ This space is a direct sum of 
copies of $\Vab,\Vhab$ built respectively on disjoint copies $A'\cup B'$ and $A"\cup B"$ of $A\cup B.$ The space
$$\Vadashbdash\bigoplus\Vhatwodashbtwodash\lrar \{[I_A \ I_{A'}\ I_{A"}]\}\bigoplus 
\left\{
\left[
\begin{matrix}
I_B & I_{B'}& 0\\
I_B & 0       & I_{B"}
\end{matrix}
\right]
\right\}
$$
is on the set $A\cup B$ and can be seen to be the collection 
$$\{(f_A,f_B):(f_A,f^1_B)\in \Vab , (f_A,f^2_B) \in \Vhab, f_B= f^1_B+f^2_B \}.$$
[The  `$\lrar $' operation with $ \{[I_A \ I_{A'}\ I_{A"}]\}\bigoplus 
\left\{
\left[
\begin{matrix}
I_B & I_{B'}& 0\\
I_B & 0       & I_{B'}
\end{matrix}
\right]
\right\}$ forces equality of the $A$ components and addition of the $B$ components.]

We thus have the following lemma.
\begin{lemma}
\label{lem:intsumvis}
$$\Vab \pb \Vhab = \Vadashbdash\bigoplus\Vhatwodashbtwodash\lrar \{[I_A \ I_{A'}\ I_{A"}]\}\bigoplus 
\left\{
\left[
\begin{matrix}
I_B & I_{B'}& 0\\
I_B & 0       & I_{B"}
\end{matrix}
\right]
\right\}.$$
\end{lemma}

The following lemma is easy to see
\begin{lemma}
\label{lem:specialperps}
$\{[I_A \ I_{A'}\ I_{A"}]\}^{\perp}=$
 \begin{align}
\left\{
\left[
\begin{matrix}
I_A & -I_{A'}& 0\\
I_A & 0       &- I_{A'}
\end{matrix}
\right]
\right\}
\end{align}
\end{lemma}

By IDT we get  the following generalization of $(F+G)^T=F^T +G^T.$\\
(Note that $(\Vsp\lrar \Vp)^{\perp}=\Vsp^{\perp}\lrar \Vp^{\perp}.$)
\begin{theorem}
\label{thm:intsumtranspose}
\begin{enumerate}
\item $$(\Vab \pb \Vhab )^{\perp}=
\Vab^{\perp} \pa \Vhab^{\perp}.$$
\item 
$(\Vab \pb \Vhab )^{\transp}=\Vab^{\transp} \pa \Vhab^{\transp}.$

\end{enumerate}
\end{theorem}
\begin{proof}
\begin{enumerate}
\item Noting that $(\Vsp\lrar \Vp)^{\perp}=\Vsp^{\perp}\lrar \Vp^{\perp},$
$$(\Vab \pb \Vhab )^{\perp}$$
$$=
\Vadashbdash^{\perp}\bigoplus\Vhatwodashbtwodash^{\perp}\lrar 
\left\{
\left[
\begin{matrix}
I_A & -I_{A'}& 0\\
I_A & 0       & -I_{A"}
\end{matrix}
\right]
\right\}
\bigoplus 
\{[I_B \ -I_{B'}\ -I_{B"}]\} 
$$
$$=(\Vab^{\perp} \pa \Vhab^{\perp})_{(-A)(-B)}.$$
$$=\Vab^{\perp} \pa \Vhab^{\perp}.$$
\item $(\Vab \pb \Vhab )^{\transp}=(\Vab \pb \Vhab )^{\perp}_{A(-B)}=(\Vab^{\perp} \pa \Vhab^{\perp})_{A(-B)}$
$$=   (\Vab^{\perp})_{A(-B)}\pa(\Vhab^{\perp})_{A(-B)}= \Vab^{\transp} \pa \Vhab^{\transp}.$$

\end{enumerate}
\end{proof}
\subsection{Scalar multiplication for linkages}
Let $\Vab $ be a linkage. We define 
$$\lambda ^b\Vab \equiv \V_{A(\lambda B)}+  \Vab\times B\equiv \{(f_A,\lambda g_B): (f_A, g_B)\in \Vab \}+ \Vab\times B.$$
 Note that, when $\lambda \ne 0,$ we have $\lambda ^b\Vab \equiv \V_{A(\lambda B)}= (\Vab\lrar \{[I_B \ \ \lambda I_{B'}]\})_{AB},$
and more generally, if we include the case of $\lambda = 0,$
we have $\lambda ^b\Vab = (\Vab\lrar \{[I_B \ \ \lambda I_{B'}]\})_{AB}+ \Vab\times B.$
\begin{remark}
We note that when $\lambda \ne 0,$ the above definition could have been 
rewritten more simply as\\
 $\lambda ^b\Vab \equiv  \{(f_A,\lambda g_B): (f_A, g_B)\in \Vab \},$
since $\0_A\oplus \Vab\times B\subseteq \{(f_A,\lambda g_B): (f_A, g_B)\in \Vab \},$ when $\lambda \ne 0.$ The contraction to $B,$ of $\{(f_A,\lambda g_B): (f_A, g_B)\in \Vab \},$ when $\lambda \ne 0,$ is simply $\Vab\times B.$

However when $\lambda =0, $ the vector space $\{(f_A,\lambda g_B): (f_A, g_B)\in \Vab \}$ (which is the first term in the right side of the definition) is equal to $\{(f_A,0_B): (f_A, g_B)\in \Vab \}.$ 
This does not contain $\0_A\oplus \Vab\times B$ and its contraction to $B$ is simply $\0_B.$ 

For reasons pertaining to the form of the duality result we state below, we prefer $\Vab\times B$  to be the contraction of $\lambda ^b\Vab $ to $B$ even when $\lambda =0.$ Hence we have explicitly added the term $\Vab\times B$ 
in the definition of $\lambda ^b\Vab $ when $\lambda =0.$ 
Thus $0^{\dws}\Vwdw = \Vwdw\circ W\oplus\Vwdw \times \dw.$
\end{remark}
\begin{theorem}
\label{thm:scalarmult}
Let $\Vab $ be a linkage. Then
\begin{enumerate}
\item
$$(\lambda_1 ^b+\lambda_2 ^b)\Vab= \lambda_1 ^b\Vab\  \pb\  \lambda_2 ^b\Vab.$$
\item
$$(\lambda ^b\Vab)^T = \lambda ^a\Vab^T.$$
\end{enumerate}
\end{theorem}
\begin{proof}
\begin{enumerate}
\item This is routine.
\item
Let $\lambda =0.$ The LHS reduces to $(\Vab \circ A \bigoplus \Vab \times B) ^{T}= (\Vabt \times A \bigoplus \Vabt \circ B),$
which is the RHS by definition of $\lambda ^a\Vab^T.$

Next let  $\lambda \ne 0.$ We have $\lambda ^b\Vab \equiv \V_{A(\lambda b)}= (\Vab\lrar \{[I_B \ \ \lambda I_{B'}]\})_{AB}.$ \\
Therefore by Theorem \ref{thm:transpose}, 
$$(\lambda ^b\Vab)^T =( (\Vab\lrar \{[I_B \ \ \lambda I_{B'}]\})_{AB})^T=(\Vabt \lrar \{[I_B \ \ \lambda I_{B'}]\}^T)_{AB}$$
$$= (\Vabt \lrar \{[\lambda I_B \ \  I_{B'}]\})_{AB}= \lambda ^a\Vab^T,$$ where we have made use of the fact that $\{[I_B \ \ \lambda I_{B'}]\}^T= \{[\lambda I_B \ \  I_{B'}]\}.$
\end{enumerate}

\end{proof}

\subsection{Distributivity of intersection-sum over matched composition  and scalar multiplication}
For ordinary maps we have the distributivity of addition over composition and scalar multiplication.
Because this result is true we are able to develop the rich spectral theory of operators.
In subsequent sections we will be speaking of generalized operators and polynomials of
such objects. So we need to study the distributivity of intersection-sum over `$\lrar $'.
The following result is the best that can be done, but proves adequate for our purpose.
\begin{theorem}
\label{thm:distributivity}
Let $\Vab, \Vhab,\Vbc,$  be linkages, $A,B,C,$ being pairwise disjoint. 
\begin{enumerate}
\item If $\Vbc \times B\subseteq \Vab \times B,$ then
$$ \Vbc\lrar(\Vab \pa \Vhab)= ( \Vbc\lrar\Vab) \pa   (\Vbc\lrar \Vhab).$$
\item If $\Vbc \times B\subseteq \Vab ^1 \times B,\cdots \Vbc \times B\subseteq \Vab ^{k-1} \times B,$ then
$$ \Vbc\lrar(\Vab ^1\pa \cdots \pa \Vab ^{k})=( \Vbc\lrar\Vab^1)\pa \cdots \pa  (\Vbc\lrar\Vab ^k).$$  
\item Dually, if $\Vbc \circ B\supseteq \Vab \circ B,$ then
$$ \Vbc\lrar(\Vab \pb \Vhab)=  (\Vbc\lrar\Vab )\pc   (\Vbc\lrar  \Vhab).$$
\item  If $\Vbc \circ B\supseteq \Vab ^1 \circ B,\cdots \Vbc \circ B\supseteq \Vab ^{k-1} \circ B,$ then
$$ \Vbc\lrar(\Vab ^1\pb \cdots \pb \Vab ^{k})=( \Vbc\lrar\Vab^1)\pc \cdots \pc  (\Vbc\lrar\Vab ^k).$$ 

\end{enumerate}
\end{theorem}
\begin{proof}
\begin{enumerate}
\item Let $(f_A,h_C)\in \Vbc\lrar(\Vab \pa \Vhab)$. 
Then there exist vectors $$(g_B,h_C)\in \Vbc , (f^1_A,g_B)\in \Vab,
(f^2_A,g_B)\in \Vhab , \ \textup{ such\  that \ } f_A= f^1_A+f^2_A.$$ 
This means $(f^1_A,h_C)\in 
 (\Vbc\lrar\Vab )$ and $(f^2_A,h_C)\in  (\Vbc\lrar\Vhab )$ and therefore
 $$(f_A,h_C)\in( \Vbc\lrar\Vab) \pa   (\Vbc\lrar \Vhab).$$
Next let $(f_A,h_C)\in( \Vbc\lrar\Vab) \pa   (\Vbc\lrar \Vhab).$  This means 
there exist 
vectors $$(f^1_A,h_C),  (f^2_A,h_C)\in  (\Vbc\lrar\Vhab ),
 (\Vbc\lrar\Vab ) \ \textup{with}\  f_A= f^1_A+f^2_A.$$
Therefore there exist
 vectors $(f^1_A,f^1_B), (f^1_B,h_C)\in \Vab, \Vbc$
respectively and vectors \\
$(f^2_A,f^2_B), (f^2_B,h_C)\in \Vhab, \Vbc,$
respectively. \\Therefore $(f^1_B-f^2_B,0_C)\in \Vbc,$ i.e.,
$f^1_B-f^2_B\in \Vbc\times B\subseteq \Vab\times B.$ \\Hence, $(0_A,f^1_B-f^2_B)\in \Vab.$ Now $(f^1_A,f^1_B)\in \Vab.$ 
Therefore $(f^1_A,f^2_B)\in \Vab.$ \\
We already have that $(f^2_A,f^2_B)\in \Vhab .$ \\
Hence $$(f^1_A+f^2_A,f^2_B)= (f_A,f^2_B)\in (\Vab \pa \Vhab).$$
We know that $(f^2_B,h_C)\in \Vbc.$
 Therefore $$(f_A,h_C)\in \Vbc\lrar(\Vab \pa \Vhab).$$
\item The proof is by induction on the number of terms within the bracket on the left hand side.
\item By the first part of the theorem  we have 
 if $\Vbc^{\transp} \times B\subseteq \Vab^{\transp} \times B,$ then
$$ \Vbc^{\transp}\lrar(\Vab^{\transp} \pa \Vhab^{\transp})= ( \Vbc^{\transp}\lrar\Vab^{\transp}) \pa   (\Vbc^{\transp}\lrar \Vhab^{\transp}).
$$
Hence, 
if $(\Vbc^{\transp} \times B)^{\transp}\supseteq (\Vab^{\transp} \times B)^{\transp},$ then
$$ (\Vbc^{\transp}\lrar(\Vab^{\transp} \pa \Vhab^{\transp}))^{\transp}= (( \Vbc^{\transp}\lrar\Vab^{\transp}) \pa   (\Vbc^{\transp}\lrar \Vhab^{\transp}))^{\transp},
$$
i.e., if 
$\Vbc^{\transp\transp} \circ B\supseteq \Vab^{\transp\transp}  \circ B,$ then
$$ \Vbc^{\transp\transp} \lrar(\Vab^{\transp\transp}  \pb \Vhab^{\transp\transp} )= ( \Vbc^{\transp\transp
} \lrar\Vab^{\transp\transp} ) \pc   (\Vbc^{\transp\transp} \lrar \Vhab^{\transp\transp} ),  
$$ i.e, if  $\Vbc \circ B\supseteq \Vab \circ B,$ then
$$\Vbc\lrar(\Vab \pb \Vhab)=  (\Vbc\lrar\Vab )\pc   (\Vbc\lrar  \Vhab).$$
\item The proof is by induction on the number of terms within the bracket on the
 left hand side.
\end{enumerate}
\end{proof}
\begin{remark}
To see that the statement of the theorem fails without the `dot-cross' conditions,
take $\Vbc$ to be spanned by $(g^1_B,f_C), (g^2_B, f_C),$   $\Vab$ to be spanned by $(h^1_A,g^1_B),$ $\Vhab$ to be spanned by $(h^2_A, g^2_B),$ and take $g^1_B\ne g^2_B, h^1_A + h^2_A \ne 0_A.$

It can be seen that in this case 
$$ \Vbc\lrar(\Vab \pa \Vhab)\ne ( \Vbc\lrar\Vab) \pa   (\Vbc\lrar \Vhab).$$

\end{remark}
The corresponding result for scalar multiplication is straight forward
and we state it but skip its proof.

\begin{theorem}
$$ \lambda ^b(\Vab \pb \Vhab)=  (\lambda ^b\Vab )\pb (\lambda ^b\Vhab),$$
\end{theorem}

\section{Dynamical Systems}
\label{subsec:DynSys}

The algebraic operations we define are natural for dynamical systems and will be illustrated for them. While the motivation is through solutions of linear constant 
coefficient differential equations, the algebraic results are true for 
vectors over arbitrary (including finite) fields.
 The preliminary definitions from
\cite{HNPS2013} are listed below.

A generalized dynamical system with additional variables (GDSA) is a vector space $\mathcal{V}_{W\mydot{W}AM }$ on the set $W\uplus\mydot{W}  \uplus  A \uplus M$ where $W,\mydot{W}$,  are copies of each other (intended to take care of the dynamical variables and their derivatives), and $A$, $M$, the additional variables and the manifest variables  respectively. Thus $w(t): W \rightarrow \mathbb{R}$, $\mydot{w}(t): \mydot{W} \rightarrow \mathbb{R}$, $a(t): A \rightarrow \mathbb{R}$, $m(t): M \rightarrow \mathbb{R}$. We write $w(t)$ etc., simply as $w$ etc. The manifest variables are the `external variables' of the system. Usually a GDSA would be available to us in the form of the solution space of constraint equations
\begin{align}
\label{eqn:GenDynSys}
\bbmatrix{ S_{\mydot{W}}  & S_W & S_A & S_M
}
\ppmatrix{
        \mydot{w} \\ w\\ a\\ m
} = 0.
\end{align}
Here $(w(t), \mydot{w}(t),a(t),m(t))$ is a real vector which is a solution
to the algebraic equation \ref{eqn:GenDynSys}. That this arises from a set of linear constant coefficient differential equations with solutions
of the kind $(w(\cdot), \mydot{w}(\cdot),a(\cdot),m(\cdot))$
is not used strongly in the present paper except in the examples and
in the definition of `regular generalized dynamical system' given below.

We write $\mathcal{V}_{W\mydot{W}AM }$ alternatively as $\mathcal{V}(w,\mydot{w},a,m)$.

We define a generalized dynamical system (GDS) as a vector space $\mathcal{V}_{W\mydot{W}M }$ on the set $W\uplus \mydot{W} \uplus  M$ where $W,\mydot{W},$ are copies of each other. This is obtained from a GDSA by suppressing the additional variables that is, 
\begin{align*}
 \mathcal{V}_{W\mydot{W}M } = \mathcal{V}_{W\mydot{W}AM }\circ W\mydot{W}  M .
\end{align*}

We study  a GDS (or a GDSA) `relative to' a vector space $\mathcal{V}_M$. This vector space is expected to capture (and is a generalization of) the usual partition of the manifest variables into `input' variables $m_u$ and `output' variables $m_y$. 
We define a generalized autonomous system ({\bf genaut}) as a vector space $\mathcal{V}_{W\mydot{W} }$ on the set $W\mydot{W} $ where $W,\mydot{W}$, are copies of each other.
A genaut $\Vwdw$ could be obtained from a GDS by taking a generalized minor
with respect to $\Vm,$ i.e., $\Vwdw=\Vwdwm \lrar \Vm.$
Often GDSs arise with a natural partition $\{M_u,M_y\}$ of the index set $M, $ corresponding to the input variables
$m_u$ and the output variables $m_y.$ 
We will call such a GDS, an input- output GDS. 
An example of such an input-output GDS would be the  system $\Vwdwmumy$ defined by
\begin{subequations}
\label{eqn:StateEqnsDynSysL}
\begin{align}
 L\mydot{w} &= Aw + Bm_u \\
 m_y &= Cw + Dm_u
\end{align}
\end{subequations}
If we allow the variables $m_u ,m_y$ to be free we would get the 
genaut $\Vwdw \equiv \Vwdwmumy\circ W \dw .$
\begin{definition}
We call a GDS $\Vwdwm,$ {\bf regular}, iff it satisfies\\
\begin{subequations}
\label{eqn:regularGDS}
\begin{align}
\Vwdwm\circ W &\supseteq (\Vwdwm\circ \dw)_W\\
\Vwdwm\times W &\supseteq (\Vwdwm\times \dw)_W
\end{align}
\end{subequations}
We call an input-output GDS $\Vwdwmumy,$ {\bf regular}, iff 
 $\Vwdwm , \ \ M\equiv M_u\uplus M_y,$ is regular.
\end{definition}
\begin{remark}
The conditions in Equations \ref{eqn:regularGDS} arise in the case 
of infinitely differentiable solutions to linear constant coefficient 
differential equations. Suppose $(w(\cdot),\mydot{w}(\cdot),m(\cdot))$ is a solution 
to the equations. Then clearly so would $(\mydot{w}(\cdot),\mydot{\mydot{w}}(\cdot),\mydot{m(\cdot)})$ be a solution. From this it is clear that whenever 
$\mydot{w}(t)\in  (\Vwdwm\circ \dw)_W$ we also have that $\mydot{w}(t)\in 
\Vwdwm\circ W.$ The first condition now follows because we can take any 
arbitrary real vector $x$ on $W$ or $\dw$ to be the vector  $\mydot{w}(t).$ 
The second condition arises if we insist that the adjoint (see Section \ref{sec:adjointGDS}
 for a description)
of a regular system be regular. This is desirable from duality considerations.
\end{remark}
\begin{definition}
A genaut $\Vwdw $ is a generalized operator ({\bf genop}) iff
\begin{equation}
\label{eqn:genop}
\Vwdw\circ W \supseteq (\Vwdw\circ \dw)_W \ and \ \Vwdw\times W \supseteq (\Vwdw\times \dw)_W.
\end{equation}

A genaut that satisfies $\Vwdw\circ W \supseteq (\Vwdw\circ \dw)_W$ is called an upper semi genop or {\bf USG} for short and 
if it satisfies
 $\Vwdw\times W \supseteq (\Vwdw\times \dw)_W,$ is called a lower semi genop or {\bf LSG} for short.  

\end{definition}
As one can see, however, genops are weaker than the usual operators.
Equation \ref{eqn:genop}
does not force the collection of permissible $w$ vectors
to be $\F_W$ and $w=0$ does not imply $\dwsmall =0.$
We show below that genops occur naturally. One can eliminate variables
and write state equations for such systems but this would be computationally expensive, destroy natural (topological ) structure making subsequent computations also expensive.
Fortunately, however, for all practical purposes, genops are sufficient.
In particular, we can define polynomials of genops in a manner similar to
those of operators, we can compute minimal annihilating polynomials etc.
as we show in Section \ref{sec:genoppoly}
.


An example of a  dynamical system with additional variables 
is  a linear electrical network with dynamic devices (capacitors ($C$) and inductors ($L$)), static devices (resistors and controlled sources ($r$)), voltage and current sources ($E$ and $J$). We will call such a network an $RLCEJ$ network. The variables associated with the devices are  $v_C$, $i_C$, $v_L$, $i_L$, $v_r$, $i_r$, $v_E$, $i_E$, $v_J$, $i_J$. In addition we can introduce output variables. Voltage outputs are obtained by introducing open circuit branches across existing nodes and sensing the voltage across them. Current outputs are obtained by splitting an existing node into two, introducing a short circuit between the split nodes and sensing the current through the short circuit. Example of such a network is given in Figure \ref{fig:circuit7}.

If all the $R,L,C$ are positive  (symmetric positive definite matrix for 
$L$ in general) and voltage sources and current sensors do not form 
loops and current sources and voltage sensors do not form cutsets,
then one can show that for arbitrary values of  sources and consistent initial conditions (capacitor voltages according to capacitor loops and inductor currents according to inductor cutsets), the solution 
is unique whether or not the sources are zero. From this, it follows that Equation \ref{eqn:regularGDS} (b) is satisfied trivially.
 The constraints on $W$ variables are the initial condition constraints and  are only topological
whereas the constraints on $\dw$ variables include the above topological conditions and additionally also other network conditions.
Therefore the corresponding $GDS$ satisfies Equation \ref{eqn:regularGDS} (a)
and thus the system 
is regular.
This argument is also sufficient to show that the $RLCEJ$ network (with the sources present) is regular.
(The adjoint of this system will satisfy Equation \ref{eqn:regularGDS} (a) trivially but Equation \ref{eqn:regularGDS} (b)
nontrivially.)
\begin{figure}[ht!]
\psfrag{C1}[][][1]{$C_1$}
\psfrag{C2}[][][1]{$C_2$}
\psfrag{C3}[][][1]{$C_3$}

\psfrag{L13}[][][1]{$L_{13}$}
\psfrag{L14}[][][1]{$L_{14}$}
\psfrag{L15}[][][1]{$L_{15}$}

\psfrag{R4}[][][1]{$R_4$}
\psfrag{R6}[][][1]{$R_6$}
\psfrag{R10}[][][1]{$R_{10}$}
 
%
%
%
%
%
\begin{center}
\includegraphics[width=6in]{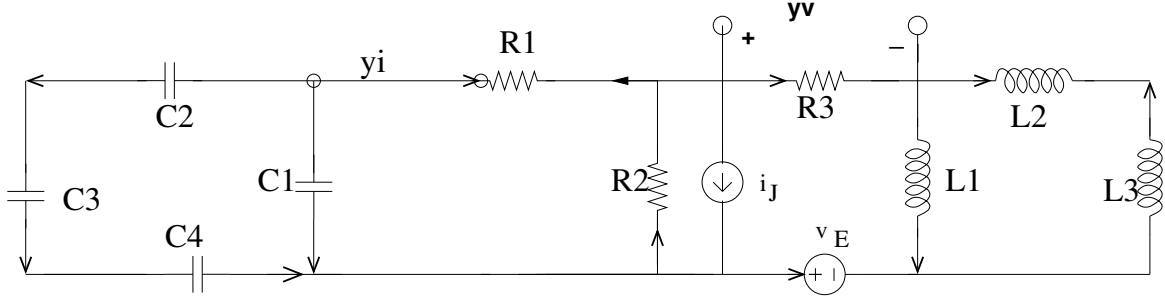}
\caption{An $RLCEJ$ Network 
}
\label{fig:circuit7}
\end{center}
\end{figure}

%
%
%
%
%
%
%
%
%


%
%

The  $RLCEJ$ network of Figure \ref{fig:circuit7}, if treated as a GDSA $\mathcal{V}(w,\mydot{w},l_{\mydot{w}},l_r,m)$  ($l_{\mydot{w}},l_r$ being the additional variables), has the variables $w$, $\mydot{w}$, $l_{\mydot{w}}$, $l_r$, $m$ as shown below:
\begin{align*}
  {w} &\equiv (v_{C1},v_{C2},v_{C3},v_{C4}, i_{L1},i_{L2},i_{L3}),\qquad  {\mydot{w} } \equiv (\mydot{v}_{C1},\mydot{v}_{C2},\mydot{v}_{C3}, \mydot{v}_{C4},\frac{di_{L1}}{dt},\frac{di_{L2}}{dt},\frac{di_{L3}}{dt}), \\
  {l}_r &\equiv (i_{R1},i_{R2},i_{R3},i_E,i_{yv},v_{yi},v_{R1},v_{R2},v_{R3},v_J), \\
  {l}_{\mydot{w}} &\equiv (i_{C1},i_{C2},i_{C3},i_{C4},v_{L1},v_{L2},v_{L3}), \\
  {m} &\equiv ( m_{u}, m_{y}), \qquad  m_{u} \equiv (i_J,v_E), \qquad
  m_{y} \equiv (yi,yv).
\end{align*}
Note that $l_r$ also includes the variables 
associated with sources and outputs  other than the manifest variables, such as, for instance, currents of voltage sources or voltages of current sensors.

Now 
\begin{align*}
\mathcal{V}(w,\mydot{w},l_{\mydot{w}},l_r,m) =  \mathcal{V}^{top}(w,\mydot{w},l_{\mydot{w}},l_r,m) \cap \mathcal{V}^{dev}(\mydot{w},l_{\mydot{w}},l_r) 
\end{align*}
where
\begin{align*}
 \mathcal{V}^{top}(w,\mydot{w},l_{\mydot{w}},l_r,m) \equiv \mathcal{V}^i(\mathcal{G}) \oplus \mathcal{V}^v(\mathcal{G}) \oplus \mathcal{V}^{top}_{\mydot{w}},
\end{align*}
$\mathcal{V}_i(\mathcal{G})$ and $\mathcal{V}^v(\mathcal{G})$ denote the current and voltage spaces of $\mathcal{G}$ and $\mathcal{V}^{top}_{\mydot{w}}$ is the space of vectors $ {\mydot{w}}$ which satisfy the topological conditions on $\mydot{w}$. In general the vector $ {\mydot{v}}_C$ satisfies, just as ${v}_C$ does, the KVL conditions of the graph obtained by open circuiting all branches other than capacitor branches and the vector $ \frac{di_L}{dt}$ satisfies, just as ${i}_L$ does, the  KCL conditions of the graph obtained by short circuiting all branches other than the inductor branches. The device characteristics constraints are
\begin{align*}
 \mathcal{V}^{dev}(\mydot{w},l_{\mydot{w}},l_r)  = \mathcal{V}^{dev}(\mydot{w},l_{\mydot{w}})  \oplus \mathcal{V}^{dev}(l_r),
\end{align*}
where $\mathcal{V}^{dev}(\mydot{w},l_{\mydot{w}})$ is the solution space of the equations $i_C = C\mydot{v}_C$, $v_L = L\frac{di_L}{dt}$, $C$, $L$ being symmetric positive definite matrices. (Observe $w$ variables are not involved in these equations). $\mathcal{V}^{dev}(l_r)$ is the solution space of the equations (in the case of the present network)
\begin{align*}
  {v}_R = R {i}_R,\qquad v_{yi} = 0, \qquad i_{yv} = 0.
\end{align*}
It is clear that the device characteristics constraints on $(\mydot{w},l_{\mydot{w}})$ do not involve $l_r$ and vice versa.

Studying this network relative to a specified $\mathcal{V}_M$
can be illustrated as follows: Let $\mathcal{V}_M\equiv 0_{mu} \oplus \mathcal{F}_{my}$. In this case this amounts to setting the sources to zero and the outputs free. If we work with $\mathcal{V}_M^\perp \equiv \mathcal{F}_{mu} \oplus 0_{my}$ it amounts to keeping the sources free but the output zero. In this case, note that the output branches have both current and voltage zero.
\section{Duality}
\label{sec:Duality1}
\subsection{Perpendicular duality}
\label{sec:perp_dual}
Duality arises in implicit linear algebra in a very natural manner.
The key idea for us is the following: Let
\begin{align*}
 \epsilon(\mathcal{V}_1,\ldots,\mathcal{V}_k,+,\cap,\oplus ,\circ , \times ,\leftrightarrow,\rightleftharpoons,\supseteq,\subseteq, =)
\end{align*}
be called the `primal' statement.. 
Consider the second statement
\begin{align*}
\epsilon^2(\mathcal{V}_1,\ldots,\mathcal{V}_k,+,\cap,\oplus ,\circ , \times ,\leftrightarrow,\rightleftharpoons,\supseteq,\subseteq, =)\ \ 
 \equiv \ \  \epsilon(\mathcal{V}_1^\perp,\ldots,\mathcal{V}_k^\perp,\cap,+,\oplus ,\times , \circ , \rightleftharpoons,\leftrightarrow,\subseteq,\supseteq, =) 
\end{align*}
which is obtained by replacing $\mathcal{V}_i$ by $\mathcal{V}_i^\perp$,\ \ \   $+$ by $\cap$,\ \ \    $\cap$ by $+$,\ \ \   $\oplus $ by $\oplus ,$\ \ \  $\circ $ by $\times ,$\ \ \  $\times$ by $\circ,$ interchanging
$\leftrightarrow,\rightleftharpoons ,$ interchanging $\supseteq$, $\subseteq,$  and retaining $=$ as  $=.$  
Any left and right brackets that occur
in the primal statement are left unchanged in the second. 
We have \\ \indent 
$$\epsilon(\mathcal{V}_1,\ldots,\mathcal{V}_k,+,\cap,\oplus ,\circ , \times ,\leftrightarrow,\rightleftharpoons,\supseteq,\subseteq,=)$$
is true \ \ iff \ \ 
 $$\epsilon(\mathcal{V}_1^\perp,\ldots,\mathcal{V}_k^\perp,\cap,+,\oplus ,\times , \circ , \rightleftharpoons,\leftrightarrow,\subseteq,\supseteq,=)$$ 
is true.\\
(If we assume that the connective `$\leftrightarrow$' is between spaces of the form $\mathcal{V}_S$, $\mathcal{V}_T$ where $S\supset T$, then there is no need to introduce `$\rightleftharpoons$'. We can simply replace `$\leftrightarrow$' by
`$\leftrightarrow$' in the second statement.)

This fact is a routine consequence of the basic results $$ (\mathcal{V}^\perp)^\perp = \mathcal{V};\ \  \ (\mathcal{V}_1+\mathcal{V}_2)^\perp = \mathcal{V}_1^\perp \cap \mathcal{V}_2^\perp;\ \  \ (\mathcal{V}_1 \leftrightarrow \mathcal{V}_2)^\perp = (\mathcal{V}_1^\perp \rightleftharpoons \mathcal{V}_2^\perp).$$


The statement
$$\epsilon(\mathcal{V}_1,\ldots,\mathcal{V}_k,\cap,+,\oplus ,\times , \circ , \rightleftharpoons,\leftrightarrow,\subseteq,\supseteq,=),$$
obtained by replacing each $\V_i$
by $\V_i^\perp $ in the second statement
$$\epsilon^2(\mathcal{V}_1,\ldots,\mathcal{V}_k,+,\cap,\oplus ,\circ , \times ,\leftrightarrow,\rightleftharpoons,\supseteq,\subseteq, =)$$
is called the `$\perp -$ dual' statement.


\begin{example}
The second statement  of 
\begin{equation}
\label{eqn:emulatorcondinv}
[(\Vonewp \bigoplus \Vtwodwdp)\lrar \Vwdw ] \lrarn  [\Vonewp\lrar\V_W]_{\mydot{P}}\ \supseteqn \ [\Vonewp\lrar\Vw],
\end{equation}
is 
\begin{equation}
\label{eqn:emulatorcontinv}
[((\Vonewp) ^{\perp}\bigoplus (\Vtwodwdp)^{\perp})\lrar \Vwdw ^{\perp}] \lrar  [(\Vonewp)^{\perp}\lrar\V_W^{\perp}]_{\mydot{P}}\ \subseteqn \ [(\Vonewp)^{\perp}\lrar\Vw^{\perp}].
\end{equation}
The $\perp -$ dual statement is 
\begin{equation}
\label{eqn:emulatorcontinv2}
[(\Vonewp \bigoplus \Vtwodwdp)\lrar \Vwdw ] \lrarn  [\Vonewp\lrar\V_W]_{\mydot{P}}\ \subseteqn\  [\Vonewp\lrar\Vw].
\end{equation}
\end{example}
In the above example, it is not necessary to replace $\lrar$ of the primal with `$\rightleftharpoons$' in the second or dual, since
one of the index sets of vector spaces, linked by the operation, 
is contained in the other.

Now suppose that the primal statement is a general statement about vector
spaces (using the permissible symbols) which is true even if each $\V_i$ were replaced
by $\V_i^\perp .$
Then,  the `$\perp -$ dual' statement is true iff the primal statement is true.
Clearly, this also follows from the basic results mentioned before.

What we have described is the most basic form of duality. Often for convenience 
one makes additional changes on the dual to retain some desirable property.
A good example is the `transpose'. Recall that $\Vabt\equiv(\Vab^{\perp})_{(-A)B}.$ 
The adjoint defined in the next subsection is an extension of the same idea.
\subsection{Adjoint duality}
\label{sec:adjointGDS}
Adjoint dynamical systems have been  natural constructs in mathematics
and have been an important component of mutivariable control theory,
ever since Kalman's seminal work  \cite{kalman}.
In our notation all the essential features of the notion of adjoint
come through in a neat way. We begin with the definition
in this section and will describe the duality properties at appropriate
places.
In what follows we have omitted the brackets in a term like $(\Vab)_{(-A)B}$ and written it
more simply as $(\Vab)_{-AB}.$
\begin{definition}
Let $\mathcal{V}_{W\mydot{W}M }$ be a GDS. Let the set $M$ be partitioned as $M_u\uplus \M_y.$
We define the adjoint of $\mathcal{V}_{W\mydot{W}M_uM_y}$ as
$$(\mathcal{V}_{{W}\mydot{{W} }M_uM_y })^a\equiv\mathcal{V}^{a}_{{W'}\mydot{{W'} }M'_uM'_y }\equiv
(\mathcal{V}^\perp_{W\mydot{W}M_uM_y })_{-\mydot{{W'} }{W'}-M'_yM'_u },
$$
where  $W',\mydot{{W'} }$ are disjoint sets which are copies of $W$ and $M'_u ,M'_y$ are disjoint
sets \\
(also disjoint from $W',\mydot{{W'} }$ ), which are copies respectively of $M_y,M_u.$
The adjoint 
$(\mathcal{V}^{a}_{{W'}\mydot{{W'} }M'_uM'_y })^a$ of $\mathcal{V}^{a}_{{W'}\mydot{{W'} }M'_uM'_y }$ is defined to be on $W\uplus \mydot{{W} }\uplus M_u\uplus M_y,$ with the same correspondence between primed and unprimed quantities as before.

Let $\mathcal{V}_{W\mydot{W} }$ be a generalized autonomous system.
We define its adjoint as
$$(\mathcal{V}_{{W}\mydot{{W} } })^a\equiv
\mathcal{V}^{a}_{{W'}\mydot{{W'} } }\equiv
(\mathcal{V}^\perp_{W\mydot{W} })_{-\mydot{{W'} }{W'} }.
$$
\end{definition}
The following theorem  lists some routine properties of the adjoint
which one expects should be true if the name `adjoint' is to be appropriate.

\begin{theorem}
\label{thm:adjointprop}
Let $\mathcal{V}_{W\mydot{W}M_uM_y }$ be a GDS, 
let $\mathcal{V}_{W\mydot{W} }$ be a vector space of the dynamic variables and let $\V_{M_uM_y }$ be a vector space of
the manifest variables.
We have 
\vspace{0.25cm}
\begin{enumerate}

\item
$\mathcal{V}_{W\mydot{W}M_uM_y  } = (\mathcal{V}^{a}_{{W'}\mydot{{W'} }M'_uM'_y })^a.$

\vspace{0.25cm}
\item
$$\mathcal{V}_{W\mydot{W} }\equiv \mathcal{V}_{W\mydot{W}M_uM_y } \lrar \V_{M_uM_y }$$
iff
$$\mathcal{V}^a_{W'\mydot{W'} }=\mathcal{V}^{a}_{{W'}\mydot{{W'} }M'_uM'_y }\lrar (\V_{M_uM_y}^{\perp} )_{-M'_yM'_u} .$$

\vspace{0.25cm}
\item
$\mathcal{V}_{W\mydot{W} } = (\mathcal{V}^{a}_{{W'}\mydot{{W'} } })^a.$

\vspace{0.25cm}
\item
$\mathcal{V}_{W\mydot{W} }$ is a USG iff  $\mathcal{V}^{a}_{{W'}\mydot{{W'} } }$ is an LSG.\\
\item $\mathcal{V}_{W\mydot{W} }$ is a genop iff  $\mathcal{V}^{a}_{{W'}\mydot{{W'} } }$ is a genop.

\vspace{0.25cm}
\item $\mathcal{V}_{W\mydot{W} }$ is decoupled iff $\mathcal{V}^{a}_{{W'}\mydot{{W'} } }$ is decoupled.
\end{enumerate}
\end{theorem}
\begin{proof}
\begin{enumerate}
\item We have
$$(\mathcal{V}^{a}_{{W'}\mydot{{W'} }M'_uM'_y })^a
= ((\mathcal{V}^\perp_{W\mydot{W}M_uM_y })_{-\mydot{{W'} }{W'}-M'_yM'_u })^\perp_{-W -\mydot{{W} }-M_u-M_y }$$
$$= (\mathcal{V}_{W\mydot{W}M_uM_y  })_{-W -\mydot{{W} }-M_u-M_y }= \mathcal{V}_{W\mydot{W}M_uM_y  },$$
(noting that $(\V_Z)_{-Z}=\V_Z)$).
\item We have
$$\mathcal{V}^a_{W'\mydot{W'} }\equiv (\mathcal{V}^\perp_{W\mydot{W}})_{-\mydot{{W'} }{W'}}.$$
So $$\mathcal{V}_{W\mydot{W} }\equiv \mathcal{V}_{W\mydot{W}M_uM_y } \lrar \V_{M_uM_y }$$
iff
$$\mathcal{V}^a_{W'\mydot{W'} } = (\mathcal{V}_{W\mydot{W}M_uM_y } \lrar \V_{M_uM_y })^\perp_{-\mydot{{W'} }{W'}}$$
$$=(\mathcal{V}_{W\mydot{W}M_uM_y } )^\perp_{-\mydot{{W'} }{W'}-M'_yM'_u  }\lrar (\V_{M_uM_y })^\perp_{-M'_yM'_u }$$
i.e., iff
$$\mathcal{V}^a_{W'\mydot{W'} }=\mathcal{V}^{a}_{{W'}\mydot{{W'} }M_u'M'_y }\lrar (\V_{M_uM_y}^{\perp} )_{-M'_yM'_u} .$$\\
\item Follows from the first part if we omit the $M$ variables.\\
\item 
 $\Vwdw $ is USG
iff
$$\Vwdw \circ W\ \supseteq \  (\Vwdw \circ \dw)_W$$
i.e., iff
$$ ((\Vwdw\circ W)^\perp)_{\mydot{W'}}\subseteq((\Vwdw\circ \mydot{W})^\perp_W)_{\mydot{W'}}$$
i.e., iff
$$(\mathcal{V}^\perp_{W\mydot{W}}\times W )_{\mydot{W'} }\subseteq((\mathcal{V}^\perp_{W\mydot{W} }\times \mydot{W})_W)_{\mydot{W'} }$$
i.e., iff
$$(\mathcal{V}^\perp_{W\mydot{W} })_{-\mydot{{W'} }{W'}  }\times \mydot{W'}\subseteq ((\mathcal{V}^\perp_{W\mydot{W} })_{-\mydot{{W'} }{W'}  }\times {W'})_{\mydot{W'}}$$
i.e., iff
$$\mathcal{V}^a_{W'\mydot{W'} }\times \mydot{W'}\subseteq(\mathcal{V}^a_{W'\mydot{W'} }\times W')_{ \mydot{W'}}.$$

i.e., iff $\mathcal{V}^a_{W'\mydot{W'} }$ is LSG.
\vspace{0.25cm}

\item Immediate from the previous part.\\
\item Follows immediately from the fact that $(\Vw^1\oplus\Vdw^2)^\perp= (\Vw^1)^\perp\oplus (\Vdw^2)^\perp.$
\end{enumerate}
\end{proof}
\begin{remark}
\label{rem:duality_proof}
{\textup{Statements about dynamical systems are naturally dualized using the adjoint 
operator which yields an adjoint system when applied to a dynamical system.\\
Parts $1$ and $3$ of Theorem \ref{thm:adjointprop}
 are basic for 
this purpose. \\ Parts $2$ and $4$ are typical instances of the dualization 
process. \\In part $2,$ we are given that through the `$\lrar $' operation with a suitable linkage, an input-output  GDS  
is converted into 
 a GDS without manifest variables.  
Then, according to the result,  through the `$\lrar $' operation with an appropriate  `dual' linkage,  the adjoint
input-output  GDS is converted into the adjoint GDS
without manifest variables. 
This makes `$\lrar , \lrar $'  a dual pair if we construct
dual statements through the operation of taking adjoint, whereas, 
if the dual were constructed through the `$\perp$' operation, then 
  `$\lrar, \rightleftharpoons$' would form 
a dual pair.}}
\end{remark}

\subsubsection{Building the adjoint dual statement}
\label{sec:adjointdual}
The adjoint dual statement is built by constructing
the adjoint based second statement
and converting it later to  the adjoint dual statement.

We construct the {\it adjoint based} second statement
by first building the $\perp -$ based second statement 
as outlined in Subsection \ref{sec:perp_dual}.

Next we replace the index set 

$W$ by $-\dwd, $

$\dw$ by $W',$ 

$M_u$ by $-M'_y,$ 

$M_y$ by $M'_u.$ 

(If there are additional dynamical variables 
with index sets, say $P, Q$ etc. we would replace $P$ by $-\dPd, $
$\dP$ by $P',$ etc..) 

The result is the {\it adjoint based} second statement.

\vspace{0.1cm}

Next we replace.
in the {\it adjoint based} second statement,
each occurrence of 

\vspace{0.1cm}

a term                 
such as $\V^a_{W'\dwd M'_uM'_y}$  by $\Vwdwmumy,$ 

\vspace{0.1cm}

a term 
such as $(\V_{WP}^\perp )_ {-\dwd -\dPd} $  by $\Vdwdp,$ 

\vspace{0.1cm}

a term such as $(\V_{\dw\dP}^\perp )_ {W' P'}$  by $\Vwp,$ 

\vspace{0.1cm}

a term such as $(\Vwmu)^\perp_{- \dwd -M'_y}$  by $\Vdwmy ,$ 

\vspace{0.1cm}

a term such as  $(\Vdwmy)^\perp_{W' M'_u}$  by $\Vwmu .$

\vspace{0.1cm}

Similar actions are performed if there are additional dynamical variables
with index sets, say $P, Q$ etc.

The result is the {\it adjoint dual} of the primal statement.
\vspace{0.1cm}

\begin{example}
The  adjoint based second statement  of
\begin{equation}
\label{eqn:emulatorcondinv3}
((\Vonewp \bigoplus \Vtwodwdp)\lrar \Vwdw ) \lrarn  (\Vonewp\lrar\V_W)_{\mydot{P}}\ \supseteqn \ \Vonewp\lrar\Vw;
\ \ \ \ \Vonewp \supseteq (\Vtwodwdp)_{WP}.
\end{equation}
is
\begin{equation}
\label{eqn:emulatorcontinv4}
\{[(\Vonewp )^\perp_{-\dwd -\dPd}\bigoplus (\Vtwodwdp)^{\perp}_{W'P'}]\lrar (\Vwdw ^{\perp})_{-\dwd W'}\} \lrarn  \{[(\Vonewp)^{\perp}_{-\dwd-\dPd}\lrar(\V_W)_{-\dwd}^{\perp}]_{-\dP}\}_{P'}$$
$$\subseteqn \{[(\Vonewp)^{\perp}_{-\dwd -\dPd}\lrar(\Vw)^{\perp}_{-\dwd}];
$$
$$    (\Vonewp)^\perp_{-\dwd -\dPd} \subseteqn \{[(\Vtwodwdp)^\perp_{W'P'}]_{WP}\}_{-\dwd -\dPd} .\end{equation}
The adjoint based second statement can be rewritten as 
\begin{equation}
\label{eqn:emulatorcontinv5}
\{[(\Vonewp )^\perp_{\dwd \dPd}\bigoplus (\Vtwodwdp)^{\perp}_{W'P'}]\lrar \V ^{a}_{W'\dwd}\} \lrarn  [(\Vonewp)^{\perp}_{\dwd\dPd}\lrar(\V_W)_{\dwd})^{\perp}]_{P'}$$
$$\subseteqn [(\Vonewp)^{\perp}_{\dwd \dPd}\lrar(\Vw)^{\perp}_{\dwd}];
$$
$$    (\Vonewp)^\perp_{\dwd \dPd} \subseteqn [(\Vtwodwdp)^\perp_{W'P'}]_{\dwd \dPd} .\end{equation}

The adjoint  dual statement is
\begin{equation}
\label{eqn:emulatorcontinv2}
[(\Vonedwdp \bigoplus \Vtwowp)\lrar \Vwdw ] \lrarn  [\Vonedwdp\lrar(\V_W)_{\dw}]_{{P}}\subseteq (\Vonedwdp\lrar(\Vw)_{\dw});
\ \ \ \ \Vonedwdp \subseteq (\Vtwowp)_{\dw\dP}
\end{equation}
\end{example}

\subsubsection{Proof through duality}
\label{sec:proof_dual}
Suppose that the primal statement is true for arbitrary dynamical systems. Then the {\it adjoint based} second statement is also true for arbitrary
dynamical systems and so will the adjoint dual be.

{\textup{ Typical ways of proving theorems through duality would be as described below.}}

\vspace{0.2cm}

Let $A,B,C,D$ be statements of the kind 
\begin{align*}
 \epsilon(\mathcal{V}_1,\ldots,\mathcal{V}_k,+,\cap,\oplus ,\circ , \times ,\leftrightarrow,\rightleftharpoons,\supseteq,\subseteq,=)
\end{align*}
Suppose we have a proof that $A \implies B.$ 
This is a statement always true for arbitrary vector spaces.

\vspace{0.2cm}

Let $A,C'$ form a primal-adjoint second pair
and so also $B,D'.$ Let $A,C$ form a primal - adjoint dual pair and so also $B,D.$\\
\\
We then have a proof that $C \implies D,$ 
which is as follows :

\vspace{0.1cm}

$$ C' \ iff \ A;\ \  A\implies B; \ \ B\ {\textup{iff}} \ D'; \ So\ \ C'\implies D'.$$

\vspace{0.1cm}


Now $C' \implies D'$ is a statement true for arbitrary vector spaces. So it will remain
true if, in $C'$ and $D',$  every $\V$ is replaced by $\V^a$ resulting in the statement  
$C\implies D,$ which is therefore proved to be true from the proof of $A \implies B.$ 
\begin{remark}
\label{rem:plain_dual}
Henceforth, we will denote {\bf adjoint dual}  by plain {\bf dual} since it is the only kind of duality
needed in dealing with dynamical systems. Similarly the {\bf adjoint based} second statement will be denoted by plain
{\bf second}.
\end{remark}

\section{Invariant subspaces}
\label{sec:Invsub}
Consider the system $\Vwdwmumy$ defined through the equations
\begin{align*}
 \mydot{w} &= Aw + Bm_u \\
 m_y &= Cw + Dm_u.
\end{align*}
We say that a subspace of states $\Vw$ is conditioned invariant 
in this system if, once having entered it, we cannot leave it.
We say it is controlled invariant in the system if, once it has entered it,
we can prevent the system from leaving it.\\

\noindent In our framework we talk in terms of a genaut $\Vwdw=\Vwdwmumy\lrar \V_{M_uM_y}.$\\

\noindent We say $\V_W$ is {\bf conditioned invariant} in $\Vwdw$ iff
$$\Vwdw \lrar \V_W\subseteq (\V_W)_{\mydot{W}}.$$
We say $\V_W$ is {\bf controlled invariant} in $\Vwdw$ iff
$$\Vwdw \lrar (\V_W)_{\mydot{W}}\supseteq \V_{{W}}.$$
If $\V_W$ is both controlled and conditioned invariant, it is said to be
{\bf invariant} in $\Vwdw.$

The terms controlled invariant and conditioned invariant probably first  appeared in \cite{basile1969controlled}.  These ideas appeared as $(A,B)$ invariance and its dual in  \cite{wonham1970decoupling}. 
Generalization of controllability and observability using invariant spaces is given in \cite{molinari1976extended}, \cite{molinari1976aStrong} \cite{molinari1976zeros}.
The concepts of conditioned and controlled invariants  are also covered in the books \cite{Wonham1978}, \cite{Trentelman2001} and \cite{basile1992controlled}. Controlled invariants have  appeared in the context of bisimulation of linear dynamical systems and hybrid dynamical systems, cf. \cite{schaft2004equivalence}, \cite{pappas2003bisimilar}, \cite{van2004bisimulation}, \cite{van2004equivalence}.

For a detailed discussion of the ideas in this section illustrating how they are 
consistent with the conventional definitions of `controlled' and `conditioned' invariant subspaces as given in say \cite{Trentelman2001}, please see \cite{HNPS2013}.

We present, in the next couple of paragraphs, a quick summary of the main ideas concerning such subspaces.\\

Given a subspace $\Vdw^{small} \subseteq \Vwdw \circ \dw,$ there is a unique minimal
conditioned invariant subspace of $\Vwdw,$ that also contains $(\Vdw^{small})_W.$
The algorithm for the construction of this subspace is simple.\\

Take $\Vw^1\equiv (\Vdw^{small})_W.$ If $\Vwdw \lrar \Vw^1\subseteq \Vdw^{small},$
stop and take $\Vw^1$ to be the desired space.\\

Otherwise compute $$\Vw^2 \equiv (\Vwdw \lrar \Vw^1)_W+\Vw^1$$ and so on
until at some stage we have $$\Vwdw \lrar \Vw^j\subseteq \Vdw^{j}.$$

Let us call this subspace $\Vw^{final}.$\\

The algorithm will terminate because as long as the stopping criterion is not
satisfied, the dimension of the subspaces $ \Vw^j$ keeps increasing.\\
At every stage $(\Vw^j)_{\dw}\subseteq \Vwdw\circ \dw.$\\
If the space $\Vwdw$ is USG, then $\Vw^{final}$
 will
also be a subspace of $\Vwdw\circ W.$ \\
This is sufficient for the controlled 
invariance condition to hold, namely
$$\Vwdw \lrar (\Vw^{final})_{\dw}\supseteq \Vw^{final}.$$
When $\Vwdwmumy$ is regular, we have that  $\Vwdw \equiv \Vwdwmumy\circ W\dw$ is USG and if we take $\Vdw^{small}\equiv \Vwdw \times \dw,$
then the invariant space $\Vw^{final}$ is the controllability subspace of $(\Vwdw\circ \dw)_W.$

The dual ideas are as follows.
Given a subspace $\Vw^{large} \supseteq \Vwdw \times W,$ there is a unique maximal
controlled invariant subspace of $\Vwdw$ that is also  contained in  $\Vw^{large}.$
The algorithm for the construction of this subspace is simple.\\

Take $\Vdw^1\equiv (\Vw^{large})_{\dw}.$ If $\Vwdw \lrar \Vdw^1\supseteq \Vw^{large},$
stop and take $\Vdw^1$ to be the desired space.\\

Otherwise compute $\Vw^2 \equiv (\Vwdw \lrar \Vdw^1)\bigcap\Vw^1$ and so on
until at some stage we have $$\Vwdw \lrar \Vdw^j\supseteq \Vw^{j}.$$

Let us call this subspace $\Vw^{final}.$\\

The algorithm will terminate because as long as the stopping criterion is not
satisfied, the dimension of the subspaces $ \Vdw^j$ keeps decreasing.\\
At every stage $\Vw^j\supseteq \Vwdw\times W.$\\
If the space $\Vwdw$ is LSG, then $\Vdw^{final}$
 will
also be a superspace of $\Vwdw\times \dw.$ \\
This is sufficient for the conditioned
invariance condition to hold, namely
$$\Vwdw \lrar \Vw^{final}\subseteq \Vdw^{final}.$$
When $\Vwdwmumy$ is regular, we have that $\Vwdw \equiv \Vwdwmumy\times W\dw$  is LSG and  if we take $\Vw^{large}\equiv \Vwdw \circ W,$
then the invariant space $\Vw^{final}$ is the observabilty superspace of 
$\Vwdw\times W.$\\
Conventionally, this would be a quotient space over the unobservable subspace
of the system.\\

We need the following results for the development of the present paper.

Theorem \ref{thm:controlledconditioneddual}
brings out the duality between conditioned and controlled invariance.

Theorem \ref{thm:conditionallyinvsum}
stresses a necessary condition that conditioned invariant subspaces must satisfy
and points out that a genop arises naturally from a conditioned invariant space in a USG.

Theorem \ref{thm:controlledinvintersection}
states the dual.

These results are counterparts of the controllable subspace and observable quotient space
of multivariable control theory.
\begin{theorem}
\label{thm:controlledconditioneddual}
$\V_W $ is conditioned  (controlled) invariant in $\Vwdw $
iff $\ \V ^{\perp}_{{W'}}$ is controlled (conditioned) invariant in 
$\mathcal{V}^{a}_{W'\mydot{{W'} } }.$
\end{theorem}
\begin{proof}
By Theorem \ref{thm:idt} we have that 
$$\Vwdw \lrar \V_W\subseteq \V_{\mydot{W}}$$ iff $$\Vwdw^{\perp} \lrar \V_W^{\perp}\supseteq \V_{\mydot{W}}^{\perp},$$
i.e., iff
$$(\Vwdw)^{\perp}_{-\dwd W'} \lrar (\V_W)^{\perp}_{-\dwd} \supseteq (\V_{\mydot{W}}^{\perp})_{W'},$$
i.e., iff
$$\mathcal{V}^{a}_{W'\mydot{{W'} } }\lrar 
(\V^{\perp}_{W'})_{\dwd} \supseteq \V_{{W'}}^{\perp}.$$

Thus 
$\V_W $ is conditioned   invariant in $\Vwdw $
iff $\ \V ^{\perp}_{{W'}}$ is controlled  invariant in
$\mathcal{V}^{a}_{W'\mydot{{W'} } }.$

The `controlled in primal' to `conditioned in dual' part of the theorem follows from the fact that for any vector space $\V,$ $(\V)^{\perp\perp}= \V,$
and for any $\Vwdw,$ we have $$\Vwdw = (\mathcal{V}^{a}_{W'\mydot{{W'} } })^a_{W\mydot{{W} } }.$$
\end{proof}
\begin{theorem}
\label{thm:conditionallyinvsum}
Let $\Vwdw $ be a genaut and 
let $\V_W $ be conditioned  invariant in $\Vwdw .$
Then
\begin{enumerate}
\item $\V_W\supseteq (\Vwdw \times \dw )_W.$ 
\item If  $\Vwdw $ is a USG, we must have $\Vwdw + (\V_W \bigoplus (\V_W)_{\mydot{W}})$ is a genop.
\end{enumerate}
\end{theorem}
\begin{proof}
\begin{enumerate}
\item By conditioned invariance of $\V_W $ in $\Vwdw ,$ and by Theorem \ref{thm:inverse}, it follows that $$\V_{\mydot{W}}\supseteq \Vwdw \lrar \V_W \supseteq \Vwdw \times \dw. $$ Therefore we have $$\V_W\supseteq (\Vwdw \times \dw )_W.$$
\item We have, $$[\Vwdw + (\V_W \bigoplus (\V_W)_{\mydot{W}})]\circ W=
\Vwdw \circ W+\Vw $$ and $$[\Vwdw + (\V_W \bigoplus (\V_W)_{\mydot{W}})]\circ
{\mydot{W}}= \Vwdw \circ {\mydot{W}}+ \V_{\mydot{W}}.$$
Since $\Vwdw $ is a USG  it follows that $(\Vwdw \circ W)_{\dw}\supseteq \Vwdw \circ {\mydot{W}}.$ Hence 
$$([\Vwdw + (\V_W \bigoplus (\V_W)_{\mydot{W}})]\circ W)_{\dw}\ \ 
\supseteq\ \  [\Vwdw + (\V_W \bigoplus (\V_W)_{\mydot{W}})]\circ {\mydot{W}}.$$ 
Next we will show that 
$$([\Vwdw + (\V_W \bigoplus (\V_W)_{\mydot{W}})]_W\times W)_{\dw} \supseteq 
[\Vwdw + (\V_W \bigoplus (\V_W)_{\mydot{W}})]\times \dw. $$ 
Let $g_{\dw}$ belong to the RHS of the above containment.
This means there exists 

$(f_W,g'_{\dw})\in \Vwdw$ and 
$(-f_W,g"_{\dw})\in (\V_W \bigoplus (\V_W)_{\mydot{W}}),$
such that $g'_{\dw}+g"_{\dw}=g_{\dw}.$ 
\vspace{0.25cm}

Thus $f_W\in \Vw, g"_{\dw}\in (\Vw)_{\mydot{W}}$ and $g'_{\dw}\in \Vwdw \lrar \Vw \subseteq (\Vw)_{\mydot{W}}.$  

Thus $g_{\dw}\in (\Vw)_{\mydot{W}}$ and therefore $g_{W}\in \Vw.$  
\vspace{0.25cm}

Now, if  $g_{W}\in  \V_W$ then 
$(g_W,0_{\dw})\in \Vwdw + (\V_W \bigoplus (\V_W))_{\dw}.$ 

\vspace{0.1cm}

We conclude that $g_{W}\in [\Vwdw + (\V_W \bigoplus (\V_W)_{\dw})]\times {{W}}.$

\vspace{0.25cm}

Hence, it is clear that 
$$([\Vwdw + (\V_W \bigoplus (\V_W)_{\mydot{W}})]_W\times W)_{\dw} \supseteq 
[\Vwdw + (\V_W \bigoplus (\V_W)_{\mydot{W}})]\times \dw, $$
and therefore that $\Vwdw + (\V_W \bigoplus (\V_W)_{\mydot{W}})$
is a genop.

\end{enumerate}
\end{proof}

The next result is the dual of Theorem \ref{thm:conditionallyinvsum}.
A detailed proof illustrating the use of duality is given in \ref{sec:controlledinvintersection}.
\begin{theorem}
\label{thm:controlledinvintersection}
Let $\Vwdw $ be a genaut and
let $\V_W $ be controlled  invariant in $\Vwdw .$
Then
\begin{enumerate}
\item $\V_W\subseteq \Vwdw \circ W.$
\item If $\Vwdw $ is an LSG, $\Vwdw \bigcap (\V_W \oplus (\V_W)_{\mydot{W}})$ is a genop.
\end{enumerate}
\end{theorem}

\section{$wm_u-$ Feedback and $m_y\mydot{w}-$ Injection}
\label{sec:feedback_injection}
In this section we present the central results about the genops that
can be obtained by state feedback and output injection from a dynamical system
$\mathcal{V}_{W\mydot{W}M_uM_y  }.$ The two operations are denoted by $wm_u-$ feedback and 
$m_y\mydot{w} -$ injection. Essentially, we characterize in a simple manner, which genaut $\Vwdw$ can be obtained
from a given dynamical system $\mathcal{V}_{W\mydot{W}M_uM_y  }$ by $wm_u-$ feedback and $m_y\mydot{w}-$  injection.
For comparing with a standard treatment of this subject the reader may consult \cite{Wonham1978}, \cite{Trentelman2001}.
\begin{definition}
Let $\mathcal{V}_{W\mydot{W}M_uM_y  }$ be a GDS and $\Vwmu $ be a linkage.
We say that the generalized autonomous system $\Vwdw$ is obtained from
 $\Vwdwmumy$ by $wm_u-$ feedback through  $\Vwmu $ iff
$$\Vwdw = (\Vwdwmumy\bigcap \Vwmu)\circ \wdw.$$
\end{definition}
\begin{definition}
Let $\Vwdwmumy$ be a GDS and $\Vmydw $ be a linkage.
We say that the generalized autonomous system $\Vwdw$ is obtained from
 $\Vwdwmumy$ by $m_y\mydot{w}-$ injection through  $\Vmydw $ iff
$$\Vwdw = (\Vwdwmumy+ \Vmydw)\times \wdw.$$
\end{definition}
We relate these ideas to the corresponding notions in classical control.
Consider the system described by
\begin{align*}
 \mydot{w} &= Aw + Bm_u \\
 m_y &= Cw + Dm_u.
\end{align*}
The space of solutions $\Vwdwmumy$ is (taking $\lambda_1,\lambda_2$ to be free vectors of appropriate dimension)
\begin{align*}
w&=\lambda_1\\
 \mydot{w} &= A\lambda_1 + B\lambda_2 \\
m_u&=\lambda_2\\
 m_y &= C\lambda_1 + D\lambda_2.
\end{align*}
The $wm_u-$ feedback (state feedback) equations are 
\begin{align*}
m_u &= Fw,
\end{align*}
with the space of solutions $\Vwmu$ given by
\begin{align*}
m_u &= F\lambda_1\\
w &= \lambda_1.
\end{align*}
The space  $(\Vwdwmumy\bigcap \Vwmu ) \circ W\dw $ is 
\begin{align*}
w&=\lambda_1\\
 \mydot{w} &= (A+BF)\lambda_1
\end{align*}
The $m_y\mydot{w}-$ injection (output injection)  equations are
\begin{align*}
\mydot{w} &= Km_y,
\end{align*}
with the space of solutions $\Vdwmy$ given by
\begin{align*}
m_y &= \sigma\\
 \mydot{w} &= K\sigma .
\end{align*}
The space  $(\Vwdwmumy+ \Vdwmy ) \times W\dw $ is obtained as
\begin{align*}
w&=\lambda_1\\
 \mydot{w} &= A\lambda_1 + B\0 +  K\sigma \\
m_u&=\0\\
 m_y &=\0= C\lambda_1 + D\0+\sigma ,
\end{align*}
which simplifies to 
\begin{align*}
 \mydot{w} &= (A-KC)w.
\end{align*}

The natural problems that arise through these operations are\\

`given $\Vwdwmumy$ and $\Vwdw,$ can we find $\Vwmu $ so that
$\Vwdw = (\Vwdwmumy\bigcap \Vwmu)\circ \wdw$ and, assuming it exists, under what conditions is it unique?'\\

`given $\Vwdwmumy$ and $\Vwdw,$ can we find $\Vmydw $ so that
$\Vwdw = (\Vwdwmumy+ \Vmydw)\times \wdw$ and, assuming it exists, under what conditions is it unique?'\\

We solve these problems in Theorem \ref{thm:statefeedback}
and Theorem \ref{thm:outputinjection}.

We begin with a result which states that $wm_u-$ feedback and $m_y\mydot{w}-$ injection
are duals.

\begin{theorem}
\label{thm:feedback_injection_duality}
We have
  $$\Vwdw = (\Vwdwmumy\bigcap \Vwmu)\circ \wdw$$
iff
 
$$\mathcal{V}^{a}_{{W'}\mydot{{W'} } }=(\mathcal{V}^{a}_{{W'}\mydot{{W'} }M'_uM'_y }+(\Vwmu)^{\perp}_{\mydot{{W'} }M'_y})\times W'\mydot{{W'} };$$

%
\end{theorem}
\begin{proof}
\item 
We have $$\Vwdw = (\Vwdwmumy\bigcap \Vwmu)\circ \wdw;$$
iff
$$\mathcal{V}^{a}_{{W'}\mydot{{W'} } }=(\Vwdw^{\perp} )_{-\mydot{{W'} }W' }=((\Vwdwmumy\bigcap \Vwmu)\circ \wdw)^{\perp}_{-\mydot{{W'} }W' },$$ i.e., iff
$$\mathcal{V}^{a}_{{W'}\mydot{{W'} } } = ((\Vwdwmumy\bigcap \Vwmu)_{-\mydot{{W'} }W'-M'_yM'_u }^{\perp}\times {\mydot{{W'} }W' },$$

$$= (\Vwdwmumy^{\perp}+ \Vwmu^{\perp})_{-\mydot{{W'} }W'-M'_yM'_u } \times {\mydot{{W'} }W' },$$  i.e., iff

$$\mathcal{V}^{a}_{{W'}\mydot{{W'} } }= ((\Vwdwmumy^{\perp})_{-\mydot{{W'} }W'-M'_yM'_u }+ (\Vwmu^{\perp})_{-\mydot{{W'} }-M'_y }) \times {\mydot{{W'} }W' },$$  i.e.,  iff

$$\mathcal{V}^{a}_{{W'}\mydot{{W'} } }=(\mathcal{V}^{a}_{{W'}\mydot{{W'} }M'_yM'_u }+(\Vwmu)^{\perp}_{\mydot{{W'} }M'_y})\times W'\mydot{{W'} }.$$

\end{proof}

We will next show that both the above, feedback and injection operations,  can be reduced to
the `$\lrar$' operation through a simple trick. We remind the reader that 
$\Iww$ refers to  the space of all $(f_W,f_{W'}),$ where $W,W'$ are copies of each other.
\begin{lemma}
\label{lem:feedbackinjectionlrar}
\begin{enumerate}
\item $$ (\Vwdwmumy\cap \Vwmu)\circ \wdw=[(\Vwdwmumy\cap \Iww)\lrar (\Vwmu\oplus \F_{my})]_{\wdw}.$$
\item
$$ (\Vwdwmumy +\Vmydw)\times \wdw=[(\Vwdwmumy+ \Iwminusw)\lrar (\Vmydw\oplus \0_{mu})]_{\wdw}.$$
\end{enumerate}
\end{lemma}
\begin{proof}
\begin{enumerate}
\item Let $(f_W,g_{\dw})\in LHS.$ \\

Then there exist $h_{Mu}, h_{My}$ such that
$(f_W,g_{\dw},h_{Mu}, h_{My})\in \Vwdwmumy,(f_W,h_{Mu})\in \Vwmu.$ \\

Hence we must have
$$(f_W,g_{\dw},h_{Mu}, h_{My})\in \Vwdwmumy,(f_W,f_{W'})\in \Iww, (f_W,h_{Mu})\in \Vwmu,$$
so that $$(f_{W'},g_{\dw})\in [(\Vwdwmumy\cap \Iww)\lrar (\Vwmu\oplus \F_{my})]$$ and therefore
$$(f_W,g_{\dw})\in [(\Vwdwmumy\cap \Iww)\lrar (\Vwmu\oplus \F_{my})]_{\wdw}= RHS.$$

Next, let $(f_W,g_{\dw})\in RHS.$\\

 Then $$(f_{W'},g_{\dw})\in [(\Vwdwmumy\cap \Iww)\lrar (\Vwmu\oplus \F_{my})]$$
and therefore there exist  $h_{Mu},  h_{My}$ such that $$(f_W,g_{\dw},h_{Mu}, h_{My})\in \Vwdwmumy,(f_W,f_{W'})\in \Iww, (f_W,h_{Mu})\in \Vwmu.$$
Thus $$(f_W,g_{\dw},h_{Mu}, h_{My})\in \Vwdwmumy,(f_W,h_{Mu})\in \Vwmu,$$ so that
$(f_W,g_{\dw})\in LHS.$
\item This is dual to the previous part of the present lemma.
Detailed proof given in \ref{sec:controlledinvintersection}.

\end{enumerate}
\end{proof}
\begin{remark}
\label{rem:feedback_injection}
One may visualize 
$(\Vwdwmumy\cap \Iww)$ as the collection of all $(f_W,f_{W'},g_{\dw},h_{m_u},k_{m_y}),$
where $(f_W,g_{\dw},h_{m_u},k_{m_y})\in \Vwdwmumy.$
Thus $$(\Vwdwmumy\cap \Iww)\circ W\dw M_uM_y= \Vwdwmumy,$$
$$(\Vwdwmumy\cap \Iww)\times W'\dw M_uM_y= \0_{W'}\oplus \Vwdwmumy\times \dw M_uM_y.$$
$$(\Vwdwmumy\cap \Iww)\times W'\dw M_y = \0_{W'}\oplus \Vwdwmumy\times \dw M_y.$$
$$(\Vwdwmumy\cap \Iww)\circ  WW'\dw M_u\times W'\dw= (\Vwdwmumy\cap \Iww)\times W'\dw M_y \circ W'\dw $$
$$= \0_{W'}\oplus \Vwdwmumy\times \dw M_y\circ \dw
= \0_{W'}\oplus \Vwdwmumy\circ W\dw M_u\times \dw.$$
\end{remark}
\begin{theorem}
\label{thm:statefeedback}
Given $\Vwdwmumy, \Vwdw,$
there exists $\Vwmu$ such that
$$ (\Vwdwmumy\cap \Vwmu)\circ \wdw\equaln \Vwdw $$
iff
\begin{equation}
\label{eqn:statefeedback}
\Vwdwmumy\circ {\wdw}\supseteqn \Vwdw;\ \ \ \ \ 
\Vwdwmumy\circ W\dw M_u\times {\dw}\subseteqn \Vwdw \times \dw .
\end{equation}
Further, if such $\Vwmu $ exists, it is unique under the condition
\begin{equation}
\label{eqn:statefeedback2}
\Vwdwmumy\circ {WM_u}\supseteqn \Vwmu; \ \ \ \ \ 
\Vwdwmumy \circ W\dw M_u\times {M_u}\subseteqn \Vwmu \times M_u 
\end{equation}
 and is equal to 
$$(\Vwdwmumy\cap   \Vwdw)\circ WM_u.$$
\end{theorem}
\begin{proof}
Let us denote  $\Vwdwmumy\cap \Iww$ by $\V_{WW'\dw M_uM_y},$
and $(\Vwdw\cap \Iww)\circ W'\dw$ by $\V_{W'\dw}.$
\vspace{0.1cm}

Thus, by Lemma \ref{lem:feedbackinjectionlrar},
 we need to examine the conditions under which there exists $\Vwmu$ such that
$$\V_{WW'\dw M_uM_y}\lrar (\Vwmu\oplus \F_{M_y})\equaln \V_{W'\dw},\ 
i.e.,\  \V_{WW'\dw M_uM_y}\circ WW'\dw M_u\lrar \Vwmu \equaln \V_{W'\dw}.$$
By Theorem \ref{thm:inverse}, this happens iff \\
$$(a)\  \V_{WW'\dw M_uM_y}\circ WW'\dw M_u\circ W'\dw\supseteqn \V_{W'\dw}, \  i.e.,\   \V_{WW'\dw M_uM_y}\circ W'\dw\supseteqn \V_{W'\dw};$$
$$ i.e., \Vwdwmumy\circ W\dw \supseteqn \Vwdw.$$
and 
$$(b)\  \V_{WW'\dw M_uM_y}\circ  WW'\dw M_u\times W'\dw\subseteqn \V_{W'\dw}.$$

The condition (a) is one of the conditions in Equation \ref{eqn:statefeedback}.

We now examine (b).

It is clear from Remark \ref{rem:feedback_injection} that  $$\V_{WW'\dw M_uM_y}\circ WW'\dw M_u\times W'\dw\equaln
\0_{W'}\oplus \V_{WW'\dw M_uM_y}\circ  W\dw M_u\times \dw\equaln0_{W'}\oplus \Vwdwmumy\circ W\dw M_u\times \dw.$$

Therefore,  $$\V_{WW'\dw M_uM_y} \circ WW'\dw M_u\times W'\dw\subseteqn \V_{W'\dw}$$ 
$$  {\textup{iff}}  
\ \
  \0_{W}\oplus \Vwdwmumy \circ W\dw M_u\times \dw\subseteqn \0_W\oplus \V_{W\dw}\times \dw, \ i.e., \  
{\textup{iff}}\ \   \Vwdwmumy \circ W\dw M_u\times \dw\subseteqn \V_{W\dw}\times \dw,$$
yielding the second condition in Equation \ref{eqn:statefeedback}.
\vspace{0.1cm}

Next, again by Theorem \ref{thm:inverse}, if such $\Vwmu$ exists, it is uniquely obtained as
$$ (((\Vwdwmumy\cap \Iww)\circ WW'\dw M_u)\lrar \Vwdw)_{WM_u},$$
 under the conditions
$$\V_{WW'\dw M_uM_y}\circ WM_u\supseteqn \Vwmu \ \  {\textup{and}} \ \ \V_{WW'\dw M_uM_y}\circ WW'\dw M_u\times WM_u\subseteqn \Vwmu .$$
The first condition is equivalent to $\Vwdwmumy\circ {WM_u}\supseteqn \Vwmu.$

\vspace{0.1cm}

We simplify  the second condition noting that
  $$\V_{WW'\dw M_uM_y}\circ WW'\dw M_u\times WM_u\equaln\0_{W}\oplus \V_{WW'\dw M_uM_y}\circ W'\dw M_u\times M_u\equaln\0_{W}\oplus \Vwdwmumy\circ {W\dw M_u\times M_u}.$$
Thus,  
$$\V_{WW'\dw M_uM_y}\circ WW'\dw M_u\times WM_u\subseteqn \Vwmu ,$$


iff $$\0_{W}\oplus \Vwdwmumy\circ W\dw M_u\times M_u
\subseteqn \0_{W}\oplus \Vwmu \times M_u,$$ i.e., 
iff
$$\Vwdwmumy\circ W\dw M_u\times M_u\subseteqn \Vwmu \times M_u.$$
\end{proof}
The next theorem is dual to the previous result.
\begin{theorem}
\label{thm:outputinjection}
Given $\Vwdwmumy, \Vwdw,$
there exists $\Vdwmy$ such that
$$ (\Vwdwmumy+ \Vdwmy)\times \wdw\equaln \Vwdw $$
iff
\begin{equation}
\label{eqn:outputinjection}
\Vwdwmumy\times {W\dw}\subseteqn \Vwdw; \ \ \\
\Vwdwmumy\times W\dw M_y\circ {W}\supseteqn \Vwdw \circ W.
\end{equation}
Further, if such $\Vdwmy $ exists, it is unique under the condition
\begin{equation}
\label{eqn:outputinjection2}
\Vwdwmumy\times {\dw M_y}\subseteqn \Vdwmy; \ \ \\
\Vwdwmumy \times W\dw M_y\circ {M_y}\supseteqn \Vdwmy \circ M_y 
\end{equation}
 and is equal to $ (((\Vwdwmumy+ \Iwminusw)\times W\dw \mydot{W'} M_y)\lrar \Vwdw)_{\mydot{W'}M_y}\equaln  (\Vwdwmumy+\Vwdw)\times \mydot{W} M_y.$
\end{theorem}
\begin{proof}
We will merely identify dual pairs of statements for the existence case.
The uniqueness is handled similarly.

Let statement $A$ denote 

Given $\Vwdwmumy, \Vwdw,$
there exists $\Vwmu$ such that
$$ (\Vwdwmumy\cap \Vwmu)\circ \wdw= \Vwdw .$$

Let statement $B$ denote

\begin{equation}
\label{eqn:statefeedback}
\Vwdwmumy\circ {\wdw}\supseteq \Vwdw; \ \ \\
\Vwdwmumy\circ W\dw M_u\times {\dw}\subseteq \Vwdw \times \dw .
\end{equation}

Let statement $C$ denote 

Given $\Vwdwmumy, \Vwdw,$
there exists $\Vdwmy$ such that
$$ (\Vwdwmumy+ \Vdwmy)\times \wdw\equaln \Vwdw $$

Let statement $D$ denote

\begin{equation}
\label{eqn:outputinjection}
\Vwdwmumy\times {W\dw}\subseteqn \Vwdw; \ \ \\
\Vwdwmumy\times W\dw M_y\circ {W}\supseteqn \Vwdw \circ W.
\end{equation}

It can be seen that $(A,C)$ is a dual pair of statements
and so also $(B,D).$

We have $A\implies B. $    

By the argument in Subsection \ref{sec:proof_dual}, $C \implies D.$

\end{proof}

\section{Polynomials of generalized autonomous systems}
\label{sec:genoppoly}
\subsection{Motivation}
In this section we present the main definitions and results concerning polynomials of 
generalized operators (genops). It is possible to define polynomials even for the more 
general case of generalized autonomous systems  (genauts). So we begin with this definition.
But for genauts  the usual  
factorisation of polynomials and the remainder theorem do not hold.
Fortunately, however, these ideas go through neatly for genops and in Theorem \ref{thm:minimalannpoly},
we present  such a
result and conclude from it the analogue, for generalized operators, of the uniqueness of the minimal annihilating 
polynomial of an operator. Some unusual notions are needed - for instance  a zero
is essentially a decoupled linkage (recall $\Vab$ is decoupled iff $\Vab=\Vab \circ A\oplus \Vab \circ B$).
Theorems \ref{USG}, \ref{LSG} contain the analogues of the usual 
results related to controllability and unobservability spaces in classical
multivariable control theory.

We give a sketch of the results in terms of a state variable description.
Let $\mathcal{V}_{W\mydot{W}M_uM_y}$ be the dynamical system given below.
\begin{subequations}
\label{eqn:StateEqnsDynSysL2}
\begin{align}
 \mydot{w} &= Aw + Bm_u \\
 m_y &= Cw + Dm_u.
\end{align}
\end{subequations}
Let $\Vonewdw \equiv \mathcal{V}_{W\mydot{W}M_uM_y}\circ W\dw.$
We can see that $\Vonewdw\circ W$ is $\F_W$ and $\Vonewdw\circ \dw$ is the range
of the matrix $A.$ Therefore $\Vonewdw$ is USG.
On the other hand,  $\Vonewdw\times W$ is the inverse image under $A$ of the range
of the matrix $B$  while $\Vonewdw\times \dw$ is the range
of the matrix $B.$ Thus, in general  $\Vonewdw\times W$ does not contain
 $(\Vonewdw\times \dw)_W.$ So $\Vonewdw$ may not be a genop.
If one performs state feedback on this system we will get a system
$\Vwdw$ governed by
\begin{subequations}
\label{eqn:StateEqnsDynSysL3}
\begin{align}
 \mydot{w} &= (A+BF) w.  
\end{align}
\end{subequations}
Clearly, $\Vwdw$ is a genop. Classical multivariable theory tells us that
there are two important maps associated with $\Vwdw.$ The first is obtained by 
restricting the map $A+BF$ to the controllability space $\Vw$ (in general
any invariant space containing $\Vonewdw \times \dw$). This is the genop
$\Vwdw\cap (\Vw\oplus (\Vw)_{\dw})$ in Theorem \ref{USG}.
The second is a map obtained by quotienting the map $A+BF$  over $\Vw.$
This map depends only  on $\Vonewdw$ and does not change under state feedback.
This is the genop $\Vonewdw +(\Vw)_{\dw}$ of Theorem \ref{USG}.

Theorem \ref{LSG} talks of the dual situation. Let us begin with 
$\Vtwowdw \equiv \mathcal{V}_{W\mydot{W}M_uM_y}\times W\dw.$
We can see that $\Vtwowdw\times \dw$ is $\0_{\dw}$ and $\Vtwowdw\times W$ is the null space
of the matrix $A.$ Therefore $\Vtwowdw$ is LSG.
On the other hand  $\Vtwowdw\circ  \dw$ is the image under $A$ of the null space of $C$
while $\Vtwowdw\circ W$ is the null space
of the matrix $C.$ Thus, in general  $(\Vtwowdw\circ \dw)_W$ is not contained in
 $\Vtwowdw\circ W.$ So $\Vtwowdw$ may not be a genop.
If one performs output injection on this system we will get a system
$\Vwdw$ governed by
\begin{subequations}
\label{eqn:StateEqnsDynSysL4}
\begin{align}
 \mydot{w} &= (A-KC) w.  
\end{align}
\end{subequations}
Clearly, $\Vwdw$ is a genop. Classical multivariable theory tells us that
there are two important maps associated with $\Vwdw.$ The first is obtained by
quotienting  the map $A-KC$ over the unobservability space $\Vw$ (in general
any invariant space contained in  $\Vtwowdw \circ W.$) This is the genop
$\Vwdw+  (\Vw\oplus (\Vw)_{\dw})$ in Theorem \ref{LSG}. 
The second is a map obtained by restricting  the map $A-KC$  to $\Vw.$
This map depends only  on $\Vtwowdw$ and does not change under output injection.
This is the genop $\Vtwowdw \cap \Vw$ of Theorem \ref{LSG}.

We reiterate that our aim is to talk of the above situation without making explicit 
use of state equations. Later in Section \ref{sec:emu}, we define the notion of emulators
which will very efficiently play the role of state equations. We will illustrate 
the use of these notions through RLC networks.
\subsection{Polynomials of genauts and genops}
To define the notion of polynomials of genauts we require the definition
of sum, scalar multiplication and `operator' multiplication.
These are captured respectively by `$\pdw,$' $\ \  $`$\ldw ,$'\ \   and `$*$' operations.
The `$*$' operation is built in terms of `$\lrar$' operation and is defined 
below.

\begin{definition}
Let $\Vwdw  ,\Vwdw  '$ be generalized autonomous systems. We define
$$\Vwdw   *\Vwdw  '  \equiv (\Vwdw  ^{} )_{WW_1}\lrar (\Vwdw  ' )_{W_1\dw}.$$
\end{definition}
In the case of an operator  $A,$ if $\Vwdw$ is the collection of all $(x^T,y^T)$ such that 
$x^TA=y^T,$  then  $\Vwdw^{(2)}$ is the collection of all $(x^T,y^T)$ such that 
$x^TA^2=y^T.$ 

Note that in the above definition, in the case of $\Vwdw,$ the `dummy' set $W_1$ is treated as a
copy of $\dw,$ and in the case of $\Vwdw  ',$ as a copy of $W.$
This operation is associative, if we introduce disjoint dummy sets whenever required.
Consider the expression
$$\Vwdw^{1}*\Vwdw^{2}*\cdots\Vwdw^{i-1}* \Vwdw^{i}*\cdots*\Vwdw^{n-1}*\Vwdw^{n}.$$
This could be treated as 
$$(\Vwdw^{1})_{WW_1}\lrar (\Vwdw^{2})_{W_1W_2}\lrar \cdots(\Vwdw^{i-1})_{W_{i-2}W_{i-1}}\lrar
(\Vwdw^{i})_{W_{i-1}W_{i-2}}\ \cdots\lrar(\Vwdw^{n-1})_{W_{n-2}W_{n-1}}\lrar (\Vwdw^{n})_{W_{n-1}\dw},$$
where all the $W_i$ can be treated as disjoint from each other and from each of $W,\dw.$ 
It then follows from Theorem \ref{thm:notmorethantwice}, that the meaning of the expression with `$*$' is independent
of where brackets are introduced, i.e., the `$*$' operation is associative.

We now define polynomials of genauts. 

\begin{definition}
Let $\Vwdw$ be a genaut.
We define 
$$\Vwdw^{(0)}\equiv \Iwdw \pdw [\Vwdw \circ W \oplus \Vwdw \times \dw ].$$
$$\Vwdw^{(1)}\equiv \Vwdw .$$
$$\Vwdw^{(k)}\equiv (\Vwdw^{(k-1)})*(\Vwdw),\ \ k=2, \cdots . $$
(By associativity of $*$ operation $\Vwdw^{(k)}\equiv (\Vwdw^{(k-j)})*(\Vwdw)^{(j)}, \ \ j=1,2, \cdots , k-1.$)
\end{definition}
\begin{remark}
Note that
$\Vwdw^{(0)}\times W$ does not in general contain $\Vwdw\times W$ and  $\Vwdw^{(0)}\circ  \dw$ is not in general contained in  $\Vwdw\circ \dw .$

In general, for a genaut $\Vwdw,$ it would not be true that $\Vwdw *\Vwdw^{(0)}$ is equal to $\Vwdw.$ But this would be true for genops
(part 5 of Lemma \ref{lem:genoppoly2}), which is adequate for our purposes.
\end{remark}

The definition of $\Vwdw^{(0)}$ is as given
so that, while being similar to the identity operator, it  shares some desirable properties with other polynomials of $\Vwdw.$ 
The characteristic properties of $\Vwdw^{(0)},$ for the most important situation, are listed in the following lemma.
\begin{lemma}
\label{lem:powerzeroprop}

  \hspace{1.3cm}   Let $(\Vwdw\circ W)_{\dw}\supseteqn \Vwdw\times \dw.$ Then
\begin{enumerate}
\item $ \Vwdw\circ W= \Vwdw^{(0)}\circ W = (\Vwdw^{(0)}\circ \dw)_W,\ \ \  \Vwdw\times {\dw}= (\Vwdw^{(0)}\times W)_{\dw}= \Vwdw^{(0)}\times \dw.$\\

\item
$$(f_W,g_{\dw}) \in \Vwdw^{(0)}\  \ {\textup{iff}}\ \   f_W \in \Vwdw\circ W,\ \   g_{\dw}\in (f_{\dw}+\Vwdw\times \dw)
,$$
 i.e., iff $$g_{\dw} \in (\Vwdw\circ W)_{\dw} ,\ \  f_{W}\in (g_{\dw}+\Vwdw\times \dw)_W.$$
\item  $\Vwdw^{(0)}$ is a genop.\\
\item $(\Vwdw^{(0)})^a= (\V_{W'\dwd}^{a})^{(0)}.$
\end{enumerate}
\end{lemma}

\begin{definition}
\label{def:poly}
Let $\Vwdw$ be a genaut and let $p(s)\equiv \Sigma_0^n \alpha_is^i$ be a polynomial in $s.$ We define
$$p(\Vwdw)\equiv \Sigma_0^n \alpha_i^{\dW}(\Vwdw^{(i)}),$$ 
where 
$$\Sigma_0^n \alpha_i^{\dW}(\Vwdw^{(i)})\equiv \alpha_0^{\dW}(\Vwdw^{(0)})\pdw \alpha_1^{\dW}(\Vwdw^{(1)})\pdw \cdots \pdw\alpha_i^{\dW}(\Vwdw^{(i)})\pdw \cdots \pdw \alpha_n^{\dW}(\Vwdw^{(n)}).$$

We define
$zero (\Vwdw)\equiv \Vwdw\circ W \oplus  
\Vwdw \times \dw.$  A polynomial $p(s)$ {\bf annihilates} $\Vwdw$ iff
$p(\Vwdw)$ is decoupled.
\end{definition}
The next theorem corresponds to the result that, for a matrix $A,$ $(p(A))^T=p(A^T).$

\begin{theorem}
\label{thm:LSG}
Let  $\Vwdw$ be a genaut and let $p(s)$ be a polynomial in $s.$
Then
 $$(p(\Vwdw))^a= p(\Vadjwdw).$$
Therefore,  $p(s)$ annihilates $\Vwdw$ iff it annihilates $\Vadjwdw.$
\end{theorem}
\begin{proof}
 We note that $$(\Vwdw^{(0)})^a= (\Vadjwdw)^{(0)},$$
$$(\Vwdw^{(k)})^a= (\Vadjwdw)^{(k)},\ \ k= 1, \cdots , $$
$$(\ldw\Vwdw^{(k)})^a= \ldwd(\Vadjwdw)^{(k)},\ \ k= 0, 1, \cdots , $$
$$(\ldw_i\Vwdw^{(i)}\pdw \ldw_j\Vwdw^{(j)})^a= \ldwd_i(\Vadjwdw)^{(i)}
\pdwd \ldwd_j(\Vadjwdw)^{(j)}.
$$
These results follow from Theorem \ref{thm:idt}, Theorem \ref{thm:intsumtranspose}, Theorem \ref{thm:scalarmult} and
Lemma \ref{lem:powerzeroprop}.
Finally we need to use the fact that the complementary orthogonal space of the direct sum
$\Va\oplus\Vb$ is $\Va^{\perp}\oplus\Vb^{\perp}.$
\end{proof}

\subsection{Polynomials of genops}
The following lemmas have the ultimate purpose of showing that 
polynomials of genops behave similar to polynomials of operators (Theorem \ref{thm:minimalannpoly}).

We begin with Lemma \ref{lem:intsumrestcont} which follows directly from the definition of intersection-sum of linkages. It relates the contraction and restriction to 
$W,\dw,$ of the original genauts, to that of their intersection-sum.
Lemmas \ref{lem:genoppoly1}, \ref{lem:genoppoly2} do this task for
the `$*$' operation. Lemma \ref{lem:genoppoly3} does this for `$\ldw$' 
operation.
This is useful in showing that polynomials of a genop are also genops and further
have the same restriction to $W$ and contraction to $\dw$ as the original genop. This leads to determining when a polynomial of a genop becomes an annihilating polynomial for it.
\begin{lemma}
\label{lem:intsumrestcont}
Let $\Vwdw  ^1 , \Vwdw ^2$ be generalized autonomous systems and let $\Vwdw \equiv \Vwdw  ^1 \pdw\Vwdw  ^2 .$
 Then,
\begin{enumerate}
\item 
$$\Vwdw  \times \dw \equaln\Vwdw ^1\times \dw +  \Vwdw ^2 \times \dw$$
$$\Vwdw  \circ W\equaln \Vwdw ^1\circ W \bigcap  \Vwdw ^2 \circ W.$$
Hence if $\Vwdw  ^1 ,\Vwdw ^2,$ have the same restriction to $W$, this space would also be the same as the restriction of $\Vwdw $ to $W$.

Similarly if $\Vwdw  ^1 ,\Vwdw ^2,$ have the same contraction to $\dw$, this space would also be the same as the contraction  of $\Vwdw  $ to $\dw$.
\item 
$$\Vwdw  \times W\supseteqn \Vwdw ^1\times W \bigcap  \Vwdw ^2 \times W$$
$$\Vwdw  \circ \dw  \subseteqn \Vwdw ^1\circ \dw +  \Vwdw ^2 \circ \dw.$$
\item 
If $\Vwdw  ^1 ,\Vwdw ^2$ are USG which have the same restriction $\Vw$ to $W,$
then  $\Vwdw  $ is a USG
with 

\vspace{0.2cm}

$\Vwdw\circ W = \Vw.$  
\item 
If $\Vwdw  ^1 ,\Vwdw ^2$ are LSG which have the same contraction $\Vdw'$ to $\dw,$
then  $\Vwdw  $ is a LSG
with 

\vspace{0.2cm}

$\Vwdw\times \dw = \Vdw'.$
\item
If $\Vwdw  ^1 ,\Vwdw ^2$ are genops which have the same restriction $\Vw$ to $W$
and 
have the same contraction $\Vdw '$ to $\dw,$ then  $\Vwdw  $ is a genop
with 

$$\Vwdw\circ W = \Vw \ \ {\textup{and}}\ \  \Vwdw\times \dw= \Vdw '. $$ 
\end{enumerate}
\end{lemma}

The next lemma is immediate from the definition of the `$*$' operation.
\begin{lemma}
\label{lem:genoppoly1}
Let $\Vwdw \equiv \Vwdw  ^1 *\Vwdw  ^2 .$
Then 
\begin{align}
\Vwdw\circ W \subseteq \Vwdw  ^1 \circ W \\
\Vwdw\circ \dw \subseteq \Vwdw  ^2\circ \dw \\
\Vwdw\times W \supseteq \Vwdw  ^1 \times W \\
\Vwdw\times \dw \supseteq \Vwdw  ^2 \times \dw .
\end{align}
\end{lemma}
Using this lemma and the definition of the `$*$' operation we get the next lemma.
\begin{lemma}
\label{lem:genoppoly2}
Let $\Vwdw  ^1,\Vwdw  ^2$ be genauts and
let $\Vwdw \equiv \Vwdw  ^1 *\Vwdw  ^2 .$
\begin{enumerate}
\item 
If $\Vwdw  ^1 ,\Vwdw ^2$ are USG which have the same restriction $\Vw$ to $W,$
then  $\Vwdw  $ is a USG
with $$\Vwdw\circ W = \Vw.$$  
\item 
If $\Vwdw  ^1 ,\Vwdw ^2$ are LSG which have the same contraction $\Vdw'$ to $\dw,$
then  $\Vwdw  $ is a LSG
with $$\Vwdw\times \dw = \Vdw'.$$  
\item
If $\Vwdw  ^1 ,\Vwdw ^2$ are genops which have the same restriction $\Vw$ to $W$
and have the same contraction $\Vdw '$ to $\dw$, then  $\Vwdw  $ is a genop
with $$\Vwdw\circ W = \Vw \ {\textup{and}}\  \Vwdw\times \dw= \Vdw '. $$ 
\item If  $\Vwdw$ is a genop, then so is $\Vwdw^{(k)},\ \ k= 1, \cdots, $ with the same
restriction to $W$ and contraction to $\dw $ as $\Vwdw .$
Further, $$\Vwdw^{(k)}\circ \dw \subseteq \Vwdw\circ \dw ,\ \ k=1, \cdots  
\ {\textup{ and}}\  \Vwdw^{(k)}\times W \supseteq \Vwdw\times W ,\ \ k=1 \cdots .$$
\item If  $\Vwdw$ is a genop, then $$\Vwdw^{(0)}*\Vwdw^{(j)}=\Vwdw^{(j)}*\Vwdw^{(0)}=\Vwdw^{(j)}, \ \ j=0, 1, \cdots .$$
\end{enumerate}
\end{lemma}
\begin{remark}
\label{rem:sumint_genops}
Sum and intersection of genops are not in general genops.
It is easy to see that sum of USGs is a USG and intersection of LSGs is an LSG.
But intersection of USGs is not in general a  USG and sum of LSGs is not in general
an LSG. 
\end{remark}
The next lemma follows from the definition of scalar multiplication of linkages.
We remind the reader that $$0^{\dws}\Vwdw\equiv \Vwdw\circ W\oplus\Vwdw \times \dw.$$ 
\begin{lemma}
\label{lem:genoppoly3}
Let $\Vwdw $ be a generalized autonomous system.\\
Then $$(\lambda ^{\dws}\Vwdw)\circ W= \Vwdw \circ W,\ \ \ (\lambda ^{\dws}\Vwdw)\times \dw= \Vwdw \times \dw.$$ 
$$(\ldws\Vwdw)  \circ \dw\subseteq  \Vwdw \circ \dw , \ \ \ 
(\ldws\Vwdw)  \times W\supseteq  \Vwdw \times W .$$

Hence, 
\begin{itemize}
\item $\ldw\Vwdw^{(j)}\ \pdw \ 0^{\dws}\Vwdw^{(0)}\equaln\ldw\Vwdw^{(j)}.$
\item
if  $\Vwdw$  is USG, so is $\lambda ^{\dws}\Vwdw $
with the same restriction to $W$  as $\Vwdw ,$ 
\item
if $\Vwdw$  is LSG, so is $\lambda ^{\dws}\Vwdw $
with the same contraction  to $\dw$  as $\Vwdw ,$   
\item
 if $\Vwdw$  is a genop, so is $\lambda ^{\dws}\Vwdw $
with the same restriction to $W$ and contraction to $\dw $ as $\Vwdw .$ 
\end{itemize}
\end{lemma}

The next lemma is a useful fact about composing or performing intersection-sum
operations with decoupled linkages.
It is immediate from the  definitions of the relevant operations and of decoupled linkages.
\begin{lemma}
\label{lem:decoupledcomposition}
\begin{enumerate}
\item

If $\Vab$ is decoupled, then so is $\Vab\lrar \Vbc.$ 
Further,
$$(\Vab\lrar \Vbc)\circ A = (\Vab\lrar \Vbc)\times  A = \Vab\circ A, \ \ (\Vab\lrar \Vbc)\circ C = (\Vab\lrar \Vbc)\times  C = \Vbc\times  C.$$
\item If  $\V^1_{AB}$ is decoupled with $$\V^1_{AB}\circ A\supseteq \V^2_{AB}\circ A,\ \ 
\V^1_{AB}\times B\subseteq \V^2_{AB}\times B, $$ then  
$$\V^1_{AB}\pb\V^2_{AB}= \V^2_{AB}\pb\V^1_{AB}=\V^2_{AB}.$$
\item If $\Vab^1,\Vab^2$ are decoupled, then  $$\Vab^1+\Vab^2= (\Vab^1\circ A +\Vab^2\circ A)\oplus  (\Vab^1\circ B +\Vab^2\circ B)$$ and $$\Vab^1\pb\Vab^2=   (\Vab^1\circ A \cap \Vab^2\circ A)\oplus  (\Vab^1\circ B +\Vab^2\circ B).$$ 
\end{enumerate}
\end{lemma}
The next lemma does preliminary work on polynomials of genauts and
also states that polynomials of genops 
are themselves genops.


\begin{lemma}
\label{lem:newgenoppoly}
Let
$\Vwdw$ be  a genaut.
Let $p(s)$ be  a polynomial in $s.$
We have the following.
\begin{enumerate}
\item Let  $p(s)=p_1(s)+p_2(s),$ where $p_1(s),p_2(s)$ are polynomials in $s.$ Then
$p(\Vwdw)= p_1(\Vwdw)\pdw \ p_2(\Vwdw). $
\item
\begin{itemize}
\item $$p(\Vwdw)\circ W\subseteq \Vwdw\circ W,\ \ \ \   p(\Vwdw)\times \dw\supseteq  \Vwdw \times \dw .$$
\item
If $\Vwdw $ is USG, then $$p(\Vwdw)\circ W= \Vwdw\circ W$$ and, if $\Vwdw $ is LSG, then   $$p(\Vwdw)\times \dw= \Vwdw \times \dw.$$
Therefore if $\Vwdw$ is respectively USG, LSG, genop, then so is
$p(\Vwdw).$
Further, when $\Vwdw $ is a genop, $p(\Vwdw)$ has the same restriction
to $W$ and contraction to $\dw$ as $\Vwdw. $
\item When  $p(s) $ has no constant term,
$$p(\Vwdw)\circ \dw \subseteq \Vwdw\circ \dw, \ \ \ \  p(\Vwdw)\times W\supseteq \Vwdw \times W.$$
\end{itemize}
\item If $\Vwdw$ is a genop,  $zero(\Vwdw)\equivn \Vwdw\circ W \oplus \Vwdw\times \dw,$
then
$$p(\Vwdw)\pdw zero(\Vwdw)= p(\Vwdw),$$
$$ (p(\Vwdw))*(zero(\Vwdw))=
(zero(\Vwdw))* (p(\Vwdw))= zero(\Vwdw).
$$
\item
If $\Vwdw$ is a genop and
 if $p(\Vwdw)$ is decoupled
then $$p(\Vwdw)= \Vwdw\circ W \oplus  
\Vwdw \times \dw= zero (\Vwdw).$$
\end{enumerate}
\end{lemma}

\begin{proof}
\begin{enumerate}
\item This follows from Lemma \ref{lem:intsumasscomm} and the fact that 
$\ldw_1\Vwdw^{(j)}\pdw \ldw _2\Vwdw^{(j)}\equaln (\lambda_1+\lambda _2)^{\dws}\Vwdw^{(j)}.$
\item We combine terms of the kind $\ldw\Vwdw^{(j)}$ through $\pdw$ intersection-sum.
We need results of the kind `when $\Vwdw$ has a relevant property,  then so do  $\Vwdw^{(j)},j\ne 0, $ $\ldw_j\Vwdw^{(j)},j\ne 0,$ 
$\pdw\{\ldws_j\Vwdw^{(j)},j\ne 0\}.$'
These results are available in 
Lemmas \ref{lem:intsumrestcont}, \ref{lem:genoppoly2}, \ref{lem:genoppoly3}. 

Note that, in general, $\Vwdw^{(0)}\circ \dw $ is not a subspace of $\Vwdw\circ \dw $ 
and $\Vwdw^{(0)}\times W $ is not a superspace of $\Vwdw\times  W. $      

\item From the previous part of the present lemma, we have that 
 $$p(\Vwdw)\circ W= \Vwdw\circ W, \ \ \ \ p(\Vwdw)\times \dw= \Vwdw \times \dw.$$
Therefore, $$zero(\Vwdw)\equaln p(\Vwdw)\circ W \oplus p(\Vwdw)\times \dw.$$

The result now follows from Lemma \ref{lem:decoupledcomposition}.
\item Follows from the fact that, when $\Vwdw$  is a genop and $p(\Vwdw)$ is decoupled, $$p(\Vwdw) \equaln  p(\Vwdw)\circ W \oplus p(\Vwdw)\times \dw\equaln  \Vwdw\circ W\oplus  \Vwdw \times \dw.$$ 
\end{enumerate}
\end{proof}
The next theorem  establishes that polynomials of genops behave similarly to
polynomials of operators, viz., they have minimal annihilating polynomials.

\begin{theorem}
\label{thm:minimalannpoly}
\begin{enumerate}
\item
Let  $\Vwdw$ be  a genop.
Suppose $p(s) \equaln p_1(s)q(s).$ Then $$p(\Vwdw)\equaln  (p_1(\Vwdw))*( q(\Vwdw)). $$
\item Let  $\Vwdw$ be  a genop and let
$p(s)\equaln p_1(s)q(s)+a(s).$  Then \\
$$p(\Vwdw)\equaln (( p_1(\Vwdw))*(q(\Vwdw)) \pdw a(\Vwdw).$$
\item 
Within nonzero multiplying factor, there is a unique polynomial 
of minimum degree $(\geq 1)$ which annihilates a genop $\Vwdw $.

\end{enumerate}

\end{theorem}

\begin{proof}
\begin{enumerate}
\item Let  $p_1(s)\equiv \Sigma_{j=0}^n \alpha_js^j,\ \     q(s)\equiv \Sigma_{i=0}^m \beta_is^i.$ Then 
$$p_1(\Vwdw)\equaln \Sigma_{j=0}^n \alpha_j^{\dws}(\Vwdw^{(j)}),\ \  q(\Vwdw)\equaln \Sigma_{i=0}^m \beta_i^{\dws}(\Vwdw^{(i)}).$$
We are given that $p(s)\equaln p_1(s)q(s).$ In order to show that $p(\Vwdw)\equaln (p_1(\Vwdw))*( q(\Vwdw)), $ we will merely show that distributivity holds in the product
$$(\Sigma_{j=0}^n \alpha_j^{\dws}(\Vwdw^{(j)}))* (\Sigma_{i=0}^m \beta_i^{\dws}(\Vwdw^{(i)})).$$
We have $$(p_1(\Vwdw))*( q(\Vwdw))\equaln (p_1(\Vwdw))*(\Sigma_{i=0}^m \beta_i^{\dws}(\Vwdw^{(i)}))\equaln (p_1(\Vwdw))_{WW_n}\lrar(\Sigma_{i=0}^m \beta_i^{\dws}(\Vwdw^{(i)}))_{W_n\dw}.$$
Since $\Vwdw$ is a genop, we have that $$(\Vwdw \times \dw)_W\subseteqn  \Vwdw \times W.$$
From the 
second part of Lemma \ref{lem:newgenoppoly},
we know that 
$$p_1(\Vwdw)\times \dw \equaln \Vwdw \times \dw \ \ \textup{and}\ \  
\Vwdw \times W\subseteqn  \beta_i^{\dws}\Vwdw^{(i)}\times W ,\ i\ne 0.$$ 
It follows
that $$(p_1(\Vwdw))_{WW_n}\times W_n \subseteqn  \beta_i^{\dws}(\Vwdw^{(i)})_{W_n\dw}\times W_n,\ i\ne 0.$$

Therefore by Theorem \ref{thm:distributivity},
 it follows that 
$$ (p_1(\Vwdw))_{WW_n}\lrar(\Sigma_{i=0}^m \beta_i^{\dws}(\Vwdw^{(i)}))_{W_n\dw}\equaln 
\Sigma_{i=0}^m (p_1(\Vwdw))_{WW_n}\lrar(\beta_i^{\dws}(\Vwdw^{(i)}))_{W_n\dw}$$
[In expanded form the equation reads
$$(p_1(\Vwdw))_{WW_n}\lrar[\beta_0^{\dws}(\Vwdw^{(0)}))_{W_n\dw}\pdw \cdots\pdw\beta_m^{\dws}(\Vwdw^{(m)}))_{W_n\dw}]$$
$$(p_1(\Vwdw))_{WW_n}\lrar(\beta_0^{\dws}(\Vwdw^{(0)}))_{W_n\dw}\pdw \cdots \pdw (p_1(\Vwdw))_{WW_n}\lrar(\beta_m^{\dws}(\Vwdw^{(m)}))_{W_n\dw}].$$
Thus,\hspace{2.5cm} $(p_1(\Vwdw))*(\Sigma_{i=0}^m \beta_i^{\dws}(\Vwdw^{(i)}))\equaln 
\Sigma_{i=0}^m (p_1(\Vwdw))*(\beta_i^{\dws}(\Vwdw^{(i)})).$

\vspace{0.6cm}

Next consider the expression 
$$(p_1(\Vwdw))*(\beta_i^{\dws}(\Vwdw^{(i)}))\equaln  (p_1(\Vwdw))_{WW_n}\lrar(\beta_i^{\dws}(\Vwdw^{(i)}))_{W_n\dw}\equaln 
(\Sigma_{j=0}^n \alpha_j^{\dws}(\Vwdw^{(j)})      )_{WW_n}\lrar(\beta_i^{\dws}(\Vwdw^{(i)}))_{W_n\dw}.$$
From the 
second part of Lemma \ref{lem:newgenoppoly},
we know that $$ \alpha_j^{\dws}(\Vwdw^{(j)}) \circ \dw \subseteqn  \Vwdw \circ \dw, \ j\ne 0 \ \ \textup{and}\ \  \beta_i^{\dws}(\Vwdw^{(i)})\circ W \equaln  \Vwdw \circ W.$$
Since $\Vwdw$ is a genop, we have that $$(\Vwdw \circ \dw)_W\subseteqn  \Vwdw \circ W.$$

It follows that $$ \alpha_j^{\dws}(\Vwdw^{(j)})) _{WW_k}\circ {W_k}\subseteqn  (\beta_i^{\dws}(\Vwdw^{(i)}))_{W_k\dw}\circ W_k,\ j\ne 0.$$
From Theorem \ref{thm:distributivity},
 we then have that $$[\Sigma_{j=0}^n \alpha_j^{\dws}(\Vwdw^{(j)})      ]_{WW_k}\lrar(\beta_i^{\dws}(\Vwdw^{(i)}))_{W_k\dw}\equaln \Sigma_{j=0}^n [(\alpha_j^{\dws}(\Vwdw^{(j)}))_{WW_k}
\lrar (\beta_i^{\dws}(\Vwdw^{(i)}))_{W_k\dw}].$$ 
[In expanded form the equation reads
$$ [\alpha_0^{\dws}(\Vwdw^{(0)})\pdw \cdots \pdw \alpha_n^{\dws}(\Vwdw^{(n)})]_{WW_k}\lrar(\beta_i^{\dws}(\Vwdw^{(i)}))_{W_k\dw}$$
$$=(\alpha_0^{\dws}(\Vwdw^{(0)}))_{WW_k}\lrar (\beta_i^{\dws}(\Vwdw^{(i)}))_{W_k\dw} \pdw \cdots \pdw (\alpha_n^{\dws}(\Vwdw^{(n)}))_{WW_k}\lrar (\beta_i^{\dws}(\Vwdw^{(i)}))_{W_k\dw}].
$$
Thus, $$(p_1(\Vwdw))*(\beta_i^{\dws}(\Vwdw^{(i)}))\equaln \Sigma_{j=0}^n [\alpha_j^{\dws}(\Vwdw^{(j)}))] *(\beta_i^{\dws}(\Vwdw^{(i)}))\equaln \Sigma_{j=0}^n[(\alpha_j^{\dws}(\Vwdw^{(j)}))*(\beta_i^{\dws}(\Vwdw^{(i)}))].$$ 
Thus 
$$(\Sigma_{j=0}^n \alpha_j^{\dws}(\Vwdw^{(j)}))* (\Sigma_{i=0}^m \beta_i^{\dws}(\Vwdw^{(i)}))
\equaln  \Sigma_{i=0}^m \Sigma_{j=0}^n[(\alpha_j^{\dws}(\Vwdw^{(j)}))*(\beta_i^{\dws}(\Vwdw^{(i)}))]
.$$\\
\item  This follows from  the 
first part of Lemma \ref{lem:newgenoppoly}
and the
previous part of the present theorem.\\
\item From the 
fourth part of Lemma \ref{lem:newgenoppoly},
we have that if $\Vwdw$ is a genop and
$p(s)$ annihilates it, then $$p(\Vwdw)\equaln  zero(\Vwdw).$$ If $p(s),q(s)$ both have the minimum degree among all polynomials that annihilate $\Vwdw,$ and do not differ only by a multiplying factor,
then we can write $$p(s)\equaln p_1(s)q(s)+a(s), \ \ \textup{where}\ \  degree(a(s))<degree(q(s)).$$
But then by the previous part of the present theorem,  we will have 
$$zero(\Vwdw)\equaln  p(\Vwdw)\equaln  (( p_1(\Vwdw))*(zero(\Vwdw))) \pdw a(\Vwdw)\equaln  a(\Vwdw),$$
which violates the minimality of degree of $p(s),q(s).$

\end{enumerate}
\end{proof}
\begin{remark}
\begin{itemize}
\item
Just as in the case of the usual operators, in general for two 
genops $\Vwdw ^1, \Vwdw ^2$ it will not be true that 
$(\Vwdw ^1)*(\Vwdw ^2)$ equals $(\Vwdw ^2)*(\Vwdw ^1),$
but for two polynomials $p(s),q(s)$ it will be true that 
$(p(\Vwdw))*(q(\Vwdw))\equaln  (q(\Vwdw))*(p(\Vwdw)).$
\item
In the above proof of Theorem \ref{thm:minimalannpoly}, note that, although it is not required, 
it is true that\\ 
$(p_1(\Vwdw))_{WW_n}\times W_n \subseteqn  \beta_0^{\dws}(\Vwdw^{(0)})_{W_n\dw}\times W_n \ \ $
and 
$ \ \ \alpha_0^{\dws}(\Vwdw^{(0)})) _{WW_n}\circ {W_n}\subseteqn  (\beta_0^{\dws}(\Vwdw^{(0)}))_{W_n\dw}\circ W_n.$
\end{itemize}
\end{remark}
%


A commonly occurring situation in multivariable control,
when translated into the language of this paper,  is that one genaut contains another
but has the same restriction or contraction to some set.

For instance, in the case of the dynamical system $\Vwdwmumy,$ defined by Equation \ref{eqn:StateEqnsDynSysL2}, the genop $$\Vwdw \equiv \Vwdwmumy \lrar (\F_{M_y}\oplus \0_{M_u})$$
is contained in $\Vwdwmumy\circ W\dw$ and contains $\Vwdwmumy\times W\dw.$
In the former case it has the same restriction to $W$ as the genaut containing it and, in the latter case, the same contraction to $\dw$
as the genaut that it contains.
This is also true of every genop which can be obtained from $\Vwdwmumy$ 
by $wm_u-$ feedback or by $m_y\dW -$ injection. 

The fundamental 
questions about annihilating polynomials of such genops, dealt with in the
next section, need the next two theorems which are duals of each other.
We prove the first in detail through a series of lemmas but skip
the  proof of the dual result.

\begin{theorem}
\label{USG}
Let $\Vonewdw $ be a USG and let $\Vwdw$ be a genop such that\\ 

$\Vwdw\subseteqn  \Vonewdw $ and $\Vwdw\circ W\equaln  \Vonewdw\circ W.$
Let $\Vw $ be invariant in $\Vonewdw .$ Then

\begin{enumerate}
\item $\Vonewdw +(\Vw)_{\dw} \equaln  \Vwdw + (\Vw \oplus (\Vw)_{\dw})$ and is a genop.\\
\item $\Vw $ is invariant in $\Vwdw$ and therefore $\Vwdw \cap \Vw \equaln  \Vwdw \bigcap (\Vw \oplus (\Vw)_{\dw})$ and $\Vwdw \cap \Vw $ is a genop.\\ 
\item   $p(\Vonewdw+(\Vw)_{\dw}) \equaln  p(\Vwdw)+(\Vw \oplus (\Vw)_{\dw}).$\\ 
\item 
$p(\Vwdw\cap \Vw) \equaln  p(\Vwdw)\bigcap \Vw\equaln  p(\Vwdw)\bigcap (\Vw \oplus (\Vw)_{\dw}).$ \\
\item  If $p(\Vwdw) $ is decoupled, so are $p(\Vonewdw +(\Vw)_{\dw})$
and $p(\Vwdw \cap \Vw).$\\
\item If $p_1(\Vonewdw +(\Vw)_{\dw}),\ \  p_2(\Vwdw \cap \Vw)$
are decoupled, then  $p_1p_2(\Vwdw) $ is decoupled.
\end{enumerate}
\end{theorem}
The proof of the theorem needs the following lemmas.

\vspace{0.1cm}

\begin{lemma}
\label{lem:twogenasrestcont}
Let $\oVonewdw, \oVtwowdw$ be generalized autonomous systems 
with $\oVonewdw \supseteqn  \oVtwowdw .$\\
\begin{enumerate}
\item If $p(s)$ is any polynomial in $s,$ then $p(\oVonewdw)\supseteqn   p(\oVtwowdw).$\\
\item Let $\oVonewdw\circ W \equaln \oVtwowdw\circ W.$ Then 
$\oVtwowdw +(\oVonewdw \times \dw) \equaln \oVonewdw .$\\
\item  Let $\oVonewdw\times \dw \equaln \oVtwowdw\times \dw.$ Then
$\oVonewdw \bigcap (\oVtwowdw \circ W) \equaln \oVtwowdw .$\\

\item Let $\oVonewdw, \oVtwowdw$ be USG and let $\oVonewdw\circ W\equaln \oVtwowdw\circ W.$ If $p(\oVtwowdw)$ is decoupled then so is $p(\oVonewdw).$\\
\item Let $\oVonewdw, \oVtwowdw$ be LSG and let $\oVonewdw\times \dw\equaln \oVtwowdw\times \dw.$ If $p(\oVonewdw)$ is decoupled then so is $p(\oVtwowdw).$
\end{enumerate}
\end{lemma}
\begin{proof}
\begin{enumerate}
\item If $\oVonewdw \supseteqn  \oVtwowdw ,$ it is clear that 
$$(\oVonewdw)^{(k)} \supseteqn  (\oVtwowdw)^{(k)},\ \  
\lambda ^{\dws}(\oVonewdw)^{(k)}\supseteqn \lambda ^{\dws} (\oVtwowdw)^{(k)},$$
$$\lambda ^{\dws}(\oVonewdw)^{(k)}\pdw\cdots \pdw \sigma ^{\dws}
(\oVonewdw)^{(m)}\supseteqn  \lambda ^{\dws}(\oVtwowdw)^{(k)}\pdw\cdots \pdw \sigma ^{\dws}
(\oVtwowdw)^{(m)}.$$ The result follows.\\
\item Clearly RHS contains LHS. 
Let $(f_W,g_{\dw})\in \oVonewdw .$ 

Since  $\oVonewdw\circ W\equaln \oVtwowdw\circ W,$
for some $g'_{\dw},\ \  (f_W,g'_{\dw})\in \oVtwowdw\subseteqn  \oVonewdw .$ 

Therefore, $g_{\dw}-g'_{\dw}\in \oVonewdw \times \dw,\ $ so that $$\ (f_W,g'_{\dw})+(0_W,g_{\dw}-g'_{\dw})\equaln   (f_W,g_{\dw})\ \ \in \ \ \oVtwowdw +\oVonewdw \times \dw.$$

Thus LHS contains RHS as required.\\
\item This is dual to the previous part.\\

\item  From the first part of the present lemma, we have that  $p(\oVonewdw)\supseteqn   p(\oVtwowdw).$

Since  $\oVonewdw, \oVtwowdw$ are USG, by 
the second part of Lemma \ref{lem:newgenoppoly},
we have that 
$$p(\oVonewdw)\circ W\equaln \oVonewdw\circ W,\ \ p(\oVtwowdw)\circ W\equaln \oVtwowdw\circ W.$$ 

By the second part of the present lemma, we have that 
$$p(\oVtwowdw) +p(\oVonewdw) \times \dw \equaln p(\oVtwowdw) + (0_{W}\oplus \oVonewdw \times \dw) \equaln p(\oVonewdw) .$$
By Lemma \ref{lem:decoupledcomposition}, the sum of decoupled linkages is decoupled. The result follows.\\

\item This is dual to the previous part.
\end{enumerate}
\end{proof}
\begin{lemma}
\label{lem:condcontinpoly}
Let $\oVonewdw,\oVtwowdw$ be genauts and let $\Vw$ be conditioned (controlled) invariant in them.
Then
\begin{enumerate}
\item $\Vw$ is conditioned (controlled) invariant in $\oVonewdw*\oVtwowdw,$ \\
\item $\Vw$ is conditioned (controlled) invariant in $\oVonewdw\pdw\oVtwowdw,$\\ 
\item $\Vw$ is conditioned (controlled) invariant in $\lambda^{\dws}\oVonewdw,$ \\
\item if $p(s)$ is a polynomial in $s,$ then $\Vw$ is conditioned (controlled) invariant in $p(\oVonewdw).$
\end{enumerate}
\end{lemma}
\begin{proof}
We will consider only the conditioned invariant case, the controlled invariant 
case following by duality.
Let  $\Vw$ be conditioned invariant in $\oVonewdw,\oVtwowdw.$ 
\begin{enumerate}
\item 
We have $$\Vw\lrar (\oVonewdw*\oVtwowdw)$$
$$ \equaln \Vw\lrar [(\oVonewdw)_{WW_1}\lrar (\oVtwowdw )_{W_1\dw}]$$
$$\equaln [\Vw\lrar (\Vonewdw)_{WW_1}]\lrar (\oVtwowdw )_{W_1\dw}$$
$$\subseteqn  (\Vw)_{W_1}\lrar (\oVtwowdw )_{W_1\dw}\subseteqn  (\Vw)_{\dw}.$$
\item We have 
$$ \Vw\lrar (\oVonewdw\pdw\oVtwowdw) \subseteqn  (\Vw\lrar \oVonewdw)+(\Vw\lrar \oVtwowdw)\subseteqn  (\Vw)_{\dw}.$$
\item When $\lambda  \ne 0,$ it is clear that
$$\Vw\lrar \lambda^{\dws}\oVonewdw\equaln \Vw\lrar \oVonewdw \subseteqn   (\Vw)_{\dw}.$$
Next, to consider the case where $\lambda  \equaln 0,$ note that $\Vw$ being
conditioned invariant in $\oVonewdw$ implies that $$(\Vw)_{\dw}\supseteqn \Vw\lrar \oVonewdw \supseteqn  \oVonewdw\times \dw.$$ We know that
$$0^{\dw}\oVonewdw\equivn \oVonewdw\circ W\oplus  \oVonewdw\times \dw,$$ so that 
$$\Vw\lrar 0^{\dw}\oVonewdw \equaln\Vw \lrar (\oVonewdw\circ W\oplus  \oVonewdw\times \dw)
\equaln\oVonewdw\times \dw\subseteqn  (\Vw)_{\dw}.$$\\
\item This follows from the previous parts of the present lemma.
\end{enumerate}
\end{proof}
\begin{lemma}
\label{lem:condintcontplus}
Let $\Vwdw$ be a genaut.
Then
\begin{enumerate}
\item
if $\Vw$ is conditioned invariant in $\Vwdw,$ $\Vw$ is also conditioned invariant in $\Vwdw \cap \Vw ;$ \\

if $\Vwdw$ is LSG(USG) then  $\Vwdw \cap \Vw $
is LSG(USG);\\
\item
if $\Vw$ is controlled invariant in $\Vwdw,$ $\Vw$ is also controlled invariant in $\Vwdw + (\Vw)_{\dw} ;$\\

  if $\Vwdw$ is LSG(USG) then  $\Vwdw + (\Vw)_{\dw} $
is LSG(USG).\\
\end{enumerate}
\end{lemma}
\begin{proof}
\begin{enumerate}
\item Let $\Vw$ be conditioned invariant in $\Vwdw.$ Then, 
$$\Vw\lrar (\Vwdw \cap \Vw)
\equaln\Vw\lrar \Vwdw \subseteqn  (\Vw)_{\dw},$$ so that $\Vw$ is conditioned invariant in $\Vwdw\cap \Vw.$

Next let $\Vwdw$ be LSG i.e., 
$(\Vwdw \times \dw )_W\subseteqn \Vwdw\times W.$\\

We know that
$$\Vwdw \times \dw \subseteqn \Vw\lrar \Vwdw \subseteqn (\Vw)_{\dw}.$$
So $$(\Vwdw\cap\Vw)\times W\equaln (\Vwdw\times W)\bigcap \Vw\supseteqn  (\Vwdw \times \dw )_W\bigcap \Vw\equaln(\Vwdw \times \dw)_W\supseteqn(\Vwdw\cap\Vw)\times \dw.$$
Thus $\Vwdw\cap\Vw$ is LSG.\\

Next let $\Vwdw$ be USG.

Since by conditioned invariance of $\Vw,$
$$\Vw\lrar \Vwdw\equaln (\Vwdw\cap \Vw )\circ \dw\subseteqn (\Vw)_{\dw},$$
we have
 $$(\Vwdw\cap\Vw)\circ W \equaln(\Vwdw\circ W)\bigcap \Vw$$
$$\supseteqn
 (\Vwdw \circ \dw )_W\bigcap \Vw\supseteqn((\Vwdw \cap  \Vw)\circ \dw)_W\bigcap \Vw\equaln ((\Vwdw\cap\Vw)\circ \dw)_W.$$ 
Thus $\Vwdw \cap\Vw$ is USG.\\
\item
This follows by duality from the previous part of the present lemma.
\end{enumerate}

\end{proof}
\begin{lemma}
\label{lem:condandops}
Let $\oVonewdw,\oVtwowdw$ be genauts and let $\Vw$ be conditioned  invariant in them.
Then
\begin{enumerate}
\item  $(\oVonewdw*\oVtwowdw)\cap \Vw\equaln (\oVonewdw\cap \Vw)*(\oVtwowdw\cap \Vw).$\\
\item $(\oVonewdw\pdw\oVtwowdw)\cap \Vw \equaln  (\oVonewdw\cap \Vw)\pdw(\oVtwowdw\cap\Vw).$\\
\item  $(\lambda^{\dws}\oVonewdw)\cap \Vw\equaln \lambda^{\dws}(\oVonewdw\cap \Vw).$\\
\item if $p(s)$ is a polynomial in $s$ then  $p(\oVonewdw)\cap \Vw \equaln p(\oVonewdw\cap \Vw) .$
\end{enumerate}
\end{lemma}
\begin{proof}
\begin{enumerate}
\item We have 
$$(\oVonewdw*\oVtwowdw)\cap \Vw\equivn ((\Vwdw^{1})_{WW_1}\lrar (\Vwdw^{2})_{W_1\dw})\cap \Vw\equaln (\Vwdw^{1}\cap \Vw)_{WW_1}\lrar (\Vwdw^{2})_{W_1\dw}.$$
Now by conditioned invariance of $\Vw$ in $\oVonewdw,$  $$\Vw\lrar (\oVonewdw)_{WW_1}\equaln ((\Vwdw^{1})_{WW_1}\cap\Vw )\circ W_1\subseteqn   (\Vw)_{W_1}.$$
So $$ (\Vwdw^{1}\cap \Vw)_{WW_1}\lrar (\Vwdw^{2})_{W_1\dw}\equaln
(\Vwdw^{1}\cap \Vw)_{WW_1}\lrar ((\Vwdw^{2})_{W_1\dw}\cap\V_{W_1})_{W_1\dw})\equaln
(\oVonewdw\cap \Vw)*(\oVtwowdw\cap \Vw).$$
\item (This does not require conditioned invariance of $\Vw.$) \\
That RHS is contained in LHS is immediate from the definition of the $\pdw$ operation.\\
Let $(f_W,g_{\dw})\inn LHS.$\\
 Then $$f_W\inn (\oVonewdw\circ W)\cap (\oVtwowdw\circ W)\cap\Vw.$$ \\
Further we must have, for some $g^1_{\dw},g^2_{\dw},$
that $$g^1_{\dw}+g^2_{\dw}\equaln g_{\dw},\ \ (f_W,g^1_{\dw})\inn \oVonewdw,\ \ (f_W,g^2_{\dw})\inn \oVtwowdw.$$ But then it is clear that $$(f_W,g^1_{\dw})\inn \oVonewdw\cap \Vw,\ \  (f_W,g^2_{\dw})\inn \oVtwowdw\cap\Vw.$$
Thus $$(f_W,g_{\dw})\inn (\oVonewdw\cap \Vw)\pdw(\oVtwowdw\cap\Vw)\equaln RHS.$$\\
\item This is routine. It is clear that it works out right also for
$\lambda \equaln 0.$ \\
\item This is a consequence of the previous parts of this lemma.
\end{enumerate}
\end{proof}
The next lemma is the dual of the above lemma and therefore its 
proof is omitted.
\begin{lemma}
\label{lem:contandops}
Let $\oVonewdw,\oVtwowdw$ be genauts and let $\Vw$ be controlled  invariant in them.
Then
\begin{enumerate}
\item  $(\oVonewdw*\oVtwowdw)\plusn (\Vw)_{\dw}\equaln (\oVonewdw+(\Vw)_{\dw})*(\oVtwowdw+(\Vw)_{\dw}).$\\
\item $(\oVonewdw\pdw\oVtwowdw)+(\Vw)_{\dw} \equaln  (\oVonewdw+(\Vw)_{\dw})\pdw(\oVtwowdw+(\Vw)_{\dw}).$\\
\item  $(\lambda^{\dws}\oVonewdw)+(\Vw)_{\dw}\equaln \lambda^{\dws}(\oVonewdw+(\Vw)_{\dw}).$\\
\item if $p(s)$ is a polynomial in $s$ then  $p(\oVonewdw)+(\Vw)_{\dw} \equaln p(\oVonewdw+(\Vw)_{\dw}) .$
\end{enumerate}
\end{lemma}


%
The Lemmas \ref{lem:condintcontplus}, \ref{lem:condandops}
,\ref{lem:contandops}
 lead immediately to the following lemma.
\begin{lemma}
\label{lem:invinpoly}
Let $\Vwdw$ be a genaut. 
\begin{enumerate}
\item Let $\Vw$ be conditioned invariant in $\Vwdw.$
Then 
$p(\Vwdw\cap \Vw) \equaln  p(\Vwdw)\bigcap \Vw.$\\
Further, if $\Vwdw$ is a genop, then so is 
$p(\Vwdw\cap \Vw).$\\
\item Let $\Vw$ be controlled invariant in $\Vwdw.$
Then 
$p(\Vwdw +( \Vw)_{\dw}) \equaln  p(\Vwdw)\plusn (\Vw)_{\dw}.$\\
Further, if $\Vwdw$ is a genop, then so is 
$p(\Vwdw+( \Vw)_{\dw}).$
\end{enumerate}

\end{lemma}

We are now ready for the proof of Theorem \ref{USG}.
%
%
%
%
%
%
\begin{proof}
Proof of Theorem \ref{USG}
\begin{enumerate}
\item 
Since $\Vw$  is controlled invariant in $\Vonewdw,$
$$\Vonewdw\circ W \supseteqn \Vonewdw\lrar (\Vw)_{\dw} \supseteqn \Vw,$$
and since $\Vw$  is conditioned invariant in $\Vonewdw,$
$$\Vonewdw\times\dw \subseteqn  \Vonewdw\lrar \Vw \subseteqn (\Vw)_{\dw}.$$
Now let $(f_W,g_{\dw})\inn  \Vwdw + (\Vw \oplus (\Vw)_{\dw}) .$\\
We are given that, $\Vonewdw\circ W\equaln\Vwdw\circ W.$
and we saw above that $\Vonewdw\circ W \supseteqn \Vw,$\\
Hence there exists $(f_W,g'_{\dw})\in \Vonewdw.$
so that  $g_{\dw}-g'_{\dw}\inn \Vonewdw\times\dw\subseteqn(\Vw)_{\dw} .$\\ 
Therefore, $(f_W,g_{\dw})\inn \Vonewdw +(\Vw)_{\dw}.$ \\
Thus,
$$\Vwdw + (\Vw \oplus (\Vw)_{\dw})\subseteqn \Vonewdw +(\Vw)_{\dw}.$$ 

\vspace{0.2cm}

Next, since $\Vonewdw\circ W\equaln\Vwdw\circ W$ and $\Vonewdw\supseteqn\Vwdw,$\\ 
 by Lemma \ref{lem:twogenasrestcont}, we know that 
 $\Vwdw + \Vonewdw\times \dw\equaln\Vonewdw.$\\
 Since 
$\Vonewdw\times\dw \subseteqn  (\Vw)_{\dw},$
$$\Vonewdw +(\Vw)_{\dw}\subseteqn \Vwdw + (\Vw \oplus (\Vw)_{\dw}).$$
Thus
$$\Vonewdw +(\Vw)_{\dw}\equaln \Vwdw + (\Vw \oplus (\Vw)_{\dw}).$$

Let  $$\oVtwowdw\equivn \Vonewdw +(\Vw)_{\dw}\equaln  \Vwdw + (\Vw)_{\dw} \equaln \Vwdw + (\Vw \oplus (\Vw)_{\dw}).$$
It is clear that $$\oVtwowdw\circ W\equaln\Vwdw \circ W+\Vw ,\ \  \oVtwowdw\circ \dw\equaln \Vwdw \circ \dw + (\Vw)_{\dw},$$
$$ \oVtwowdw\times W\equaln\Vwdw \times W+\Vw ,\ \ 
\oVtwowdw\times \dw\equaln \Vwdw\times \dw + (\Vw)_{\dw},$$ and it follows that $\oVtwowdw$ is a genop since $\Vwdw$ is one.
\\
\item  
From the conditioned invariance of $\Vw$ in $\Vonewdw$ it follows that $$\Vw\lrar \Vwdw\subseteqn \Vw\lrar \Vonewdw \subseteqn (\Vw)_{\dw}.$$
So $\Vw$ is conditioned invariant in $\Vwdw.$
Next we examine controlled invariance.
We will show that $$\Vdw \lrar \Vwdw \equaln \Vdw \lrar \Vonewdw.$$ The controlled invariance of $\Vw$ in $\Vwdw$ will follow
from its controlled invariance in $\Vonewdw.$\\

Since $\Vwdw\subseteq \Vonewdw,$ clearly $$\Vdw \lrar \Vwdw \subseteqn \Vdw \lrar \Vonewdw.$$
To see the reverse containment, let $f_W\inn  \Vdw \lrar \Vonewdw.$ \\
 Then there exist $(f_W,g_{\dw})\inn \Vonewdw,\ \  g_{\dw}\in \Vdw.$\\
We know that $\Vwdw\circ W\equaln \Vonewdw \circ W.$\\
So there exists $(f_W,g'_{\dw})\inn \Vwdw\subseteqn \Vonewdw.$\\
It follows that $(g_{\dw}-g'_{\dw}) \inn  \Vonewdw \times \dw\subseteqn \Vdw.$\\
So $g'_{\dw}\inn \Vdw,$ so that $f_W\inn \Vdw \lrar \Vwdw .$\\
Thus $$\Vdw \lrar \Vonewdw\subseteqn \Vdw \lrar \Vwdw,$$ and the equality between the two sides follows.

\vspace{0.4cm}

From the conditioned invariance of
$\Vw$ in $\Vwdw,$ 
it follows that 
 $$\Vwdw \cap \Vw \equaln  \Vwdw \bigcap (\Vw \oplus (\Vw)_{\dw}).$$

Let  $$\oVtwowdw\equivn  \Vwdw \cap \Vw \equaln   \Vwdw \bigcap (\Vw \oplus (\Vw)_{\dw}).$$
It is clear that $$\oVtwowdw\circ W\equaln\Vwdw \circ W\bigcap   \Vw ,\ \  \oVtwowdw\circ \dw\equaln \Vwdw \circ \dw \bigcap (\Vw)_{\dw},$$
$$ \oVtwowdw\times W\equaln\Vwdw \times W\bigcap \Vw ,\ \ 
\oVtwowdw\times \dw\equaln \Vwdw\times \dw \bigcap (\Vw)_{\dw},$$ and it follows that $\oVtwowdw$ is a genop since $\Vwdw$ is one.
\\

%
%
\item
We  have, 
$\Vonewdw, \Vwdw$ are USG, genop respectively and therefore by Lemma \ref{lem:newgenoppoly},
$p(\Vonewdw), p(\Vwdw)$ are also respectively, USG, genop.

$\Vw$ is invariant in $\Vonewdw, \Vwdw$ and therefore also in $p(\Vonewdw)$ and in $p(\Vwdw)$ (Lemma \ref{lem:condcontinpoly}).\\
From Lemma \ref{lem:contandops}, $$p(\Vonewdw+(\Vw)_{\dw}) \equaln  p(\Vonewdw)+(\Vw)_{\dw}.$$\\
We have  $p(\Vonewdw)\supseteq p(\Vwdw)$ and, by Lemma \ref{lem:newgenoppoly}, $$p(\Vonewdw)\circ W\equaln  p(\Vwdw)\circ W.$$\\
From the first part of the present theorem,
 $$p(\Vonewdw+(\Vw)_{\dw}) \equaln p(\Vonewdw)+(\Vw)_{\dw}\equaln p(\Vwdw)+(\Vw \oplus (\Vw)_{\dw}).$$ 
\item By Lemma \ref{lem:invinpoly}, we have that $p(\Vwdw\cap \Vw)\equaln p(\Vwdw)\bigcap \Vw.$ \\
Since $\Vw$ is invariant 
in $\Vwdw,$ it is also invariant in $p(\Vwdw)$ (Lemma \ref{lem:condcontinpoly}).\\
By a previous 
part of the present theorem we have that 
$$p(\Vwdw\bigcap \Vw)\equaln p(\Vwdw)\bigcap \Vw\equaln p(\Vwdw)\bigcap (\Vw\oplus (\Vw)_{\dw}).
$$
\item 
If $p(\Vwdw)$ is decoupled, so are $$p(\Vwdw)\plusn (\Vw \oplus (\Vw)_{\dw}),\ \ p(\Vwdw)\bigcap (\Vw\oplus (\Vw)_{\dw}).$$
The result follows since $$p(\Vonewdw+(\Vw)_{\dw})\equaln p(\Vwdw)\plusn (\Vw \oplus (\Vw)_{\dw}) \ \ \textup{and}\ \ 
p(\Vwdw\bigcap \Vw)\equaln p(\Vwdw)\bigcap (\Vw\oplus (\Vw)_{\dw}).$$
\item 
We have,  
$\Vw$ is conditioned invariant in $p_1(\Vonewdw)$ and therefore $p_1(\Vonewdw)\times \dw \subseteqn  (\Vw)_{\dw}.$\\
Therefore,   
$$(p_1(\Vonewdw+(\Vw)_{\dw}))\times \dw \equaln  (p_1(\Vonewdw)+(\Vw)_{\dw})\times \dw \equaln p_1(\Vonewdw)\times \dw +(\Vw)_{\dw}\equaln (\Vw)_{\dw}.$$
If $p_1(\Vonewdw+(\Vw)_{\dw})$ is decoupled so is $p_1(\Vonewdw)+(\Vw)_{\dw}.$ \\
Since $p_1(\Vonewdw)+(\Vw)_{\dw}$ is decoupled, 
$$(p_1(\Vonewdw)+(\Vw)_{\dw})\times \dw\equaln (p_1(\Vonewdw)+(\Vw)_{\dw})\circ \dw
\equaln p_1(\Vonewdw)\circ \dw +(\Vw)_{\dw}\equaln  (\Vw)_{\dw}.$$ 
It follows that
$ p_1(\Vonewdw)\circ \dw\subseteqn  (\Vw)_{\dw}.$ Therefore, since $p_1(\Vwdw)\subseteqn p_1(\Vonewdw),$ 
$$p_1(\Vwdw)\circ \dw\subseteqn  p_1(\Vonewdw)\circ \dw\subseteqn  (\Vw)_{\dw}.$$ 

Next $$p_1p_2(\Vwdw)\equiv p_1(\Vwdw)*p_2(\Vwdw)\equaln p_1(\Vwdw)*(p_2(\Vwdw)\bigcap \Vw)\equaln p_1(\Vwdw)*(p_2(\Vwdw\cap \Vw)).$$
Since $p_2(\Vwdw\cap \Vw)$ is decoupled, so is $p_1(\Vwdw)*(p_2(\Vwdw\cap \Vw)) $  and therefore $p_1p_2(\Vwdw).$
\end{enumerate}

\end{proof}
The next result is a line by line dual of the previous theorem and therefore its proof is omitted.

\begin{theorem}
\label{LSG}
Let $\Vtwowdw $ be a LSG and let $\Vwdw$ be a genop such that\\ 
$\Vwdw\supseteq \Vtwowdw $ and $\Vwdw\times \dw\equaln  \Vtwowdw\times \dw.$
Let $\Vw $ be invariant in $\Vtwowdw .$ Then
\begin{enumerate}
\item $\Vtwowdw \bigcap \Vw \equaln  \Vwdw \bigcap (\Vw \oplus (\Vw)_{\dw})$ and is a genop.\\
\item $\Vw $ is invariant in $\Vwdw$ and therefore $\Vwdw \plusn (\Vw)_{\dw} \equaln  \Vwdw \plusn (\Vw \oplus (\Vw)_{\dw})$ and\\ $\Vwdw \plusn (\Vw)_{\dw}$ is a genop. \\
\item   $p(\Vtwowdw\cap \Vw) \equaln  p(\Vwdw)\bigcap (\Vw \oplus (\Vw)_{\dw}).$ \\
\item   $p(\Vwdw+ (\Vw)_{\dw}) \equaln  p(\Vwdw)\plusn (\Vw)_{\dw}\equaln  p(\Vwdw)\plusn (\Vw \oplus (\Vw)_{\dw}).$ \\
\item  If $p(\Vwdw) $ is decoupled, so are $p(\Vtwowdw \cap \Vw)$
and $p(\Vwdw + (\Vw)_{\dw}).$\\
\item If $p_1(\Vtwowdw \cap \Vw), p_2(\Vwdw + (\Vw)_{\dw})$
are decoupled, then  $p_1p_2(\Vwdw) $ is decoupled.
\end{enumerate}
\end{theorem}

\section{Eigenvalue problem through annihilating polynomials}
\label{sec:eigen_ann}
A fundamental problem in control theory is to `place poles'
in a desired region. The classical result here is that
if the system is fully controllable, the poles can be placed where we wish
using state feedback and dually, if the system is fully observable, the poles can be placed where we wish
using output injection.
There are neat refinements of this result when the system is not fully
controllable or fully observable. In this section we derive implicit
versions of this result. 


\subsection{Pole placement through state feedback and output injection}
\label{subsec:statefeedback_ouputinjection}
We will work with regular dynamical systems (RGDS), i.e., $$\Vwdwmumy\circ W\supseteq (\Vwdwmumy\circ \dw)_W, \ \ \  \Vwdwmumy\times W\supseteq (\Vwdwmumy\times \dw)_W.$$
This implies that $\Vwdwmumy\circ \wdw$ is a USG and $\Vwdwmumy \times \wdw$
is an LSG. Therefore ideas from Theorems \ref{USG} and \ref{LSG} are applicable.

We now present the basic questions and remark upon their answers.
\begin{question}
Given an  RGDS $\Vwdwmumy$ what are the genops which can be obtained by
$wm_u-$ feedback and $m_y\mydot{w}-$ injection?
\end{question}

This question has been answered for general autonomous spaces in
Theorems \ref{thm:statefeedback} and \ref{thm:outputinjection} and therefore also for genops. In brief the answer is: \\
The genop $\Vwdw$  can be obtained by $wm_u-$ feedback iff  
it satisfies \\ $\Vwdwmumy\circ {\wdw}\supseteq \Vwdw; \ \ 
\Vwdwmumy\circ W\dw M_u\times {\dw}\subseteq \Vwdw \times \dw .$ \\ 
The genop $\Vwdw$  can be obtained by $m_y\mydot{w}-$ injection iff  
it satisfies\\
$\Vwdwmumy\times {\wdw}\subseteq \Vwdw; \ \ 
\Vwdwmumy\times W\dw M_y\circ W\supseteq \Vwdw \circ W.$ \\

\begin{question}
What polynomials can be made into minimal annihilating polynomials for 
genops obtainable by $wm_u-$ feedback and for those obtainable
by $m_y\mydot{w}-$ injection from the RGDF $\Vwdwmumy?$
\end{question}

We will only answer this question in a manner sufficient for 
pole placement purposes in order to simplify the discussion.

\vspace{0.4cm}

{\it $wm_u-$ feedback}

Let $\Vonewdw \equiv \Vwdwmumy\circ \wdw.$ We saw earlier that $\Vonewdw$ is a USG.
If $\Vonewdw$ is already a genop, it cannot be altered by state feedback.
We will therefore assume that $\Vonewdw\times W$ does not contain $(\Vonewdw\times \dw)_W.$ 

Let $\Vw^{com}\equiv \Vonewdw\times W\bigcap (\Vonewdw\times \dw)_W.$

Let $\Vw $ be the minimal invariant space of $\Vonewdw$ containing
$(\Vonewdw\times \dw)_W.$
In the case of $wm_u-$ feedback, the answer is :
\begin{enumerate}
\item A necessary condition is that the given polynomial must contain
as a factor, the minimal annihilating polynomial $p_u(s)$ of the genop 
$\Vonewdw+(\Vw)_{\dw} .$
\item Let $p_2(s)$ be any polynomial of degree equal to $r(\Vw)-r(\Vw^{com}).$
Then, through $wm_u-$ feedback, a genop $\Vwdw$ can be obtained from
$\Vwdwmumy$ for which  $p_up_2(s)$ is an annihilating polynomial.  
\end{enumerate}
In other words, the poles which are roots of $p_u(s)$ cannot be shifted
but we have full control over the remaining poles.

\vspace{0.4cm}

{\it $m_y\mydot{w}-$ feedback}

Let $\Vtwowdw \equiv \Vwdwmumy \times \wdw.$ We saw earlier that $\Vtwowdw$ is an LSG. 
If $\Vtwowdw$ is already a genop, it cannot be altered by output injection.
We will therefore assume that $\Vtwowdw\circ \dw$ is not  contained in  $(\Vtwowdw\circ W)_{\dw}.$\\
Let $\Vw^{com2}\equiv (\Vtwowdw\circ W)+(\Vtwowdw\circ \dw)_W.$

Let $\Vw $ be the maximal invariant space of $\Vtwowdw$ contained in
$\Vtwowdw\circ W.$
 In the case of $m_y\mydot{w}-$ injection, the answer is :
\begin{enumerate}
\item A necessary condition is that the given polynomial must contain
as a factor, the minimal annihilating polynomial $p_l(s)$ of the genop
$\Vtwowdw\bigcap \Vw .$
\item Let $p_1(s)$ be any polynomial of degree equal to $r(\Vw^{com2})-r(\Vw).$
Then, through $m_y\mydot{w}-$ injection, a genop $\Vwdw$ can be obtained from 
$\Vwdwmumy$ for which  $p_1p_l(s)$ is an annihilating polynomial.
\end{enumerate}
In other words, the poles which are roots of $p_l(s)$ cannot be shifted
but we have full control over the remaining poles.

\subsection{Constructing a genop with desired annihilating polynomial}
\label{sec:genopconsruction}

\vspace{0.4cm}

Next, let us consider 
\begin{question}
Let $\Vw$ 
be  the minimal invariant space of $\Vonewdw$ containing
$(\Vonewdw\times \dw)_W.$
How should we choose $\Vwdw $ in such a way that 
$\Vwdw\bigcap \Vw $
has a desired annihilating polynomial?
\end{question}

The solution is well known
in traditional control theory. We will give a simple solution in our frame work.

\subsubsection{Building the genop $\Vwdw^{start}$}
\label{sec:start_genop}

\vspace{0.4cm}

Our construction involves the idea of a `basic sequence' which is analogous 
to the sequence $$x^0, Ax^0, \cdots , A^kx^0, \cdots ,$$ for an operator $A.$
We will begin by building a genop $\Vwdw^{start}$ (Algorithm \ref{alg:basic_sequence})
with the annihilating polynomial
of degree equal to $r(\Vw)- r(\Vw^{com}),$ where $\Vw^{com}\equiv \Vonewdw\times W\bigcap (\Vonewdw\times \dw)_W.$
Next, we modify the basic sequence of $\Vwdw^{start}$ suitably (Equation \ref{eqn:xy}) to obtain 
another sequence whose genop $\Vwdw^{end}$
has the desired minimal annihilating polynomial.

\begin{definition}
Let $\Vonewdw$ be a USG and let $\Vw$ be invariant in it.
We say $x^0_W , x^1_W  ,\cdots  , x^k_W,$ is a basic sequence in $\Vonewdw\bigcap \Vw,$ iff 
$$x^j_W \in \Vw, \ \ j=0,1, \cdots , k$$  
$$\Vw \equaln sp\{x^j_W ,j=0,1,2,\cdots k-1\}+\Vw^{com}.$$
$$x^{j+1}_{\dw}\inn (x^j_W\lrar \Vonewdw),\ \ j=0,1, \cdots , k-1  
\ \ \textup{and}\ \   x^k_W \in \Vw. $$

\end{definition}

We first construct a basic sequence in $\Vonewdw\bigcap \Vw,$ 
and from  it build a genop $\Vwdw^{start},$ for which it is a basic sequence.

\begin{algorithm}
{\bf Algorithm Basic sequence}\\
\label{alg:basic_sequence}
Input: genaut $\Vonewdw$\\
Output: 
\begin{enumerate}
\item  A basic sequence $x^0_W , x^1_W  ,\cdots  , x^k_W,$
\item $\Vw ,$
the minimal space invariant in $\Vonewdw,$ containing $(\Vonewdw \times \dw)_W$
\item The genop $\Vwdw^{start}.$
\end{enumerate}

Step 1. Pick $x^0_W\inn [(\Vonewdw \times \dw)_W-\Vonewdw \times W].$

\vspace{0.2cm}

Step 2. Pick, if possible $x^{j+1}_{\dw}\inn (x^j_W\lrar \Vonewdw),$ and 
$x^0_W , x^1_W  ,\cdots  , x^{j+1}_W$ independent modulo $\Vw^{com}
.$

\vspace{0.2cm}

STOP if $ x^{j}_W\lrar \Vonewdw\subseteq (sp\{x^j_W ,j=0,1,2,\cdots j\}+\Vw^{com})_{\dw}.$

\vspace{0.1cm}

 $k\leftarrow j+1.$

\vspace{0.2cm}

$\Vw \leftarrow  sp\{x^j_W ,j=0,1,2,\cdots k-1\}+\Vw^{com}.$

\vspace{0.2cm}

$\Vwdw^{start}\leftarrow sp\{(x^j_W,x^{j+1}_{\dw}) ,j=0,1,2,\cdots k-1\}+(\Vw^{com}
\oplus (\Vw^{com})_{\dw})
.$

\vspace{0.2cm}

Algorithm ends.
\end{algorithm}

Note that the algorithm terminates, the first time the affine space
$ x^{j}_W\lrar \Vonewdw$ is contained in \\
$(sp\{x^i_W ,i=0,1,2,\cdots j\}+\Vw^{com})_{\dw}.$
This is when $j= k-1.$ Thus we have,\\
$(\Vw')_{\dw} \equiv  
(sp\{x^j_W ,j=0,1,2,\cdots k-1\}+\Vw^{com})_{\dw}$
contains $x^{k-1}_{\dw}+\Vonewdw \times \dw$ and therefore also 
$\Vonewdw \times \dw.$  \\
For $j< k-1,$ this does not happen.

Every vector in $x^{j+1}_{\dw}+\Vonewdw \times \dw=x^j_W\lrar \Vonewdw, j=0,1,2,\cdots k-1,$ belongs to $(\Vw')_{\dw}$ and 
$$\Vw^{com}\lrar\Vonewdw\subseteqn  \Vwdw\times \dw \subseteqn (\Vw')_{\dw}.$$
Thus $$\Vonewdw \lrar \Vw' \subseteq(\Vw')_{\dw}.$$
Now $\Vonewdw\circ W \supseteq \Vw',$ so that $$\Vonewdw \lrar (\Vw')_{\dw} \supseteq\Vw'.$$

Thus $\Vw'$ is invariant in $\Vonewdw.$

Next any invariant space in $\Vonewdw$ containing $(\Vonewdw \times \dw)_W $
must contain $x^0_W,$ \\
therefore $(x^{j}_W\lrar \Vonewdw)_W, \ \ j=0, \cdots , k-1.$\\
It follows therefore that it must contain $\Vw'.$\\
Thus, $sp\{x^j_W ,j=0,1,2,\cdots k-1\}+\Vw^{com} =\Vw,$ the minimal invariant space of
$\Vonewdw$\\
 containing $(\Vonewdw \times \dw)_W .$

\vspace{0.4cm}

It is clear that $\Vwdw^{start}$ is USG since $\Vwdw^{start}\circ W =\Vw$ and $\Vw$ is invariant in  $\Vonewdw .$

\vspace{0.1cm}

Next the vectors $x^0_W , x^1_W  ,\cdots  , x^{k}_W$ are independent modulo $\Vw^{com}.$

\vspace{0.1cm}

So the only way in which vectors
in $\{(x^j_W,x^{j+1}_{\dw}) ,j=0,1,2,\cdots k-1\}$ can be linearly combined to yield a vector of the kind $(f_W,g_{\dw}),$ with $f_W \in \Vw^{com},$ is through the trivial linear combination,\\
 which means that
$g_{\dw}\in (\Vw^{com})_{\dw}.$\\
Therefore, $\Vwdw^{start}\times \dw \subseteq (\Vw^{com})_{\dw}.$\\ 
It is LSG since $\Vwdw^{start}\times W \supseteq \Vw^{com}=(\Vwdw^{start}\times \dw )_W. $

Thus  $\Vwdw^{start}$ is a genop.
\begin{definition}
We refer to  $r(\Vw")- r(\Vw^{com}),$ as the essential rank of the invariant space
$\Vw"$ of $\Vonewdw .$ For a genop $\Vwdw$ the essential rank $er(\Vwdw),$ is defined similarly
to be $r(\Vwdw\circ W)- r(\Vwdw\times \dw).$ 
\end{definition}
(Note that for a genop
$\Vwdw \times W \supseteq (\Vwdw \times \dw)_W.$) 
\begin{definition}
The basic sequence $\{ x^j_W ,j=0,1,2,\cdots k \}$ 
output by the algorithm is referred to as a basic sequence for the genop 
$\Vwdw^{start}$ in the genaut $\Vonewdw ,$
while  the genop 
$$ sp\{(x^j_W,x^{j+1}_{\dw}) ,j=0,1,2,\cdots k-1\}+(\Vw^{com}
\oplus (\Vw^{com})_{\dw})$$ is the unique genop in the genaut $\Vonewdw $ for which $ \{x^j_W ,j=0,1,2,\cdots k\}$
is a basic sequence. We will refer to this genop as the genop of the basic sequence.
\end{definition}
\begin{remark}
Every genop does not have a basic sequence. A genop having a basic sequence is analogous to  a matrix $A$
for which there exists a vector $x$ in its domain which is annihilated by the minimal annihilating polynomial of $A.$
\end{remark}

\begin{claim}
\label{cl:annihilate_start}
Let $x^k_W+b _{k-1} x^{k-1}_W + \cdots + b _0   x^0_W\in \Vw^{com}.$
The polynomial 
$p^{start}(s) \equiv s^k +b_{k-1}s^{k-1}+\cdots b_0$
is the minimal annihilating polynomial for $\Vwdw^{start}.$
\end{claim}

To prove the claim, we need to show that $p^{start}(\Vwdw^{start})$ is decoupled
and if $p'(s)$ has degree lower than that of $p^{start}(s),$ then  
$p'(\Vwdw^{start})$ is not decoupled.

\vspace{0.2cm}

Since $\Vwdw^{start}$ is a genop, by Lemma \ref{lem:newgenoppoly}
, we have that
for any polynomial $p(s),$ $$p(\Vwdw^{start})\circ W\equaln  \Vwdw^{start}\circ W\equaln \Vw, \ \ \ p(\Vwdw^{start})\times \dw \equaln \Vwdw^{start}\times \dw\equaln  (\Vw^{com})_{\dw}.$$

This means that  $p(\Vwdw^{start})$ is decoupled if $$
p(\Vwdw^{start})\circ W\lrar  p(\Vwdw^{start})\equaln \Vw\lrar p(\Vwdw^{start})\subseteqn  p(\Vwdw^{start})\times \dw\equaln \Vwdw^{start}\times \dw\equaln  (\Vw^{com})_{\dw}.$$

We first show that for any vector $x^j_W,\ \ \  j=0, 1, \cdots , k-1,$ we must have  
$$x^j_W\lrar p^{start}(\Vwdw^{start})\subseteqn (\Vw^{com})_{\dw}.
$$
We have, $$(x^j_W)_{\dw}\in ((\Vwdw^{start})^{(j)}\lrar x^0_W).$$
Let $p_j(s)\equiv p^{start}(s)s^j.$
So $$p^{start}(\Vwdw^{start})\lrar x^j_W\subseteqn p^{start}(\Vwdw^{start})\lrarn((\Vwdw^{start})^{(j)}\lrar x^0_W)\equaln  p^{start}(\Vwdw^{start})* (\Vwdw^{start})^{(j)}\lrar x^0_W\equaln  p_j(\Vwdw^{start})\lrar x^0_W$$ 
$$=    [(\Vwdw^{start})^{(j)} * p^{start}(\Vwdw^{start})]\lrar x^0_W
\equaln (\Vwdw^{start})^{(j)}\lrar [p^{start}(\Vwdw^{start})\lrar x^0_W
]_{W}
\subseteqn  (\Vwdw^{start})^{(j)} \lrar\Vw^{com}\subseteqn (\Vw^{com})_{\dw}.$$
Since every $x^j_W, j=1, \cdots , k-1,$ satisfies $x^j_W\lrar p^{start}(\Vwdw^{start})\subseteq (\Vw^{com})_{\dw},$ it follows that for every linear combination $x_W$ of these vectors
we must have that $x_W\lrar p^{start}(\Vwdw^{start})\subseteq (\Vw^{com})_{\dw}.$\\
Next for any polynomial $p(s),$ we have $$(p(\Vwdw^{start})\times W) \lrarn p(\Vwdw^{start})\equaln p(\Vwdw^{start})\times \dw\equaln \Vwdw ^{start}\times \dw\equaln (\Vw^{com})_{\dw}.$$
Now $p(\Vwdw^{start})\times W\supseteq \Vwdw^{start}\times W\supseteq \Vw^{com}.$

Hence, $\Vw^{com}\lrarn p(\Vwdw^{start})\equaln  (\Vw^{com})_{\dw}.$\\
Since $sp\{x^j_W ,j=0,1,2,\cdots k-1\}+\Vw^{com} =\Vw,$
we therefore have  $$[p^{start}(\Vwdw^{start})
\circ W]\lrarn p^{start}(\Vwdw^{start})\equaln \Vw \lrar p^{start}(\Vwdw^{start})
\subseteqn (\Vw^{com})_{\dw}\equaln  p^{start}(\Vwdw^{start})
\times \dw,$$ 

i.e.,  $p^{start}(s)$ annihilates $\Vwdw^{start}.$

\vspace{0.2cm}

Next, let $p(s)\equiv \sum_{i=0}^t c_is^i$ and let $t\leq k-1.$
We have, \ \ $x^{t}_W \ \ \notin  \ \ sp\{x^0_W,x^1_W  ,\cdots  , x^{t-1}_W\}+\Vw^{com}.$ 

Let us temporarily, for the next couple of lines, denote $\Vwdw^{start}$ by $\Vwdw$ to simplify 
the notation.

Now $x^{t}_{\dw}\in \Vwdw^{(t)}\lrar x^0_W.$ It follows that 
$$(\Vwdw^{(t)}\pdw c_{t-1}^{\dws}\Vwdw^{(t-1)}\pdw \cdots \pdw  c_{1}^{\dws}\Vwdw\pdw c_{0}^{\dws}\Vwdw^{(0)})\lrar x^0_W  \not\subseteq (\Vw^{com})_{\dw}.$$
Thus $p(s)\equiv \sum_i^t c_is^i$ does not annihilate $\Vwdw.$

It follows that $p^{start}(s)$ is a polynomial of minimal degree that annihilates
$\Vwdw^{start}.$ This proves the claim.
\\

\subsubsection{Building the genop $\Vwdw^{end}$}
\label{sec:end_genop}

\vspace{0.4cm}

%
Let $p^{end}(s)\equiv s^k +c_{k-1}s^{k-1}+\cdots c_0,$
be a specified polynomial. We will show how to modify $\Vwdw^{start}$
to $\Vwdw^{end}$ so that the latter has $p^{end}(s)$ as its minimal annihilating polynomial.
We do this by building a basic sequence for $\Vwdw^{end}.$

Let  $y^0_W , y^1_W  ,\cdots  , y^{k}_W$
be defined through 
\begin{equation}
\label{eqn:xy}
\begin{matrix}
y^0_W  &&=&&& x^0_W \\
y^1_W  &&=& &&x^1_W+\lambda _1x^0_W\\ 
    & &\vdots&& & \\
y^k_W &&=& &&x^k_W+\lambda _1 x^k_W+\cdots + \lambda _k x^0_W
\end{matrix}
\end{equation}
We will adjust the $\lambda _i$ so that we have 
$$y^k_W+c _{k-1} y^{k-1}_W + \cdots + c _0   y^0_W\in \Vw^{com}.$$

Consider the vector equation, treating $x^i_W$ as symbols
\begin{align}
\bbmatrix{
        x^0_W & x^1_W  &\cdots  & x^k_W \\
        0     &x^0_W   & \cdots  & x^{k-1}_W \\
             &        &\vdots  &  \\
        0     &  0      &\cdots  &x^0_W  \\
       }
\ppmatrix{
  c_0 \\   \\  \vdots  \\ c_k
}
= \ppmatrix{
  \alpha_0 \\   \\  \vdots  \\ \alpha_k
}
\end{align}
(In the above equation $x^i_W, \alpha _i $ can be thought of as
column vectors.)
It is easy to see that this equation in the symbols $x^i_W,$ can be rewritten as 
\begin{align}
\bbmatrix{
        c_0 & c_1  &\cdots  & c_k \\
         c_1     & \cdots  & c_{k}& 0 \\
         \vdots    &        &\vdots  &0  \\
        c_{k}     &  0      &\cdots  &0  \\
       }
\ppmatrix{
  x^0_W \\   \\  \vdots  \\ x^k_W 
}
=\ppmatrix{
  \alpha_0 \\   \\  \vdots  \\ \alpha_k
}.
\end{align}
(Here however the  symbols $x^i_W, \alpha _i $ are to be interpreted as row vectors.)

Now consider the expression
\begin{align}
 \label{eqn:xyc}
\ppmatrix{\lambda _0& \cdots & \lambda _k
}
\bbmatrix{
        x^0_W & x^1_W  &\cdots  & x^k_W \\
        0     &x^0_W   & \cdots  & x^{k-1}_W \\
             &        &\vdots  &  \\
        0     &  0      &\cdots  &x^0_W  \\
       }
\ppmatrix{
  c_0 \\   \\  \vdots  \\ c_k
}
= z .
\end{align}
If we choose the $\lambda _i$ so that the Equation \ref{eqn:xy}
is satisfied,  Equation \ref{eqn:xyc} would reduce to 
\begin{align}
\bbmatrix{
        y^0_W & y^1_W  &\cdots  & y^k_W 
       }
\ppmatrix{
  c_0 \\   \\  \vdots  \\ c_k
}
= z.
\end{align}
We would like to choose $y^i_W$ so that $z \in \V^{com}_W.$  
So we first rewrite the expression  in Equation \ref{eqn:xyc}
as 
\begin{align}
\ppmatrix{\lambda _0& \cdots & \lambda _k
}
\bbmatrix{
        c_0 & c_1  &\cdots  & c_k \\
         c_1     & \cdots  & c_{k}& 0 \\
         \vdots    &        &\vdots  &0  \\
        c_{k}     &  0      &\cdots  &0  \\
       }
\ppmatrix{
  x^0_W \\   \\  \vdots  \\ x^k_W 
}
=z  
\end{align}
and choose the $\lambda _i$ so that 
\begin{align}
\label{eqn:c2b}
\ppmatrix{\lambda _0& \cdots & \lambda _k
}
\bbmatrix{
        c_0 & c_1  &\cdots  & c_k \\
         c_1     & \cdots  & c_{k}& 0 \\
         \vdots    &        &\vdots  &0  \\
        c_{k}     &  0      &\cdots  &0  \\
       }=
 \ppmatrix{b_0& \cdots & b_k
}
\end{align}
This is always possible since the coefficient matrix has nonzero entries
$c_k=1$ along the skew diagonal and zero entries beneath it.
But now the expression in Equation \ref{eqn:xyc}
 reduces to 
\begin{align}
\ppmatrix{b_0& \cdots & b_k
}
\ppmatrix{
  x^0_W \\   \\  \vdots  \\ x^k_W 
}
\in \Vw^{com},
\end{align}
since $x^k_W+b _{k-1} x^{k-1}_W + \cdots + b _0   x^0_W\in \Vw^{com}.$

\vspace{0.4cm}

Thus if we choose 
\begin{equation}
\label{eqn:find_lambda}
\ppmatrix{\lambda _0& \cdots & \lambda _k
}=\ppmatrix{b_0& \cdots & b_k}
{\bbmatrix{
        c_0 & c_1  &\cdots  & c_k \\
         c_1     & \cdots  & c_{k}& 0 \\
         \vdots    &        &\vdots  &0  \\
        c_{k}     &  0      &\cdots  &0  \\
       }}^{-1},
\end{equation}

\vspace{0.4cm}

and use this vector 
$({\lambda _0, \cdots , \lambda _k})$ in Equation \ref{eqn:xy},
the resulting basic sequence 
$y^0_W , y^1_W  ,\cdots  , y^{k}_W$\\
would define the genop $\Vwdw^{end}$ through
$$\Vwdw^{end}\equiv sp\{(y^j_W,y^{j+1}_{\dw}) ,j=0,1,2,\cdots k-1\}+(\Vw^{com}\oplus (\Vw^{com})_{\dw}).$$
Since
$y^k_W+c _{k-1} y^{k-1}_W + \cdots + c _0   y^0_W\in \Vw^{com},$
the polynomial 
$p^{end}(s)\equiv s^k +c_{k-1}s^{k-1}+\cdots c_0,$ annihilates $\Vwdw^{end}.$ 

We summarize the above procedure to construct a genop with the desired annihilating polynomial 
in the following algorithm.

\begin{algorithm}
\label{alg:c2g}
{\bf Algorithm Annihilating polynomial to genop }\\

Input: A genop $\Vwdw^{start}\subseteq \Vonewdw $ with annihilating polynomial $p_b(s)\equiv \sum_{i=0}^kb_is^i,$
and a desired annihilating polynomial  $\sum_{i=0}^kc_is^i.$

Output:  A genop $\Vwdw^{end}\subseteq \Vonewdw $ with annihilating polynomial $p_c(s)\equiv \sum_{i=0}^kc_is^i.$

Step 1. Compute $\lambda _i, i= 0, \cdots , k$ from $p_b(s),p_c(s)$ using Equation \ref{eqn:find_lambda}.

Step 2. Compute the basic sequence $y_W^0, \cdots , y_W^k$ from a basic sequence $x_W^0, \cdots , x_W^k$
of $\Vwdw^{start}$ using Equation \ref{eqn:xy}.

Step 3.  $\Vwdw^{end}\leftarrow sp\{(y^j_W,y^{j+1}_{\dw}) ,j=0,1,2,\cdots k-1\}+(\Vw^{com}
\oplus (\Vw^{com})_{\dw})
.$
 
Algorithm Ends.
\end{algorithm}

\subsubsection{Building the genop $\Vwdw^{new}$}
\label{sec:new_genop}

\vspace{0.4cm}

Now we could grow $\Vwdw^{end}$ to a genop $\Vwdw^{new},$ where $\Vwdw^{new}\circ W=
\Vonewdw \circ W.$ For this we use the following algorithm which starts with a basic sequence 
in $\Vonewdw \bigcap \Vw $ and outputs $\Vwdw^{new}.$
\begin{algorithm}
\label{alg:b2g}
{\bf Algorithm  Basic2genop}\\

Input: A genop $\Vwdw^{end}\subseteq \Vonewdw .$

Output:
A genop $\Vwdw^{new},$ where $\Vwdw^{new}\circ W=
\Vonewdw \circ W$ and $\Vwdw^{new}\supseteq \Vwdw^{end}.$


Step 1. Let $t=r(\Vonewdw\circ W)-r(\Vwdw^{end}\circ W).$
Construct a sequence of vectors $u^1_W , u^2_W  ,\cdots  , u^{t}_W$
from $\Vonewdw\circ W,$ which are independent modulo $\Vw\equiv \Vwdw^{end}\circ W.$ 

(Thus if $sp\{y^0_W , y^1_W  ,\cdots  , y^{k}_W\}+\Vw^{com}= \Vwdw^{end}\circ W,$ then  $$sp\{y^0_W , y^1_W  ,\cdots  , y^{k}_W, u^1_W , u^2_W  ,\cdots  , u^{t}_W\}+\Vw^{com}= \Vonewdw\circ W.)$$

Step 2.
Pick vectors $v^1_{\dw} , v^2_{\dw}  ,\cdots  , v^{t}_{\dw}$
such that $(u^1_W,v^2_{\dw}), \cdots , (u^{t}_W,v^{t}_{\dw}) \in \Vonewdw .$\\

Step 3.
$\Vwdw^{new}\leftarrow sp\{ (u^1_W,v^1_{\dw}), \cdots , (u^{t}_W,v^{t}_{\dw})\} +  sp\{(y^j_W,y^{j+1}_{\dw}) ,j=0,1,2,\cdots k-1\}+(\Vw^{com}\oplus (\Vw^{com})_{\dw}).$

Algorithm ends.
\end{algorithm}

\subsubsection{Proof that $\Vwdw^{new}$ is a genop
with desired annihilating polynomial}
\label{sec:proof_genop_ann_poly}

\vspace{0.4cm}

We will now show that $\Vwdw^{new}$ is a genop 
and  has as an annihilating polynomial 
$p_up^{end}(s),$ where $p_u(s) $ is the minimal annihilating polynomial of
$\Vonewdw +(\Vw\oplus(\Vw)_{\dw})$ and
$p^{end}(s)\equiv
s^k +c_{k-1}s^{k-1}+\cdots c_0,$
the annihilating polynomial of $\Vwdw^{end}.$

\vspace{0.2cm}

We know that  $$\Vwdw^{new}\circ W= \Vonewdw \circ W\supseteq (\Vonewdw \circ \dw)_W \supseteq (\Vwdw^{new}\circ \dw)_W.$$
Next the the vectors $y^0_W , y^1_W  ,\cdots  , y^{k}_W, u^1_W , u^2_W  ,\cdots  , u^{t}_W$ are independent modulo $\Vw^{com}.$\\
So the only way in which vectors 
in $$\{(y^j_W,y^{j+1}_W) ,j=0,1,2,\cdots k-1\}\bigcup \{(u^1_W,v^1_W), \cdots , (u^{t}_W,v^{t}_W)\}$$ 
can be linearly combined to yield a vector of the kind $(f_W,g_{\dw}),$ with $f_W \in \Vw^{com}$ is through the trivial linear combination, which means that
$$g_{\dw}\in (\Vw^{com})_{\dw}.$$
Therefore, $\Vwdw^{new}\times \dw \subseteq (\Vw^{com})_{\dw}.$ We already know that
$\Vwdw^{new}\times \dw \supseteq (\Vw^{com})_{\dw},$ so that we must have 
$$\Vwdw^{new}\times \dw \ (\equaln \Vw^{com})_{\dw}.$$ It is clear that $\Vwdw^{new}\times W\supseteq \Vw^{com}$ and $(\Vwdw^{new}\times W)_{\dw} \supseteq (\Vw^{com})_{\dw}.$ 
Thus we conclude $$\Vwdw^{new}\times W\supseteq \Vw^{com}= (\Vwdw^{new}\times \dw)_W.$$
Hence $\Vwdw^{new}$ is a genop and has as an annihilating polynomial 
(using part 6 of Theorem \ref{USG})
$p_up^{end}(s),$ where $p_u(s) $ is the minimal annihilating polynomial of
$\Vonewdw +(\Vw\oplus(\Vw)_{\dw})$ and 
$p^{end}(s),$ 
the annihilating polynomial of $\Vwdw^{end}.$

Further, the  minimal annihilating polynomial of $\Vwdw^{new}$
contains both $p_u(s) $ and $p^{end}(s)$ as factors by Theorem \ref{USG}.
Thus the poles of the controlled dynamical system are precisely the roots of
$p_u(s) $ and $p^{end}(s).$ 

\vspace{0.4cm}

Finally recall  that any generalized autonomous system $\Vwdw$ that satisfies 
\begin{enumerate}
\item $\Vwdwmumy\circ \wdw \supseteq \Vwdw,$
\item $\Vwdwmumy\circ \wdw M_u\times \dw  \subseteq \Vwdw \times \dw,$
\end{enumerate}
can be realized by  $wm_u-$ feedback  through $\Vwmu,$ i.e.,\\ $\Vwdw=(\Vwdwmumy\bigcap \Vwmu)\circ \wdw,$ where
$\Vwmu \equiv (\Vwdwmumy\bigcap \Vwdw)\circ WM_u.$ 

The first condition states simply that $\Vonewdw\supseteq \Vwdw.$
This is clearly satisfied by the genop $\Vwdw^{new}$ constructed by the Algorithm \ref{alg:b2g}.

To satisfy the  second condition we need  that $$(\Vwdwmumy\circ \wdw M_u\times \dw)_{W}
\subseteq \Vw^{com}\equiv \Vonewdw\times W \bigcap (\Vonewdw\times\dw)_{W}.$$
In commonly encountered situations if $w(0)=0$ and the input is zero,
the derivative $\mydot{w}(t)$ will be zero for $t\geq 0$ (and therefore also for 
$t=0$), which means
$\Vwdwmumy\circ \wdw M_u\times \dw=\0_{\dw},$ so that the second condition is also
satisfied by  $\Vwdw^{new}.$ 

\begin{remark}
\label{rem:arbitrayannpoly}
Let $\Vw$
be  the minimal invariant space of a USG $\Vonewdw$ containing
$(\Vonewdw\times \dw)_W.$ We have seen above that we can choose a genop $\Vwdw $ in such a way that
$\Vwdw\bigcap \Vw $ has any desired annihilating polynomial of degree $k,$ where
$$k\equiv r(\Vw) - r(\Vw^{com}),\ \ 
\Vw^{com}\equiv \Vonewdw\times W \bigcap (\Vonewdw\times\dw)_{W}.$$

A consequence of part $4$  of Lemma \ref{lem:twogenasrestcont}
 is that, since $$ \Vw = \Vwdw \bigcap \Vw\circ W = \Vonewdw\bigcap \Vw\circ W, \ \   \Vwdw \bigcap \Vw \subseteq \Vonewdw\bigcap \Vw,$$ the space $\Vonewdw\bigcap \Vw,$ which is USG,  can also be annihilated by any annihilating polynomial of degree $k,$ in particular, by $s^k.$

\vspace{0.4cm}

Dually, if $\Vw$
be  the maximal invariant space of an LSG $\Vtwowdw,$ contained in
$\Vtwowdw\circ W,$\\ then  by part $5$ of Lemma \ref{lem:twogenasrestcont},

$\Vtwowdw + (\Vw)_{\dw},$ can be annihilated by 
any annihilating polynomial of degree $k,$ $$ k \equiv r[\Vtwowdw\circ  W \plusn (\Vtwowdw\circ \dw)_{W}]- r(\Vw), $$ in particular, by $s^k.$ 

\end{remark}
The foregoing algorithms  can be used indirectly to do pole placement
for $m_y\dW-$ injection.
One does this by simply going to the adjoint and using the algorithms of the previous subsections
to obtain suitable genops, taking their adjoints and returning to the primal dynamical system.
It would, however, be of some interest to recast the algorithm in such a manner that they can be
line by line dualized. This we do in \ref{sec:ppa_dualization}.
\section{Emulators}
\label{sec:emu}
\subsection{Introduction}
Dynamical systems do not always come in a form that is suitable
for theoretical and practical studies. We usually convert the original system
into another dynamical system that is more convenient for our purposes.
The dynamical system as described by state variables is a very good example.
We call such transformed systems, which capture the original dynamical
system sufficiently well so that theoretical and computational studies 
can be carried out,
 `emulators'. A formal definition and examples follow.

\begin{definition}
\label{def:emulators}
Dynamical systems $\Vwdwm,\Vpdpm$ are said to be {\bf emulators}
of each other iff 
$\exists \Vonewp, \Vtwodwdp$  with $\Vonewp \supseteq (\Vtwodwdp)_{WP}$ such that 
$$\Vwdwm=\Vpdpm \lrar (\Vonewp \bigoplus \Vtwodwdp)$$ 
and
$$\Vpdpm=\Vwdwm \lrar (\Vonewp \bigoplus \Vtwodwdp).$$
In particular, when $M$ is void,  $\Vwdw,\Vpdp$ are said to be {\bf emulators}
of each other iff 
$\exists \Vonewp, \Vtwodwdp,$ with $\Vonewp \supseteq (\Vtwodwdp)_{WP},$ such that,
$$\Vwdw=\Vpdp \lrar (\Vonewp \bigoplus \Vtwodwdp)$$ 
and
$$\Vpdp=\Vwdw \lrar (\Vonewp \bigoplus \Vtwodwdp).$$
If such $\Vonewp,\Vtwodwdp$ exist, we say that $\{\Vwdwm,\Vpdpm\}$
is {\bf $\E-$ linked} through $\Vonewp,\Vtwodwdp,$ or that $\{\Vonewp,\Vtwodwdp\}$
is an {\bf $\E-$ linkage pair} for $\{\Vwdwm,\Vpdpm\}.$
\end{definition}
An immediate consequence is that elimination 
of the $M$ variables continues to result in a pair of spaces
linked through the same {\bf $\E-$ linkage pair}.
\begin{lemma}
\label{lem:emlinking}
Let $\{\Vwdwm,\Vpdpm\}$
be { $\E-$ linked} through $\Vonewp,\Vtwodwdp .$\\
Then so is $\{\Vwdwm\lrar \Vm,\Vpdpm\lrar \Vm\}$
{ $\E-$ linked} through $\Vonewp,\Vtwodwdp .$
\end{lemma}
\begin{proof}
Observe that, since no index set occurs more than twice in the expression,
by Theorem \ref{thm:notmorethantwice},
$$(\Vwdwm\lrar \Vm)\lrar (\Vwp\oplus \Vdwdp)= (\Vwdwm\lrar(\Vwp\oplus \Vdwdp))\lrar \Vm.$$
\end{proof}

Example \ref{eg:emulator}
illustrates the notion of an emulator.

The $\E-$ linkage is analogous to similarity transformation but much less
restrictive in scope. For instance, it permits linking between systems with
different numbers of variables (i.e., $|W|$ and $|P|$ need not be the same).
If one thinks of $\Vonewp\lrar \Vwdw\lrar\Vtwodwdp$ as analogous to $T^{-1}AT,$
$\Vonewp$ takes on the role of $T^{-1},$ while $\Vtwodwdp$ that of
$T.$ Some flexibility is provided by $\Vonewp, \Vtwodwdp$
to treat the $W,P$ variables differently from $\dw,\dP$ variables.
This is appropriate because the derivatives of dynamic variables
usually satisfy constraints of the dynamic variables and some additional ones.
However, the cancellation that occurs when we multiply $T^{-1}AT$ by $T^{-1}AT$
resulting in $T^{-1}A^2T$
is mimicked by $\E-$ linkage. Lemma \ref{lem:cancel} shows that $\Vonewp\lrar \Vwdw \lrar \Vtwodwdp$ when `multiplied' by itself results in $\Vonewp\lrar \Vwdw^{(2)} \lrar \Vtwodwdp.$

When systems are emulators of each other, we will show that, when we need to make computations such as finding suitable conditioned and controlled invariant
spaces   on $\Vwdwm,$ we can first go to $\Vpdpm$ through
the $\E-$ linkage, make the computations there and get back to $\Vwdwm$
again through the $\E-$ linkage. As another example, we will show that to perform
$wm_u-$ feedback we can first go to $\Vpdpm$ through
the $\E-$ linkage, do the $pm_u-$ feedback $\V_{PM_u}\equiv \Vonewp \lrar \V_{WM_u}$ and get back to $\Vwdwm$
through the linkage $\{\Vonewp, \Vtwodwdp\}.$ The effect would be the same
as performing $wm_u-$ feedback directly on $\Vwdwm.$
A similar statement holds for $m_y\dW-$ injection.

Lastly, and most importantly, spectral computations can be performed
on emulators and transferred to the original dyanmical system.

We will show that for a large class of electrical circuits,
which could be regarded as generic, some emulators can be built in near linear time and these are in fact more useful than state equations.

The following result describes an essential characteristic of emulator linkage pairs.
\begin{theorem}
\label{thm:elinkagedotcross}
Let $(\Vonewp,\Vtwodwdp)$
be an {\bf $\E-$ linkage} for $\{\Vwdwm,\Vpdpm\}.$
Then the following `dot-cross' conditions 
hold
\begin{subequations}
\label{eqn:emulatorsW}
\begin{align}
\Vonewp\circ W&\supseteq \Vwdwm\circ W\\
\Vonewp\times W&\subseteq \Vwdwm\times W\\
\Vtwodwdp\circ \dw&\supseteq \Vwdwm\circ \dw\\
\Vtwodwdp\times \dw&\subseteq \Vwdwm\times \dw\\
\end{align}
\end{subequations}
and also the same conditions replacing $W$ by $P,$
\begin{subequations}
\label{eqn:emulatorsW}
\begin{align}
\Vonewp\circ P&\supseteq \Vpdpm\circ P\\
\Vonewp\times P&\subseteq \Vpdpm\times P\\
\Vtwodwdp\circ \dotp&\supseteq \Vpdpm\circ \dotp\\
\Vtwodwdp\times \dotp&\subseteq \Vpdpm\times \dotp .
\end{align}
\end{subequations}
Further, if the above conditions hold, then 
$$\Vwdwm=\Vpdpm \lrar (\Vonewp \bigoplus \Vtwodwdp)$$
iff
$$\Vpdpm=\Vwdwm \lrar (\Vonewp \bigoplus \Vtwodwdp).$$
\end{theorem}
\begin{proof}
The proof follows from Theorem \ref{thm:inverse} and the fact that \\
$(\Vab\oplus\Vcd) \circ A= \Vab \circ  A$
and $(\Vab\oplus \Vcd) \times A= \Vab \times A,$ when $A,B,C,D$ are mutually disjoint.
\end{proof}
The next lemma stresses an easy consequence of the definition of
$\E-$ linkage, which is nevertheless useful.
\begin{lemma}
\label{lem:elinking}
 Let $\{\Vwdw,\Vpdp\}$
be { $\E-$ linked} through $\Vonewp,\Vtwodwdp .$
Then
\begin{enumerate}
\item $$\Vwdw \circ W\lrar \Vonewp =\Vpdp \circ P,\ \ \ \Vwdw \times W\lrar \Vonewp =\Vpdp \times P,$$
$$\Vwdw \circ \dw \lrar \Vtwodwdp =\Vpdp \circ \dP,\ \ \ \Vwdw \times \dw \lrar \Vtwodwdp =\Vpdp \times \dP.$$
\item If either of $\Vwdw,\Vpdp$ is decoupled so would the other be.
\end{enumerate}
\end{lemma}
\begin{proof}
\begin{enumerate}
\item
Follows directly from the definition of { $\E-$ linkage}.
\item If $\Vwdw $ is decoupled then $\Vwdw = (\Vwdw \circ W)\oplus (\Vwdw \circ \dw).$ The result now follows from the previous part of the present lemma.
\end{enumerate}
\end{proof}

\subsection{Motivation}
In this subsection we illustrate the idea of emulators through examples involving electrical networks. Our aim is two fold. Firstly we try to show that something
of the level of generality of emulators is needed to deal {\it implicitly} with the objects we are concerned
with.  Secondly our aim is to show that for large systems
which interlink small dynamical systems through graphs, emulators can be built 
very {\it efficiently} and can serve as convenient starting points for practical
computations.
 

%
%
%
%
\label{eg:emulator}
\begin{figure}
\begin{center}
 \includegraphics[width=3in]{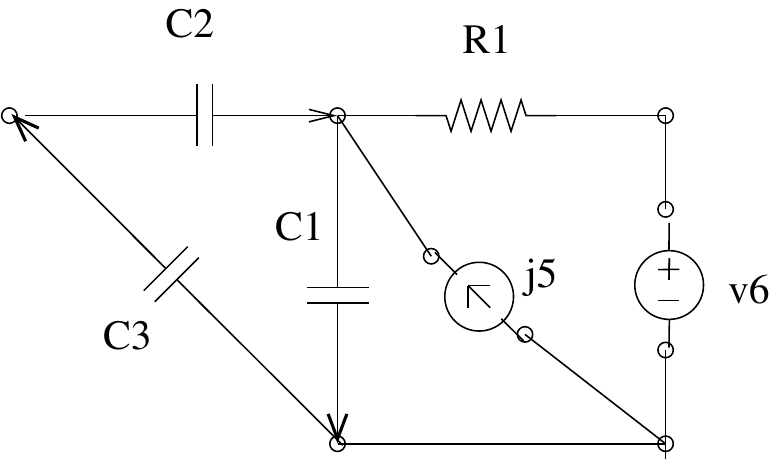}
 \caption{ An $RC$ circuit
}
 \label{fig:circuit6}
\end{center}
\end{figure}
\begin{example}
Consider the GDSA represented by the circuit $\mathcal{N}$ in Fig \ref{fig:circuit6}.
Let us take all capacitors of value $1 F$ and all resistors of value $1\Omega .$
The  corresponding GDS  $\ \Vwdwm$ is the solution space of the equations given below. Here the $W$ variables are $v_{C1},v_{C2},v_{C3},$ the $\dw$ variables are ${\mydot{v}}_{C1} ,{\mydot{v}}_{C2} ,{\mydot{v}}_{C3} ,$ and the $M$ variables
are $j_5 ,v_6,  v_5, i_6 .$
\begin{subequations}
\label{eqn:example1}
\begin{align}
 \begin{pmatrix}
  {\mydot{v}}_{C1} \\ 
  {\mydot{v}}_{C2} \\ 
  {\mydot{v}}_{C3}  
 \end{pmatrix}
 &=
 \begin{bmatrix}
 -2/3   & 0 &0 \\
  1/3 &0  &0 \\
  1/3 &0  &0 
 \end{bmatrix}
 \begin{pmatrix}
   v_{C1} \\ 
   v_{C2} \\ 
   v_{C3}  
 \end{pmatrix}
+
 \begin{bmatrix}
   2/3 \\ -1/3 \\ -1/3
 \end{bmatrix}
 v_6 
 + \begin{bmatrix}
   2/3 \\ -1/3 \\ -1/3
 \end{bmatrix}
 j_5\\
0  &= 
\begin{bmatrix}
1&  1 &  1 
\end{bmatrix}
\begin{pmatrix}
  v_{C1} \\ 
  v_{C2} \\
  v_{C3} 
 \end{pmatrix}\\
\begin{pmatrix}
 i_6 \\ 
v_5
 \end{pmatrix}
 &=
 \begin{bmatrix}
  1 &0  &0 \\
  -1 &0  &0 
 \end{bmatrix}
 \begin{pmatrix}
   v_{C1} \\ 
   v_{C2} \\ 
   v_{C3}  
 \end{pmatrix}
+
 \begin{bmatrix}
  -1 \\ 0
 \end{bmatrix}
 v_6 
 + \begin{bmatrix}
  0 \\ 0 
 \end{bmatrix}
 j_5.
\end{align}
\end{subequations}
Since $P$ variables and $W$ variables have to be disjoint and
so also the variables $\dw$ and ${\mydot{P}}$ we define the new $P,\dP$
variables
$v_{C1}' , v_{C2}',{\mydot{v}}_{C1}', {\mydot{v}}_{C2}'$
taking  $v_{C1}'=v_{C1},v_{C2}'=v_{C2},{\mydot{v}}_{C1}'={\mydot{v}}_{C1},
{\mydot{v}}_{C2}'={\mydot{v}}_{C2}.$
The GDS $\Vpdpm, $ is the   solution space of the state and output equations (of the circuit in Figure \ref{fig:circuit6})  given below:
\begin{subequations}
\label{eqn:example1state}
\begin{align}
 \begin{pmatrix}
  {\mydot{v}}_{C1}' \\ 
  {\mydot{v}}_{C2}' 
 \end{pmatrix}
 &=
 \begin{bmatrix}
 -2/3   & 0  \\
  1/3 & 0   
 \end{bmatrix}
 \begin{pmatrix}
   v_{C1}' \\ 
   v_{C2}'  
 \end{pmatrix}
+
 \begin{bmatrix}
   2/3 \\ -1/3
 \end{bmatrix}
 v_6 
 + \begin{bmatrix}
   2/3 \\ -1/3 
 \end{bmatrix}
 j_5\\
\begin{pmatrix}
 i_6 \\ 
v_5
 \end{pmatrix}
 &=
 \begin{bmatrix}
  1 &0   \\
  -1 &0   
 \end{bmatrix}
 \begin{pmatrix}
   v_{C1}' \\ 
   v_{C2}'  
 \end{pmatrix}
+
 \begin{bmatrix}
  -1 \\ 0
 \end{bmatrix}
 v_6 
 + \begin{bmatrix}
  0 \\ 0 
 \end{bmatrix}
 j_5.
\end{align}
\end{subequations}
The GDS pair $\{\Vwdwm,\Vpdpm\}$ is $\mathcal{E}$- linked and the two GDS's
are emulators of each other.\\
Next let us build a suitable $\mathcal{E}$-linkage $\{\Vonewp,\Vtwodwdp\}$ for this pair.
Observe that the constraint on $ v_{C1} , v_{C2} , v_{C3} $
is $$v_{C1} + v_{C2} + v_{C3}=0. $$
The constraints on 
$ {\mydot{v}}_{C1}, {\mydot{v}}_{C2}, {\mydot{v}}_{C3},$ 
are 
$$
\begin{pmatrix}
 {\mydot{v}}_{C1}+ {\mydot{v}}_{C2}+ {\mydot{v}}_{C3}\\ 
 {\mydot{v}}_{C2}/C_2- {\mydot{v}}_{C3}/C_3 
\end{pmatrix}
=
\begin{pmatrix}
0\\
0
\end{pmatrix}
.
$$
Thus the space of all vectors  $ (v_{C1} , v_{C2} , v_{C3}) $
is the space spanned by the rows of 
$$\begin{bmatrix}
2& -1& -1 \\
0& 1 & -1
\end{bmatrix}
.$$
Similarly the space of all vectors  
$( {\mydot{v}}_{C1}, {\mydot{v}}_{C2}, {\mydot{v}}_{C3})$
is the space spanned by the row 
$$
\begin{bmatrix}
2& -1& -1 
\end{bmatrix}
.$$
The space $\Vonewp$ can be the space of all vectors  $ (v_{C1} , v_{C2} , v_{C3}, v_{C1}' , v_{C2}') $
spanned by the rows of
$$ \begin{bmatrix}
2& -1& -1 & 2& -1\\
0& 1 & -1 & 0& 1
\end{bmatrix}
.$$

The space $\Vtwodwdp $ can be the space of all vectors
$( {\mydot{v}}_{C1}, {\mydot{v}}_{C2}, {\mydot{v}}_{C3}, {\mydot{v}}_{C1}', {\mydot{v}}_{C2}')$
spanned by the row
$$\begin{bmatrix}
2& -1& -1  & 2& -1
\end{bmatrix}
.$$
One can verify just by inspection that 
\begin{equation}
\label{eqn:linkagefromemulators}
\Vwdwm\lrar (\Vonewp \bigoplus \Vtwodwdp)=\Vpdpm ,
\end{equation}
and also that 
all the following `dot-cross' conditions are satisfied :
$$\Vonewp\circ W=  \Vwdwm\circ W,
\Vonewp\times W\subseteq  \Vwdwm\times W,\ \ \ 
\Vtwodwdp\circ \dw= \Vwdwm\circ \dw,
\Vtwodwdp\times \dw= \Vwdwm\times \dw,$$
$$\Vonewp\circ P=  \Vpdpm\circ P,
\Vonewp\times P\subseteq   \Vpdpm\times P,\ \ \ 
\Vtwodwdp\circ \dP= \Vpdpm\circ \dP,
\Vtwodwdp\times \dP=  \Vpdpm\times \dP.$$
This is equivalent to working with the reduced set of variables 
$ v_{C1} , v_{C2}, {\mydot{v}}_{C1}, {\mydot{v}}_{C2}$
when we move from $W$ to $P$ and $\dw$ to $\dP$ variables.
It must be noted that other linkages would also work in this case.
For instance we could have taken both $\Vonewp$ and $\Vtwodwdp$ to be spanned
by the vectors 
$$\begin{bmatrix}
2& -1& -1 & 2& -1\\
0& 1 & -1 & 0& -1
\end{bmatrix}
.$$
However, this latter linkage does not reveal as much about the relationship
between the emulators as the linkage defined earlier and in this case 
the   conditions 
$$\Vtwodwdp\circ \dw \supseteq \Vwdwm\circ \dw,\ \ \ 
%
\Vtwodwdp\circ \dP\supseteq \Vpdpm\circ \dP$$

are not satisfied with equality.

Next let us consider an emulator which does not correspond to state equations of the system.
It is  important because it is easy to construct algorithmically 
(near linear time) but captures the eigen spaces of the system, corresponding to 
nonzero eigen values, exactly. It is sufficient for complete solution of the network
as we show in Example \ref{eg:strongemulator}.

 We define the new $Q,\dQ$
variables
$v_{C} , {\mydot{v}}_{C},$
taking  $$v_{C}=v_{C1},\ \  v_C = -v_{C2}-v_{C3},\ \ {\mydot{v}}_{C}={\mydot{v}}_{C1}+
\frac{1}{2}({\mydot{v}}_{C1}+{\mydot{v}}_{C2}).$$
The GDS $\Vqdqm, $ is the   solution space of the equations (of the circuit in Figure \ref{fig:circuit6})  given below:

\begin{subequations}
\label{eqn:example1state}
\begin{align}
  {\mydot{v}}_{C} 
 &= 
 \begin{bmatrix}
 -2/3   
 \end{bmatrix}
   v_{C} 
+
 \begin{bmatrix}
   2/3 
 \end{bmatrix}
 v_6 
 + \begin{bmatrix}
   2/3 
 \end{bmatrix}
 j_5\\
\begin{pmatrix}
 i_6 \\ 
v_5
 \end{pmatrix}
 &=
 \begin{bmatrix}
  1    \\
  -1    
 \end{bmatrix}
   v_{C}  
+
 \begin{bmatrix}
  -1 \\ 0
 \end{bmatrix}
 v_6 
 + \begin{bmatrix}
  0 \\ 0 
 \end{bmatrix}
 j_5.
\end{align}
\end{subequations}

The GDS pair $\{\Vwdwm,\Vqdqm\}$ is $\mathcal{E}$- linked and the two GDS's
are emulators of each other.\\
Next let us build a suitable $\mathcal{E}$-linkage $\{\Vonewq,\Vtwodwdq\}$ for this pair.

The space $\Vonewq$ can be the space of all vectors  $ (v_{C1} , v_{C2} , v_{C3}, v_{C} ) $
spanned by the rows of
$$ \begin{bmatrix}
2& -1& -1 &  2\\
0& 1 & -1 & 0 
\end{bmatrix}
.$$

The space $\Vtwodwdq $ can be the space of all vectors
$( {\mydot{v}}_{C1}, {\mydot{v}}_{C2}, {\mydot{v}}_{C3}, {\mydot{v}}_{C})$
spanned by the row
$$\begin{bmatrix}
2& -1& -1  & 2
\end{bmatrix}
.$$
One can verify just by inspection that
\begin{equation}
\label{eqn:linkagefromemulators}
\Vwdwm\lrar (\Vonewq \bigoplus \Vtwodwdq)=\Vqdqm ,
\end{equation}
and also that
all the following `dot-cross' conditions are satisfied (with equality):
$$\Vonewq\circ W= \Vwdwm\circ W,\ \ 
\Vonewq\times W= \Vwdwm\times W,\ \ 
\Vtwodwdq\circ \dw= \Vwdwm\circ \dw,\ \ 
\Vtwodwdq\times \dw= \Vwdwm\times \dw,$$
$$\Vonewq\circ Q= \Vqdqm\circ Q,\ \ 
\Vonewq\times Q= \Vqdqm\times Q,\ \ 
\Vtwodwdq\circ \dQ=\Vqdqm\circ \dQ,\ \ 
\Vtwodwdq\times \dQ= \Vqdqm\times \dQ.$$

This is equivalent to working with the reduced set of variables
$ v_{C} ,  {\mydot{v}}_{C}$
when we move from $W$ to $Q$ and $\dw$ to $\dQ$ variables.

\end{example}
\begin{example}
\label{eg:strongemulator}
In this example we outline the topological procedure of multiport decomposition
for building an emulator.

\textup{Consider the GDSA represented by the circuit $\mathcal{N}$ in Figure \ref{fig:circuit2}(a).}
\begin{figure}
%
%
%
%
%
%
%
%
%
\begin{center}
 \includegraphics[width=7in]{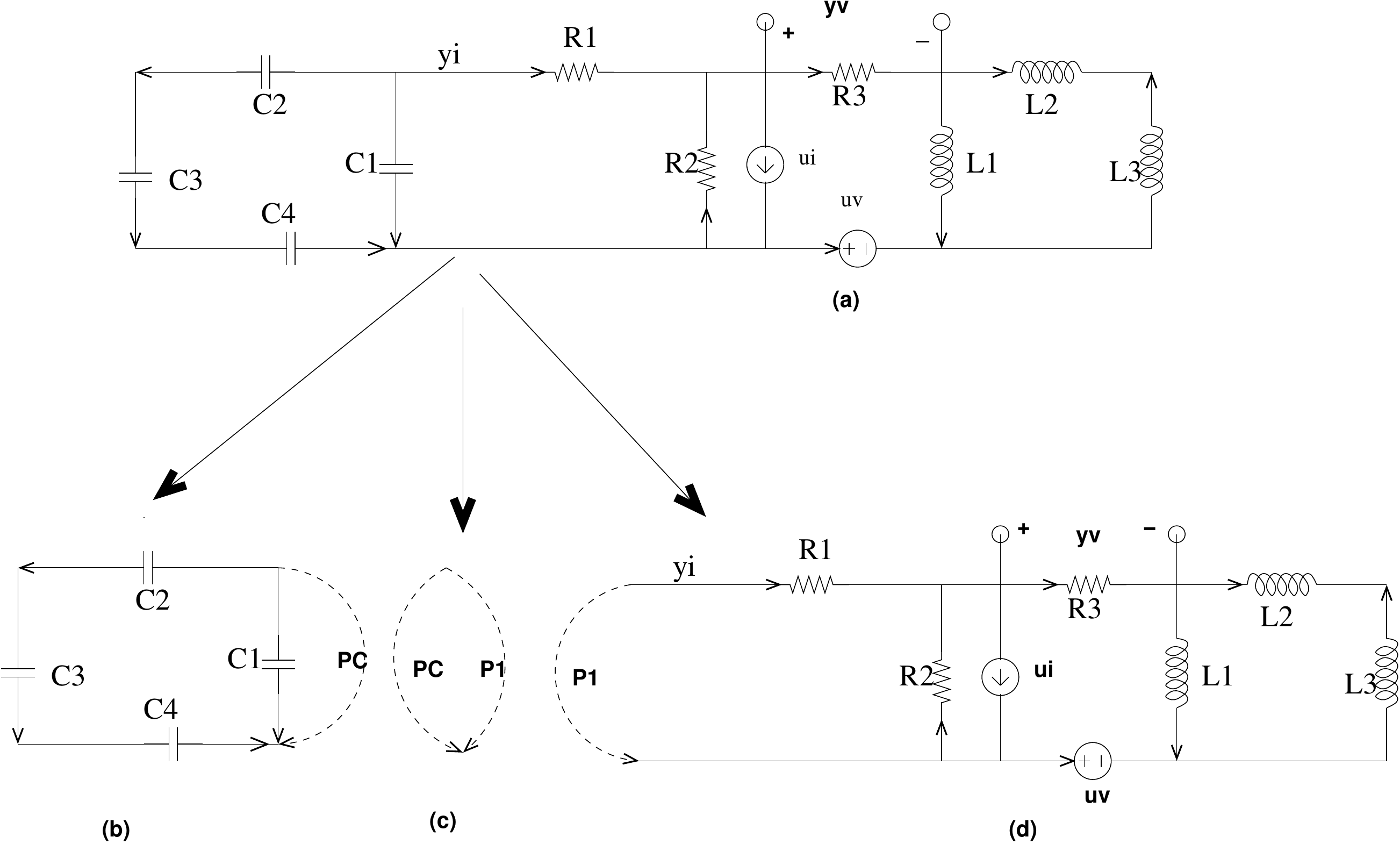}
 \caption{An $RLCEJ$ Network and its multiport decomposition
into capacitive multiport and $(static + \ L)$ multiport}
 \label{fig:circuit2}
\end{center}
\end{figure}
\begin{figure}
\begin{center}
 \includegraphics[width=7in]{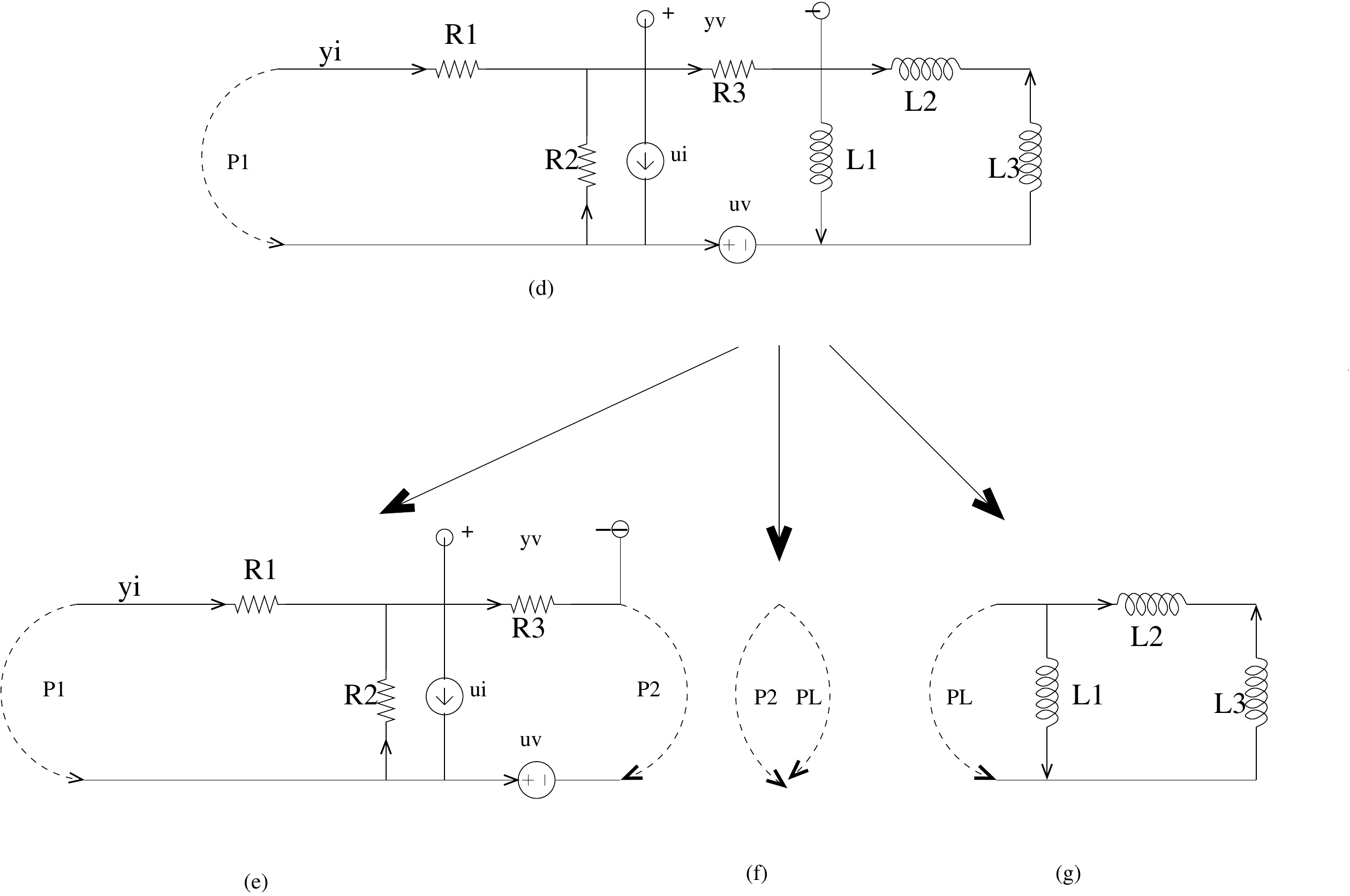}
\caption{Multiport decomposition of $(static +\ L)$ multiport
into static multiport and inductive multiport}
 \label{fig:circuit3}
\end{center}
\end{figure}
%
%

\textup{In order to efficiently build an emulator for the system we follow the procedure 
given below.
}

\textup{{\bf Topological procedure for emulator construction}\\
{\bf Step 1.}
Decompose the network $\mathcal{N}$ into  two multiports 
$\mathcal{N}_{CP_C}, \mathcal{N}_{2P_1}$ and the port connection diagram 
$\mathcal{G}_{P_CP_1}$ with minimum number of ports as in Fig \ref{fig:circuit2}.}
\textup{Here $\mathcal{N}_{CP_C}$ contains all the capacitors and $\mathcal{N}_{2P_1},$ all the other devices.
}

\textup{It can be shown that in any multiport 
decomposition with minimum number of ports, the ports will not contain cutsets or circuits of the 
graphs which figure in the decomposition.}

\textup{[A linear time algorithm for doing this, based only on the graph of the network, is given in Subsection \ref{subsec:MultiportDecomposition}. This is justified in
 \cite{HNarayanan1986a}
and \cite{HNarayanan1997}. For general vector space based systems slower decomposition algorithms are available in the same references.]
}

\textup{Figure \ref{fig:circuit2} shows (a) an $RLCEJ$ Network $\mathcal{N},$
its multiport decomposition into (b) the capacitive multiport $\mathcal{N}_{CP_C},$
(c) the port connection diagram $\mathcal{G}_{P_CP_1},$
and (d) the $(static+L)$ multiport $\mathcal{N}_{2P_1}.$
}

\vspace{0.2cm}
\noindent \textup{{\bf Step 2.}  Decompose the network $\mathcal{N}_{2P_1}$ into  two multiports
$\mathcal{N}_{P_1SP_2}, \mathcal{N}_{LP_L}$ and the port connection diagram 
$\mathcal{G}_{P_2P_L}$ with minimum number of ports as in Figure \ref{fig:circuit3}
.
Here $\mathcal{N}_{LP_L}$ contains all the inductors and $\mathcal{N}_{P_1SP_2},$
all the static devices  such as resistors, linear controlled sources, input and output ports of $ \mathcal{N},$ 
the `old' ports $P_1$ (which were present in the multiport $\mathcal{N}_{2P_1}$
during  the decomposition of $\mathcal{N}$ into $\mathcal{N}_{CP_C},$
$\mathcal{N}_{2P_1}$ and the port connection diagram $\mathcal{G}_{P_CP_1}$),
and the new ports $P_2$ that arose when $\mathcal{N}_{2P_1}$ was decomposed
into two multiports and a port connection diagram.
}

\textup{Figure \ref{fig:circuit3} shows  the $(static + \ L) $ multiport (d), its 
multiport decomposition into (e)  the static multiport $ \mathcal{N}_{P_1SP_2},$  (f) the  port connection diagram $\mathcal{G}_{P_2P_L},$
and (g) the inductive multiport $\mathcal{N}_{LP_L}.$
}

\textup{Let us denote the original GDS $\Vwdwm$
corresponding to the network $\mathcal{N}$ 
by $\V_{v_C,i_L,{\mydot{v}}_C,{\mydot{i}}_L,M}$ since the latter notation
is natural for $RLC$ networks. The set $M$ contains all the input and output variables.
In this case they are $ui,uv,yi,yv.$
}

\textup{The emulator for this GDS is $\V_{v_{PC},i_{PL},{\mydot{v}}_{PC},{\mydot{i}}_{PL},M}$
which is represented implicitly by the three multiports 
 $\mathcal{N}_{CP_C},\mathcal{N}_{P_1SP_2}, \mathcal{N}_{LP_L},$
and the port connection diagrams $\mathcal{G}_{P_CP_1},$ and
$\mathcal{G}_{P_2P_L}.$
}

\textup{The original dynamic variables are $w=(v_C,i_L)\equiv (v_{C1},v_{C2},v_{C3},v_{C4},i_{L1},i_{L2},i_{L3}).$
These variables are not state variables because of the presence of 
capacitor loops and inductor cutsets which cause the voltages to form a dependent set and similarly the currents.
A proper subset
say $v_{C1},v_{C2},v_{C3},i_{L1},i_{L2}$ can be taken as the state variables.
However, the network has capacitor cutsets and inductor loops 
which result in zero eigen values. For computational purposes,
it is more convenient to work with the dynamic variables  $(v_{P_C},i_{P_L})$
which form a smaller subset than the set of state variables  but from which the dynamics of the system can be inferred very easily.\\
 An example, in which we spell out all the spaces which implicitly go to make up the emulator,
 is given in \ref{sec:implicitemulatormultiport}.
}

\textup{One of the uses of implicit linear algebra is to make dynamical system 
computations using implicit descriptions. However, in the present case, for greater  clarity, we describe below how to obtain an explicit description of 
the dynamical system $\V_{v_{PC},i_{PL},{\mydot{v}}_{PC},{\mydot{i}}_{PL},M}$
and show how computations made on it can be easily translated to computations 
on $\Vwdwm.$
}

\vspace{0.2cm}

\textup{{\bf An explicit description for $\V_{v_{PC},i_{PL},{\mydot{v}}_{PC},{\mydot{i}}_{PL},M}.$}\\
}

\textup{{\bf Step 1.} Using the constraints of $\mathcal{N}_{CP_C},$
obtain the relationship $i_{PC}= C_{PC}{\mydot{v}}_{PC}.$\\
The constraints are} \\

\textup{(a). $({\mydot{v}}_C,{\mydot{v}}_{PC})\in \V^v(\mathcal{G}_{CP_C}),$
where $\mathcal{G}_{CP_C}$ is the graph of $\mathcal{N}_{CP_C}.$
This is the same as satisfying the Kirchhoff voltage equations of $\mathcal{G}_{CP_C}.$}
\\

\textup{(b).  $(i_C,i_{P_C})\in \V^i(\mathcal{G}_{CP_C}).$
This is the same as satisfying Kirchhoff current equations of $\mathcal{G}_{CP_C}.$}
\\

\textup{(c). $i_{C}= C{\mydot{v}}_{C}.$  Here $C$ is symmetric positive definite.
}\\

\textup{{\bf Step 2.} Using the constraints of $\mathcal{N}_{LP_L},$
obtain the relationship $i_{PL}= L_{PL}{\mydot{i}}_{PL}.$\\
The constraints are} \\

\textup{(a). $({\mydot{i}}_L,{\mydot{i}}_{PL})\in \V^i(\mathcal{G}_{LP_L}),$
where $\mathcal{G}_{LP_L}$ is the graph of $\mathcal{N}_{LP_L}.$
This is the same as satisfying the Kirchhoff current equations of $\mathcal{G}_{LP_L}.$}
\\

\textup{(b).  $(v_L,v_{P_L})\in \V^v(\mathcal{G}_{LP_L}).$
This is the same as satisfying Kirchhoff voltage equations of $\mathcal{G}_{LP_L}.$}
\\

\textup{(c). $v_{L}= L{\mydot{i}}_{L}.
$
Here $L$ is symmetric positive definite.
}\\

\textup{{\bf Step 3.} Similarly using the constraints of $\mathcal{N}_{P_1SP_2},$
we obtain
\begin{subequations}
\label{eqn:static}
\begin{align}
 \begin{pmatrix}
  i_{P_1} \\ v_{P_2}
 \end{pmatrix}
 &=
 \begin{bmatrix}
   H_{11} & H_{12} \\
  H_{21} & H_{22}
 \end{bmatrix}
 \begin{pmatrix}
   v_{P_1} \\ i_{P_2}
 \end{pmatrix}
+
 \begin{bmatrix}
   H_{13} \\ H_{23}
 \end{bmatrix}
 u \\
y & = [ H_{31}\ H_{32}  ]
 \begin{pmatrix}
  v_{P_1} \\ i_{P_2}
 \end{pmatrix}
+
[H_{33}] u.
\end{align}
\end{subequations}
}

\textup{{\bf Step 4.} Use the topological constraints of $\mathcal{G}_{P_CP_1}$
to transform $i_{P_1}$ to $i_{P_C}$ and $v_{P_1}$ to $v_{P_C}.$ \\
(This is always possible because when the port decomposition has minimum 
number of ports, the ports cannot contain circuits or cutsets of 
the port connection diagram (Theorem \ref{thm:minimalmultiport}).
This means that the graph $\mathcal{G}_{P_CP_1}$
has $P_C,P_1$ as cospanning forests.
From this it follows that voltage (current) of $P_C$ ($P_1$)  can be written
in terms  of voltage (current) of $P_1$ ($P_C$).)
}\\

\textup{{\bf Step 5.} Similarly use the topological constraints of $\mathcal{G}_{P_LP_2}$
to transform $i_{P_2}$ to $i_{P_L}$ and $v_{P_2}$ to $v_{P_L}.$ \\
(The graph $\mathcal{G}_{P_2P_L}$
has $P_2,P_L$ as cospanning forests.)
}\\ 

\textup{{\bf Step 6.} The transformed version of Equation \ref{eqn:static}
now involves $i_{P_C},v_{P_C}$ in place of $i_{P_1},v_{P_1}$
and $v_{P_L},i_{P_L}$ in place of $v_{P_2},i_{P_2}.$
}\\

\textup{{\bf Step 7.}  Now substitute 
$v_{PC}= C_{PC}{\mydot{v}}_{PC}, v_{PC}= L_{PL}{\mydot{i}}_{PL}$
in the above transformed equation to get rid of the variables 
$i_{P_C},v_{P_L}$ and obtain
\begin{subequations}
\label{eqn:strongstate}
\begin{align}
 \begin{pmatrix}
  {\mydot{v}}_{P_C} \\ {\mydot{i}}_{P_L}
 \end{pmatrix}
 &=
 \begin{bmatrix}
   {\hat{H}}_{11} & {\hat{H}}_{12} \\
  {\hat{H}}_{21} & {\hat{H}}_{22}
 \end{bmatrix}
 \begin{pmatrix}
   v_{P_C} \\ i_{P_L}
 \end{pmatrix}
+
 \begin{bmatrix}
   {\hat{H}}_{13} \\ {\hat{H}}_{23}
 \end{bmatrix}
 u \\
y & = [ {\hat{H}}_{31}\ {\hat{H}}_{32}  ]
 \begin{pmatrix}
  v_{P_C} \\ i_{P_L}
 \end{pmatrix}
+
[{\hat{H}}_{33}] u.
\end{align}
\end{subequations}
}\\

\textup{This final equation only involves $v_{P_C},{\mydot{v}}_{P_C},i_{P_L},{\mydot{i}}_{P_L}$ 
and $u,y$ variables as required.
The solution space of this equation is the dynamical system 
$\V_{v_{PC},i_{PL},{\mydot{v}}_{PC},{\mydot{i}}_{PL},M}.$
}\\

\textup{We now show that $\V_{v_{PC},i_{PL},{\mydot{v}}_{PC},{\mydot{i}}_{PL},M},$
captures the essential dynamics of the network and therefore of the dynamical system $\Vwdwm,$ which it represents.
}\\

\textup{Using the constraints of the capacitive multiport $\mathcal{N}_{CP_C},$
we can write 
${\mydot{v}}_C= T_C {\mydot{v}}_{PC}.$
(This is a network solution problem and here we are treating $({\mydot{v}}_C,{\mydot{v}}_{PC})$ as the voltage variables and $({{i}}_C,{{i}}_{PC})$
as the current variables of $\mathcal{N}_{CP_C}.$
Now in $\mathcal{G}_{CP_C},$ the set of edges $P_C$ does not  contain 
loops or cutsets (since the port decomposition is minimal).
If the matrix $C$ is symmetric positive definite such a network will have a unique solution
for any given $i_{PC}$ or ${\mydot{v}}_{PC}.$ )
%
\begin{align*}
v_C(t) &= v_C(0) + T_C \int_0^t \mydot{v}_{P_C} dt \\
   &= v_C(0) + T_C(v_{P_C}(t)-v_{P_C}(0) ).
\end{align*}}

\textup{Using the constraints of the inductive multiport $\mathcal{N}_{LP_L},$
we can write
${\mydot{i}}_L= T_L {\mydot{i}}_{PL}.$
(This is a network solution problem and here we are treating $({\mydot{i}}_L,{\mydot{i}}_{PL})$ as the current variables and $({{v}}_L,{{v}}_{PL})$
as the voltage variables of $\mathcal{N}_{LP_L}.$
Now in $\mathcal{G}_{LP_L},$ the set of edges $P_L$ does not  contain
loops or cutsets (since the port decomposition is minimal).
If the matrix $L$ is symmetric positive definite such a network will have a unique solution
for any given $v_{PL}$ or ${\mydot{i}}_{PL}.$ )
%
\begin{align*}
i_L(t) &= i_L(0) + T_L \int_0^t \mydot{i}_{P_L} dt \\
   &= i_L(0) + T_L(i_{P_L}(t)-i_{P_L}(0) ).
\end{align*}}

\textup{We thus see that, even though $v_{PC},i_{PL}$ are not state variables,
their evolution in time along with $v_C(0),i_L(0)$ fully determines 
$v_C(t),i_L(t).$ 
The notion of emulator is meant to capture this property of
$\V_{v_{PC},i_{PL},{\mydot{v}}_{PC},{\mydot{i}}_{PL},M},$
with respect to $\Vwdwm \equiv \V_{v_C,i_L,{\mydot{v}}_C,{\mydot{i}}_L,M} .$
Usually in the process of moving from the original system to the 
emulator we could lose invariant space information about
zero eigenvalues of the system (see for instance Theorem \ref{thm:emulatorpoly}).
In the present case , working with the emulator $\V_{v_{PC},i_{PL},{\mydot{v}}_{PC},{\mydot{i}}_{PL},M}$
does not give us information about the zero eigenvectors of the original
system which has to be independently provided. But the original system  $\V_{v_C,i_L,{\mydot{v}}_C,{\mydot{i}}_L,M}$ provides complete information
about the emulator,
namely, knowing $v_C,i_L,$ we can 
uniquely obtain $v_{PC},i_{PL}.$ 
}

\end{example}

\subsection{Transitivity and duality for emulators}
In order for the notion of emulator to be useful, we need
it to have two essential properties 
\begin{itemize}
\item emulators of emulators should continue to be 
emulators of the original dynamical system;
\item when we dualize  { $\E-$ linked} dynamical 
systems, this should result in { $\E-$ linked} dynamical
systems for which the  linkages are dual to the original.
\end{itemize}
Theorem \ref{thm:transitivity_emu} of Subsection \ref{sec:transitivity} shows that the first of these properties  does indeed hold.
The second property is handled by Theorem \ref{thm:emuadjoint} in the next subsection.
\subsubsection{Transitivity}
\label{sec:transitivity}
\begin{theorem}
\label{thm:transitivity_emu}
Let $\{\Vwdw,\Vpdp\}$
 be { $\E-$ linked} through $\Vonewp,\Vtwodwdp $
and let
$\{\Vwdw,\Vqdq\}$
 be { $\E-$ linked} through $\Vonewq,\Vtwodwdq .$
Then  $\{\Vpdp,\Vqdq\}$
is { $\E-$ linked} through $\Vonepq,\Vtwodpdq ,$\\
where $\Vonepq\equiv (\Vonewp\lrar \Vonewq), \ \Vtwodpdq\equiv (\Vtwodwdp\lrar \Vtwodwdq).$
\end{theorem}
\begin{proof}
By the definition of a linkage we have that 
$$\Vqdq= (\Vonewq\oplus \Vtwodwdq)\lrar \Vwdw$$
$$= (\Vonewq\oplus \Vtwodwdq)\lrar ((\Vonewp\oplus \Vtwodwdp)\lrar\Vpdp).$$
By Theorem \ref{thm:notmorethantwice} we then have
$$\Vqdq=((\Vonewq\lrar\Vonewp)\oplus ( \Vtwodwdq\lrar \Vtwodwdp))\lrar \Vpdp.$$
By the definition of a linkage we have that 
$$\Vpdp=(\Vonewp\oplus \Vtwodwdp)\lrar \Vwdw, \ \ \ \Vwdw=(\Vonewq\oplus \Vtwodwdq)\lrar\Vqdq.$$
Therefore, in the above sequence of statements, we could have interchanged the roles of $\Vpdp, \Vqdq$ so that we get
$$\Vpdp=((\Vonewq\lrar\Vonewp)\oplus ( \Vtwodwdq\lrar \Vtwodwdp))\lrar \Vqdq     .$$
\end{proof}

\subsubsection{Adjoints of emulators}
\label{sec:adjoint_emu}
We show in this subsection that adjoints of emulators can be obtained through
construction of emulators of adjoints.
We introduce the following definition primarily as a typographical simplification.
\begin{definition}
\label{def:adjoint_emu}
Let $\{\Vwdwmumy,\Vpdpmumy\}$ be $\mathcal{E}-$ linked through $\Vonewp,\Vtwodwdp$ \\
and 
let $\{\Vwdw,\Vpdp\}$ also be $\mathcal{E}-$ linked through $\Vonewp,\Vtwodwdp.$

Let $$\Vtonewp \equiv (\Vtwodwdp)^a\equiv (\Vtwodwdp)^{\perp}_{-W'P'}=(\Vtwodwdp)^{\perp}_{W'-P'},$$
$$ \Vttwodwdp \equiv (\Vonewp)^a\equiv (\Vonewp)^{\perp}_{-\dwd\dPd}=(\Vonewp)^{\perp}_{\dwd -\dPd},$$
$$\Vtwdwm\equiv (\Vwdwmumy)^a\equiv  (\Vwdwmumy)^{\perp}_{(-\dwd)W'-M_y'M_u'},$$
$$\Vtpdpm\equiv (\Vpdpmumy)^a\equiv  (\Vpdpmumy)^{\perp}_{(-\dPd)P'-M_y'M_u'},$$
$$\Vtwdw\equiv (\Vwdw)^a\equiv (\Vwdw)^{\perp}_{(-\dwd)W'},\ \ \ 
\Vtpdp\equiv (\Vpdp)^a\equiv  (\Vpdp)^{\perp}_{(-\dPd)P'}.$$
\end{definition}
\begin{theorem}
\label{thm:emuadjoint}
\begin{enumerate}
\item Let $\Vpdpmumy=(\Vonewp\oplus\Vtwodwdp)\lrar \Vwdwmumy.$
Then $$\Vtpdpm=(\Vtonewp\oplus\Vttwodwdp)\lrar \Vtwdwm.$$ 
\item Let $\Vpdp=(\Vonewp\oplus\Vtwodwdp)\lrar \Vwdw.$
Then $\Vtpdp=(\Vtonewp\oplus\Vttwodwdp)\lrar \Vtwdw.$ 
\end{enumerate}
\end{theorem}
\begin{proof}
\begin{enumerate}
\item
We have  $$\Vtpdpm\equiv (\Vpdpmumy)^a= (\Vpdpmumy)^{\perp}_{(-\dPd)P'-M_y'M_u'},$$
$$=(((\Vonewp\oplus\Vtwodwdp)\lrar \Vwdwmumy)^{\perp})_{(-\dPd)P'-M_y'M_u'},$$
$$=((\Vonewp\oplus\Vtwodwdp)^{\perp}\rightleftharpoons (\Vwdwmumy)^{\perp})_{(-\dPd)P'-M_y'M_u'},$$
$$=((\Vonewp)^{\perp}\oplus(\Vtwodwdp)^{\perp})\rightleftharpoons (\Vwdwmumy)^{\perp})_{(-\dPd)P'-M_y'M_u'},$$
$$=((\Vonewp)^{\perp}_{\dwd\dPd}\oplus(\Vtwodwdp)^{\perp}_{W'P'}\rightleftharpoons(\Vwdwmumy)^{\perp}_{\dwd W'M_y'M_u'})_{(-\dPd)P'-M_y'M_u'},$$
$$=(\Vonewp)^{\perp}_{\dwd -\dPd}\oplus(\Vtwodwdp)^{\perp}_{W'P'}\rightleftharpoons(\Vwdwmumy)^{\perp}_{\dwd W'-M_y'M_u'},$$
(now on left side of `$\rightleftharpoons$' we change sign of $W'$ and on right side, sign of $\dwd)$
$$=(\Vonewp)^{\perp}_{\dwd -\dPd}\oplus(\Vtwodwdp)^{\perp}_{-W'P'}\lrar(\Vwdwmumy)^{\perp}_{-\dwd W'-M_y'M_u'},$$
i.e.,
$$\Vtpdpm=(\Vttwodwdp\oplus\Vtonewp)\lrar \Vtwdwm=(\Vtonewp\oplus\Vttwodwdp)\lrar \Vtwdwm.$$ 
\item The proof is identical to the previous part, but omitting the $M$ variables.
\end{enumerate}
\end{proof}

An example illustrating the construction of adjoint for an $RLC$ network 
is given in \ref{sec:exadjoint}.

\subsection{Polynomials of $\E-$ linked spaces}
We will now describe, through a series of results, how operations
on a GDS $\Vwdwm$ transfer across an $\E-$ linkage to another GDS
$\Vpdpm$ and, similarly, in the case of  generalized autonomous systems
$\Vwdw,\Vpdp .$ In this subsection our primary concern is with polynomials
of linked genauts.
Our final result of this subsection, Theorem \ref{thm:emulatorpoly}, states that polynomials,
without constant terms, of linked spaces continue to be linked.
Under extra conditions this is true for even the ones with constant terms.
This implies that annihilation of a genaut by a polynomial, holds true across linkages.
So eigenvalue, eigenvector computations can be carried out through the most convenient
emulators that we decide to choose. The qualification, about not having constant terms, is not a serious impediment.
It only means that zero eigenvector computations, which any way are usually topological and therefore very inexpensive,
may have to be done directly on the original dynamical system.

The following result describes  a form of `cancellation' which makes { $\E-$ linkages} 
behave like similarity transformations. This enables the continued linkage of
polynomials of linked spaces.

\begin{lemma} 
\label{lem:cancel}
Let $\Vab,\Vbc$ be linkages and let linkages $\V_{BQ}^1, \V_{BQ}^2$
satisfy $$\V_{BQ}^1\supseteq  \V_{BQ}^2,\ \ \  
\V_{BQ}^2\circ B\supseteq \Vab\circ B,\ \ \  \V_{BQ}^1\times B\subseteq \Vbc\times B.$$
Then $$\Vab\lrar \Vbc = \Vab\lrar \V_{BQ}^2\lrar (\V_{BQ}^1)_{B_1Q}\lrar (\Vbc)_{B_1C}.$$
\end{lemma} 
\begin{proof}
Let $(f_A, h_C)\in \Vab\lrar \Vbc.$ 
Then there exists $g_B$ such that
$$(f_A,g_B)\in \Vab,\ \  (g_B,h_C)\in \Vbc .$$ Since $\V_{BQ}^2\circ B\supseteq \Vab\circ B,$
there exists $k_Q$ such that $(g_B,k_Q)\in  \V_{BQ}^2$ and since $\V_{BQ}^1\supseteq  \V_{BQ}^2,$
it follows that $(g_B,k_Q)\in  \V_{BQ}^1.$ \\
Noting that $B_1$ is a copy of $B.$ we must have
$((g_{B_1},k_Q)\in  (\V_{BQ}^1)_{B_1Q}.$ 
Thus we have $$(f_A,g_B)\in \Vab,\ (g_B,k_Q)\in  \V_{BQ}^2,\  (g_{B_1},k_Q)\in  (\V_{BQ}^1)_{B_1Q},\ 
 (g_{B_1},h_C)\in (\Vbc)_{B_1C},$$
so that $$(f_A, h_C)\in\Vab\lrar \V_{BQ}^2\lrar (\V_{BQ}^1)_{B_1Q}\lrar (\Vbc)_{B_1C}.$$
Thus $LHS\subseteq RHS.$

Next let $(f_A, h_C)\in  \Vab\lrar \V_{BQ}^2\lrar (\V_{BQ}^1)_{B_1Q}\lrar (\Vbc)_{B_1C}.$
Then there exist $$(f_A,g_B)\in \Vab,\ (g_B,k_Q)\in  \V_{BQ}^2,\  (g_{B_1}',k_Q)\in  (\V_{BQ}^1)_{B_1Q},
\  (g_{B_1}',h_C)\in (\Vbc)_{B_1C}.$$
But this means $$(g_{B_1},k_Q),(g_{B_1}',k_Q)\in  (\V_{BQ}^1)_{B_1Q}\ \ 
(\textup{since}\  \V_{BQ}^2\subseteq \V_{BQ}^2).$$ Therefore, $$(g_{B_1}-g_{B_1}')\in (\V_{BQ}^1)_{B_1Q}\times B_1\subseteq (\Vbc)_{B_1C}\times B_1.$$ We conclude that $ (g_{B_1},h_C)\in (\Vbc)_{B_1C}$ and 
equivalently $ (g_{B},h_C)\in \Vbc.$
Thus $$(f_A,g_B)\in \Vab,\ (g_{B},h_C)\in (\Vbc).
\ \textup{So}\  (f_A, h_C)\in  \Vab\lrar\Vbc.$$ Thus $RHS\subseteq LHS.$

\end{proof}
The next lemma sets up the machinery to deal with polynomials across
linkages by dealing with the building blocks for constructing polynomials.
\begin{lemma}
\label{lem:elinking2}
 Let $\{\Vwdw,\Vpdp\}$
be { $\E-$ linked} through $\Vonewp,\Vtwodwdp .$
Then
\begin{enumerate}
\item 
$\{\ldws(\Vwdw),\ldp(\Vpdp)\}$ is  { $\E-$ linked} through $\Vonewp,\Vtwodwdp .$
\item
$\{\Vwdwk,\Vpdpk\}, k\ne 0$ is { $\E-$ linked} through $\Vonewp,\Vtwodwdp .$

\item 
If further, 
$$\Vwdw \circ W = (\Vwdw \circ \dw)_{W},\ \Vwdw \times W= (\Vwdw \times \dw)_W,\ \ \ 
\Vpdp \circ P = (\Vpdp \circ \dP)_{P},\ \Vpdp \times P=(\Vpdp \times \dP)_P,$$
then we have
$\{\Vwdw^{(0)},\Vpdp^{(0)}\}$ { $\E-$ linked} through $\Vonewp,\Vtwodwdp .$
\end{enumerate}
\end{lemma}
\begin{proof}
\begin{enumerate}
\item When $\lambda =0,\ \ \ \ldws(\Vwdw)=\Vwdw \circ W\oplus \Vwdw \times \dw.$
From Lemma \ref{lem:elinking}, it follows that 
$$\ldws(\Vwdw)\lrar(\Vonewp \oplus \Vtwodwdp)= (\Vwdw \circ W\oplus \Vwdw \times \dw)\lrar (\Vonewp \oplus \Vtwodwdp)=\Vpdp \circ P\oplus \Vpdp \times \dP=\ldp(\Vpdp).$$

Next let $\lambda \ne 0$ and let $(f_P, \lambda g_{\dP}) \in \ldws(\Vwdw)\lrar(\Vonewp \oplus \Vtwodwdp).$ 

Then there exist $$(h_W, \lambda k_{\dw})\in \ldws(\Vwdw),\ \ (h_W, f_P) \in \Vonewp, \ \ (\lambda k_{\dw},\ \  \lambda g_{\dP})\in \Vtwodwdp .$$
Therefore,
 $$(h_W, k_{\dw})\in \Vwdw,\ \ (h_W, f_P) \in \Vonewp, \ \ ( k_{\dw},  g_{\dP})\in \Vtwodwdp,$$
so that 
$(f_P,g_{\dP})\in\Vpdp.$
It follows that $$(f_P, \lambda g_{\dP}) \in \ldp(\Vpdp), \ \textup{i.e.,}\ \   \ldws(\Vwdw)\lrar(\Vonewp \oplus \Vtwodwdp)\subseteq \ldp(\Vpdp).$$ 
The same argument, using $\Vpdp\lrar(\Vonewp \oplus \Vtwodwdp)= \Vwdw ,$ proves the 
reverse containment.
\item
We are given that $\{\Vwdw,\Vpdp\},$ is { $\E-$ linked} through $\Vonewp,\Vtwodwdp .$

Suppose $\{\Vwdwk,\Vpdpk\}, 0<k< r,$ is { $\E-$ linked} through $\Vonewp,\Vtwodwdp .$

We will show that the statement is true also for $k=r.$

\vspace{0.2cm}

By Lemma \ref{lem:genoppoly1}, we have that for any genaut $\Vwdw ,$ 
$$\Vwdwk\circ \dw \subseteq \Vwdw \circ \dw ,\ \ \  \Vwdwk\times W\supseteq \Vwdw \times W .
$$
 
We therefore have $$\Vonewp \times W\subseteq \Vwdw\times W\subseteq \Vwdwk\times W\ \textup{and}\  
\Vtwodwdp \circ \dw\supseteq \Vwdw\circ \dw\supseteq \Vwdwk\circ \dw.$$ 
We have,
$$ \Vonewp\lrar \Vwdw^{(r)}\lrar \Vtwodwdp= \Vonewp\lrar (\Vwdw^{(r-1)}*\Vwdw) \lrar \Vtwodwdp= \Vonewp\lrar (\Vwdw^{(r-1)})_{WW_1}\lrar (\Vwdw)_{W_1\dw}\lrar \Vtwodwdp .$$
Now we expand $ (\Vwdw^{(r-1)})_{WW_1}\lrar (\Vwdw)_{W_1\dw},$ using Lemma \ref{lem:cancel} as 
$$ (\Vwdw^{(r-1)})_{WW_1}\lrar (\Vtwodwdp)_{W_1P_1}\lrar (\Vonewp)_{W_1'P_1}\lrar (\Vwdw)_{W_1'\dw}.$$ 
So $\Vonewp\lrar (\Vwdw^{(r-1)}*\Vwdw) \lrar \Vtwodwdp$  can be rewritten as 

$$ (\Vonewp\lrar (\Vwdw^{(r-1)})_{WW_1}\lrar (\Vtwodwdp)_{W_1P_1})\lrar ((\Vonewp)_{W_1'P_1} \lrar (\Vwdw)_{W_1'\dw}\lrar \Vtwodwdp ),$$
i.e., as 
$$((\Vpdp^{(r-1)})_{PP_1}\lrar (\Vpdp)_{P_1\dP}= \Vpdp^{(r-1)}*\Vpdp =\Vpdp^{(r)} .$$
%
\item
We need to show that
$$ \Vonewp\lrar \Vwdw^{(0)} \lrar \Vtwodwdp =  \Vpdp^{(0)}.$$
Since the conditions specified in the hypothesis are symmetric with respect to
$W,P, $ and $\dw, \dP ,$ the same arguments would be valid for showing
$$ \Vonewp\lrar \Vpdp^{(0)} \lrar \Vtwodwdp =  \Vwdw^{(0)}.$$
We will denote $ \Vonewp\lrar \Vwdw^{(0)} \lrar \Vtwodwdp ,$
by $T(\Vwdw^{(0)})$ and $ \Vonewp\lrar \Vpdp^{(0)} \lrar \Vtwodwdp$ by
$T'(\Vpdp^{(0)}).$

Let $(h_P, k_{\dP}) \in T(\Vwdw^{(0)}).$

Then there exist $(f_W, h_{P}),(f_W, g_{\dw}),(g_{\dw}, k_{\dP})$
respectively in $\Vonewp, \Vwdw^{(0)},\Vtwodwdp.$

Since $ \Vonewp\supseteq (\Vtwodwdp)_{WP}, $ we must have 
$(g_{W}, k_{P})\in \Vonewp.$

Hence, $(f_W-g_{W},h_{P}-k_{P})\in \Vonewp.$

Since $(f_W, g_{\dw})\in \Vwdw^{(0)},$ we must have $(f_{\dw}- g_{\dw})\in 
\Vwdw \times \dw = \Vwdw^{(0)}\times \dw$ (part 2 of Lemma \ref{lem:powerzeroprop}
). 

Thus $$(h_{P}-k_{P}) \in\Vonewp \lrar (\Vwdw \times \dw )_W=\Vonewp \lrar \Vwdw \times W= \Vpdp \times P= (\Vpdp \times \dP)_P= (\Vpdp^{(0)}\times \dP)_P,$$
i.e., $(h_{\dP}-k_{\dP})\in \Vpdp^{(0)}\times \dP.$

Now $h_P\in \Vonewp\lrar \Vwdw^{(0)}\circ W= \Vonewp\lrar \Vwdw\circ W= \Vpdp \circ P= \Vpdp^{(0)}\circ P.$

Therefore, $(h_{P}, h_{\dP}-(h_{\dP}-k_{\dP}))= (h_{P},k_{\dP}) \in \Vpdp^{(0)}.$

Thus $T(\Vwdw^{(0)}) \subseteq \Vpdp^{(0)}.$

Similarly $T'(\Vpdp^{(0)}) \subseteq \Vwdw^{(0)}.$

By the definition of $T(\cdot),$ it is clear that it is an 
increasing function, i.e., $T(\V)\subseteq T(\V'), \V\subseteq \V'.$
This is also true of $T'(\cdot).$

Using the definition of $\mathcal{E}-$ linkage, the hypothesis of this part of the lemma, the properties of $\Vwdw^{(0)}$ listed  in Lemma \ref{lem:powerzeroprop},
 and Theorem \ref{thm:inverse}
, it is clear that
$T'(T(\Vwdw^{(0)}))= \Vwdw^{(0)}$ and $T(T'(\Vpdp^{(0)}))= \Vpdp^{(0)}.$

Therefore $\Vpdp^{(0)}\supseteq T(\Vwdw^{(0)}) \supseteq TT'(\Vpdp^{(0)})=\Vpdp^{(0)}.$

It follows that $T(\Vwdw^{(0)}) = \Vpdp^{(0)}$
and $T'(\Vpdp^{(0)}) = \Vwdw^{(0)}.$

\end{enumerate}
\end{proof}
The next lemma continues the task begun by  the previous lemma
for dealing with polynomials by showing $\pdw$ becomes $\pdp$ across
a linkage. It also contains similar results for $+$ and $\bigcap$ 
operations. These latter are useful while dealing with invariance
across a linkage.
\begin{lemma}
\label{lem:linkagepdw}
\begin{enumerate}
\item 
Let $\{\hat{\V}_{\wdw},\hat{\V}_{P\mydot{P}}\}, \ \{\tilde{\V}_{\wdw},\tilde{\V}_{P\mydot{P}}\}$
be both { $\E-$ linked} through $\Vonewp,\Vtwodwdp .$\\ 
Then so is $\{\hat{\V}_{\wdw}\pdw\tilde{\V}_{\wdw},\hat{\V}_{P\mydot{P}}\pdp \tilde{\V}_{P\mydot{P}}\}$ 
{ $\E-$ linked} through $\Vonewp,\Vtwodwdp .$

\item Let
$\{\Vwdw,\Vpdp\},\ 
\{\Vwdw',\Vpdp'\}$
be { $\E-$ linked} through $\Vonewp,\Vtwodwdp .$\\
Then so is $\{\Vwdw+\Vwdw',\Vpdp+\Vpdp'\}$ 
{ $\E-$ linked} through $\Vonewp,\Vtwodwdp .$
\item  Let
$\{\Vwdw,\Vpdp\},\ 
\{\Vwdw',\Vpdp'\}$
be { $\E-$ linked} through $\Vonewp,\Vtwodwdp .$\\
Then so is $\{\Vwdw\bigcap\Vwdw',\Vpdp\bigcap\Vpdp'\}$
{ $\E-$ linked} through $\Vonewp,\Vtwodwdp .$
\end{enumerate}
\end{lemma}
\begin{proof}
\begin{enumerate}
\item
The result follows by using the distributivity of `$\lrar$' over the intersection-sum
operation for appropriate conditions (Theorem \ref{thm:distributivity}
).
We have 
$$\Vonewp\lrar(\hat{\V}_{\wdw}\pdw\tilde{\V}_{\wdw})\lrar \Vtwodwdp$$
$$= ((\Vonewp\lrar\hat{\V}_{\wdw})\pdw(\Vonewp\lrar\tilde{\V}_{\wdw}))
\lrar \Vtwodwdp$$
(since $\Vonewp\times W \subseteq \hat{\V}_{\wdw}\times W,\ \Vonewp\times W \subseteq \tilde{\V}_{\wdw}\times W$)
$$=(\Vonewp\lrar\hat{\V}_{\wdw}\lrar\Vtwodwdp)\pdp (\Vonewp\lrar\tilde{\V}_{\wdw}\lrar\Vtwodwdp)$$
(since $(\Vonewp\lrar\hat{\V}_{\wdw})\circ \dw\subseteq \hat{\V}_{\wdw}\circ \dw\subseteq  
\Vtwodwdp\circ \dw, \ (\Vonewp\lrar\tilde{\V}_{\wdw})\circ \dw\subseteq \tilde{\V}_{\wdw}\circ \dw\subseteq  
\Vtwodwdp\circ \dw$)
$$=\hat{\V}_{{P\mydot{P}}}+_{\dP}\tilde{\V}_{{P\mydot{P}}}.$$
The same argument, because of the properties of $\mathcal{E}-$ linkage being symmetric 
with respect to $W,P$ and $\dw,\dP,$  works for showing $$\Vonewp\lrar(\hat{\V}_{{P\mydot{P}}}+_{\dP}\tilde{\V}_{{P\mydot{P}}})\lrar \Vtwodwdp\equaln \hat{\V}_{\wdw}\pdw\tilde{\V}_{\wdw}.$$
\item 
It is easily seen that
$$\Vonewp\lrar (\Vwdw+\Vwdw')\lrar \Vtwodwdp\supseteqn 
(\Vonewp\lrar \Vwdw) + (\Vonewp\lrar \Vwdw')\supseteqn 
\Vpdp+\Vpdp'$$
and similarly 
$$\Vonewp\lrar (\Vpdp+\Vpdp')\lrar \Vtwodwdp\supseteqn 
\Vwdw+\Vwdw'.$$
Further, $$\Vonewp\circ W\supseteqn \Vwdw\circ W +\Vwdw' \circ W 
\equaln (\Vwdw+\Vwdw')\circ W,$$
$$ \Vonewp\times W\subseteqn \Vwdw\times W +\Vwdw' \times W \subseteqn (\Vwdw+\Vwdw')\times W,$$
$$ \Vtwodwdp\circ \dw\supseteqn \Vwdw\circ \dw +\Vwdw' \circ \dw \equaln (\Vwdw+\Vwdw')\circ
 \dw,$$
$$ \Vtwodwdp\times \dw\subseteq \Vwdw\times \dw +\Vwdw' \times \dw \subseteq (\Vwdw+\Vwdw')\times \dw.$$
A similar set of conditions is clear in terms of the $P,\dP$ variables.

Thus the dot-cross  $\mathcal{E}-$ linkage  conditions are satisfied for
$\{(\Vwdw+\Vwdw'),(\Vpdp+\Vpdp')\}.$

Let us now denote $\Vonewp\lrar (\Vwdw+\Vwdw')\lrar \Vtwodwdp$ by
$T(\Vwdw+\Vwdw')$ 

and $\Vonewp\lrar (\Vpdp+\Vpdp')\lrar \Vtwodwdp$ by
$T'(\Vpdp+\Vpdp').$ 

Using Theorem \ref{thm:inverse}
, we have 
$$T'T(\Vwdw+\Vwdw')\equaln \Vwdw+\Vwdw',\ \  TT'(\Vpdp+\Vpdp')\equaln \Vpdp+\Vpdp'.$$
It is clear that $T(\cdot), T'(\cdot)$ are increasing functions.
Therefore we have 
$$\Vpdp+\Vpdp'\equaln  TT'(\Vpdp+\Vpdp')\supseteqn T(\Vwdw+\Vwdw')\supseteqn \Vpdp+\Vpdp',$$
$$\Vwdw+\Vwdw'\equaln  T'T(\Vwdw+\Vwdw')\supseteqn T'(\Vpdp+\Vpdp')\supseteqn \Vwdw+\Vwdw'.$$

We conclude,
$$T(\Vwdw+\Vwdw')\equaln\Vpdp+\Vpdp',\ \  T'(\Vpdp+\Vpdp')\equaln\Vwdw+\Vwdw'.$$
Therefore, 
$\{\Vwdw+\Vwdw',\Vpdp+\Vpdp'\}$
is { $\E-$ linked} through $\Vonewp,\Vtwodwdp .$\\
\item This result is dual to the previous part of the present lemma.
\end{enumerate}

\end{proof}
\begin{remark}
In Lemma \ref{lem:elinking2}, the conditions about { $\E-$ linking}
$(\Vwdw^{(0)},\Vpdp^{(0)})$ that are stated  imply that polynomials 
of { $\E-$ linked} genops may not necessarily be also so linked
unless the constant term is zero. In particular if a polynomial with
constant term 
`annihilates' a genop, it may not necessarily annihilate another 
that is { $\E-$ linked} to it unless the genauts satisfy special properties. We will illustrate these ideas through an example later (see 
Example \ref{sec:implicitemulatormultiport} and Remark \ref{rem:implicitemulatormultiport} at the end of it).
\end{remark}
From Lemma \ref{lem:elinking}, Lemma \ref{lem:elinking2} and the first part of Lemma \ref{lem:linkagepdw}
we obtain immediately,
\begin{theorem}
\label{thm:emulatorpoly}
Let $p(s)$ be a polynomial with no constant term.
Let $\Vwdw,\Vpdp$ be genops such that $\{\Vwdw,\Vpdp\}$
is { $\E-$ linked} through $\Vonewp,\Vtwodwdp .$

Then so is $\{p(\Vwdw),p(\Vpdp)\}$ { $\E-$ linked} through $\Vonewp,\Vtwodwdp .$ If one of $p(\Vwdw),p(\Vpdp)$ is decoupled so would the other be.

If however,\\
$\Vwdw \circ W = (\Vwdw \circ \dw)_{W},\ \ \Vwdw \times W= (\Vwdw \times \dw)_W,$ \\
$\Vpdp \circ P = (\Vpdp \circ \dP)_{P},\ \ \Vpdp \times P=(\Vpdp \times \dP)_P,$\\

then since $(\Vwdw^{(0)},\Vpdp^{(0)})$ is { $\E-$ linked} through $\Vonewp,\Vtwodwdp ,$

$\{p(\Vwdw),p(\Vpdp)\}$ is { $\E-$ linked} through $\Vonewp,\Vtwodwdp ,$
even if $p(s)$ has a constant term.\\
If one of $p(\Vwdw),p(\Vpdp)$ is decoupled so would the other be.

\end{theorem}
\subsection{Invariance across { $\E-$ linking}}
In this subsection, our concern is with showing that controlled and conditioned
invariant spaces of a genaut transfer in a natural manner to the corresponding spaces 
in a linked genaut.
\begin{theorem}
\label{thm:invlinking}
Let $\{\Vwdw,\Vpdp\}$
be { $\E-$ linked} through $\Vonewp,\Vtwodwdp .$
\begin{enumerate}
\item Let $\V_W $ be controlled  invariant in $\Vwdw.$
Then $\Vonewp\lrar\V_W, (\Vtwodwdp)_{WP}\lrar \V_W $ are controlled  invariant in $\Vpdp .$
\item Let $\V_W $ be conditioned  invariant in $\Vwdw.$
Then $\Vonewp\lrar\V_W, (\Vtwodwdp)_{WP}\lrar \V_W $ are conditioned  invariant in $\Vpdp .$
\item Let $\V_W $ be invariant in $\Vwdw.$
Then $\Vonewp\lrar\V_W, (\Vtwodwdp)_{WP}\lrar \V_W $ are invariant in $\Vpdp .$
\end{enumerate}
\end{theorem}
\begin{proof}
\begin{enumerate}
\item We are given that 
$\Vwdw \lrar \Vdw \supseteq \Vw.$ We will first show that 
\begin{equation}
\label{eqn:emulatorcondinv}
((\Vonewp \bigoplus \Vtwodwdp)\lrar \Vwdw ) \lrar  (\Vonewp\lrar\V_W)_{\mydot{P}}\supseteq (\Vonewp\lrar\Vw).
\end{equation}
The LHS of equation \ref{eqn:emulatorcondinv} can be rewritten as 
$$ ((\Vonewp \lrar \Vtwodwdp)\lrar \Vwdw )\lrar (\Vonedwdp \lrar \Vdw),$$ 
where $\Vonedwdp\equiv (\Vonewp)_{\dwdp}$ and $ \Vdw\equiv (\Vw)_{\dw}$ 
and we note that $\Vonewp \bigoplus \Vtwodwdp$ can be written also as $\Vonewp \lrar \Vtwodwdp.$

We will rewrite the term $(\Vonedwdp \lrar \Vdw)$ as 
$(\Vonetildedwdp\lrar \Vtildedw),$ where $\Vonetildedwdp\equiv (\Vonedwdp)_{\tildedwdp}$
and $\Vtildedw\equiv (\Vw)_{\tildedw}.$
The LHS of equation \ref{eqn:emulatorcondinv} now appears as 
$$ ((\Vonewp \lrar \Vtwodwdp)\lrar \Vwdw )\lrar (\Vonetildedwdp\lrar \Vtildedw).$$
Observe that in this expression no index set appears more than twice.
Therefore, by Theorem \ref{thm:notmorethantwice}, the expression can have only a unique meaning even without brackets
and we can rearrange the terms according to convenience.
We now rearrange terms and write the expression as 
$$(  (\Vonetildedwdp \lrar \Vtwodwdp)\lrar \Vwdw ) \lrar \Vtildedw \lrar \Vonewp ).$$ 
We observe that $(\Vtwotildedwdp \lrar \Vtwodwdp)\lrar \Vwdw =(\Vwdw)_{\wtildedw},$
by Theorem \ref{thm:pseudoidentity} and  further that $\Vonetildedwdp\supseteq \Vtwotildedwdp.$
Hence $$(  (\Vonetildedwdp \lrar \Vtwodwdp)\lrar \Vwdw ) \lrar \Vtildedw \lrar \Vonewp )
\supseteq ((\Vwdw)_{\wtildedw} \lrar \Vtildedw )\lrar \Vonewp \supseteq \Vw \lrar \Vonewp. $$
This proves the required result.
\\
We will next show that 
\begin{equation}
\label{eqn:emulatorcondinv2}
((\Vonewp \bigoplus \Vtwodwdp)\lrar \Vwdw ) \lrar  (\Vtwowp\lrar\V_W)_{\mydot{P}}\supseteq (\Vtwowp\lrar\Vw),
\end{equation}
where $\Vtwowp\equivn (\Vtwodwdp)_{WP}.$\\
The LHS of equation \ref{eqn:emulatorcondinv2} can be rewritten as 
$$((\Vonewp \lrar \Vtwodwdp)\lrar \Vwdw ) \lrar  (\Vtwowp\lrar\V_W)_{\mydot{P}}).
$$
As before this expression can be rewritten as 
$$ ((\Vonewp \lrar \Vtwodwdp)\lrar \Vwdw )\lrar (\Vtwotildedwdp\lrar \Vtildedw).$$
Rearranging terms (Theorem \ref{thm:notmorethantwice}), this reduces to 
$$(  (\Vtwotildedwdp \lrar \Vtwodwdp)\lrar \Vwdw ) \lrar \Vtildedw \lrar \Vonewp ).$$ 
This simplifies, using Theorem \ref{thm:pseudoidentity}, to 
$$(\Vwdw)_{\wtildedw} \lrar \Vtildedw \lrar \Vonewp.$$
We note that $\Vonewp
\supseteq (\Vtwodwdp)_{WP}.$ Hence 
 $$(  (\Vonetildedwdp \lrar \Vtwodwdp)\lrar \Vwdw ) \lrar \Vtildedw \lrar \Vonewp )
\supseteq ((\Vwdw)_{\wtildedw} \lrar \Vtildedw )\lrar \Vtwowp \supseteq \Vw \lrar \Vtwowp. $$
This proves the required result.
\item This is dual to the previous part.
\item Follows from the preceding parts of the lemma.

\end{enumerate}

\end{proof}
%
%
%
%
%
The following lemma is used in the proof of Theorem \ref{thm:controlobservelink}.

\begin{lemma}
\label{lem:supsetelrar}

Let $\V^1_{WP}\supseteq \V^2_{WP}.$
\begin{enumerate}
\item Suppose $\V^1_{WP}\circ W = \V^2_{WP}\circ W, \ \Vp \supseteq \V^1_{WP}\times P.$
Then, $\V^1_{WP}\lrar \Vp =  \V^2_{WP}\lrar \Vp.$\\
Hence, $(\V^1_{WP}\lrar \Vp)_{\mydot{W}} =  \Vtwodwdp\lrar (\Vp)_{\mydot{P}}.$\\

\item Suppose 
$\V^1_{WP}\times W = \V^2_{WP}\times W, \ \Vp \subseteq \V^2_{WP}\circ P.$
Then, $\V^1_{WP}\lrar \Vp =  \V^2_{WP}\lrar \Vp.$\\
Hence, $(\V^1_{WP}\lrar \Vp)_{\mydot{W}} =  (\Vtwodwdp\lrar (\Vp)_{\mydot{P}}).$
\end{enumerate}
\end{lemma}
\begin{proof}
\begin{enumerate}
\item It is clear that  $\V^1_{WP}\lrar \Vp \supseteq   \V^2_{WP}\lrar \Vp.$
Let $g_W\in \V^1_{WP}\lrar \Vp.$ \\ 
Then there exist 
$(f_P,g_W)\in \V^1_{WP}, f_P \in \Vp.$\\ 

Since $\V^1_{WP}\circ W = \V^2_{WP}\circ W,$ there exists $(f'_P,g_W)\in \V^2_{WP}\subseteq\V^1_{WP}.$\\
But then $(f_P-f'_P)\in \V^1_{WP}\times P\subseteq \Vp,$ so that 
$f'_P \in \Vp.$ 
\\

Thus 
$(f'_P,g_W)\in \V^2_{WP}, f'_P \in \Vp,$
so that $g_W\in \V^2_{WP}\lrar \Vp.$ \\ 
\item Dual to the previous result.
\end{enumerate}
\end{proof}

In Theorem \ref{USG}, we had shown the analogue of controllable
and uncontrollable spaces in the context of USGs and in Theorem \ref{LSG},
the dual notions of the unobservable
and observable spaces in the context of LSGs.
In the former case, starting with a USG $\Vonewdw, $ a genop contained 
in it and a space $\Vw$ invariant in it, we had shown that if
$p_1(s)$ annihilates the `uncontrollable' space $\Vonewdw +(\Vw)_{\dw},$
and  $p_2(s),$ the `controllable' space $\Vwdw \bigcap \Vw,$
then $p_1p_2(s)$ annihilates $\Vwdw.$ A dual statement was proved 
for an LSG. The primal part of the following result states, essentially, that linked versions 
of the uncontrollable and controllable spaces can be computed 
in the emulator, and the result transferred to the original across
{ $\E-$ linkage}. It also contains the dual statement about 
observable and unobservable spaces.

\begin{theorem}
\label{thm:controlobservelink}
Let $\{\Vwdw,\Vpdp\}$
be { $\E-$ linked} through $\Vonewp,\Vtwodwdp. $
\begin{enumerate}
\item
Let $\Vp$ be invariant in $ \Vpdp $ with $
 (\Vp )_{\mydot{P}}\subseteq\Vpdp \circ \mydot{P}
,$
let $(\Vonewp\times W)_{\dw} = \Vtwodwdp\times \dw$\\
and let
$\Vw \equiv (\Vtwodwdp)_{WP}\lrar \Vp .$ 
\\Then 
\begin{enumerate}
\item $\Vw = \Vonewp\lrar \Vp $ and $\Vw$ is invariant in $\Vwdw$
with $\Vwdw\circ \dw\supseteq (\Vw)_{\dw} .$\\

\item 
 $\{\Vwdw+(\Vw\bigoplus (\Vw)_{\mydot{W}}),\Vpdp+(\Vp\bigoplus (\Vp)_{\mydot{P}})\}$
is also { $\E-$ linked} through $\Vonewp,\Vtwodwdp. $
\end{enumerate}
\item
Let $\Vp$ be invariant in $ \Vpdp $ with $
 \Vp \supseteq\Vpdp \times {P}
,$
and let $(\Vonewp\circ W)_{\dw} = \Vtwodwdp\circ \dw.$\\
Let
$\Vw \equiv \Vonewp\lrar \Vp .$ 
\\Then 
\begin{enumerate}
\item $(\Vw)_{\dw} = \Vtwodwdp \lrar (\Vp)_{\dP} $ and $\Vw$ is invariant in $\Vwdw$
with $\Vwdw\times W\subseteq \Vw.$\\
\item 
 $\{\Vwdw\bigcap (\Vw\bigoplus (\Vw)_{\mydot{W}}),\Vpdp\bigcap (\Vp\bigoplus (\Vp)_{\mydot{P}})\}$
is also { $\E-$ linked} through $\Vonewp,\Vtwodwdp. $
\end{enumerate}

\end{enumerate}
\end{theorem}
\begin{proof}
\begin{enumerate}
\item
\begin{enumerate}
\item We have, by the property of { $\E-$ linkages} (Theorem \ref{thm:elinkagedotcross})
and the hypothesis,
$$\Vtwodwdp\circ \mydot{P}\supseteqn\Vpdp \circ \mydot{P}\supseteqn(\Vp )_{\mydot{P}}.$$ Further, 
$(\Vonewp\times W)_{\dw} \equaln \Vtwodwdp\times \dw.$
So by part 2 of Lemma \ref{lem:supsetelrar}, $$\Vonewp\lrar \Vp = (\Vtwodwdp\lrar(\Vp)_{\dP})_W=\Vw.$$ By Theorem \ref{thm:invlinking}, it follows that $\Vw $ is invariant in $\Vwdw.$

Next, since $\{\Vonewp,\Vtwodwdp\} $
is an { $\E-$ linkage} for $\{\Vwdw,\Vpdp\},$
by Lemma \ref{lem:elinking},\\ 
$$\Vwdw \circ \dw= \Vtwodwdp\lrar(\Vpdp\circ \dP)\supseteq\Vtwodwdp\lrar(\Vp)_{\dP}=(\Vw)_{\dw}.$$\\
\item By the previous subpart of the theorem, it is clear that
$$\{\Vw\bigoplus (\Vw)_{\mydot{W}}, \Vp\bigoplus (\Vp)_{\mydot{P}}\}$$
is { $\E-$ linked} through $\Vonewp,\Vtwodwdp. $

Since $\{\Vwdw,\Vpdp\}$
is also { $\E-$ linked} through $\Vonewp,\Vtwodwdp,$
by Lemma \ref{lem:linkagepdw}, it follows that\\
 $$\{\Vwdw+(\Vw\bigoplus (\Vw)_{\mydot{W}}),\Vpdp+(\Vp\bigoplus (\Vp)_{\mydot{P}})\}$$
is { $\E-$ linked} through $\Vonewp,\Vtwodwdp. $\\
\end{enumerate}
\item This is dual to the previous result.
\end{enumerate}
\end{proof}
\begin{remark}
\label{rem:zeroeigen}
Note that, if $\Vwdw\times W= \0_W$ and $ \Vpdp\times P= \0_P,$ equivalently, if the system
and the emulator have 
no zero eigenvalue, then there is a one to one correspondence 
between invariant spaces of $\Vwdw$ and $\Vpdp.$

Even when $\Vwdw\times W\ne \0_W,$
a convenient emulator with $\Vpdp\times P= \0_P$ can usually be built.
See for instance Example \ref{sec:implicitemulatormultiport}.

Now define $$\Vwdw'\equiv ((\Vtwodwdp)_{WP}\oplus \Vtwodwdp)\lrar \Vpdp.$$
Theorem \ref{thm:controlobservelink} indicates that invariant spaces can be computed in
$\Vpdp$ and transformed through $(\Vtwodwdp)_{WP}$ to invariant subspaces
of $\Vwdw'.$ If to each of these spaces, we add invariant subspaces of $\Vwdw\times W$
we will account for all invariant subspaces of $\Vwdw.$
\end{remark}
\subsection{State feedback and output injection through { $\E-$ linkages}}
In this subsection we show that we can perform $wm_u- feedback/ m_y\dW- injection$
using $\Vwmu/\Vdwmy $
on a GDS $\Vwdwmumy $ by first moving to an emulator $\Vpdpmumy $ through the { $\E-$ linkage} $(\Vonewp \oplus \Vtwodwdp ), $ performing the operation using $(\Vwmu \lrar \Vwp)/(\Vdwmy \lrar \Vdwdp)$ on it, obtaining $\Vpdp$ and then returning to 
$\Vwdw$ through  $(\Vwp \oplus \Vdwdp ).$ 

The following theorem essentially formalizes
the above informal statement on $wm_u-$ feedback.
\begin{theorem}
\label{thm:wmfeedback}
Let $$(\Vonewp\oplus \Vtwodwdp)\lrar \Vwdwmumy = \Vpdpmumy  \ \textup{and\ let}\   \Vonewp \times W\subseteq \Vwdwmumy\circ W\dw M_u \times W.$$
Let $\Vwmu$ be a space on $  W \uplus M_u,$  \ \ $\Vpmu\equiv \Vwmu \lrar  \Vonewp.$
Let $\Vwdw \equiv (\Vwdwmumy\cap \Vwmu)\circ W\dw$
and  let $\Vpdp \equiv  (\Vpdpmumy\cap \Vpmu)\circ P\dP.$
Then
$$(\Vonewp\oplus \Vtwodwdp)\lrar \Vwdw = \Vpdp.$$
\end{theorem}
We need the following lemmas to prove this theorem.
\begin{lemma}
\label{lem:wm1}
Let $\Vwdwmu\equiv \Vwdwmumy \circ W\dw M_u.$
Let $\Vonewp \times W\subseteq \Vwdwmu\times W.$ 
Then,
$$\Vonewp \lrar (\Vwdwmu\cap\Vwmu)= (\Vonewp \lrar \Vwdwmu)\bigcap (\Vonewp \lrar \Vwmu).$$
\end{lemma}
\begin{proof}
It is easy to see that
$$\Vonewp \lrar (\Vwdwmu\cap\Vwmu)\subseteq (\Vonewp \lrar \Vwdwmu)\bigcap (\Vonewp \lrar \Vwmu).$$
Let $(f_P,g_{\dw},f_{M_u})\in $ R.H.S.

Then there exist $(f_W,f_P),(f'_W,f_P)\in \Vonewp$ such that 
$(f_W,g_{\dw},f_{M_u})\in \Vwdwmu$ and $(f'_W,f_{M_u}) \in \Vwmu.$

But this means  $$f_W-f'_W\in \Vonewp\times W\subseteq 
\Vwdwmu\times W.\ \ \textup{ It\  follows\  that}\  (f_W-f'_W,0_{\dw},\ 0_{M_u})\in \Vwdwmu$$
so that $(f'_W,g_{\dw},f_{M_u})$ belongs to $\Vwdwmu$ and therefore also belongs to 
$(\Vwdwmu\cap\Vwmu).$ 

Since $(f'_W,f_P)\in \Vonewp,$
it follows that $$(f_P,g_{\dw},f_{M_u})\in \Vonewp \lrar (\Vwdwmu\cap\Vwmu).$$

\end{proof}
The following lemma is routine and the proof is omitted.
\begin{lemma}
\label{lem:wm2}
Let $\Vpdwmu \equiv \Vonewp \lrar \Vwdwmu$ and let $\Vpmu \equiv \Vonewp \lrar \Vwmu.$ 
Then
$$\Vtwodwdp \lrar (\Vpdwmu\cap \Vpmu)=(\Vtwodwdp \lrar\Vpdwmu)\bigcap 
\Vpmu.$$
\end{lemma}

\begin{proof} (of Theorem \ref{thm:wmfeedback}).\\
From Lemmas \ref{lem:wm1}, \ref{lem:wm2}, we have that 
$$(\Vonewp\oplus \Vtwodwdp)\lrar (\Vwdwmu\cap\Vwmu)$$
$$=
\Vtwodwdp\lrar ((\Vonewp\lrar \Vwdwmu)\bigcap (\Vonewp\lrar \Vwmu))$$
$$=
((\Vonewp\oplus \Vtwodwdp)\lrar \Vwdwmu)\bigcap (\Vonewp\lrar \Vwmu).
$$
It follows that 
$$[(\Vonewp\oplus \Vtwodwdp)\lrar (\Vwdwmu\cap\Vwmu)]\lrar \F_{M_u}=
[((\Vonewp\oplus \Vtwodwdp)\lrar \Vwdwmu)\bigcap (\Vonewp\lrar \Vwmu)]\lrar \F_{M_u}.
$$
By Theorem \ref{thm:notmorethantwice}
 the L.H.S can be rewritten as 
$$(\Vonewp\oplus \Vtwodwdp)\lrar [(\Vwdwmu\cap\Vwmu)\lrar \F_{M_u}] =  (\Vonewp\oplus \Vtwodwdp)\lrar [(\Vwdwmumy\cap\Vwmu)\lrar \F_{M_uM_y}] $$
and is therefore equal to $(\Vonewp\oplus \Vtwodwdp)\lrar \Vwdw$
and the R.H.S. is 
$$(\Vpdpmu\cap \Vpmu)\circ P\dP= (\Vpdpmumy\cap \Vpmu)\circ P\dP= \Vpdp .$$
The theorem follows.
\end{proof}
The following result is a formalization of the statement about
performing $m\dW -$ injection through an emulator, at the beginning of this subsection.
It is the dual of Theorem \ref{thm:wmfeedback}.
\begin{theorem}
\label{thm:dwminjection}
Let $$(\Vonewp\oplus \Vtwodwdp)\lrar \Vwdwmumy \equaln \Vpdpmumy, \ \textup{and\ let} \ \Vtwodwdp\circ \dw \supseteqn 
\Vwdwmumy\times W\dw M_y\circ \dw .$$
 Let $\Vdwmy$ be a space on $ \dw\uplus  M_y,$  \ \ $\Vdpmy\equiv \Vdwmy \lrar  \Vtwodwdp.$
 Let $\Vwdw \equiv (\Vwdwmumy+ \Vdwmy)\times W\dw$
 and  let $\Vpdp = (\Vpdpmumy+ \Vdpmy)\times P\dP.$
 Then
 $$(\Vonewp\oplus \Vtwodwdp)\lrar \Vwdw = \Vpdp.$$
\end{theorem}
\begin{proof}
We will merely show that if $\{\Vwdwmumy, \Vpdpmumy\}$ satisfies Theorem \ref{thm:wmfeedback},\\
the adjoint pair $\{\Vtwdwm, \Vtpdpm\}$ satisfies the present theorem.
\\

We follow Definition \ref{def:adjoint_emu}. We have
$$\Vtonewp \equiv (\Vtwodwdp)^{\perp}_{-W'P'}, \ \ \Vttwodwdp \equiv (\Vonewp)^{\perp}_{\dwd -\dPd},$$
$$\Vtwdwm\equiv (\Vwdwmumy)^{\perp}_{-\dwd W'-M_y'M_u'},\ \ \Vtpdpm\equiv(\Vpdpmumy)^{\perp}_{-\dPd P'-M_y'M_u'},$$
$$\Vtwdw\equiv  (\Vwdw)^{\perp}_{-\dwd W'},\ \ \ 
\Vtpdp\equiv  (\Vpdp)^{\perp}_{-\dPd P'}.$$
Further, let $\Vtdwmy\equiv (\Vwmu)^\perp_{\dwd M_y'}.$

By Theorem \ref{thm:wmfeedback}, if $\{\Vwdwmumy, \Vpdpmumy\}$ is $\mathcal{E}-$ linked through $\Vonewp, \Vtwodwdp $
so will $\{\Vwdw, \Vpdp\}$ be, where 
$\Vwdw \equiv (\Vwdwmumy\bigcap \Vwmu)\circ W\dw$
 and   $\Vpdp \equiv  (\Vpdpmumy\bigcap \Vpmu)\circ P\dP.$

By Theorem \ref{thm:emuadjoint}, $\{\Vwdwmumy, \Vpdpmumy\}$ is $\mathcal{E}-$ linked through $\Vonewp, \Vtwodwdp $ \\iff $\{\Vtwdwm, \Vtpdpm\}$ is $\mathcal{E}-$ linked through $\Vtonewp, \Vttwodwdp $\\
and $\{\Vwdw, \Vpdp\}$ $\mathcal{E}-$ linked through $\Vonewp, \Vtwodwdp $ iff
$\{\Vtwdw, \Vtpdp\}$ is so linked through $\Vtonewp, \Vttwodwdp .$ 
\\

Now $$\Vtwdw\equiv (\Vwdw)^{\perp}_{-\dwd W'}= ((\Vwdwmumy\cap \Vwmu)\circ W\dw)^{\perp}_{-\dwd W'}=(\Vtwdwm +\Vtdwmy)\times W'\dwd,$$
$$\Vtpdp\equiv (\Vpdp)^{\perp}_{-\dPd P'}= ((\Vpdpmumy\cap \Vpmu)\circ P\dP)^{\perp}_{-\dPd P'}=(\Vtpdpm +\Vtdpmy)\times P'\dPd.$$

This proves the result.

%
%
%
%
%
\end{proof}


\section{Conclusions}
\label{sec:Conclusions}
In this paper we made a beginning towards developing an implicit form of linear algebra, basing it on linear relations  (`linkages')
rather than on linear transformations.

The basic operation was the `matched composition $\lrar$' which generalized ordinary composition of maps to that of linkages.
We defined the `intersection-sum' and `scalar multiplication' for linkages which were analogues of sum and scalar multiplication of maps. 
Using  the basic results (\cite{HNarayanan1986a}) - implicit inversion theorem and implicit duality theorem -
we proved results for linkages (eg. that intersection-sum of transposes of linkages is the transpose of intersection-sum of linkages,
that, under mild conditions, matched composition is distributive under intersection-sum) analogous to those for maps.

We generalized the notion of an operator in our frame work and developed a primitive spectral theory for it 
using intersection-sum as `sum' and matched composition as `multiplication'. 

We developed  computationally efficient counterparts (`emulators') to state variables  and showed that standard computations
such as those involving invariant subspaces, feedback etc.,  and also spectral computations, could be done with  emulators.

We tried to show that our ideas are useful from a computational point of view by first systematically developing the 
standard linear multivariable control theory, as expounded for instance in \cite{Wonham1978}, in terms of linkages
and showing that all the operations can be done implicitly without eliminating variables until the very end, when it would be required.

We illustrated our ideas with a paradigmatic instance namely a resistor, inductor, capacitor circuit.
In particular we showed for such circuits that emulators could be built in linear time.

It is hoped that the `necessity' of implicit linear algebra is evident since the illustrative examples show
that operators are usually available in practice as composition of linkages
(eg. associated with the graph of a network) and more elementary operators
(eg. the device characteristics of a network). To efficiently exploit 
topology (for instance through port decompositon) we need to work in terms of linkages rather than in terms of maps. 
\appendix

\section{Computing the basic operations}
\label{app:ComputingBasicOperations}
\subsection{Representation}
A vector space $\V_A$ may be specified in one of the two following ways:
\begin{enumerate}
 \item Through generator set for $\V_A$:
\begin{align*}
 f_A = \lambda^T \begin{bmatrix}R_A\end{bmatrix},
\end{align*}
where the rows of $R_A$ generate $\V_A$ by linear combination. When the rows of $R_A$ are linearly independent $R_A$ is called a representative matrix of $\V_A$.
\item Through constraint equations:
\begin{align*}
\begin{bmatrix} S_A \end{bmatrix} f_A = 0,
\end{align*}
where the rows of $S_A$ generate $\V_A^\perp$. The solution space of this equation is $(\V_A^\perp)^\perp = \V_A$.

\end{enumerate}

One or the other of these representations might be convenient depending upon the context.

\subsection{Computing $\V_{AB}\circ A$ and $\V_{AB}\times B$  explicitly}
Let $\V_{AB}$ be defined through 
\begin{align*}
 \begin{bmatrix}
          S_A & S_B
         \end{bmatrix}
 \begin{pmatrix}
           f_A \\ f_B
 \end{pmatrix}
 = 0.
\end{align*}
We do row operations so that the coefficient matrix is in the form 
\begin{align*}
\begin{bmatrix}
         S_{1A} & 0 \\
         S_{2A} & S_{2B}
        \end{bmatrix},
\end{align*}
where the rows of $S_{2B}$ are linearly independent (or, if more convenient, columns of $S_{2A}$ are dependent on columns of $S_{2B}$ ). Then $[S_{1A}]f_A= 0$ defines $\V_{AB}\circ A$ and $[S_{2B}]f_B= 0$ defines $\V_{AB}\times B$. The statement about $\V_{AB}\circ A$ may be seen as follows: Clearly if any $(f_A,f_B) \in \V_{AB}$, then $f_A$ satisfies $[S_{1A}]f_A= 0$ so that $\V_{AB}\circ A$ is contained in the solution space of the latter. Next let $f_A$ satisfy $[S_{1A}]f_A= 0$. Since columns of $S_{2A}$ are linearly dependent on columns of $S_{2B}$, the equation
\begin{align*}
\begin{bmatrix}S_{2B}\end{bmatrix} f_B = -\begin{bmatrix}S_{2A}\end{bmatrix}f_A
\end{align*}
has a solution $\widehat{f}_B$.
Thus for any solution of $[S_{1A}]f_A = 0$, we can always find $\widehat{f}_B$ such that $(f_A,\widehat{f}_B) \in \V_{AB}$. This completes the proof that
 $\V_{AB}\circ A$ is the solution space of $[S_{1A}]f_A = 0$.

The statement about $\V_{AB}\times B$ is routine.

\subsection{Computing $\V^1_A \cap \V^2_A$ Explicitly}
Suppose $[S^1_A]f_A = 0$, $[S^2_A]f_A = 0$ define $\V^1_A$, $\V^2_A$ respectively. Let $A'$, $A''$ be copies of $A$ with $A$, $A'$, $A''$ disjoint. Build the space $\V_{AA'A''}$ defined through
\begin{align}
 \label{eqn:ComputeIntersection}
\begin{bmatrix}
        S^1_A &       &    \\
              & S^2_A &    \\
        0     & I     & -I \\
        I     & 0     & -I
       \end{bmatrix}
\begin{pmatrix}
 f^1_{A'} \\
f^2_{A''} \\
f_{A}
\end{pmatrix}
=
\begin{pmatrix}
         0 \\ 0 \\ 0
        \end{pmatrix}
\end{align}
It is clear that the vectors $f_A$ are precisely those in both $\V^1_A$ and $\V^2_A$.  Thus $\V_{AA'A''} \circ A $ is $\V^1_A \cap \V^2_A$. The explicit construction of the dot operation has been discussed above.

\subsection{Computing $\V^1_A + \V^2_A$ Explicitly}
Let $\V^1_A$ and $\V^2_A$ be defined as above. We define $\widehat{\V}_{AA'A''}$ through 
\begin{align}
 \label{eqn:ComputeSum}
\begin{bmatrix}
        S^1_A &       &    \\
              & S^2_A &    \\
        -I    & -I    & I
       \end{bmatrix}
\begin{pmatrix}
 f^1_{A'} \\
f^2_{A''} \\
f_{A}
\end{pmatrix}
=
\begin{pmatrix}
          0 \\ 0 \\ 0
\end{pmatrix}.
\end{align}
In this case again it is easy to see that $\widehat{\V}_{AA'A''}\circ A = \V^1_A + \V^2_A$ and we complete the explicit construction
as above.

Often the spaces $\V^1_A$, $\V^2_A$ may themselves be available to us only implicitly --- usually in the form of a generalized minor $\V_{AB} \leftrightarrow \V_B$. We discuss this situation below after we discuss the computation of the generalized minor.

\subsection{Computing $\V_{AB} \leftrightarrow \V_B$ Implicitly}
Let $\V_{AB}$ be defined by 
\begin{align*}
 \begin{bmatrix} 
         S_A & S_B
        \end{bmatrix}
\begin{pmatrix}
 f_A \\
 f_B
\end{pmatrix}
= 0
\end{align*}
and $\V_B$ by 
\begin{align*}
\begin{bmatrix} \widehat{S}_B \end{bmatrix} f_B = 0.
\end{align*}
Let $\widehat{\V}_{AB}$ be defined by 
\begin{align}
\label{eqn:ComputeDoubleArrow} 
 \begin{bmatrix}
  S_A & S_B \\
   0 & \widehat{S}_B
       \end{bmatrix}
\begin{pmatrix}
 f_A \\ f_B
\end{pmatrix} =
 \begin{pmatrix}
        0 \\ 0
       \end{pmatrix}.
\end{align}
Then 
\begin{align*}
 \widehat{\V}_{AB}\circ A & \equiv \{ f_A \ | \  (f_A,f_B) \in \V_{AB}, f_B\in\V_B \} \\
& = \V_{AB} \leftrightarrow \V_{B}.
 \end{align*}

\subsection{Computing $\V_{AB} \leftrightarrow \V_B$ Explicitly}
\label{sec:CompDblArrowExplictly}
In order to explicitly represent $\V_{AB}\leftrightarrow \V_{B}$ as the solution space of an equation
\begin{align*}
\begin{bmatrix}\widetilde{S}_A \end{bmatrix} f_A = 0,
\end{align*}
one has to do row operations and recast \eqref{eqn:ComputeDoubleArrow} as
\begin{align}
\label{eqn:ComputeDoubleArrowRecast}
 \begin{bmatrix}
  S'_{1A} & S'_{1B} \\
  S'_{2A} & 0 \\
   0 & \widehat{S}_B
       \end{bmatrix}
\begin{pmatrix}
 f_A \\ f_B
\end{pmatrix} =
\begin{pmatrix}
        0 \\ 0
       \end{pmatrix},
\end{align}
where the rows of $S'_{1B}$ together with those of  $\widehat{S}_B$ are independent.  The solution space of 
\begin{align*}
\bbmatrix{S'_{2A} }f_A = 0,
\end{align*}
is clearly $\widehat{\V}_{AB} \circ A$ ($\widehat{\V}_{AB}$ as in Equation \ref{eqn:ComputeDoubleArrow}) that is, $\V_{AB}\leftrightarrow \V_{B}.$ 

In general when one is performing operations such as intersection or sum with vector spaces, their representation may not be explicitly available but in the form  $\V_{AB}\leftrightarrow \V_{B}$.  

Let $\widehat{f}_B \in (\V_{AB} \cap \V_{B}) \circ B$. To find $\V_{AB} \leftrightarrow \{\widehat{f}_B\} $ we first find one solution $\widehat{f}_A$ to the equation using the first rows of \eqref{eqn:ComputeDoubleArrowRecast}
\begin{align*}
\bbmatrix{
S'_{1A}  \\ S'_{2A} 
} \widehat{f}_A = 
\bbmatrix{
S'_{1B}\widehat{f}_B  \\ 0
}.
\end{align*}
Then 
\begin{align*}
\V_{AB} \leftrightarrow \{\widehat{f}_B\} = \widehat{f}_A + \V_{AB}  \times A.
\end{align*}
$\V_{AB}  \times A$ can be seen to be the solution to 
\begin{align*}
\bbmatrix{
S'_{1A}  \\ S'_{2A} 
}\widetilde{f}_A = 
\ppmatrix{0  \\ 0
}
\end{align*}
and $\V_{AB} \leftrightarrow \{\widehat{f}_B\} $ as the collection of all solutions to 
\begin{align*}
\bbmatrix{
S'_{1A}  \\ S'_{2A} 
} f_A = 
\bbmatrix{
S'_{1B}\widehat{f}_B  \\ 0
}.
\end{align*}

\subsection{Computing $\V^1_A \cap \V^2_A$ Implicitly}
Suppose $\V^1_A \cap \V^2_A$ is to be computed where
\begin{align*}
 \V^1_A &= \V^1_{AB}\leftrightarrow \V^1_{B} \\
 \V^2_A &= \V^2_{AB}\leftrightarrow \V^2_{B}.
\end{align*}
Let $\V^1_{AB}$, $\V^2_{AB}$ be defined respectively through
\begin{align*}
\bbmatrix{
        S^1_A & S^1_B
       }
\ppmatrix{
 f^1_A \\ f^1_B
}
= 0,\\
\bbmatrix{
        S^2_A & S^2_B
       }
\ppmatrix{
 f^2_A \\ f^2_B
}
= 0.
\end{align*}
Let $\V^1_B$, $\V^2_B$ be defined  respectively through
\begin{align*}
 \bbmatrix{\widehat{S}^1_B} f_B &=0,\\
\bbmatrix{ \widehat{S}^2_B} f_B &=0.
\end{align*}
We proceed now as in \eqref{eqn:ComputeIntersection} by first copying set $A$ into $A'$, $A''$ and $B$ into $B'$, $B''$ and write the final equations as
\begin{align}
 \label{eqn:ComputeDoubleIntersection}
\bbmatrix{
        S^1_A & S^1_B & & & \\
        & \widehat{S}^1_B & & & \\
        & & S^2_A & S^2_B & \\
        & & & \widehat{S}^2_B & \\
        & & +I & & -I & \\
        I & & & & -I 
       }
\ppmatrix{
  f^1_{A'} \\ f^1_{B'} \\ f^2_{A''} \\  f^2_{B''} \\ f_A
}
= 0
\end{align}
If the above equation has the solution space $\mathcal{V}_{AA'A''B'B''}$ then we can show that 
\begin{align*}
 \mathcal{V}_{AA'A''B'B''}\circ A = \mathcal{V}^1_A \cap \mathcal{V}^2_A.
\end{align*}
To see this, note that $\mathcal{V}^1_{A'}$ is the collection of all $f^1_{A'}$ which can be part of a solution vector $(f^1_{A'},f^1_{B'})$ of the first two rows of \eqref{eqn:ComputeDoubleIntersection}. $\mathcal{V}^2_{A''}$ is the space of all vectors $f^2_{A''}$ which can be a part of the solution vector $(f^2_{A''},f^2_{B''})$ of the next two rows of \eqref{eqn:ComputeDoubleIntersection}. Now $f_A$, by the last two rows of the equation, is seen to a copy of both $f^1_{A'}$ as well as $f^2_{A''}$. Any such vector therefore belongs to $\mathcal{V}^1_A \cap \mathcal{V}^2_A$. Conversely any vector in $\mathcal{V}^1_A \cap \mathcal{V}^2_A$ can clearly be a copy of a part of some solution vector of the first two rows and therefore of the middle two rows also --- which means it is a part of a solution vector of \eqref{eqn:ComputeDoubleIntersection}.

\subsection{Computing  $\mathcal{V}^1_A + \mathcal{V}^2_A$ Implicitly}
By similar arguments we can show that we can write the constraints for $\mathcal{V}^1_A + \mathcal{V}^2_A$ as an implicit version \eqref{eqn:ComputeSum} as follows:
\begin{align}
 \label{eqn:ComputeDoubleSum}
\bbmatrix{
        S^1_A & S^1_B & & & \\
        & \widehat{S}^1_B & & & \\
        & & S^2_A & S^2_B & \\
        & & & \widehat{S}^2_B & \\
        -I & & -I & & I 
       }
\ppmatrix{
  f^1_{A'} \\ f^1_{B'} \\ f^2_{A''} \\  f^2_{B''} \\ f_A
}
= 0
\end{align}
If $\widetilde{\mathcal{V}}_{AA'A''B'B''}$ is the solution space of the above equation then 
\begin{align*}
 \mathcal{V}^1_A + \mathcal{V}^2_A = \widetilde{\mathcal{V}}_{AA'A''B'B''} \circ A.
\end{align*}

\subsection{Computing $\mathcal{V}_{\mydot{W}W}\leftrightarrow \mathcal{V}_W$, $\mathcal{V}_{\mydot{W}W}\leftrightarrow \mathcal{V}_{\mydot{W}}$ for Networks}

The computation of $\mathcal{V}_{\mydot{W}W}\leftrightarrow \mathcal{V}_W$ or $\mathcal{V}_{\mydot{W}W}\leftrightarrow \mathcal{V}_{\mydot{W}},$
$\Vonewp\lrar \Vw,$ $\Vtwodwdp\lrar (\Vw)_{\dw}.$
are important examples of the computations involved in implicit linear algebra. We describe how to perform such computations for a $RLCEJ$ network  for one of the  instances of $\mathcal{V}_M$.  The essential idea is to transform the problem to one of \nw{solving} the three multi-ports under appropriate conditions.

\subsubsection{ $\mathcal{V}_{\mydot{W}W}\leftrightarrow \mathcal{V}_{\mydot{W}}$,   $\mathcal{V}_M = 0_{mu} \oplus \mathcal{F}_{my}$}
\label{subsec:ComputeState0uFy}

For the case $\mathcal{V}_{\mydot{W}W}\leftrightarrow \mathcal{V}_{{W}}$,   $\mathcal{V}_M = 0_{mu} \oplus \mathcal{F}_{my},$
 see \cite{HNPS2013}.

To compute $\mathcal{V}_{\mydot{W}W}\leftrightarrow \mathcal{V}_{\mydot{W}}$, we begin once again by setting all sources to zero. Let $\mydot{w} = (\mydot{v}_C,\frac{di_L}{dt}) \in \mathcal{V}_{\mydot{W} }$. We will show how to find all $w$ such that $(w,\mydot{w}) \in \mathcal{V}_{\mydot{W}W}$.

The first step is to compute all $w = (v_c, i_L)$ corresponding to $\mydot{w}=0$. When the device characteristic 
matrices $C,L,$ are positive definite, the collection of all such $v_C$ is the voltage space of $\mathcal{G}_{CP_C} \times E_C$ and the collection of all such $i_L$ is the current space of $\mathcal{G}_{LP_L} \circ E_L$.
\\

Setting $\mydot{w}=0$ implies $\mydot{v}_C=0$ and $\frac{di_L}{dt}=0$. This implies $C\mydot{v}_C=0$ and $L\frac{di_L}{dt}=0$. Hence $i_C=0$, $v_L=0$. 

Therefore in the multi-ports $\mathcal{N}_{CP_C}$, $\mathcal{N}_{LP_L}$, since $P_C$, $P_L$ do not contain cutsets, $i_{P_C}=0$, $v_{P_L} =0$. 

Using KCL, KVL of $\mathcal{G}_{P_CP_C'}$, $\mathcal{G}_{P_LP_L'}$ this implies $i_{P_C'}=0$, $v_{P_L'} =0$. 

In the multi-port $\mathcal{N}_{E_rE_mP_C'P_L'}$, the sources have already been set to zero (since $\mathcal{V}_M = 0_{mu} \oplus \mathcal{F}_{my}$ ). 

In this network $P_C'$, $P_L'$ contain no loops or cutsets even when all sources are set to zero. 

Under mild genericity conditions on the static devices, this network can be shown to have a unique solution for arbitrary values of $i_{P_C'}$, $v_{P_L'}$. 

So in particular if $i_{P_C'}=0$, $v_{P_L'} =0$ we would have only the zero solution. Hence $v_{P_C'}$, $i_{P_L'}$ are zero vectors. 

This means, through KCL, KVL of $\mathcal{G}_{P_CP_C'
 }$, $\mathcal{G}_{P_LP_L'}$ that $v_{P_C}=0$, $i_{P_L}=0$.  

The corresponding space of $v_C$ (with $v_{P_C}=0$) would therefore be the space of solutions of KVL constraints of $\mathcal{G}_{CP_C}$ with $P_C$ short circuited, that is, the voltage space of $\mathcal{G}_{CP_C} \times E_C$. 

(Under  minimal multi-port decomposition this is the space of solutions of KVL constraints of $\mathcal{G}$ with all branches other than the $C$ branches short circuited). 

The corresponding space of $i_L$ (with $i_{P_L}=0$ ) can similarly be shown to be the space of KCL constraints of $\mathcal{G}_{LP_L}$ with $P_L$ open circuited, that is, the current space of $\mathcal{G}_{LP_L}\circ E_L$. 

(Under minimal multi-port decomposition this is the space of solutions of KCL constraints of $\mathcal{G}$ with all branches of $\mathcal{G}$ other than $L$ branches open circuited).
\\

Next we compute a single $w_1\in \mathcal{V}_{W}$ corresponding to a $\mydot{w}_1\in \mathcal{V}_{\mydot{W}}$.

Let $\mydot{w}_1\in \mathcal{V}_{\mydot{W}}$. Now let $\mydot{w}_1 = (\mydot{v}_C^1,\frac{di_L^1}{dt})$. 

Using $i_C^1=C\mydot{v}_C^1$, $v_L^1 = L\frac{di_L^1}{dt}$ we get $i_C^1$, $v_L^1$. 

Using KVL, KCL of $\mathcal{G}_{CP_C}$, $\mathcal{G}_{LP_L}$, since $P_C$, $P_L$ contain no loops or cutsets in these graphs we can compute $i_{P_C}^1$, $v_{P_L}^1$ (uniquely). 

Using KCL, KVL of $\mathcal{G}_{P_CP_C'}$, $\mathcal{G}_{P_LP_L'}$  we can compute $i_{P_C'}^1$, $v_{P_L'}^1$. 

In the multi-port $\mathcal{N}_{E_rE_mP_C'P_L'}$, the sources have already been set to zero (since $ \mathcal{V}_M = 0_{mu} \oplus \mathcal{F}_{my}$). 

As before we can compute $v_{P_C'}^1$, $i_{P_L'}^1$ uniquely. 

Through KCL, KVL of 
$\mathcal{G}_{P_LP_L'}$, $\mathcal{G}_{P_CP_C'}$ $v_{P_C}^1$, $i_{P_L}^1$ can be computed uniquely.
\\


One solution $v_C^1$ can be obtained by extending $P_C$ to a tree $t_C$ in $\mathcal{G}_{CP_C}$ assigning arbitrary voltages to branches in $t_C-P_C$ and hence by KVL of $\mathcal{G}_{CP_C}$, computing voltages for $E_C-t_C$. 

Similarly, one solution $i_L^1$ corresponding to $i_{P_L}^1$ can be obtained by extending $P_L$ to a cotree $\bar{t}_L$ in $\mathcal{G}_{LP_L}$ assigning arbitrary currents to branches in $\bar{t}_L-P_L$ and hence by KCL of $\mathcal{G}_{LP_L}$ computing currents for $E_L-\bar{t}_L$. 

The vectors $v_C^1$, $i_L^1$ together constitute the desired $w_1\in \mathcal{V}_{W}$.
\\

The collection of all vectors $\mathcal{K}_C^1 \oplus \mathcal{K}_L^1$ corresponding to $v_{P_C}^1$, $i_{P_L}^1$ (and therefore also corresponding to $\mydot{w}_1 = (\mydot{v}_C^1$, $\frac{di_L^1}{dt})$ ) is obtained as follows: 

$\mathcal{K}_C^1 = v_C^1 + \mathcal{V}_C^0$, $\mathcal{K}_L^1 = i_L^1 + \mathcal{V}_L^0$, where $\mathcal{V}_C^0$, $\mathcal{V}_L^0$ correspond to $\mydot{v}_C = 0$, $\frac{di_L}{dt}=0$. 

In the above, the computations which involve only KCL, KVL are, being linear time, inexpensive. The significant computation is the one that involves solution of the multi-port $\mathcal{N}_{E_rE_mP_C'P_L'}$ for given $i_{P_C'}^1$, $v_{P_L'}^1$.


\section{Proof of Theorem \ref{thm:controlledinvintersection}}
\label{sec:controlledinvintersection}
\begin{proof}
\begin{enumerate}
\item
By Theorem \ref{thm:controlledconditioneddual}, any $\Vwdw$ that  has $\Vw$ controlled invariant in it,
is adjoint to the genaut  $\mathcal{V}^{a}_{{W'}\mydot{{W'} } },$
which has $\V^{\perp}_{W'}\equiv (\V^{\perp}_W)_{W'}$ conditioned invariant in it.

\vspace{0.1cm}

Now $\Vwdw =(\mathcal{V}^{a}_{{W'}\mydot{{W'} } })^a,$ and 
$((\V^{\perp}_{W'})^{\perp})_W=\Vw.$ 

So it is sufficient for us to show that if $\Vwdw$ has  
$\Vw$ conditioned  invariant in it, then 
$\V^{\perp}_{W'}\subseteq\mathcal{V}^{a}_{{W'}\mydot{{W'} } }\circ W'.$

\vspace{0.25cm}

Let $\Vwdw$ be a genaut and have $\Vw$ conditioned invariant in it.

From Theorem \ref{thm:controlledconditioneddual}
, we know that
 $\V_W\supseteq (\Vwdw \times \dw )_W.$

It follows that
 $$\V^{\perp}_W\subseteq (\Vwdw \times \dw )^{\perp}_W,$$ i.e.,
 $$\V^{\perp}_W\subseteq  (\Vwdw^{\perp}\circ{\mydot{W}})_W,$$ i.e.,
$$(\V^{\perp}_W)_{W'}\subseteq  ((\Vwdw^{\perp})_{-\mydot{W'}W'}\circ{W'})_W,$$  i.e., $$\V^{\perp}_{W'} \subseteq \mathcal{V}^{a}_{{W'}\mydot{W'} }\circ W'.$$
\item
From Theorem \ref{thm:controlledconditioneddual},
 $\Vwdw$ is LSG with  $\Vw$ controlled invariant in it
 iff
  $\mathcal{V}^{a}_{{W'}\mydot{{W'} } }$
is USG and has $\V^{\perp}_{W'}$  conditioned invariant in it.

\vspace{0.1cm}

By Theorem \ref{thm:conditionallyinvsum}, when $\mathcal{V}^{a}_{{W'}\mydot{{W'} } }$
is USG, we have that 
$$\mathcal{V}^{a}_{{W'}\mydot{{W'} } }+ (\V^{\perp}_{W'}\oplus (\V^{\perp}_{W'})_{\mydot{{W'} }})$$ is a genop.

\vspace{0.1cm}

So it is sufficient to show that,  if $\Vwdw$ is USG and has $\Vw$ conditioned invariant in it,

then 
$\mathcal{V}^{a}_{{W'}\mydot{{W'} } }\bigcap (\V^{\perp}_{W'}\oplus (     \V^{\perp}_{W'})_{\mydot{{W'} }})$ is a genop.

\vspace{0.25cm}


Let $\Vwdw$ be USG and have $\Vw$ conditioned invariant in it.

\vspace{0.1cm}

From Theorem \ref{thm:conditionallyinvsum}, we know that,  
 $\Vwdw + (\V_W \oplus (\V_W)_{\mydot{W}})$ is a genop.

\vspace{0.1cm}

But then by Theorem \ref{thm:adjointprop} we have that 
 $(\Vwdw + (\V_W \oplus (\V_W)_{\mydot{W}}))_{-\dwd W'} ^{\perp}$
is a genop, \\

\vspace{0.1cm}

i.e., $(\Vwdw^{\perp} \bigcap (\V_W^{\perp} \oplus (\V_W^{\perp})_{\mydot{W}}))_{-\dwd W'}$ is a genop,\\

\vspace{0.1cm}

 i.e., 
$((\Vwdw^{\perp})_{-\mydot{W'}W'} \bigcap ((\V_W)^{\perp}_{-\mydot{W'}} \oplus (\V_W^{ \perp})_{{W'}}))$ is a genop, \\

\vspace{0.1cm}

i.e., 
$\mathcal{V}^{a}_{{W'}\mydot{{W'} } }\bigcap (\V^{\perp}_{W'}\oplus (     \V^{\perp}_{W'})_{\mydot{{W'} }})$ is a genop.
\end{enumerate}
\end{proof}
\section{Proof of Lemma \ref{lem:feedbackinjectionlrar} part 2}
\begin{proof}
Note that 
$$ ((\Vwdwmumy\bigcap \Vwmu)\circ W\dw)^{\perp}\equaln([(\Vwdwmumy\bigcap \Iww)\lrarn  (\Vwmu\oplus \F_{M_y})]_{W\dw})^{\perp},$$
which reduces to
$$ (\Vwdwmumy^{\perp}+\Vwmu^{\perp})\times W\dw \equaln ((\Vwdwmumy^{\perp}+I_{W -W'})\lrarn  (\Vwmu^{\perp}\oplus \0_{M_y}))_{W\dw},$$
which is equivalent to 
$$((\Vwdwmumy^{\perp})_{-\mydot{W"}W"-M"_yM"_u}+(\Vwmu^{\perp})_{-\mydot{W"}-M"_y})\times \mydot{W"}W"$$
$$\equaln [(\Vwdwmumy^{\perp})_{-\mydot{W"}W"-M"_yM"_u }+(I_{W -W'})_{\mydot{W"}\mydot{W'}}]\lrarn  [(\Vwmu^{\perp})_{ {-\mydot{W"}-M"_y}}\oplus (\0_{M_y})_{M"_u}]_{W"\mydot{W"}}.$$
This can be interpreted as $$ (\mathcal{V}^{a}_{{W"}\mydot{{W"} }M"_yM"_u }+\V_{\mydot{{W"}}M"_y})\times  W"\mydot{W"}
$$ $$\equaln (\mathcal{V}^{a}_{{W"}\mydot{{W"} }M"_yM"_u }+ I_{\mydot{W"}-\mydot{W'}})\lrarn(\V_ {\mydot{W"}M"_y}\oplus \0_{M"_u})_{W"\mydot{W"}},$$ where we have denoted by $(\Vwmu^{\perp})_{ {-\mydot{W"}-M"_y}}$ by $\V_ {\mydot{W"}M"_y}.$
Since any dynamical system is the adjoint of its adjoint and for any space $\V,$ we have $(\V^{\perp})^{\perp},$ this proves this part of the lemma.
\end{proof}
\section{Dualization of pole placement algorithms}
\label{sec:ppa_dualization}
In this section we recast  the algorithms of Subsection \ref{subsec:statefeedback_ouputinjection}
  in such a manner that they can be 
line by line dualized. 

The rule is simply to work only with vector spaces (and not with vectors), avoid the `belongs to' relation
and use `contained in' instead.
We give interpretations to all the entities that occur in the algorithms of the previous subsections.

{\it Interpreting basic sequences}

$ x^0_{\dw}$ is a one dimensional subspace satisfying
$$x^0_{\dw} \subseteq \Vonewdw \times \dw,\ \ \  x^0_{\dw} \not\subseteq (\Vonewdw \times W)_{\dw}.$$
(This is possible since $\Vonewdw$ is a USG but has been assumed not to be a genop.)

All the $ x^j_{\dw}$ are one dimensional subspaces.
$$x^{j+1}_{\dw} \subseteq (\Vonewdw\cap x^{j}_{W}) \circ \dw,\ \ \ \ \ 
x^{j+1}_{\dw} \not \subseteq x^0_{\dw} + x^1_{\dw}  +\cdots  + x^{j}_{\dw}\plusn\V^{com}_{\dw}.$$
We stop when 
$(\Vonewdw\cap x^{j}_{W}) \circ \dw \subseteq 
 x^0_{\dw} + x^1_{\dw}  +\cdots  + x^{j}_{\dw}\plusn\V^{com}_{\dw},$
 i.e., at $j=k-1.$

We build the following subspaces 
$$\Xwdw^{00}\equiv x^0_{W}\cap  (\Vonewdw)^{(0)}\cap x^0_{\dw}\plusn\V^{com}_{W}\oplus\V^{com}_{\dw},$$
$$\Xwdw^{01}\equiv x^0_{W}\cap  (\Vonewdw)^{(1)}\cap x^1_{\dw}\plusn\V^{com}_{W}\oplus\V^{com}_{\dw},$$
$$ \vdots $$
$$\Xwdw^{0j}\equiv x^0_{W}\cap  (\Vonewdw)^{(j)}\cap x^j_{\dw}\plusn\V^{com}_{W}\oplus\V^{com}_{\dw},$$
till $j=k-1.$
Here $$(\Vonewdw)^{(0)}\equiv (I_{W\dw}\cap \Vonewdw \circ W) +\Vonewdw \times \dw,\ \  (\Vonewdw)^{(1)}\equiv \Vonewdw ,\ \  (\Vonewdw)^{(j+1)}\equiv (\Vonewdw)^{j}*\Vonewdw .$$
At $j=k,$ we can pick $b _i$ such that
$$b _0^{\mydot{w}}\Xwdw^{00}\pdw \cdots \pdw b _{k-1}^{\mydot{w}}\Xwdw^{0(k-1)}\pdw \Xwdw^{0(k)}\subseteq  x^0_{W} \oplus \V^{com}_{W}\oplus \V^{com}_{\dw}.$$ 
We will show that this implies that the LHS is decoupled. 

Let $\Vwdw'\equiv LHS.$

Each term $b _{j}^{\mydot{w}}\Xwdw^{0j}$ satisfies 
$b _{j}^{\mydot{w}}\Xwdw^{0j}\times \dw\supseteq \V^{com}_{\dw}.$

Therefore $\Vwdw'\supseteq \0_W\oplus \V^{com}_{\dw},$
so that $\Vwdw'\times\dw \supseteq  \V^{com}_{\dw}.$

But  $\Vwdw'\subseteq x^0_{W} \oplus \V^{com}_{W}\oplus\V^{com}_{\dw},$ implies that 
$\Vwdw'\circ \dw \subseteq \V^{com}_{\dw}.$ 

Therefore $\Vwdw'\circ \dw\subseteq \Vwdw'\times \dw .$

Since we always have $\Vwdw'\circ \dw\supseteq \Vwdw'\times \dw ,$
this means that $\Vwdw'\circ \dw= \Vwdw'\times \dw ,$
and that $\Vwdw'$ is decoupled.

{\it Constructing genops from basic sequence analogues}

Define
$$\Xwdw^{j(j+1)}\equiv x^j_{W}\cap  \Vonewdw\cap x^{j+1}_{\dw}\plusn\V^{com}_{W}\oplus\V^{com}_{\dw},$$

 $$\Vwdw^{start}\equiv  \Xwdw^{01}+\cdots +\Xwdw^{j(j+1)}+ \cdots + \Xwdw^{(k-1)k}\plusn \V^{com}_{W}\oplus\V^{com}_{\dw}.$$
 
Observe that $\Xwdw^{0j}, \Xwdw^{j(j+1)}$ are LSGs and would be one dimensional if $\V^{com}_{\dw}=\0_{\dw}.$
The space $\Vwdw^{start},$ which is a sum of spaces $\Vwdw^{j(j+1)}$ and $\V^{com}_{W}\oplus\V^{com}_{\dw},$
is a genop since  
\begin{enumerate}
\item
 $$\Vwdw^{start}\circ W = x^0_{W}+\cdots + x^{k-1}_{W}\plusn \V^{com}_{W}\supseteq x^1_{W}+\cdots + x^{k}_{W}\plusn \V^{com}_{W}=
(\Vwdw^{start}\circ \dw)_{W},$$
\item 
$\Vwdw^{start}\circ W= x^0_{W}+\cdots + x^{k-1}_{W}\plusn \V^{com}_{W}$ and the only way of obtaining
 a vector in $\V^{com}_{W}$ as a linear combination of nonzero vectors of $x^0_{W},\cdots , x^{k-1}_{W}$
is by the trivial linear combination, so that $$\Vwdw^{start}\times \dw =  \V^{com}_{\dw}=(\V^{com}_{W})_{\dw}\subseteq (\Vwdw^{start}\times W)_{\dw}.$$
\end{enumerate}
$\Vwdw^{start}$ would be annihilated by the polynomial $b _0+ \cdots + b_{k-1}s^{k-1}+s^k.$
The proof is essentially the same as the one for  Claim \ref{cl:annihilate_start}.

{\it Transforming basic sequence analogues}

Next we build spaces which are analogues of $y^i_{W}$ by constructing a subspace version of
Equation \ref{eqn:xy}.
\begin{equation}
\begin{matrix}
\Ywdw^{00}  &&=&&& \Xwdw^{00} \\
\Ywdw^{01}  &&=& &&\Xwdw^{01}\pdw\ldw_1 \Xwdw^{00}\\ 
    & &\dots&& & \\
\Ywdw^{0k} &&=& &&\Xwdw^{0k}\pdw\ldw_1\Xwdw^{0(k-1)} \pdw\cdots \pdw \ldw_k\Xwdw^{00} 
\end{matrix}
\end{equation}

Define
$$\Ywdw^{j0}\equiv (\Ywdw^{0j})_{\dw W}, \ \ 
\Ywdw^{j(j+1)}\equiv 
\Ywdw^{j0}* \Ywdw^{0(j+1)}$$
and $$\Vwdw^{end}\equiv 
\Ywdw^{01}+\cdots +\Ywdw^{j(j+1)}+ \cdots + \Ywdw^{(k-1)k}\plusn  \V^{com}_{W}\oplus\V^{com}_{\dw}.$$

$\Vwdw^{end}$ would be annihilated by the polynomial $c _0+ \cdots + c_{k-1}s^{k-1}+s^k,$
where $(c_0, \cdots ,c_k), c_k=1$ satisfies Equation \ref{eqn:c2b}.

{\it Line by line dualization}

Once ideas have been cast in the subspace context the algorithms can be dualized line by line as follows.

Our notation is as follows:

The dual of $x^j_{W}$ is $x'^{j}_{\dw}\equiv ((x^j_{W})^{\perp})_{\dw},$ dual of $x^j_{\dw}$ is $x'^{j}_{W}\equiv  ((x^j_{\dw})^{\perp})_{W}.$

The dual of $\Vwdw$ is $\Vwdw'\equiv ((\Vwdw)^{\perp})_{-\dw W}.$

The dual of $\Vwdw^1+\Vwdw^2$ is  $(\Vwdw^1)'\cap (\Vwdw^2)',$

the dual of $\Vwdw^1\cap \Vwdw^2$ is  $(\Vwdw^1)'+ (\Vwdw^2)',$

the dual of $\ldw \Vwdw$ is  $\ldw \Vwdw'.$

The dual of $\Vwdw^1\pdw\Vwdw^2$ is  $(\Vwdw^1)'\pdw (\Vwdw^2)'.$

The dual of $\Vwdw^{\alpha}*\Vwdw^{\beta}$ is $(\Vwdw^{\alpha}*\Vwdw^{\beta})'\equiv (\Vwdw'^{\alpha})_{W'\dw}\lrar (\Vwdw'^{\beta})_{WW'}= \Vwdw'^{\beta}*\Vwdw'^{\alpha}.$

{\it Interpreting basic sequences}

We are given an LSG $\Vtwowdw\equiv (\Vonewdw)',$
$ x'^0_{W}$ is an $n-1 \ \  (n= |W|)$  dimensional subspace satisfying
$$x'^0_{W} \supseteqn \Vtwowdw \circ W, \ \ x'^0_{W}\  \not\supseteq \  (\Vtwowdw \circ \dw)_{W}.$$
(This is possible since $\Vtwowdw$ is an LSG but has been assumed not to be a genop.)

All $ x'^j_{W}$  are $n-1$ dimensional subspaces.
$$x'^{j+1}_{W} \supseteqn (\Vtwowdw +x'^{j}_{\dw}) \times W,\ \ x'^{j+1}_{W} \  \not \supseteq \ x'^0_{W} \cap x'^1_{W}  \cap\cdots  \cap x'^{j}_{W}\cap \V'^{com}_{W}.$$
($\V'^{com}_{W}\equiv  \Vtwowdw \circ W + (\Vtwowdw \circ \dw)_{W}.$) 

We stop when
$(\Vtwowdw +x'^{j}_{\dw}) \times W \supseteq 
 x'^0_{W} \cap x'^1_{W}  \cap \cdots  \cap  x'^{j}_{W}  \cap \V'^{com}_{W},$
 i.e., at $j=k-1.$

We build the following subspaces
$$\Xwdw'^{00}\equiv (x'^0_{W}+   (\Vtwowdw)^{(0)}+  x'^0_{\dw})\cap (\V'^{com}_{W}\oplus\V'^{com}_{\dw}),$$
$$\Xwdw'^{10}\equiv (x'^1_{W} + (\Vtwowdw)^{(1)} + x'^0_{\dw})\cap (\V'^{com}_{W}\oplus\V'^{com}_{\dw}),$$
$$ \vdots $$
$$\Xwdw'^{j0}\equiv (x'^j_{W} + (\Vtwowdw)^{(j)} +x'^0_{\dw})\cap (\V'^{com}_{W}\oplus\V'^{com}_{\dw}),$$
till $j=k-1.$

Here $$(\Vtwowdw)^{(0)}\equiv (I_{W\dw}\cap \Vtwowdw\circ W) +\Vtwowdw\times \dw, \ \ (\Vtwowdw)^{(1)}\equiv \Vtwowdw ,\ \  (\Vtwowdw)^{(j+1)}\equiv (\Vtwowdw)^{j}*\Vtwowdw .$$
(Note that we could equivalently have written $(\Vtwowdw)^{(0)}\equiv (I_{W\dw}+ \Vtwowdw\times \dw ) \bigcap \Vtwowdw\circ W.$ )


At $j=k,$ we can pick $b _i$ such that
$$b _0^{\mydot{w}}\Xwdw'^{00}\pdw \cdots \pdw b _{k-1}^{\mydot{w}}\Xwdw'^{(k-1)0}\pdw \Xwdw'^{k0}\supseteq  x'^0_{\dw} \oplus \V'^{com}_{W}\oplus \V'^{com}_{\dw}.$$ 

We will show that this implies that the LHS is decoupled.

 Let $\Vwdw'\equiv LHS.$
Each term $b _{j}^{\mydot{w}}\Xwdw'^{j0}$ satisfies $$ b _{j}^{\mydot{w}}\Xwdw'^{j0}\circ W\subseteq \V'^{com}_{W}.$$

Therefore $\Vwdw'\subseteq \F_{\dw}\oplus \V'^{com}_{W},$
so that $\Vwdw'\circ W \subseteq  \V'^{com}_{W}.$

But  $\Vwdw'\supseteq x'^0_{\dw} \oplus \V'^{com}_{W}\oplus \V'^{com}_{\dw},$ implies that
$\Vwdw'\times W \supseteq \V'^{com}_{W}.$ 

Since we always have $\Vwdw'\times W\subseteq \Vwdw'\circ W ,$
this means that $\Vwdw'\circ W= \Vwdw'\times W ,$
and that $\Vwdw'$ is decoupled.

{\it Constructing genops from basic sequence analogues}

Define
$$\Xwdw'^{(j+1)j}\equiv (x'^j_{\dw}+  \Vtwowdw+ x'^{j+1}_{W})\cap (\V'^{com}_{W}\oplus\V'^{com}_{\dw}),$$
 $$\Vwdw'^{start}\equiv  (\Xwdw'^{10}\cap \cdots \cap \Xwdw'^{(j+1)j}\cap  \cdots \cap  \Xwdw'^{k(k-1)})\bigcap  (\V'^{com}_{W}\oplus\V'^{com}_{\dw}).$$

Observe that $\Xwdw'^{j0}, \Xwdw'^{(j+1)j}$ are USGs and would be $|W|-1$ dimensional if $\V'^{com}_{W}=\F_{\W}.$
The space $\Vwdw'^{start},$ which is an intersection of spaces $\Vwdw'^{(j+1)j}$ and $\V'^{com}_{W}\oplus\V'^{com}_{\dw},$
is a genop since we saw earlier that its dual is a genop.
$\Vwdw'^{start}$ would be annihilated by the polynomial $b _0+ \cdots + b_{k-1}s^{k-1}+s^k.$

{\it Transforming basic sequence analogues}

Next we build spaces which are analogues of $y'^i_{W}$ by constructing a subspace version of
Equation \ref{eqn:xy}.
\begin{equation}
\begin{matrix}
\Ywdw'^{00}  &&=&&& \Xwdw'^{00} \\
\Ywdw'^{10}  &&=& &&\Xwdw'^{10}\pdw \ldw _1\Xwdw'^{00}\\ 
    & &\dots&& & \\
\Ywdw'^{k0} &&=& &&\Xwdw'^{k0}\pdw \ldw _1\Xwdw'^{(k-1)0} \pdw \cdots \pdw  \ldw _k\Xwdw'^{00} 
\end{matrix}
\end{equation}

Define
$$\Ywdw'^{0j}\equiv (\Ywdw'^{j0})_{\dw W}, \ \ 
\Ywdw'^{(j+1)j}\equiv 
\Ywdw'^{(j+1)0}* \Ywdw'^{0j}$$

We define $\Vwdw'^{end}\equiv 
(\Ywdw'^{10}\cap \cdots \cap \Ywdw'^{(j+1)j}\cap  \cdots \cap  \Ywdw'^{k(k-1)})\bigcap  (\V'^{com}_{W}\oplus\V'^{com}_{\dw}).$

$\Vwdw'^{end}$ would be annihilated by the polynomial $c _0+ \cdots + c_{k-1}s^{k-1}+s^k,$
where $(c_0, \cdots ,c_k), c_k=1$ satisfies Equation \ref{eqn:c2b}.
\section{Building an emulator implicitly for a network based dynamical system}
\label{sec:implicitemulatormultiport}
We give an example where we spell out all the vector spaces, which implicitly
determine the multiport based emulator.

\begin{example}
\label{eg:emulator_multiport_decomposition}
\begin{figure}
 \includegraphics[width=7in]{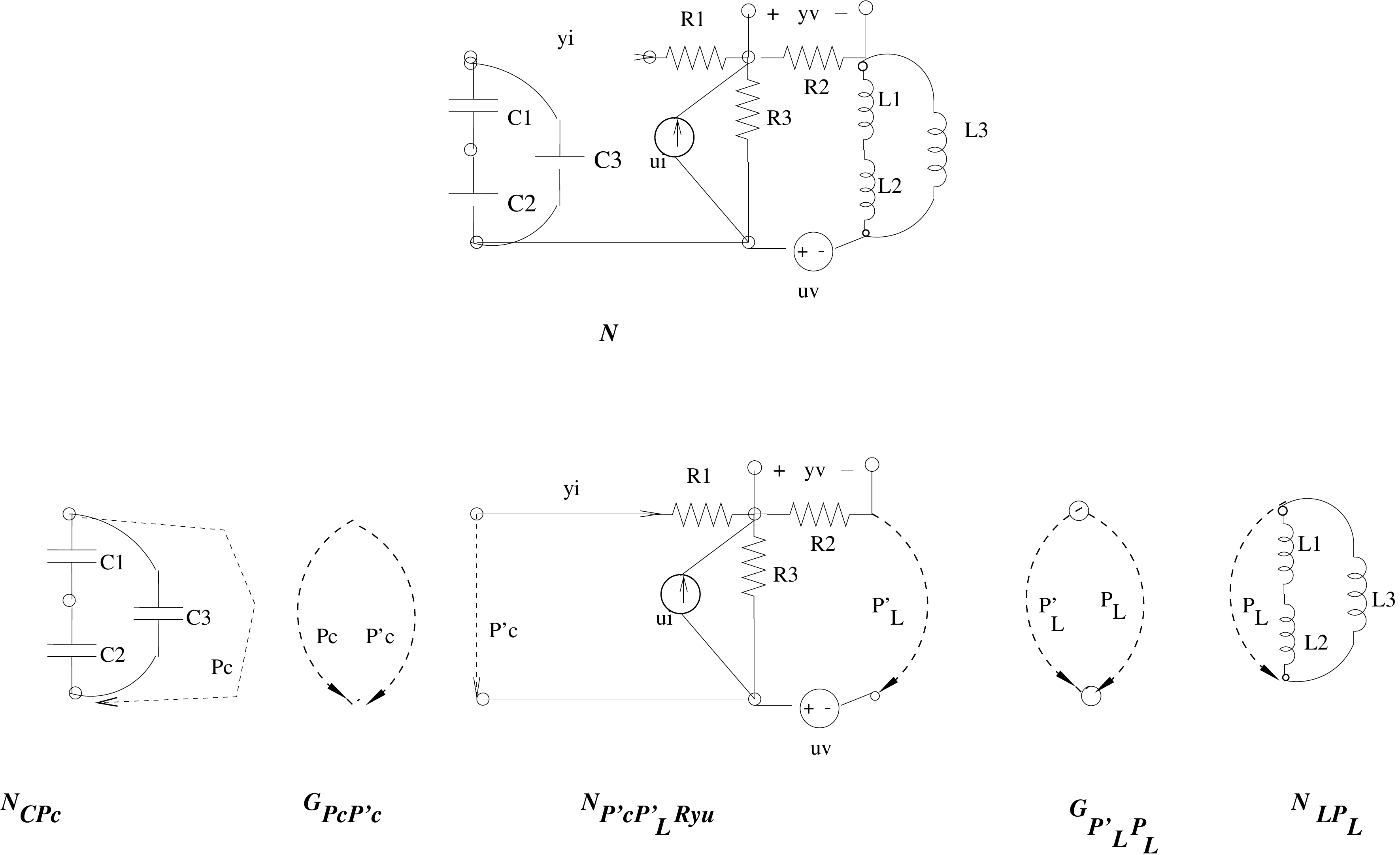}
 \caption{The primal network with port decomposition into capacitor inductor and static
multiports}
 \label{fig:circuit03}
\end{figure}

{\textup{
The primal $\mathcal{N}$ shown in Figure \ref{fig:circuit03} (or rather its solution space $\V^{sol}(\mathcal{N})),$
is a GDSA (generalized dynamical system with additional variables)
with the variables
$$v_C,i_L,i_C,v_L,  {\mydot{v}}_C,{\mydot{i}}_L, v_R,i_R,ui,uv,yi,yv, v_{ui},i_{uv}, v_{yi}, i_{yv},$$
(each of $v_C,i_L, \cdots $ is itself a subvector, for instance $v_C \equiv (v_{C1},v_{C2},v_{C3})$).
Of these the dynamical variables are $v_C,i_L,{\mydot{v}}_C,{\mydot{i}}_L,$
the manifest variables  $M_u$ are $ui,uv,$ and the manifest variables $M_y$ are $yi,yv.$ The additional variables are
$i_C,v_L, v_R,i_R,v_{ui},i_{uv}, v_{yi}, i_{yv}.$ For greater readability, we will make this total set of variables, the index set. We also use a convenient notation
$\{x=Ky\}$ to denote the solution space $\V_{xy}$ of the equation $x=Ky.$\\
Let $\mathcal{G}$
be the graph of the network.
Let
the device characteristic constraints of the network be
$$i_C=C{\mydot{v}}_C;\ \  v_L=L{\mydot{i}}_L;\ \ v_R=Ri_R;\ \ v_{yi}=0;\ \ i_{yv}=0;\ \  v_{ui}\ free,\ \  i_{uv}\ free,$$
where the matrices $C,L,R$ are symmetric positive definite.
The topological constraints are
$$v\in \V^{v}(\mathcal{G}), \ i\in \V^{i}(\mathcal{G}), \ {\mydot{v}}_C \in \V^{v}(\mathcal{G})\circ E_C= \V^{v}(\mathcal{G}\circ E_C), \ {\mydot{i}}_L \in \V^{i}(\mathcal{G})\circ E_L= \V^{i}(\mathcal{G}\times E_L).$$
We will denote }}

\hspace{4cm}the above vector space to which ${\mydot{v}}_C$ belongs,
by $\V^{top}_{\mydot{v_C}},$ \\

\hspace{4cm}the  vector space to which ${\mydot{i}}_L$ belongs,
by $\V^{top}_{\mydot{i_L}}$
and 
$$\V^{v}(\mathcal{G})\oplus  \V^{i}(\mathcal{G})\oplus \V^{top}_{\mydot{v_C}}
\oplus \V^{top}_{\mydot{i_L}}\ \ \ \  \textup{by} \ \ \ \ \V^{top}(\mathcal{N}).$$

{\textup{The original GDS $\Vwdwmumy$ corresponding to $\mathcal{N}$
has the index set $\{v_C,i_L,{\mydot{v}}_C,{\mydot{i}}_L,ui,uv,yi,yv\}.$
We have
 $$\Vwdwmumy $$
$$ = [(\V^{top}(\mathcal{N})\lrar (\{v_R=Ri_R\}\oplus\0_ {v_{yi}i_{yv}}\oplus \F_{v_{ui}i_{uv}}))\  \bigcap\    (\{i_C=C{\mydot{v}}_C \}\oplus \{v_L=L{\mydot{i}}_L\})]
\ \ \lrar \ \ (\F_{i_C}\oplus \F_{v_L}).$$}}
%
%

\textup{In order to build the emulator $\Vpdpmumy,$ the network $\mathcal{N}$ is now decomposed into multiports and port connection
diagrams, with minimum number of ports, as outlined in Section \ref{subsec:MultiportDecomposition}. This is shown in Figure \ref{fig:circuit03}. The relevant multiports are 
$$\mathcal{N}_{CP_C},\ \ \ 
\mathcal{N}_{P'_CP'_LRuy}, \ \ \ 
\mathcal{N}_{LP_L}.$$
We denote their graphs respectively by 
$$\mathcal{G}_{CP_C},\ \ 
\mathcal{G}_{P'_CP'_LRuy},\ \ 
\mathcal{G}_{LP_L}.$$
The device characteristics of the multiport  $\mathcal{N}_{CP_C}
$ are 
$$ 
\{i_C=C{\mydot{v}}_C \}, \ \ v_C, i_{P_C}, v_{P_C} \  free,$$
those of the multiport $\mathcal{N}_{LP_C}
$ are
$$ 
\{v_L=L{\mydot{i}}_L \}, \ \ i_L, i_{P_L}, v_{P_L} \  free,$$
those of  the multiport  $\mathcal{N}_{P'_CP'_LRuy} 
$ are
$$\{v_R=Ri_R\}, \ \ v_{ui},  i_{ui},  v_{uv}, i_{uv},   v_{yv}, i_{yi} , 
,  i_{P'_C}, v_{P'_C} , i_{P'_L}, v_{P'_L} \  free,\ \ i_{yv}, v_{yi}\ zero .$$
The topological constraints of the multiport  $\mathcal{N}_{CP_C}$
are
$$(v_C,v_{P_C})\in \V^{v}(\mathcal{G}_{CP_C} ),\ \ (i_C,i_{P_C})\in \V^{i}(\mathcal{G}_{CP_C} ),\ \ 
\mydot{v}_C\in \V^{v}(\mathcal{G}_{CP_C}) \circ{E_C}= \V^{v}(\mathcal{G})\circ E_C= V^{v}(\mathcal{G}\circ E_C)=\V^{top}_{\mydot{v}_C}
 .$$
The topological constraints of the multiport  $\mathcal{N}_{LP_L}$
are
$$(v_L,v_{P_L})\in \V^{v}(\mathcal{G}_{LP_L} ),\ \ (i_L,i_{P_L})\in \V^{i}(\mathcal{G}_{LP_L} ),\ \ 
\mydot{i}_L\in \V^{i}(\mathcal{G}_{LP_L}) \circ{E_L}= \V^{i}(\mathcal{G}\times E_L)= \V^{top}_{\mydot{i}_L}
 .$$
(The statements  $\V^{v}(\mathcal{G}_{CP_C}) \circ{E_C}= \V^{v}(\mathcal{G})\circ E_C , \ \ 
\V^{i}(\mathcal{G}_{LP_L}) \circ{E_L}= \V^{i}(\mathcal{G})\circ E_L$ \\
follow from the properties of multiport decomposition.)}

\textup{The port connection diagram is $\mathcal{G}_{P_CP'_C}\oplus \mathcal{G}_{P_LP'_L}.$
For simplicity let us write for any graph $\mathcal{G'},$} 
$$\V^{top}(\mathcal{G'}) \equiv (\V^{v}(\mathcal{G'}) \oplus \V^{i}(\mathcal{G'}) ).$$ 
\textup{By the definition of multiport decomposition, we have 
$$\V^{top}(\mathcal{G}) 
=(\V^{top}(\mathcal{G}_{CP_C})\oplus \V^{top}(\mathcal{G}_{LP_L})\oplus \V^{top}(\mathcal{G}_{P'_CP'_LRuy})) 
\lrar (\V^{top}(\mathcal{G}_{P_CP'_C})\oplus \V^{top}(\mathcal{G}_{P_LP'_L}) 
).$$}
\textup{We denote the space of solutions of the device characteristic of a linear network $\mathcal{N}$
by $\V^{dev}(\mathcal{N})$ and the space of solutions of the network by
$\V^{sol}(\mathcal{N}).$
In each of the multiports, the port edges have no constraints on their currents or voltages.
The multiports do not have common variables.
Thus 
$$\V^{dev}(\mathcal{N})=$$ 
$$[\V^{dev}(\mathcal{N}_{CP_C})\circ \{v_C,i_C, {\mydot{v}}_C \}]\oplus [\V^{dev}(\mathcal{N}_{LP_L})
\circ \{v_L,i_L, {\mydot{i}}_L \}]\oplus [\V^{dev}(\mathcal{N}_{P'_CP'_LRuy})\circ\{v_R,i_R,ui,uv,yi,yv,v_{ui},i_{uv}, v_{yi}, i_{yv} \}].$$
Since the port branches have no device characteristic constraints, we can write 
$$\V^{sol}(\mathcal{N})
=(\V^{sol}(\mathcal{N}_{CP_C})\oplus \V^{sol}(\mathcal{N}_{LP_L})\oplus \V^{sol}(\mathcal{N}_{P'_CP'_LRuy}))
\lrar (\V^{top}(\mathcal{G}_{P_CP'_C})\oplus \V^{top}(\mathcal{G}_{P_LP'_L})
)$$
The dynamical system $\V_{W\dw M_uM_y}$ is the restriction of the above space to
$\{v_C,i_L, {\mydot{v}}_C,{\mydot{i}}_L, ui,uv,yi,yv\},$
with $W\equiv \{v_C,i_L\},\dw\equiv \{{\mydot{v}}_C,{\mydot{i}}_L\},
M_u\equiv \{ui,uv\}, M_y\equiv \{yi,yv\}.$
This can be rewritten as 
$$\V_{W\dw M_uM_y}=$$
$$ [((\V^{top}(\mathcal{G}_{CP_C})\oplus \V^{top}_{\mydot{v_C}})\bigcap \{i_C=C{\mydot{v}}_C\})\lrar \F_{i_c}]$$
$$\oplus [((\V^{top}(\mathcal{G}_{LP_L})\oplus \V^{top}_{\mydot{i_L}})\bigcap \{v_L=L\mydot{i}_L\})\lrar \F_{v_L}]$$
$$\oplus [\V^{top}(\mathcal{G}_{P'_CP'_LRuy})\lrar (\{v_R=Ri_R\}\oplus \0_{ v_{yi}i_{yv}}
\oplus \F_{i_{uv}v_{ui}})]
$$
$$\lrar [\V^{top}(\mathcal{G}_{P_CP'_C})\oplus \V^{top}(\mathcal{G}_{P_LP'_L})].$$}
\textup{We now describe the emulator dynamical system $\V_{P\dP M_uM_y},$ implicitly.
This has the index set of variables $\{v_{PC},i_{PL},{\mydot{v}}_{PC},{\mydot{i}}_{PL},ui,uv,yi,yv\}.$ Here $ P\equiv \{v_{PC},i_{PL}\}, \dP\equiv \{{\mydot{v}}_{PC},{\mydot{i}}_{PL}\},
M\equiv \{ui,uv,yi,yv\}.$ The variables ${\mydot{v}}_{PC},{\mydot{i}}_{PL}$ are newly introduced
and are defined by $$({\mydot{v}}_C, {\mydot{v}}_{P_C})\in \V^v(\mathcal{G}_{CP_C});\ \ 
({\mydot{i}}_L, {\mydot{i}}_{P_L})\in \V^i(\mathcal{G}_{LP_L}).$$
If $P_C,P_L$ do not contain circuits or cutsets, which happens when the port decomposition 
is minimal,\\
 ${\mydot{v}}_{P_C}$ is uniquely determined by ${\mydot{v}}_C$ and 
${\mydot{i}}_{P_L}$ is uniquely determined by ${\mydot{i}}_L$ through the above definitions.}

\vspace{0.2cm}

\textup{ Let us now modify  $\mathcal{N}_{CP_C}$ to a new object $\mathcal{N}'_{CP_C}$ 
by adding an additional constraint  involving 
$({\mydot{v}}_C, {\mydot{v}}_{P_C}) .$}

\vspace{0.2cm}

\textup{Thus the constraints of  $\mathcal{N}'_{CP_C}$  are:
$$(i_C, i_{P_C}) \in \V^i(\mathcal{G}_{CP_C}); \ \  
(v_C, v_{P_C}) \in \V^v(\mathcal{G}_{CP_C}); \ \ 
({\mydot{v}}_C, {\mydot{v}}_{P_C}) \in \V^v(\mathcal{G}_{CP_C}); \ \  i_C=C{\mydot{v}}_C.$$
Let us denote the solution space of these constraints restricted to 
$\{i_{P_C},{\mydot{v}}_{P_C}\}$
by $\V_{i_{P_C}{\mydot{v}}_{P_C}}.$
Clearly 
$$\V_{i_{P_C}{\mydot{v}}_{P_C}}=[(\V^v(\mathcal{G}_{CP_C}))_{{\mydot{v}}_{C}{\mydot{v}}_{P_C}}
\oplus \V^i(\mathcal{G}_{CP_C})]\lrar \{i_C=C{\mydot{v}}_C\}.$$
Since the port decomposition is minimal, the ports $P_C$ will contain neither circuits
nor cutsets of $\mathcal{G}_{CP_C}.$ If the matrix $C$ is positive 
definite, it can be shown that from either $i_{P_C}$ or ${\mydot{v}}_{P_C} $ we can determine all the variables
in $ \{i_C, {\mydot{v}}_C, i_{P_C}, {\mydot{v}}_{P_C} \},$ uniquely.}

\vspace{0.2cm}

\textup{Next let us modify  $\mathcal{N}_{LP_L}$ to a new object $\mathcal{N}'_{LP_L}$
by adding an additional constraint  involving
$({\mydot{i}}_L, {\mydot{i}}_{P_L}) .$
Thus the  constraints of $\mathcal{N}'_{LP_L}$ are
$$(v_L, v_{P_L}) \in \V^v(\mathcal{G}_{LP_L}); \ \  
(i_L, i_{P_L}) \in \V^i(\mathcal{G}_{LP_L}); \ \ 
({\mydot{i}}_L, {\mydot{i}}_{P_L}) \in \V^i(\mathcal{G}_{LP_L}); \ \  v_L=L{\mydot{i}}_L.$$
Let us denote the solution space of these constraints restricted to
the index set $\{v_{P_L},{\mydot{i}}_{P_L}\},$ by $\V_{v_{P_L},{\mydot{i}}_{P_L}}.$
Clearly 
$$\V_{v_{P_L}{\mydot{i}}_{P_L}}=[(\V^i(\mathcal{G}_{LP_L}))_{{\mydot{i}}_{L}{\mydot{i}}_{P_L}}
\oplus \V^v(\mathcal{G}_{LP_L})]\lrar \{v_L=L{\mydot{i}}_L\}.$$
Since the port decomposition is minimal, the ports $P_L$ will contain neither circuits
nor cutsets of $\mathcal{G}_{LP_L}.$
If the matrix $L$ is positive
definite, it can be shown that from either
$v_{P_L}$ or ${\mydot{i}}_{P_L},$ we can determine all the variables
in $ \{v_L, {\mydot{i}}_L, v_{P_L}, {\mydot{i}}_{P_L} \},$ uniquely.}

\vspace{0.2cm}

\textup{The constraints of $\mathcal{N}_{P'_CP'_LRuy}$
 are
$$i_{P'_C P'_LRuy} \in \V^i(\mathcal{G}_{P'_C P'_LRuy});\ \  v_{P'_C P'_LRuy} \in \V^v(\mathcal{G}_{P'_C P'_LRuy});$$ 
$$ v_R=Ri_R;\ \ i_{uv}\ free;\ \ v_{ui}\ free;\ \ v_{yi}=0;\ \ i_{yv}=0.$$
}

\textup{Let us denote the  solution space of these constraints restricted
to the index set $\{v_{P'_C}, v_{P'_L},v_{uv},v_{yv},
i_{P'_C}, i_{P'_L},i_{ui},i_{yi}\},$  
briefly by
$\V^{stat}.$
}
\\

\textup{The emulator $\V_{P\dP M_uM_y}$ on the index set 
$\{v_{PC},i_{PL},{\mydot{v}}_{PC},{\mydot{i}}_{PL},ui,uv,yi,yv\}$\\
($P\equiv \{v_{PC},i_{PL}\},\dP\equiv \{{\mydot{v}}_{PC},{\mydot{i}}_{PL}\},$
$M\equiv \{ui,uv,yi,yv\}$),
is  expressed implicitly by
$$\V_{P\dP M_uM_y}$$
$$=\V_{i_{P_C}{\mydot{v}}_{P_C}} \lrar [\V^v(\mathcal{G}_{P_CP'_C})
\oplus \V^i(\mathcal{G}_{P_CP'_C})]\lrar \V^{stat}\lrar [\V^v(\mathcal{G}_{P_LP'_L})\oplus \V^i(\mathcal{G}_{P_LP'_L})]
\lrar \V_{v_{P_L}{\mydot{i}}_{P_L}}.$$
}

\textup{ Next let us consider the linking between the primal dynamical system $\Vwdwmumy$
and its emulator 
$\Vpdpmumy,$
by examining the most natural candidates for  the spaces $\Vonewp$ and $\Vtwodwdp.$}\\
\textup{Here $W\equiv \{v_C,i_L\}, P \equiv \{v_{P_C}, i_{P_L}\},$ and $\dw\equiv \{ {\mydot{v}}_{C}, {\mydot{i}}_{L}\},\dP\equiv \{{\mydot{v}}_{P_C}, {\mydot{i}}_{P_L}\}.$} 

\vspace{0.2cm}

\textup{The relationship between $v_C, v_{P_C}$ is entirely topological, merely that  
$ (v_C, v_{P_C})\in \V^v(\mathcal{G}_{CP_C}).$ 
Similarly, the relationship between $i_L, i_{P_L}$ is  
merely that                
$ (i_L, i_{P_L})\in \V^i(\mathcal{G}_{LP_L}).$ 
Thus $$\Vonewp \equiv  (\V^v(\mathcal{G}_{CP_C})\oplus \V^i(\mathcal{G}_{LP_L})).$$
On the other hand the relationship between $\dw,\dP$ involves more constraints.
$$({\mydot{v}}_C, {\mydot{v}}_{P_C}) \in \V^v(\mathcal{G}_{CP_C});\ \ (i_C, i_{P_C}) \in \V^i(\mathcal{G}_{CP_C});\ \ i_C=C{\mydot{v}}_{C}.$$
Let us denote  the restriction of the solution space of these constraints to 
$\{{\mydot{v}}_C, {\mydot{v}}_{P_C}\}$ by $\V_{{\mydot{v}}_C {\mydot{v}}_{P_C}}.$
Thus
$$\V_{{\mydot{v}}_C {\mydot{v}}_{P_C}}
=   [(\V^v(\mathcal{G}_{CP_C})_{{\mydot{v}}_C {\mydot{v}}_{P_C}}
\oplus \V^i(\mathcal{G}_{CP_C}) )\bigcap \{i_C=C{\mydot{v}}_{C}\}]\circ \{{\mydot{v}}_C, {\mydot{v}}_{P_C}\} .$$
Next,
$$({\mydot{i}}_L, {\mydot{i}}_{P_L}) \in \V^i(\mathcal{G}_{LP_L});\ \ (v_L, v_{P_L}) \in \V^v(\mathcal{G}_{LP_L});\ \ v_L=L{\mydot{i}}_{L}.$$
Let us denote  the restriction of the solution space of these constraints to $\{{\mydot{i}}_L, {\mydot{i}}_{P_L}\}$ by $\V_{{\mydot{i}}_L {\mydot{i}}_{P_L}}.$
Thus
$$\V_{{\mydot{i}}_L {\mydot{i}}_{P_L}}
=   [(\V^i(\mathcal{G}_{LP_L})_{{\mydot{i}}_L, {\mydot{i}}_{P_L}}
\oplus \V^v(\mathcal{G}_{LP_L}) )\bigcap \{v_L=L{\mydot{i}}_{L}\}]\circ \{{\mydot{i}}_L, {\mydot{i}}_{P_L}\}.$$
We then have
$$\Vtwodwdp\equiv \V_{{\mydot{v}}_C {\mydot{v}}_{P_C}}\oplus \V_{{\mydot{i}}_L {\mydot{i}}_{P_L}}.$$
}

\textup{ From the construction of  $ \Vonewp ,\Vtwodwdp,$ it is clear that $$\Vpdpmumy= (\Vonewp\oplus \Vtwodwdp)\lrar \Vwdwmumy .$$ 
From the properties of the multiport decomposition,
it follows that  
$$\Vonewp \circ W= \Vwdwmumy \circ W,\ \  \Vonewp \times W= \Vwdwmumy \times  W.$$ 
From the construction of  $\Vtwodwdp,$ it follows that 
$$\Vtwodwdp \circ \dw= \Vwdwmumy \circ \dw,\ \   \Vtwodwdp \times \dw= \Vwdwmumy \times  \dw.$$ 
By Theorem \ref{thm:inverse}, it follows that 
 $$\Vwdwmumy= (\Vonewp\oplus \Vtwodwdp)\lrar \Vpdpmumy .$$
Note that, in general, $(\Vtwodwdp)_{WP}\subset \Vonewp .$ Indeed knowing
$\dP$ variables we can uniquely fix $\dw $ variables if the 
$C,L$ matrices are symmetric positive definite, 
while  knowing
$P$ variables we cannot uniquely fix $W $ variables 
if there are capacitor cutsets and inductor loops in the original
circuit.}
\begin{remark}
\label{rem:implicitemulatormultiport}
Let $\Vwdw \equiv \Vwdwmumy\lrar (\0_{M_u}\oplus \F_{M_y}),
\Vpdp \equiv \Vpdpmumy\lrar (\0_{M_u}\oplus \F_{M_y}).$
Let us examine the emulator pair $\{\Vwdw,
\Vpdp\}$ ${\mathcal{E}}-$ linked through $\{\Vonewp,\Vtwodwdp\}.$
When the port decomposition is minimal, in the three multiports,
the ports $P_C,P'_C,P_L,P'_L$ will not contain cutsets or circuits.
The generic situation for the static multiport, when $P'_C,P'_L$
are thus, is to have a unique solution 
for arbitrary
values of $i_{P'_C},v_{P'_L}.$
If in addition, $C,L$ matrices are symmetric positive definite, 
it is easy to show that  if ${\mydot{p}}=0,$ we must have $p=0.$
This means $\Vpdp\times P = \0_P,$  and equivalently, that $\Vpdp$
possesses no zero eigenvalue. However, in the present example,
it is clear that $\Vwdw \times W\ne \0_W.$ This happens because 
of the presence of capacitor cutsets and inductor loops. (Pick 
$v_{C3}=0$ and $v_{C1},v_{C2}$ consistently, pick $i_{L1},i_{L2},i_{L3}$ so that they constitute a circulating current within the inductor loop.
The rest of the circuit will not be affected.)
This is  an instance where the minimal annihilating polynomial $p(s)$ for
$\Vpdp,$ which has a constant term, does not annihilate $\Vwdw.$
But on the other hand, as it should follow from Theorem \ref{thm:emulatorpoly}, the polynomial $sp(s)$ annihilates both $\Vwdw$ and $\Vpdp.$
\end{remark}

\end{example}
\section{Implicit construction of adjoint and its emulator}
\label{sec:exadjoint}
\begin{example}

{\bf Building the adjoint implicitly for a network based dynamical system}\\

\textup{We have seen in Examples \ref{eg:strongemulator} and \ref{eg:emulator_multiport_decomposition},
 that in order to study the network $\mathcal{N}$
one can use an emulator constructed implicitly through a multiport decomposition. We first build the adjoint network $\mathcal{N}^a$
and then show that the 
 (multiport decomposition based) emulator of the adjoint 
can be built by constructing adjoints of the multiport decomposition of 
$\mathcal{N}.$ We have, however, simplified the description so that the adjoint has negative resistors and capacitors requiring less 
changes of sign when replacing variables of the primal by the corresponding ones in the adjoint.}

\textup{
The primal $\mathcal{N}$ (or rather its solution space $\V^{sol}(\mathcal{N}))$, shown in Figure \ref{fig:circuit03},
is a GDSA (generalized dynamical system with additional variables)
with the variables 
$$v_C,i_L,i_C,v_L,  {\mydot{v}}_C,{\mydot{i}}_L, v_R,i_R,ui,uv,yi,yv, v_{ui},i_{uv}, v_{yi}, i_{yv}.$$
Of these the dynamical variables are $v_C,i_L,{\mydot{v}}_C,{\mydot{i}}_L,$
the manifest variables  $M_u$ are $ui,uv,$ and the manifest variables $M_y$ are $yi,yv.$ The additional variables are
$i_C,v_L, v_R,i_R,v_{ui},i_{uv}, v_{yi}, i_{yv}.$ For greater readability, we will make this total set of variables, the index set.}

\vspace{0.2cm}

\textup{Let $\mathcal{G}$
be the graph of the network.} 
\textup{Let 
the device characteristic constraints of the network be
$$i_C=C{\mydot{v}}_C;\ \  v_L=L{\mydot{i}}_L;\ \ v_R=Ri_R;\ \ v_{yi}=0;\ \ i_{yv}=0;\ \  v_{ui}\ free,\ \  i_{uv}\ free,$$
where the matrices $C,L,R$ are symmetric positive definite.
The original GDS $\Vwdwmumy $ corresponding to $\mathcal{N}$
has the index set $\{v_C,i_L,{\mydot{v}}_C,{\mydot{i}}_L,ui,uv,yi,yv\}.$ 
We have, as we saw before,
 $$\Vwdwmumy $$
$$ = [(\V^{top}(\mathcal{N})\lrar (\{v_R=Ri_R\}\oplus\0_ {v_{yi}i_{yv}}\oplus \F_{v_{ui}i_{uv}}))\  \bigcap\    (\{i_C=C{\mydot{v}}_C \}\oplus \{v_L=L{\mydot{i}}_L\})]
\ \ \lrar \ \ \F_{i_Cv_L},$$
where we denote
$$\V^{v}(\mathcal{G})\oplus  \V^{i}(\mathcal{G})\oplus \V^{top}_{\mydot{v_C}}
\oplus \V^{top}_{\mydot{i_L}}\ \ \ \  \textup{by} \ \ \ \ \V^{top}(\mathcal{N}).$$
So
$$\Vwdwmumy^{\perp}$$
$$ = [(\V^{v}(\mathcal{G})\oplus  \V^{i}(\mathcal{G})\oplus  \V^{top}_{\mydot{v_C}}
\oplus \V^{top}_{\mydot{i_L}})^{\perp}\lrar (\{v_R=Ri_R\}^{\perp}\oplus\ \0_ {v_{yi}i_{yv}}^{\perp}\oplus \F_{v_{ui}i_{uv}}^{\perp}))\ \  +\ \  (\{i_C=C{\mydot{v}}_C \}^{\perp}\oplus \{v_L=L{\mydot{i}}_L\}^{\perp})]
 \lrar  \ \0_{i_Cv_L}.$$
$$= [(\V^{v}(\mathcal{G}))\oplus  \V^{i}(\mathcal{G}) \oplus (\V^{top}_{\mydot{v_C}}
\oplus \V^{top}_{\mydot{i_L}})^{\perp}\lrar \{-Rv_R=i_R\}\oplus\F_ {v_{yi}i_{yv}}\oplus \ \0_{v_{ui}i_{uv}}\ \ +\ \  (\{-Ci_C={\mydot{v}}_C \}\oplus \{-Lv_L={\mydot{i}}_L\})]\ \ \lrar \ \ \0_{i_Cv_L}.$$
}
\textup{
Next we make the following changes of variables:
replace $v$ by $-i"$, $i$ by $-v"$, ${\mydot{v}}_C $ by $q_C,$
${\mydot{i}}_L$ by $\psi_L.$ We also make use of $(\V^{v}(\mathcal{G}))^{\perp}=\V^{i}(\mathcal{G});\ \ (\V^{i}(\mathcal{G}))^{\perp} = \V^{v}(\mathcal{G}).$
So the above expression reduces to
$$ [((\V^{i}(\mathcal{G}))_{i"}\oplus  (\V^{v}(\mathcal{G}))_{v"} \oplus (\V^{top}_{\mydot{v_C}})^{\perp}_{q_C}
\oplus (\V^{top}_{\mydot{i_L}})^{\perp}_{\psi_L}\lrar \{-Ri"_R=v"_R\}\oplus\F_ {i"_{yi}v"_{yv}}\oplus \ \0_{i"_{ui}v"_{uv}}) \ + \   (\{Cv"_C=q_C \}\oplus \{Li"_L=\psi_L\})]$$  $$ \lrar \ \ \0_{v"_Ci"_L}.$$
We remind the reader that $\V^{top}_{\mydot{v_C}}= \V^v(\mathcal{G}\circ E_C),\ \  \V^{top}_{\mydot{i_L}}=\V^i(\mathcal{G}\times E_L)$
so that $$(\V^{top}_{\mydot{v_C}})^{\perp}_{q_C}= (\V^i(\mathcal{G}\circ E_C))_{q_C},\ \ (\V^{top}_{\mydot{i_L}})^{\perp}_{\psi_L}= (\V^v(\mathcal{G}\times E_L))_{\psi_L}.$$
We will denote $(\V^i(\mathcal{G}\circ E_C))_{q_C}\ \ \textup{by}\ \ \V^{top}_{{q_C}}\ \  \textup{and}\ \  (\V^v(\mathcal{G}\times E_L))_{\psi_L}
\ \ \textup{by} \ \ \V^{top}_{{\psi_L}}.$\\
\vspace{0.1cm}\\
Since $(\V_{ABC}+\V_{BC})\times AC= (\V_{ABC}\cap\V_{-BC})\circ AC,$ the above expression reduces to\\ 
$$ [((\V^{i}(\mathcal{G}))_{i"}\oplus  (\V^{v}(\mathcal{G}))_{v"}\oplus \V^{top}_{{q_C}}\oplus \V^{top}_{{\psi_L}}\lrar \{-Ri"_R=v"_R\}\oplus\F_ {i"_{yi}v"_{yv}}\oplus \ \0_{i"_{ui}v"_{uv}}) \ \bigcap \   (\{-Cv"_C=q_C \}\oplus \{-Li"_L=\psi_L\})]$$ $$\lrar  \F_{v"_Ci"_L}.$$
Next 
change the names of the edges $yi,yv,ui,uv$ respectively to
$u"v,u"i,y"v,y"i$ respectively so that the expression reduces to
$$ [((\V^{i}(\mathcal{G}))_{i"}\oplus  (\V^{v}(\mathcal{G}))_{v"}\oplus \V^{top}_{{q_C}}\oplus \V^{top}_{{\psi_L}}\lrar \{-Ri"_R=v"_R\}\oplus\F_ {i"_{u"v}v"_{u"i}}\oplus \ \0_{i"_{y"v}v"_{y"i}}) \ \bigcap \   (\{-Cv"_C=q_C \}\oplus \{-Li"_L=\psi_L\})]$$
$$ \lrar  \F_{v"_Ci"_L}.$$
This space is on the variables (which we also treat as index set)
$i"_C,v"_L, q_C,\psi _L, u"v,u"i,y"i,y"v.$
We now rename $i"_C$ as ${\mydot{q}}_C,$ $v"_L$ as ${\mydot{\psi}}_L$ and
rename  $-y"v,-y"i$ as $y'v,y'i$ (which can be done by reversing the arrows of $y"v,y"i),$  
and rename 
$$v"_C,i"_L, i"_R, v"_R,u"_v,u"_i, y"i,y"v,v"_{u"i},i"_{u"v},i"_{y"v},v"_{y"i},\ \  \textup{respectively\  as}$$ 
$$v'_C,i'_L, i'_R, v'_R,u'_v,u'_i,-y'i,-y'v,v'_{u'i},i'_{u'v},-i'_{y'v},-v'_{y'i}.$$}
\textup{Let  us call the graph obtained from $\mathcal{G}$ by reversing the arrows of the edges
corresponding to $y"v,y"i,$ the graph  $\mathcal{G}'.$
This results in the (final!) expression
$$ [((\V^{i}(\mathcal{G}'))_{i'}\oplus  (\V^{v}(\mathcal{G}'))_{v'}\oplus \V^{top}_{{q_C}}\oplus \V^{top}_{{\psi_L}}\lrar \{-Ri'_R=v'_R\}\oplus\F_ {i'_{u'v}v'_{u'i}}\oplus \ \0_{i'_{y'v}v'_{y'i}}) \ \bigcap\   (\{-Cv'_C=q_C \}\oplus \{-Li'_L=\psi_L\})]$$
$$ \lrar  \F_{v'_Ci'_L}.$$
The above constraints are those corresponding to the adjoint network shown in Figure \ref{fig:circuitA3}
. Note that the constraints on the initial
conditions on $q_C$ are according to the space $\V^{top}_{{q_C}}$ which is the current space  $\V^{i}(\mathcal{G}\circ E_C)$
and those on $\psi_L$ are according to the space $\V^{top}_{{\psi_L}}$ which is the voltage space  $\V^{v}(\mathcal{G}\times E_L),$ whereas in the original network the initial
conditions on $v_C$ are according to the voltage space  $\V^{v}(\mathcal{G}\circ E_C)$
and those on $i_L$ are according to the current space  $\V^{i}(\mathcal{G}\times E_L).$ \\
The resulting space corresponding to the above final expression is $$(\V_{W\dw M_uM_y}^{\perp})_{-\dwd W'-M'_yM'_u}=\V^a_{W'{\mydot{W'}} M'_uM'_y}.$$
Here 
$W' \equiv \{q_C,\psi_L\}, \dwd\equiv \{{\mydot{q}}_C,{\mydot{\psi}}_L\},
M'_u\equiv \{u'v,u'i\}, M'_y\equiv \{y'i,y'v\}.$
}

\vspace{1cm}

\textup{{\bf Building the emulator implicitly for the adjoint}}\\
\begin{figure}
 \includegraphics[width=7in]{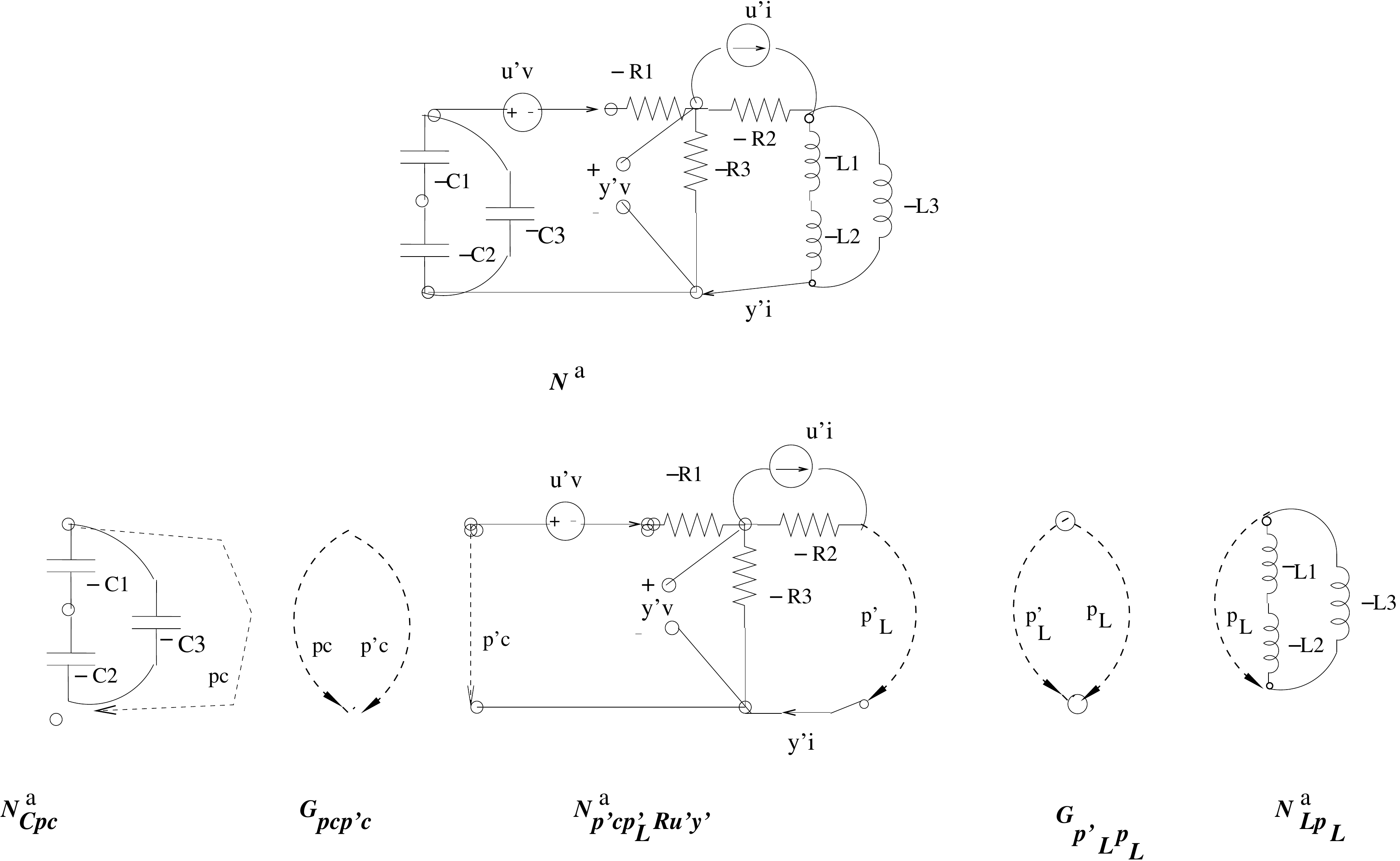}
 \caption{The adjoint network with port decomposition into capacitor inductor and static
multiports}
 \label{fig:circuitA3}
\end{figure}

\vspace{0.2cm}

\textup{To build the emulator $\V^a_{P'\dPd  M'_u M'_y} $  of $\V^a_{W'\dwd M'_u M'_y}, $
we could start with the latter and build its emulator, or equivalently,
build the emulator $\V_{P\dP M_u M_y}$ for the primal 
and then construct the adjoint of the primal emulator.
We will follow the latter procedure.
We have $$\V^a_{P'\dPd  M'_u M'_y}\equiv (\V_{P\dP M_u M_y}^{\perp})_{-\dPd P' -M'_y M'_u}.$$
The primal emulator is given to us in an implicit form in terms of
the multiport decomposition. We will build its adjoint also implicitly.
We begin with the expression for $\V_{P\dP M_u M_y}$ in terms of the multiport decomposition, rearranging terms and
expanding $\V^{top}(\mathcal{G}')$ for each occurrence of $\mathcal{G}':$ 
$$\V_{P\dP M_u M_y}$$
$$=(\V_{i_{P_C}{\mydot{v}}_{P_C}} \oplus  \V^{stat}
\oplus  \V_{v_{P_L}{\mydot{i}}_{P_L}})\lrar (\V^v(\mathcal{G}_{P_CP'_C})
\oplus \V^i(\mathcal{G}_{P_CP'_C})).
$$
We remind the reader that  
$$\V_{i_{P_C}{\mydot{v}}_{P_C}}=[(\V^v(\mathcal{G}_{CP_C}))_{{\mydot{v}}_{C},{\mydot{v}}_{P_C}}
\oplus \V^i(\mathcal{G}_{CP_C})]\lrar \{i_C=C{\mydot{v}}_C\},$$
$$\V_{v_{P_L}{\mydot{i}}_{P_L}}=[(\V^i(\mathcal{G}_{LP_L}))_{{\mydot{i}}_{L},{\mydot{i}}_{P_L}}
\oplus \V^v(\mathcal{G}_{LP_L})]\lrar \{v_L=L{\mydot{i}}_L\},$$
$$\V^{stat}\equiv (\V^v(\mathcal{G}_{P'_C, P'_L,R,ui,uv,yi,yv})\oplus \V^i(\mathcal{G}_{P'_C, P'_L,R,ui,uv,yi,yv}))
\lrar (\{v_R=Ri_R\}\oplus \ \0_{v_{yi}i_{yv}}
\oplus \F_{i_{uv}v_{ui}}),
$$
\hspace{3.5cm}index set for $\V^{stat}$ is $\{v_{P'_C},i_{P'_C},v_{P'_L},i_{P'_L},v_{uv},i_{ui}, v_{yv},i_{yi} \},$
$$P\equiv\{ v_{P_C}, i_{P_L}\}, \dP\equiv \{{\mydot{v}}_{P_C}, {\mydot{i}}_{P_L}\},
M_u\equiv \{uv,ui\}, M_y\equiv \{yv,yi\}.$$
$$(\V_{P\dP M_u M_y})^{\perp}$$
$$=(\V_{i_{P_C}{\mydot{v}}_{P_C}})^{\perp} \oplus  (\V^{stat})^{\perp}
\oplus  (\V_{v_{P_L}{\mydot{i}}_{P_L}})^{\perp}\lrar (\V^v(\mathcal{G}_{P_CP'_C})
\oplus \V^i(\mathcal{G}_{P_CP'_C}))^{\perp}$$
}

\textup{
We will now expand each of these terms and change the variables as follows:\\
$$v, i, {\mydot{v}}_C, {\mydot{v}}_{P_C}, {\mydot{i}}_{L}, {\mydot{i}}_{P_L},   \ \ \textup{respectively} \ \  \textup{to}\ \ -i', -v',q_C,q_{P_C}, \psi_{L}, \psi_{P_L},$$  
$$i'_C, i'_{P_C},  v'_{L}, v'_{P_L},  \ \ \textup{respectively} \ \  \textup{to}\ \ {\mydot{q}}_{C},{\mydot{q}}_{P_C},{\mydot{\psi}}_{L}, {\mydot{\psi}}_{P_L},$$
$$ \ \ \ \ \ \ ui,uv,yi,yv \ \ \ \textup{respectively \ \ \textup{to}\  }\   -y'v,-y'i,u'v,u'i.$$
We have,
$$\V_{{i_{P_C} {\mydot{v}}_{P_C}}}^{\perp}
=   ((\V^v(\mathcal{G}_{CP_C}))^{\perp}_{{\mydot{v}}_C, {\mydot{v}}_{P_C}}
\oplus (\V^i(\mathcal{G}_{CP_C}) )^{\perp}_{i_Ci_{P_C}})\lrar \{i_C=C{\mydot{v}}_{C}\}^{\perp}$$
So
$$(\V_{i_{P_C} {\mydot{v}}_{P_C}}^{\perp})_{-v'_{P_C}{{q}}_{P_C}}= ((\V^i(\mathcal{G}_{CP_C}))_{{{q}}_C {{q}}_{P_C}}
\oplus (\V^v(\mathcal{G}_{CP_C}) )_{-v'_C -v'_{P_C}}) \lrar \{-Ci_C ={\mydot{v}}_{C}\}_{-v'_Cq_C}  .$$
Therefore,
$$(\V_{i_{P_C} {\mydot{v}}_{P_C}}^{\perp})_{v'_{P_C}{{q}}_{P_C}}= ((\V^i(\mathcal{G}_{CP_C}))_{{{q}}_C {{q}}_{P_C}}
\oplus (\V^v(\mathcal{G}_{CP_C}) )_{v'_Cv'_{P_C}}) \lrar \{-Ci_C ={\mydot{v}}_{C}\}_{v'_Cq_C}  .$$
So, denoting $(\V_{i_{P_C} {\mydot{v}}_{P_C}}^{\perp})_{v'_{P_C}{{q}}_{P_C}}$ by $\V_{v'_{P_C}{{q}}_{P_C}},$ 
$$\V_{v'_{P_C}{{q}}_{P_C}}= ((\V^i(\mathcal{G}_{CP_C}))_{{{q}}_C {{q}}_{P_C}}
\oplus (\V^v(\mathcal{G}_{CP_C}) )_{v'_Cv'_{P_C}} \lrar \{-Cv'_C ={{q}}_{C}\}  .$$
Proceeding analogously and denoting 
$$(\V_{v_{P_L} {\mydot{i}}_{P_L}}^{\perp})_{i'_{P_L}{{\psi}}_{P_L}} \ \ \textup{by}\ \  \V_{i'_{P_L}{{\psi}}_{P_L}},$$ 
we get
$$\V_{i'_{P_L}{{\psi}}_{P_L}}= ((\V^v(\mathcal{G}_{LP_L}))_{{{\psi}}_L {{\psi}}_{P_L}}
\oplus (\V^i(\mathcal{G}_{LP_L}) )_{i'_Li'_{P_L}} \lrar \{-Li'_L ={{\psi}}_{L}\}  .$$
Similarly in the case of $\V^{stat},$ denoting\\ 
$$((\V^{stat})^{\perp})_{i"_{P'_C},v"_{P'_C},i"_{P'_L},v"_{P'_L},i"_{uv},v"_{ui}, i"_{yv},v"_{yi}}
\ \ \textup{by}\ \  \V_{i"_{P'_C},v"_{P'_C},i"_{P'_L},v"_{P'_L},i"_{uv},v"_{ui}, i"_{yv},v"_{yi}}$$\\
(we have replaced $v_{P'_C},i_{P'_C},v_{P'_L},i_{P'_L},v_{uv},i_{ui}, v_{yv},i_{yi} $\\ 
respectively by $-i"_{P'_C},-v"_{P'_C},-i"_{P'_L},-v"_{P'_L},-i"_{uv},-v"_{ui}, -i"_{yv},-v"_{yi}$),
and the graph $$\mathcal{G}_{P'_C, P'_L,R,ui,uv,yi,yv} \ \ \textup{by}\ \  \mathcal{G}^{stat}$$
we get 
$$\V_{i"_{P'_C},v"_{P'_C},i"_{P'_L},v"_{P'_L},i"_{uv},v"_{ui}, i"_{yv},v"_{yi}}$$
$$=(\V^i(\mathcal{G}^{stat}))_{i"_{P'_C}, i"_{P'_L},i"_R,i"_{uv},i"_{ui},i"_{yv},i"_{yi}}\oplus (\V^v(\mathcal{G}^{stat}))_{v"_{P'_C}, v"_{P'_L},v"_R,v"_{uv},v"_{ui},v"_{yv},v"_{yi}}$$
$$\lrar (\{-Ri"_R=v"_R\}\oplus \ \0_{v"_{uv}i"_{ui}}
\oplus \F_{i"_{yi}v"_{yv}})).
$$
This space is on the variables $ i"_{P'_C},v"_{P'_C},i"_{P'_L},v"_{P'_L} , i"_{uv},v"_{ui},v"_{yi},i"_{yv}.$\\
Next we replace $i"$ variables by $i'$ except in the case of $i"_{uv},i"_{ui},$
where we replace them respectively by $-i'_{uv},-i'_{ui}.$\\
Similarly we replace $v"$ variables by $v'$ except in the case of $v"_{uv},v"_{ui},$
where we replace them respectively by $-v'_{uv},-v'_{ui}.$ 
\\
We change the names of the branches $uv,ui,yv,yi$ respectively to $y'i,y'v,u'i,u'v,$
and reverse the arrows of the branches $y'i,y'v.$ This causes the required sign change - $\ $
$uv$ becoming $-y'i$ and $ui$ becoming $-y'v.$
.}\\

\textup{
Thus the graph $\mathcal{G}^{stat}_{P'_C, P'_L,R,ui,uv,yi,yv}$
is modified to $\mathcal{G}^{'stat}$}
\textup{on the set of branches $P'_C, P'_L,R,y'i,y'v,u'i,u'v,$  }\\

\textup{respectively, with the arrows
$yi,yv$ reversed in $u'v,u'i.$}\\

\textup{The resulting space is on the variables $ i'_{P'_C},v'_{P'_C},i'_{P'_L},v'_{P'_L} , i'_{y'i},v'_{y'v},v'_{u'v},i'_{u'i}$ and will be denoted by $\V^{astat}.$
We therefore have 
$$\V^{astat}$$
$$=(\V^i(\mathcal{G}^{'stat}))_{i'_{P'_C}, i'_{P'_L},i'_R,i'_{u'v},i'_{u'i},i'_{y'v},i'_{y'i}}\oplus (\V^v(\mathcal{G}^{'stat}))_{v'_{P'_C}, v'_{P'_L},v'_R,v'_{u'v},v'_{u'i},v'_{y'v},v'_{y'i}}$$
$$\lrar (\{-Ri'_R=v'_R\}\oplus \ \0_{v'_{y'i}i'_{y'v}}
\oplus \F_{i'_{u'v}v'_{u'i}}).
$$
Thus $$\V^a_{P'\dPd  M'_u M'_y}= (\V_{P\dP M_u M_y})^{\perp}_{-\dPd P'-M'_y M'_u}$$ 
$$=\V_{v'_{P_C}{{q}}_{P_C}}\lrar \V^{astat} \lrar \V_{i'_{P_L}{{\psi}}_{P_L}}.$$
Here $$P'\equiv \{{{q}}_{P_C},{{\psi}}_{P_L}\}, \dPd \equiv  \{{\mydot{q}}_{P_C},{\mydot{\psi}}_{P_L}\},M'_u\equiv \{v'_{u'v},i'_{u'i}\}, M'_y \equiv \{i'_{y'i},v'_{y'v}\}.$$}
\textup{The multiport decomposition corresponding to the emulator of the adjoint which is also the adjoint of the emulator of the original
dynamical system is shown in Figure \ref{fig:circuitA3}.}

\textup{Next let us consider the linking between the adjoint dynamical system $\V^a_{W'\dw' M_u'M_y'}$
and its emulator
$\V^a_{P'\dPd M'_uM'_y}.$
We have by Theorem \ref{thm:emuadjoint},
$$\V^1_{W'P'}=((\Vtwodwdp)^{\perp})_{-W'P'}$$
$$ =((\V_{{\mydot{v}}_C {\mydot{v}}_{P_C}}\oplus \V_{{\mydot{i}}_L {\mydot{i}}_{P_L}})^{\perp})_{-q_Cq_{P_C} -\psi_L \Psi_{P_L}}.$$
$$ =(\V_{{\mydot{v}}_C {\mydot{v}}_{P_C}}^{\perp})_{-q_Cq_{P_C}}\oplus (\V_{{\mydot{i}}_L {\mydot{i}}_{P_L}}^{\perp})_{-\psi_L \Psi_{P_L}}.$$
When $C,L$ are symmetric positive definite and the port branches do not contain loops or cutsets, as mentioned before, we have
unique ${\mydot{v}}_{C}$ for arbitrary choice of ${\mydot{v}}_{P_C}$ and unique ${\mydot{i}}_{L}$ for arbitrary choice of ${\mydot{i}}_{P_L}.$ 
Thus we have ${\mydot{v}}_{C}=K_C{\mydot{v}}_{P_C}$ and
${\mydot{i}}_{L}=K_L{\mydot{i}}_{P_L},$
 for some $K_C,K_L.$
In this case, in the adjoint, we will have 
${{q}}_{P_C}=K_C^T{{q}}_{C}$ and
${{\psi}}_{P_L}=K_L^T{{\psi}}_{L}.$\\
Again, by Theorem \ref{thm:emuadjoint},
$$\V^{2}_{\dwd\dPd}=({\Vonewp}^{\perp})_{-\dwd\dPd }$$
$$= (\V^v(\mathcal{G}_{CP_C})\oplus \V^i(\mathcal{G}_{LP_L}))^{\perp}_{-\dwd \dPd }
=(\V^i(\mathcal{G}_{CP_C}))_{-\mydot{q}_C\mydot{q}_{P_C}}
\oplus (\V^v(\mathcal{G}_{LP_L}))_{-\mydot{\psi}_L\mydot{\psi}_{P_L}}.$$
}

\end{example}

%
%
%
%
%
%











\bibliographystyle{elsarticle1-num}
\bibliography{references}

\end{document}